\documentclass[12pt]{article}

\usepackage{graphicx,amsmath,amsthm,amssymb,mathrsfs}

\usepackage[center,it]{titlesec}

\DeclareMathOperator{\divergence}{div}
\DeclareMathOperator{\diver}{div}
\DeclareMathOperator{\curl}{curl}
\DeclareMathOperator{\Riem}{Riem}
\DeclareMathOperator{\Ric}{Ric}
\DeclareMathOperator{\tr}{tr}

\newcommand{\ud}{\mathrm{d}}

\newcommand{\dm}[1]{\ud \mu_{#1}}
\newcommand{\epsu}[2]{\epsilon_{#1}^{\phantom{#1} #2}}

\newcommand{\Cb}{\underline{C}}

\newcommand{\ub}{\underline{u}}
\newcommand{\vb}{\underline{v}}
\newcommand{\rs}{{r_\ast}}
\newcommand{\vth}{\vartheta}

\newcommand{\us}{{u_\ast}}
\newcommand{\vs}{{v_\ast}}
\newcommand{\uu}{\underline{u}}
\newcommand{\vv}{\underline{v}}

\newcommand{\uk}{u_{\mathcal{K}}}
\newcommand{\vk}{v_{\mathcal{K}}}
\newcommand{\Omegau}{\Omega_{(u)}}

\newcommand{\Lh}{\hat{{L}}}
\newcommand{\Lp}{L^\prime}
\newcommand{\Lb}{\underline{L}}
\newcommand{\Lbh}{\hat{\underline{L}}}
\newcommand{\Lbp}{\Lb^\prime}

\newcommand{\Lies}{\mathcal{L}\!\!\!/}
\newcommand{\MLie}[1]{\tilde{\mathcal{L}}_{#1}}
\newcommand{\MJ}[2]{{}^{(#1)}J(#2)}
\newcommand{\Dh}{\hat{D}}
\newcommand{\Db}{\underline{D}}
\newcommand{\Dbh}{\hat{\Db}}
\newcommand{\pid}[1]{{}^{(#1)}\pi}
\newcommand{\pih}[1]{{}^{(#1)}\hat{\pi}}

\newcommand{\alphab}{\underline{\alpha}}
\newcommand{\betab}{\underline{\beta}}
\newcommand{\omegab}{\underline{\omega}}
\newcommand{\omegabh}{\hat{\omegab}}
\newcommand{\omegah}{\hat{\omega}}
\newcommand{\etab}{\underline{\eta}}
\newcommand{\chib}{\underline{\chi}}
\newcommand{\chibh}{\hat{\chib}}

\newcommand{\chih}{\hat{\chi}}

\newcommand{\mb}{\underline{m}}
\newcommand{\nb}{\underline{n}}
\newcommand{\ih}{\hat{i}}

\newcommand{\alphat}{\tilde{\alpha}}
\newcommand{\betat}{\tilde{\beta}}
\newcommand{\rhot}{\tilde{\rho}}
\newcommand{\sigmat}{\tilde{\sigma}}
\newcommand{\betabt}{\tilde{\betab}}
\newcommand{\alphabt}{\tilde{\alphab}}
\newcommand{\Xib}{\underline{\Xi}}
\newcommand{\Xibt}{\tilde{\underline{\Xi}}}
\newcommand{\Xit}{\tilde{\Xi}}
\newcommand{\Lambdab}{\underline{\Lambda}}
\newcommand{\Lambdabt}{\tilde{\underline{\Lambda}}}
\newcommand{\Lambdat}{\tilde{\Lambda}}
\newcommand{\Kb}{\underline{K}}
\newcommand{\Kbt}{\tilde{\underline{K}}}
\newcommand{\Kt}{\tilde{K}}
\newcommand{\Thetab}{\underline{\Theta}}
\newcommand{\Thetabt}{\tilde{\underline{\Theta}}}
\newcommand{\Thetat}{\tilde{\Theta}}
\newcommand{\Ib}{\underline{I}}
\newcommand{\Ibt}{\tilde{\underline{I}}}
\newcommand{\It}{\tilde{I}}
\newcommand{\Ps}{{P\!\!\!\!/\,}}
\newcommand{\ps}{p\!\!\!/\,}

\newcommand{\gs}{g\!\!\!/}

\newcommand{\gammao}{\stackrel{\circ}{\gamma}}
\newcommand{\gQ}{\stackrel{{\scriptscriptstyle Q}}{g}}
\newcommand{\epsilons}{\epsilon\!\!/}
\newcommand{\nablas}{\nabla\!\!\!\!/}

\newcommand{\divs}{\diver\!\!\!\!\!/\:}
\newcommand{\curls}{\curl\!\!\!\!\!/\:}
\newcommand{\ds}{\mathrm{d}\!\!\!/}
\newcommand{\gb}[1]{\overline{g}_{#1}}
\newcommand{\gbb}{\overline{g}}
\newcommand{\nablab}{\overline{\nabla}}
\newcommand{\divb}{\overline{\divergence}}
\newcommand{\curlb}{\overline{\curl}}
\newcommand{\laplaceb}{\overline{\triangle}}

\newcommand{\Rb}{\overline{R}}

\newcommand{\ld}{{}^\ast}
\newcommand{\otimesh}{\hat{\otimes}}

\newcommand{\lambdaq}{\frac{\Lambda}{3}}
\newcommand{\lambdasq}{\sqrt{\frac{\Lambda}{3}}}
\newcommand{\lambdaf}{\frac{\Lambda r^2}{3}+\frac{2m}{r}-1}

\newcommand{\rc}{{r_{\mathcal{C}}}}
\newcommand{\rh}{{r_{\mathcal{H}}}}
\newcommand{\rcv}{\lvert \overline{r_{\mathcal{C}}}\rvert}
\newcommand{\alphah}{\alpha_{\mathcal{H}}}
\newcommand{\alphacv}{\overline{\alpha_{\mathcal{C}}}}
\newcommand{\kc}{\kappa_\mathcal{C}}

\newcommand{\Ceq}[1]{(#1) in \cite{ch:blue}}
\newcommand{\CCh}[1]{Chapter~#1 in \cite{ch:blue}}
\newcommand{\CProp}[1]{Proposition~#1 in \cite{ch:blue}}
\newcommand{\CKeq}[1]{(#1) in \cite{ch:kl}}
\newcommand{\CKCh}[1]{Chapter~#1 in \cite{ch:kl}}
\newcommand{\CKProp}[1]{Proposition~#1 in \cite{ch:kl}}
\newcommand{\CKLemma}[1]{Lemma~#1 in \cite{ch:kl}}
\newcommand{\CLemma}[1]{Lemma~#1 in \cite{ch:blue}}

\newcommand{\VS}[1]{Section~#1 of \cite{glw}}
\newcommand{\CKS}[1]{Section~#1 in \cite{ch:kl}}

\newcommand{\lVerts}{\lVert\!\!\!\!-}
\newcommand{\rVerts}{\rVert\!\!\!\!-}
\newcommand{\nLp}[1]{\lVerts #1 \rVerts_{\mathrm{L}^{p}(S)}}
\newcommand{\nLpq}[2]{\lVerts #1 \rVerts_{\mathrm{L}^{#2}(S)}}
\newcommand{\nLq}[1]{\nLpq{#1}{4}}

\theoremstyle{plain}
\newtheorem*{theorem}{Theorem}
\newtheorem{proposition}{Proposition}[section]
\newtheorem{lemma}[proposition]{Lemma}
\newtheorem{corollary}[proposition]{Corollary}

\theoremstyle{definition}

\theoremstyle{remark}
\newtheorem{remark}[proposition]{Remark}

\numberwithin{equation}{section}
\begin{document}

\title{\textbf{\Large Decay of the Weyl curvature\\ in expanding black hole cosmologies}}
\author{Volker Schlue\bigskip\\{\small University of Melbourne}\\{\small volker.schlue@unimelb.edu.au}}
%\address{Universit\'e Pierre et Marie Curie}
%\email{}
\date{{\small \today}}

%\begin{abstract}
%\end{abstract}

\maketitle

\begin{abstract}
  This paper is motivated by the non-linear stability problem for the expanding region of Kerr de Sitter cosmologies in the context of Einstein's equations with positive cosmological constant.
  We show that under dynamically realistic  assumptions the conformal Weyl curvature of the spacetime decays towards future null infinity.
  More precisely we establish decay estimates for Weyl fields which are (i)~uniform (with respect to a global time function) (ii)~optimal (with respect to the rate) and (iii)~consistent with a global existence proof (in terms of regularity).
  The proof relies on a geometric positivity property of compatible currents which is a manifestation of the global redshift effect capturing the expansion of the spacetime.
\end{abstract}

\tableofcontents

\section{Introduction}
\label{sec:intro}

In this paper we are interested in solutions to Einstein's field equations with a positive cosmological constant $\Lambda>0$,
\begin{equation}\label{eq:eve:lambda:intro}
  \Ric(g)=\Lambda g\,.
\end{equation}

These equations --- in fact a more general form including sources of matter ---  were proposed by Einstein to model the universe in the large \cite{einstein:kosmos}, $(\mathcal{M},g)$ being an unknown $3+1$ dimensional Lorentzian manifold which represents the geometry of space-time. \footnote{The equations \eqref{eq:eve:lambda:intro} in its more general form including sources of matter can be thought of as a general relativistic version of the classical analogue $\triangle \psi-\Lambda \psi=4\pi \rho$, namely a modification of Newton's law which Einstein considered to model homogeneous mass densities $\rho$ in space  \cite{einstein:kosmos}.} 
The simplest solutions to \eqref{eq:eve:lambda:intro} have non-trivial topology, and are not asymptotically flat: The Einstein universe (which contains a homogeneous fluid) is topologically a cylinder $\mathbb{S}^3\times \mathbb{R}$, and thus represents a \emph{closed} universe. While Einstein's solution is \emph{static} (in time), de Sitter found a solution to the vacuum equations \eqref{eq:eve:lambda:intro} shortly after the cosmological constant was introduced \cite{desitter}, which is \emph{expanding}:\footnote{Neither was this solution discovered in this form, nor was it immediately understood that this ``universe'' is expanding. This was realized only later by Lema\^\i tre \cite{lemaitre}. For a historical discussion of the confusions and controversies surrounding its discovery see e.g.~\cite{nussbaumer}.} The de Sitter space-time can be embedded as a (time-like) hyperboloid $H$ in $5$-dimensional Minkowski space $(\mathbb{R}^{4+1},m)$, with metric $h$ simply induced by the ambient metric $m$. Its geometric properties are relevant to this paper,\footnote{In particular our choice of gauge is informed by the behavior of  solutions to the eikonal equation on de Sitter spacetime; see \cite{schlue:optical}.} and are discussed in detail in \cite{schlue:optical}.

A model of a \emph{black hole} in an expanding universe is provided by the Schwarzschild de Sitter geometry \cite{kottler:sds,weyl:sds} discussed in Section~\ref{sec:schwarzschild:de:sitter}. It is a solution to \eqref{eq:eve:lambda:intro} with less symmetries than the de Sitter solution  but still spherically symmetric, which means that the metric takes the form:
\begin{equation}\label{eq:g:intro:spherical}
  g=\gQ+r^2\gammao
\end{equation}
where $\gQ$ is a Lorentzian metric on a $1+1$-dimensional manifold $Q$, $r:Q\to(0,\infty), q\mapsto r(q)$ is the radius of a sphere $q\in Q$, and $\gammao$ the standard metric on $\mathbb{S}^2$. Its causal geometry is best understood if we depict the level sets of $r$ in $Q$ while keeping the null lines of $\gQ$ at $45^\circ$, namely in the form of the Penrose diagram of Fig.~\ref{fig:penrose:schwarzschild:de:sitter}.

\begin{figure}[tb]
  \centering
  \includegraphics{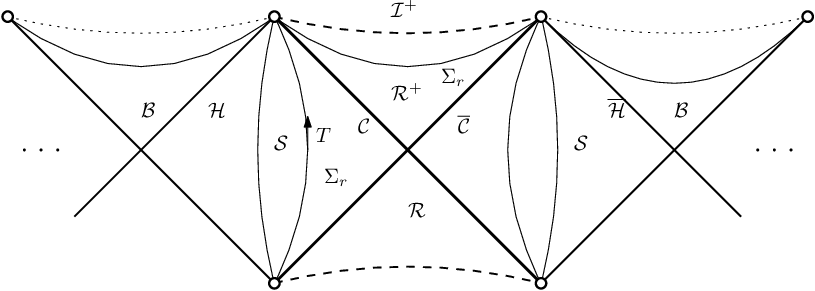}
  \caption{Penrose diagram of Schwarzschild de Sitter geometry.}
  \label{fig:penrose:schwarzschild:de:sitter}
\end{figure}

On Schwarzschild de Sitter spacetime we distinguish between the black hole region $\mathcal{B}$, the stationary black hole exterior $\mathcal{S}$, and the cosmological region $\mathcal{R}$. The stationary region $\mathcal{S}$ has a time-like Killing vectorfield $T$, and is bounded by an event horizon $\mathcal{H}$ towards the interior at $r=\rh$, and a \emph{cosmological horizon} $\mathcal{C}$ towards the exterior at $r=\rc$. Beyond $\mathcal{C}$ lies the cosmological region, whose future component $\mathcal{R}^+$  we depict separately in Fig.~\ref{fig:penrose:expanding}. $\mathcal{R}^+$ is bounded to the past by the null hypersurfaces $\mathcal{C}^+\cup\bar{\mathcal{C}}^+$, and foliated by the level sets $\Sigma_r$ of $r$:
\begin{equation}
  \mathcal{R}=\bigcup_{r> \rc}\Sigma_r
\end{equation}
Each leaf $\Sigma_r$ is topologically a cylinder $\mathbb{S}^2\times\mathbb{R}$, and a \emph{spacelike} hypersurface for $r> \rc$.
Since $r$ is increasing along any future-directed causal curve in $\mathcal{R}^+$ we also call this region \emph{expanding}.
It is future geodesically complete, yet has the property that any two observers are eventually causally disconnected.\footnote{Here we simply mean that given any two time-like geodesics $\Gamma_1$, $\Gamma_2$ in $\mathcal{R}^+$ (with different ``end points'' on $\mathcal{I}^+$) parametrized by arc length, we can find proper times $s_1$, and $s_2$ such that the causal futures of $\Gamma_1(s_1)$ and $\Gamma_2(s_2)$ have \emph{empty} intersection.} This has the consequence that the ``ideal boundary at infinity'' $\mathcal{I}^+$ is a \emph{spacelike} surface.\footnote{Here $\mathcal{I}^+$ can be identified with the endpoints all outgoing null geodesics from $\mathcal{S}$ along which $r\to\infty$, and is thus still referred to as \emph{null infinity}. It being spacelike (as a surface in a suitably extended space-time) marks an important difference to the asymptotically flat setting, and is prototypical for $\Lambda>0$.} ($\mathcal{I}^+$ is not part of the spacetime, but it is intrinsically a cylinder $\mathbb{R}\times\mathbb{S}^2$ and can be thought of as attached to the spacetime in the topology of the Penrose diagram.) Finally $\mathcal{B}$ is referred to as the black hole region, because it lies in the complement of the past of $\mathcal{I}^+$.

\begin{figure}[bt]
  \centering
  \includegraphics{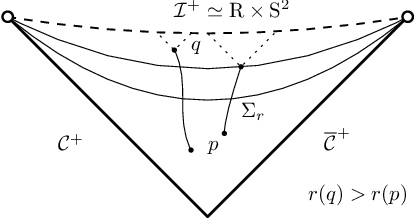}
  \caption{Expanding region of a Schwarzschild de Sitter cosmology.}
  \label{fig:penrose:expanding}
\end{figure}

The maximal extension of the Schwarzschild de Sitter spacetime consists of an infinite chain of black hole regions $\mathcal{B}$, separated by exteriors $\mathcal{S}$ to the future and past of which lie the cosmological regions $\mathcal{R}$. We shall restrict attention to a given cosmological region, and its adjacent black hole exteriors, up to the event horizons; in particular the interior of the black hole is not considered here.\footnote{It is expected that the analysis in \cite{dl:C0} (adapted to the $\Lambda>0$ setting) can be combined with the results of \cite{hi:vasy:stability} to prove the $C^0$-stability of the interior up to the Cauchy horizon, and to \emph{disprove unconditionally} the $C^0$-formulation of strong cosmic censorship. However it is an open question to understand in which sense strong cosmic censorship may hold in this setting; see Section 1.6~in \cite{dl:C0}.}

\begin{figure}[hbt]
  \centering
  \includegraphics[scale=1.2]{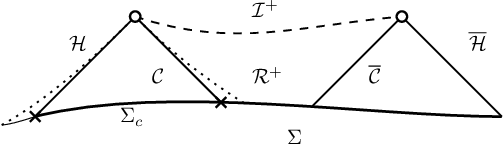}
  \caption{Cauchy problem in the domain of dependence of $\Sigma$.}
  \label{fig:Cauchy}
\end{figure}

While all classical solutions to \eqref{eq:eve:lambda:intro} referred to here were found explicitly, 
a natural question to ask from the evolutionary point of view is the following:
\begin{quote} \label{quote:stability:sds}
  \itshape Is the picture of Fig.~\ref{fig:Cauchy} dynamically stable?\smallskip\\
In other words, does a perturbation of Schwarzschild de Sitter data on a Cauchy hypersurface $\Sigma$ give rise to a maximal development $\mathcal{D}$ with similar features? In particular does $\mathcal{D}$ contain a \emph{future geodesically complete} region $\mathcal{R}$ with \emph{spacelike} boundary $\mathcal{I}^+$ at infinity, relative to which $\mathcal{D}$ contains a black hole region $\mathcal{B}$. Moreover, is the black hole exterior $\mathcal{S}=I^-(\mathcal{C})\cap I^-(\mathcal{H})$ --- where $\mathcal{C}$ and $\mathcal{H}$ are defined to be the future boundary of the past of $\mathcal{B}$ and~$\mathcal{I^+}$, respectively --- asymptotically \emph{stationary}?
\end{quote}

% \begin{quote}\itshape  Let $\Sigma$ be a spacelike hypersurface  as depicted in Fig.~\ref{fig:Cauchy}. Let $(\gb{},k)$ be initial data on $\Sigma$ for the Einstein equations \eqref{eq:eve:lambda:intro}, which is close (in a suitable topology) to the data induced by the Schwarzschild de Sitter geometry. Then does the maximal domain of development $\mathcal{D}$ ``have the same Penrose diagram as Schwarzschild de Sitter''?

% \upshape
% More precisely, does $\mathcal{D}$ decompose disjointly in $\mathcal{D}=\overline{\mathcal{B}\cup\mathcal{S}\cup\mathcal{R}\cup\bar{\mathcal{S}}\cup\bar{\mathcal{B}}}$  such that with $\mathcal{C}=\partial^+ J^-(\mathcal{\mathcal{B}})$, $\bar{\mathcal{C}}=\partial^+ J^-(\bar{\mathcal{B}})$, 
% \begin{enumerate}
% \item $\mathcal{R}$ is \emph{future geodesically complete}, $\mathcal{R}$ is bounded in the past by $\mathcal{C}\cup\bar{\mathcal{C}}$, $\partial^-\mathcal{R}\supseteq \mathcal{C}\cup\bar{\mathcal{C}}$, and future null infinity $\mathcal{I}^+$ can attached as $\partial^+\mathcal{R}$ at infinity as a \emph{spacelike} hypersurface in a suitably regular manner?
% \item $\mathcal{D}$ contains black hole regions, namely $\mathcal{B}\cup\bar{\mathcal{B}}=\mathcal{D}\setminus J^-(\mathcal{I}^+)\neq\varnothing$ and $\mathcal{H}\cup\bar{\mathcal{H}}=\partial^-\mathcal{B}\cup\partial^-\bar{\mathcal{B}}=\partial^+ J^-(\mathcal{I}^+)$?
% \item $\mathcal{S}$, and  $\bar{\mathcal{S}}$, are asymptotically \emph{stationary}?
% \end{enumerate}
% \end{quote}

In view of the domain of dependence property of solutions to \eqref{eq:eve:lambda:intro} the \emph{stability of the black hole exterior} $\mathcal{S}$ can be treated \emph{independently} of the cosmological region $\mathcal{R}$ (and the black hole interior $\mathcal{B}$). Indeed, since $\mathcal{S}$ is contained in the domain of dependence of a compact subset $\Sigma_c\subset\Sigma$, the behavior of the solution in $\mathcal{S}$ is not influenced by data in the complement of $\Sigma_c$.
In a remarkable series of papers \cite{hi:vasy:stability,hi:quasi,hintz:vasy:16:global,hi:vasy:trapping,vasy:13} Hintz and Vasy have recently proven that solutions to the Cauchy problem for \eqref{eq:eve:lambda:intro} arising from a perturbation of Schwarzschild de Sitter data on $\Sigma_c$  \emph{converge exponentially fast to a member of the Kerr de Sitter family} on $\mathcal{S}\cup\mathcal{H}\cup\mathcal{C}$, (and in particular become stationary). \footnote{Here $\mathcal{S}$ is of course not defined with the help of $\mathcal{I}^+$, but instead found dynamically as the stationary region of the Kerr de Sitter geometry that the solution asymptotes to, and the solution is controlled in a slightly larger domain, as indicated in Fig.~\ref{fig:Cauchy}.}

The Kerr de Sitter geometry --- given by an explicit 2-parameter family of \emph{axi-symmetric} solutions to \eqref{eq:eve:lambda:intro}, containing Schwarzschild de Sitter as a subfamily --- plays a central role for the understanding of solutions in $\mathcal{S}$.  Indeed, the result of Hintz and Vasy shows that they parametrize all possible final states for the evolution of perturbations of Schwarzschild de Sitter data, \emph{in the domain bounded by the event horizon $\mathcal{H}$ and the cosmological horizon $\mathcal{C}$.} We will see that for the evolution beyond the cosmological horizon this explicit family of solutions does \emph{not} play an equally prominent role.

The problem that motivates this paper is then the following:
\begin{quote}\label{quote:stability:sds:goursat}
  \itshape Consider the characteristic initial value problem (or \emph{Goursat} problem) for \eqref{eq:eve:lambda:intro} with data on the (future geodesically complete) cosmological horizons $\mathcal{C}\cup\bar{\mathcal{C}}$, cf.~Fig.~\ref{fig:Goursat}. Suppose the characteristic data converges exponentially fast to the geometry induced by a Kerr de Sitter horizon, then is the maximal development \emph{future geodesically complete}, and can future null infinity $\mathcal{I}^+$ be attached at infinity as a \emph{spacelike} surface in a suitably regular manner?
\end{quote}

\begin{figure}
  \centering
  \includegraphics[scale=0.4]{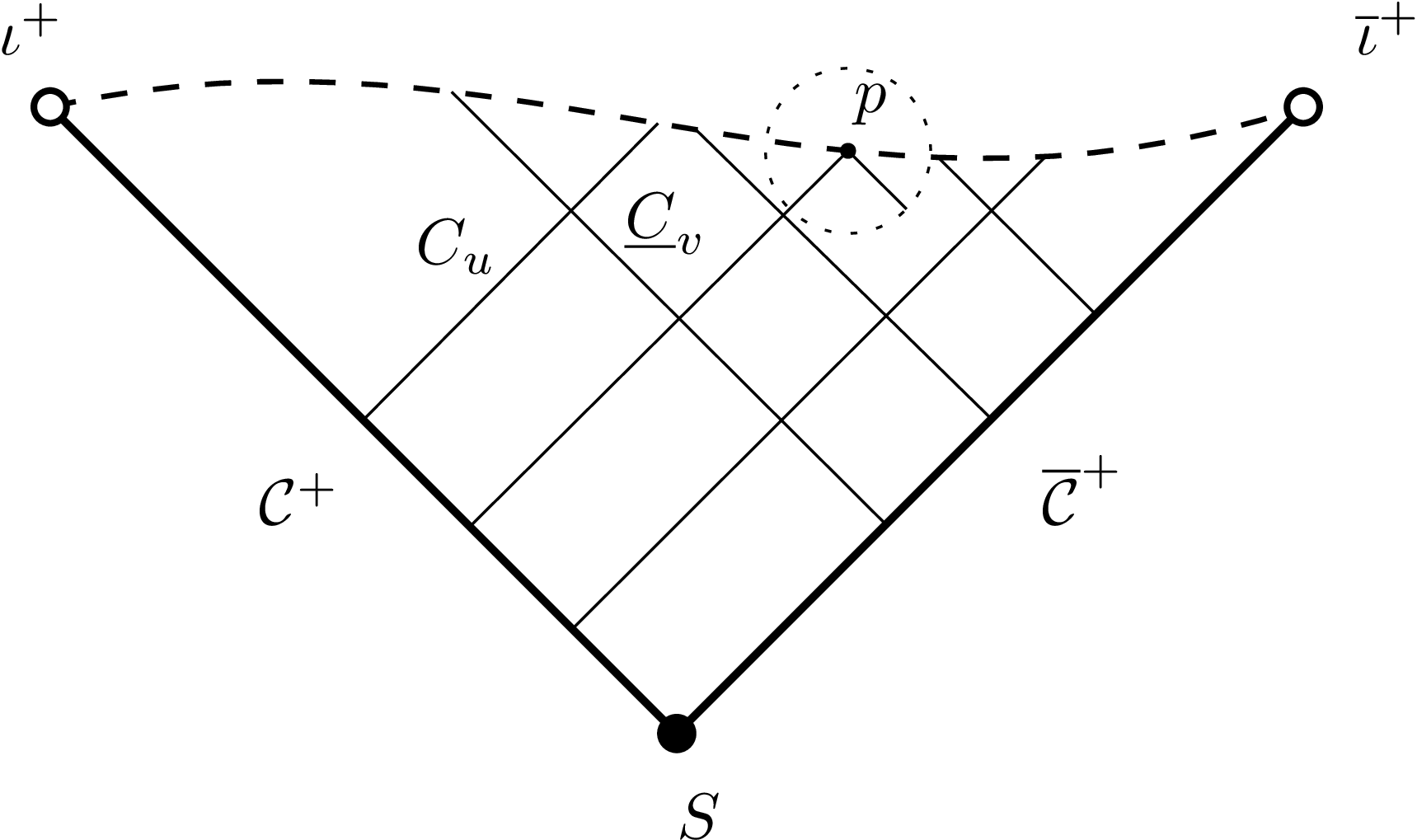}
  \caption{Goursat problem from the cosmological horizons.}
  \label{fig:Goursat}
\end{figure}

% \begin{quote}
%   \itshape Consider \emph{characteristic} initial data for \eqref{eq:eve:lambda:intro} on two future complete null hypersurfaces $\mathcal{C}$, $\bar{\mathcal{C}}$ --- intersecting in a sphere $S$ in the past --- which converge to Kerr de Sitter data exponentially fast. Then is the maximal development future geodesically complete, and bounded \emph{at infinity} by a \emph{spacelike} hypersurface $\mathcal{I}^+$ which can be attached in suitable regular manner?
%\end{quote}

In \emph{spherical symmetry} the analogue of this problem has been addressed in \cite{costa:nohair} in the context of the Einstein-Maxwell-Scalar field system with $\Lambda>0$. In this paper we approach the above problem \emph{without any symmetries}, but we restrict ourselves to the \emph{linear} analysis of the Bianchi equations, and do not yet pose initial data on the cosmological horizons, but instead on a \emph{spacelike} hypersurface (arbitrarily close to the cosmological horizons) in the cosmological region.

Note that the assumption  --- that the data be exponentially decaying to a Kerr de Sitter geometry along the cosmological horizons --- is justified by virtue of the result of Hintz and Vasy \cite{hi:vasy:stability}.\footnote{Strictly speaking \cite{hi:vasy:stability} does not construct the cosmological horizon as a null hypersurface, but this is a minor point given that one expects to obtain $\mathcal{C}$ in their setting by an application of the stable manifold theorem.

  In any case, in this paper ---  in particular in the main theorem of Section~\ref{sec:intro:result} --- we actually do \emph{not} use the exponential decay along the cosmological horizons. While the results in this paper hold under broader assumptions, one could impose these stronger assumptions to derive the behaviour of all geometric quantities not only in $r(u,v)$ but also track the behaviour as a function of ``retarded time'' $u$; see Section~\ref{sec:intro:basic} for definitions. The detailed analysis of the characteristic constraint equations under the exponential decay assumption, and the investigation of finer asymptotics towards $\iota^+$ are not included in this paper; see also Remark~\ref{remark:ivp} below.}   However, also note that in the discussion of our expectation for the asymptotics no reference is made to the Kerr de Sitter solution.

In \cite{glw} I have considered a linear model problem, namely the corresponding Cauchy problem for the linear wave equation
\begin{equation}\label{eq:wave:intro}
  \Box_g\psi =0
\end{equation}
on a fixed Kerr de Sitter background $(\mathcal{M},g)$. It was shown that any solution to \eqref{eq:wave:intro} arising from finite energy data on $\Sigma$, remains globally bounded yet has a limit $\psi^+$ on $\mathcal{I}^+$ which as a function on the standard cylinder $\mathbb{R}\times\mathbb{S}^2$ has finite energy.
Moreover, if exponential decay is assumed along the cosmological horizons (which in this setting is justified by the results of Dyatlov \cite{sd:beyond,sd:quasi}), then this ``rescaled'' energy of $\psi^+$ on $\mathcal{I}^+$ decays towards time-like infinity $\iota^+$, but still need not vanish globally on $\mathcal{I}^+$. This means that even in the context of the linear theory, \emph{there is a non-trivial degree of freedom at infinity}.
Finally, the results in \cite{glw} depend by no means on the symmetries of Kerr de Sitter geometry, and have been proven therein for a large class of spacetimes without any symmetries near the Schwarzschild de Sitter cosmology.

The intuition gained in the linear problem tells us that \emph{in the context of the fully non-linear problem we cannot expect convergence to a member of the Kerr de Sitter family, but merely a ``nearby'' geometry, which is however a priori unknown.} In fact, the setting to be presented in this paper is consistent with the asymptotic geometry to differ from Kerr de Sitter --- indeed \emph{de Sitter} spacetime ---  even at the ``leading order''.\footnote{As we shall see the geometry \emph{is} expected to be asymptotically de Sitter \emph{locally} in the past of any time-like geodesic, however it cannot be expected to agree \emph{globally}.
  Indeed the stability results in Appendix~C of \cite{hi:vasy:stability} for the static model of de Sitter space are of a local nature, in the sense that they apply to the past of a point $p\in\mathbb{S}^3$ on the future boundary.
  If the cosmological region of the Kerr de Sitter solution is stable then we could choose a point $p\in \mathcal{I}^+$, and find than in a suitable small neighborhood of $p$ as depicted in Figure~\ref{fig:Goursat} the past of $p$ is a perturbation of the static region of de Sitter spacetime, to which Theorem~C.4 in \cite{hi:vasy:stability} applies.
  However, convergence of the metric to de Sitter as the point $p$ on the future boundary is approached can only be achieved \emph{in a gauge that depends on $p$}. In fact, their proof proceeds in by an iteration scheme which eliminates growing modes of the linearised equations \emph{by a suitable gauge choice which depends on the solution in the past of $p$}.
  This behaviour of the solutions at late times observed \emph{locally} also runs under the name of ``cosmic no-hair'', and has been proven \emph{in spherical symmetry} in \cite{costa:nohair}.}

This view is echoed in an instructive series of papers by Ashtekar, Bonga and Kesavon \cite{ashtekar:PRL,ashtekar:I,ashtekar:II}.  In \cite{ashtekar:I} it is argued that future null infinity $\mathcal{I}^+$ cannot be conformally flat in the presence of gravitational waves.\footnote{In general a spacelike hypersurface is conformally flat if the Bach tensor $B$ vanishes identically, where $B$ is related to the magnetic part of the ambient Weyl curvature $W$; see Appendix~\ref{sec:bach}. Here future null infinity $\mathcal{I}^+$ is a not an embedded hypersurface but rather the boundary of a conformally changed spacetime, and the statement that $\mathcal{I}^+$ is conformally flat refers to the vanishing of the Bach tensor of the \emph{rescaled} Weyl curvature $r^3 W$. Thus the conformal flatness of $\mathcal{I}^+$ is related to the vanishing of certain components of the \emph{rescaled} Weyl curvature, and our setting is compatible with all components of $r^3 W$ to be supported, and $\mathcal{I}^+$ \emph{not} to be conformally flat. } They show in particular that the condition that $\mathcal{I}^+$ be intrinsically conformally flat (as it is for Schwarzschild de Sitter geometry) would suppress ``half'' of the gravitational degrees of freedom. This is consistent with the setting in this paper, where will will allow the spheres foliating $\mathcal{I}^+$ to be \emph{not perfectly round}.\footnote{The ``functional degrees of freedom at infinity'' referred to here already appear in the context of Friedrich's proof of the stability of de Sitter spacetime\cite{friedrich:desitter}, and are precisely captured by the free data for the ``conformal constraint equations''.
 In Section~2.3 of \cite{valiente} it is pointed out that Friedrich's conformal constraint equations reduce to  $\tilde{\nabla}\cdot D=0$, where $D$ is the electric part of the \emph{rescaled} conformal Weyl curvature, and $\tilde{\nabla}$ is the connection associated the intrinsic metric on $(\mathcal{I}^+,h)$; see also Section~\ref{sec:electromagnetic} for the definition of electromagnetic decomposition. The statement that $\mathcal{I}^+$ is not intrinsically locally conformally flat is here naturally related to the non-vanishing of the Cotton tensor of the metric $h$; cf.~the discussion of the related Bach tensor in Appendix~\ref{sec:bach}.\label{footnote:conformal:constraints}.}

Now we will treat Einstein's  equations not as a system of wave equations for the metric, but rather using the \emph{electromagnetic analogy}, which has been employed so successfully in the seminal work of Christodoulou and Klainerman \cite{ch:kl}. In other words, we use that \eqref{eq:eve:lambda:intro} imply the homogeneous contracted Bianchi equations for the Riemann curvature tensor $R$:
\begin{equation}
  \divergence R = \curl \Ric =0
\end{equation}
The Riemann curvature however --- in the role of the Faraday tensor $F$ --- is not a suitable quantity to consider in this setting, and cannot be expected to decay; indeed the de Sitter solution is a \emph{constant curvature space}. In this work we pass from the Riemannian curvature $R$ to the \emph{conformal Weyl curvature} $W$, which for any solution to \eqref{eq:eve:lambda:intro} is related to $R$ by\footnote{de Sitter space is in fact \emph{conformally flat}. The passing from $R$ to $W$ can thus be thought of as a renormalisation of the curvature by its de Sitter values. Note also that for $\Lambda=0$ the Weyl and Riemann curvature coincide, and thus play a notably different role only in the cosmological setting.}
\begin{equation}\label{eq:weyl:lambda:intro}
  W(X,Y,U,V)=R(X,Y,U,V)+\frac{\Lambda}{3}\Bigl[g(X,V) g(Y,U)-g(X,U)g(Y,V)\Bigr]
\end{equation}
and thus also satisfies
\begin{equation}\label{eq:bianchi:W:intro}
  \divergence W = 0\,.
\end{equation}
The conformal Weyl curvature is \emph{the} prototypical Weyl field $W$ in the sense of \cite{ch:kl} and its algebraic properties allow us to construct energies  using the \emph{Bel-Robinson tensor} $Q(W)$, which can be viewed as a generalisation of the energy-momentum tensor of electromagnetic theory.
A key advantage of this approach is then that certain methods developed for the treatment of the linear equation \eqref{eq:wave:intro} --- in particular our understanding of the decay mechanism for solution to \eqref{eq:wave:intro} in the cosmological region --- carry over to the study of solutions to \eqref{eq:bianchi:W:intro}. We will elaborate on this in more detail in Section~\ref{sec:intro:proof}.

\bigskip
In this paper we will treat the first part of the global non-linear stability problem, as formulated above as a characteristic initial value problem for the cosmological region. Namely following the strategy laid out in \cite{ch:kl}, we will make certain assumptions on the metric $g$, and the connection coefficients,\footnote{Connection coefficients are on the level of \emph{one} derivative of the metric. Alternatively these assumptions can be thought of as conditions on the geometric properties of a chosen \emph{foliation}, or as conditions on the deformation tensors of a number of relevant vectorfields.} and then prove a non-trivial statement for the Weyl curvature.
Independently in a second part we will prove that \emph{assuming the bounds on the Weyl curvature} all assumptions that were made here on the metric and the connection coefficients can be derived by a suitable gauge choice on the cosmological horizons.\footnote{In a simplified setting this is precisely what is achieved in \cite{schlue:optical}. Therein a global double null foliation of de Sitter spacetime --- namely under the assumption that the Weyl curvature vanishes identically --- is constructed by a suitable choice of spheres on a bifurcate null hypersurface.}
In the context of an overarching bootstrap argument, it remains to show that the argument closes under the assumption of \emph{small initial data} on the cosmological horizons\footnote{Characteristic initial data will in particular be prescribed such that \emph{along} the cosmological horizons the geometry converges exponentially fast to the geometry induced by Kerr de Sitter. The detailed geometric setup is not discussed in this paper. For a general discussion of \emph{characteristic} initial data we refer to Chapter~2 in~\cite{ch:blue}.} to yield a full existence result.

A significant challenge of this part  lies in \emph{identifying} a set of assumptions which are on one hand sufficiently general to encompass the actual dynamics of the metric under the evolution of \eqref{eq:eve:lambda:intro} (too restrictive assumptions would be inconsistent, and have no chance of being recovered), while on the other hand sufficiently restrictive for the decay mechanism to come into play. The latter is predominantly the \emph{expansion}, which can be captured adequately on the level of \emph{mean curvatures}. 

Informally speaking, we establish the following:
\begin{quote}
  \itshape The conformal Weyl curvature decays uniformly in the cosmological region~$\mathcal{R}$,  provided the metric and connection coefficients satisfy a set of assumptions which capture in particular the expansion of the spacetime.
\end{quote}

In Schwarzschild de Sitter spacetime the Weyl curvature has only one non-vanishing component
\begin{equation}\label{eq:weyl:sds:intro}
   \rho[W] =-\frac{2m}{r^3} 
\end{equation}
where $m$ is a constant, the mass of the black hole; cf.~Section~\ref{sec:eve:lambda}, \ref{sec:schwarzschild:de:sitter}.

``Decay'', and its ``uniformity'' refer to a parameter like $r$ in the Schwarzschild de Sitter example, but as we shall see even the definition of a suitable time function \[r:\mathcal{R}\to (0,\infty)\] is non-trivial.
The reason the definition of $r$ is a non-trivial question is that we subject it
to the following two requirements:
\begin{itemize}
\item The function $r$ is a time-function on $\mathcal{R}$ --- i.e.~it is strictly increasing along any future-directed time-like curve --- and the zero level set of $1/r$ can be identified with $\mathcal{I}^+$.
  \item The function $r$ is defined by a purely geometric construction, and fully determined by gauge choices on $\mathcal{C}\cup\overline{\mathcal{C}}$.
\end{itemize}
In \cite{schlue:optical} it is shown in particular that --- in the context of using optical functions as the purely geometric means by which $r$ is defined --- the first requirement is not stable at all under perturbations of the gauge choices on $\mathcal{C}\cup\overline{\mathcal{C}}$; see Section~\ref{sec:intro:basic} for a qualitative discussion. (The paper \cite{schlue:optical} then also identifies a useful criterion for both requirements to be satisfied, and gives a construction of non-trivial optical functions, and hence time functions in a simplified setting.)  In this paper we essentially assume that both requirements are fulfilled, and we will then establish that under our assumptions  \emph{all} components of the Weyl curvature decay at precisely the rate indicated in \eqref{eq:weyl:sds:intro}.

In Section~\ref{sec:intro:basic} we discuss the basic difficulties related to a suitable choice of coordinates which covers the domain of development, and correctly parametrizes future null infinity.\footnote{Interestingly, in a different context, namely in the study of relativistic compressible fluids, a similar issue appears for the characterisation of the ``singular hypersurface'' in the formation of shocks; see Chapter~15 in \cite{ch:shocks}.} In Section~\ref{sec:intro:foliation} we highlight some of the assumptions made in this paper, in particular we identify a suitable notion of expansion at the level of mean curvatures. Then we proceed to a more precise statement of the result in Section~\ref{sec:intro:result}, and discuss some of the ideas and difficulties of the proof in Section~\ref{sec:intro:proof}.
%In Section~\ref{sec:intro:preview} we briefly sketch some of the ideas relevant for the recovery of the assumptions, which is the topic of a subsequent paper. (Nonetheless, it is important to illustrate the consistency of the assumptions already at this stage.)
Finally, we discuss the relation of this result to earlier work, in particular Friedrich's proof of the stability of the de Sitter solution, in Section~\ref{sec:intro:earlier}.%, and the relevance of this work beyond the black hole stability problem in Section~\ref{sec:intro:beyond}

\subsection{Basic Difficulties}
\label{sec:intro:basic}

In view of the expectation that the geometry of a dynamical solution to \eqref{eq:eve:lambda:intro} in the cosmological region \emph{does not globally converge to a member of the Kerr de Sitter family, but merely to a ``nearby geometry'' which is a priori unknown},  the choice of suitable coordinate system --- which covers the entire region of existence --- is non-trivial. 

 Consider for example any given coordinate system $(x^a;y^A)$ on $\mathcal{R}\subset \mathcal{Q}$ for the Schwarzschild de Sitter solution. A na\"\i ve approach would be to formulate the assumptions on the metric in this coordinate system, and try to establish the decay with respect to a parameter $r=r(x^a)$ formally defined as in the Schwarzschild de Sitter geometry. However, such an approach turns out to be \emph{inconsistent}, the reason being that the surface $r=\infty$ thus defined does not coincide with the true future boundary found in evolution.

To overcome this problem we work in a \emph{double null gauge},\footnote{This gauge --- and its very general geometric formulation suitable for the study of the Einstein equations outside spherical symmetry --- has been introduced by Christodoulou \cite{demetri:notes} and used notably in his \cite{ch:blue}, whose conventions we follow closely in this paper.} namely coordinates $(u,v;\vartheta^1,\vartheta^1)$ such that the level sets of $u$, $v$ are null hypersurfaces whose intersections $S_{u,v}$ are diffeomorphic to $\mathbb{S}^2$. This choice is natural for the treatment of a characteristic initial value problem, and it allows us in particular to introduce the function $r$ as the \emph{area radius} of the spheres of intersection $(S_{u,v},\gs)$,
\begin{equation}\label{eq:area:intro}
  4\pi r^2(u,v)=\int_{S_{u,v}}\dm{\gs}\,.
\end{equation}
It is then tempting to define future null infinity $\mathcal{I}^+$ as the collection of spheres with infinite area radius.
However, a general double null foliation still allows considerable freedom in the choice of the spheres of intersection, and it can happen that a sphere with infinite area radius is in fact partly contained in the spacetime. In such a scenario future null infinity is not correctly identified by the set of points where $r=\infty$.

This  subtle yet imporant point is proven in Section~4 of~\cite{schlue:optical}, where we give an explicit construction of a double null foliation of de Sitter space which illustrates this phenomenon:

\begin{quote}
  \itshape There exist double null foliations  of de Sitter space such that the union of all spheres with area radius $r\in (\rc,\infty)$  is contained, but does not exhaust the cosmological region.
\end{quote}

In fact, these examples are constructed as explicit solutions to the eikonal equation on de Sitter spacetime $(H,h)$,
\begin{equation}\label{eq:intro:eikonal:desitter}
  h(\nabla u,\nabla u)=0
\end{equation}
with the property that the level sets $C_u$ intersect the cosmological horizon $\mathcal{C}$ in a small ellipsoidal deformation $C_u\cap\mathcal{C}$ of the round sphere, yet the ``corresponding sphere near infinity'' ---  namely the intersection $S_{u,0}=C_u\cap\underline{C}_0$ with a fixed ``incoming'' null hypersurface $\underline{C}_0$  --- is only partially contained in $H$. In these examples,  $S_{u,0}$ first ``touches infinity at a point'' and then gradually ``disappears on annular regions'' as $u$ varies, see Fig.~\ref{fig:sphere:infty}. 

\begin{figure}
  \centering
  \includegraphics{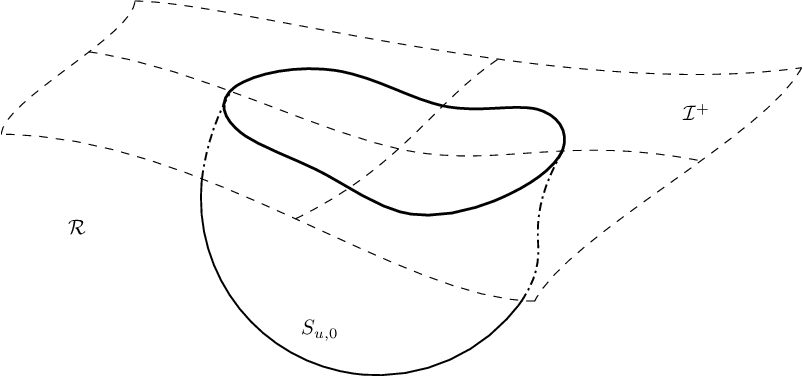}
  \caption{Sphere with infinite area radius partially contained in the spacetime.}
  \label{fig:sphere:infty}
\end{figure}

We expect this to be the generic behavior, and it is  important to note that our assumptions on the foliations to be discussed in Section~\ref{sec:intro:foliation} rule out such behavior of the spheres near infinity. \footnote{Essentially, any assumption that requires the smallness of the deviation of a quantity  $f(u,v,\vartheta^1,\vartheta^2)$ from its average $\overline{f}(u,v)$ on the sphere $S_{u,v}$ cannot be satisfied if the area radius of $S_{u,v}$ diverges while a subset of points $q\in S_{u,v}$ remain in the spacetime.} In fact,  the assumptions of Section~\ref{sec:intro:foliation}  ensure that the spheres of the foliation indeed exhaust the expanding region; see \cite{schlue:optical}.

Furthermore, in \cite{schlue:optical} we show --- in the case of a fixed de Sitter spacetime with identically vanishing Weyl curvature  --- that \emph{using a final gauge choice} a global double null foliation \emph{can be constructed}, which has all the properties assumed in Section~\ref{sec:intro:foliation}. This is achieved by a global analysis of the \emph{null structure equations} on de Sitter space, which arise in the decomposition of \eqref{eq:eve:lambda:intro} in double null coordinates $(u,v;\vartheta^A)$, where $u$, and $v$ satisfy \eqref{eq:intro:eikonal:desitter}; see Section~5 in \cite{schlue:optical}.

%In we study also the transformation of the structure coefficients under the gauge transformations provided by these examples, and demonstrate explicitly that some of the assumptions we make in Section~\ref{sec:intro:foliation} are \emph{not satisfied}.\footnote{Essentially, any assumption that requires the smallness of the deviation of a quantity  $q(u,v,\vartheta^1,\vartheta^2)$ from its average $\overline{q}(u,v)$ on the sphere $S_{u,v}$ cannot be satisfied if the area radius of $S_{u,v}$ diverges while a subset of points $q\in S_{u,v}$ remain in the spacetime.}

%Even with a ``good'' choice of a foliaton with the property that all spheres with infinite area radius are indeed ``at infinity'', it would be too restrictive to assume that the spheres foliating null infinity are intrinsically round. Our assumptions are certainly consistent with such a non-trivial asymptotic geometry, and we will already see in Section~\ref{sec:intro:preview} an obstruction to the roundness of the spheres, but the precise geometric characterisation of future null infinity is the subject of a subsequent paper.\footnote{See also the discussion in \cite{ashtekar:PRL}. The prospect of a fundamental obstruction to the roundness of the spheres composing null infinity raises a number of interesting questions, in particular related to the definition of asymptotic quantites, such as mass, and angular momentum, whose classical definitions all require the asymptotic roundness, which here appears to fail, at least globally. See also \cite{tod:bondi,chrusciel:energy}.}

\subsection{Assumptions on the foliation}
\label{sec:intro:foliation}

Consider a $3+1$-dimensional Lorentzian manifold $(\mathcal{M},g)$ with past boundary $\mathcal{C}\cup\overline{\mathcal{C}}$, two null hypersurfaces $\mathcal{C}$ and $\overline{\mathcal{C}}$ intersecting in a sphere $S$, whose null geodesic generators have no future end points. We think of initial data prescribed along $\mathcal{C}$ and $\overline{\mathcal{C}}$, and of $\mathcal{R}=J^+(\mathcal{C}\cup\overline{\mathcal{C}})$ --- the ``cosmological region'' --- as its future development.

 Consider further a double null foliation of $\mathcal{R}$ by null hypersurfaces $C_u$, and $\Cb_v$, namely the level sets of functions
\begin{equation}
  u:\mathcal{R}\longrightarrow(0,\infty)\qquad v:\mathcal{R}\longrightarrow (0,\infty)
\end{equation}
satisfying the eikonal equations
\begin{equation}\label{eq:eikonal}
  g(\nabla u,\nabla u)=0\qquad g(\nabla v,\nabla v)=0
\end{equation}
which are increasing towards the future, such that $\mathcal{C}=\Cb_0$, $\overline{\mathcal{C}}=C_0$, and 
\begin{equation}\label{eq:spheres}
  \Bigl(S_{u,v}=C_u\cap\Cb_v ,\gs = g\rvert_{\mathrm{T}S_{u,v}}\Bigr)
\end{equation}
 is diffeomorphic to $\mathbb{S}^2$.
Following the conventions in \cite{ch:blue} we define
\begin{equation}\label{eq:Lp}
  \Lp=-2\ud u^\sharp\qquad \Lbp=-2\ud v^\sharp
\end{equation}
to be the null geodesic normals, and $\Omega$ to be the null lapse:
\begin{equation}\label{eq:null:lapse}
  \Omega=\sqrt{\frac{2}{-g(\Lbp,\Lp)}}
\end{equation}

\begin{remark}
  The reader may find it useful to refer in parallel to Section~\ref{sec:schwarzschild:de:sitter} where all of the following geometric quantities are computed  for the Schwarzschild de Sitter spacetime in spherically symmetric double null foliations.
\end{remark}

Then normalised null normals are given by
\begin{equation}\label{eq:Lh}
  \Lbh=\Omega\Lbp\qquad \Lh=\Omega\Lp
\end{equation}
and used to define the null second fundamental forms of the spheres $S_{u,v}$ as surfaces embedded in $C_u$, and $\Cb_v$ respectively:
\begin{equation}\label{eq:chi}
  \chi(X,Y)=g(\nabla_X\Lh, Y)\qquad \chib(X,Y)=g(\nabla_X\Lbh, Y)\qquad (X,Y\in\mathrm{T}S_{u,v})
\end{equation}
The null expansions, namely the traces $\tr\chib$, and $\tr\chi$ (with respect to $\gs$), measure pointwise the change of the area element $\dm{\gs}$, in the null directions $\Lbh$, and $\Lh$, respectively. Our first main assumption is that ``the cosmological region is expanding'':
\begin{equation}
  \label{eq:assumption:intro:tr}
  \tr\chi>0\qquad \tr\chib>0 \qquad \text{: on }\mathcal{R} \tag{\emph{\textbf{BA:I}.i}}
\end{equation}
The positivity of the null expansions alone is not enough, and we will assume that they are always close to the ``geodesic accelerations'' $2\omegabh$, and $2\omegah$ defined by
\begin{equation}\label{eq:omegah}
  \nabla_{\Lh}\Lh=\omegah \Lh \qquad \nabla_{\Lbh}\Lbh=\omegabh \Lbh
\end{equation}
Equivalently, they are given by
\begin{equation}
  \omegah=\Lh\log \Omega\qquad \omegabh=\Lbh \log \Omega
\end{equation}
and our second assumption is that for some constant $C_0>0$:
\begin{equation}\label{eq:assumption:intro:omega}
  \Omega \lvert 2\omegah - \tr\chi \rvert \leq C_0 \tr\chi\qquad \Omega \lvert 2\omegabh - \tr\chib \rvert \leq C_0 \tr\chib  \tag{\emph{\textbf{BA:I}.ii}}
\end{equation}

Our third assumption is crucially related to the discussion in Section~\ref{sec:intro:basic}, and amounts to the condition that the null expansions are pointwise close to their spherical averages
\begin{equation}
  \overline{\Omega\tr\chi}=\frac{1}{4\pi r^2}\int_{S_{u,v}}\Omega\tr\chi\dm{\gs} \qquad   \overline{\Omega\tr\chib}=\frac{1}{4\pi r^2}\int_{S_{u,v}}\Omega\tr\chib\dm{\gs}\,.
\end{equation}
We require that for some constant $C_0>0$
\begin{equation}\label{eq:assummption:intro:average}
  \lvert \Omega \tr\chi - \overline{ \Omega\tr\chi } \rvert \leq  C_0\Omega^{-1} \overline{\Omega \tr\chi } \qquad   \lvert \Omega \tr\chib - \overline{ \Omega\tr\chib } \rvert \leq  C_0 \Omega^{-1} \overline{\Omega \tr\chib }  \tag{\emph{\textbf{BA:I}.iii}}
\end{equation}

Finally we assume for the remaing connection coefficients, namely the trace-free parts of the null second fundamental forms above, and the torsion 
\begin{equation}\label{eq:zeta}
  \zeta(X)=\frac{1}{2}g(\nabla_X\Lh,\Lbh)\qquad (X,Y\in\mathrm{T}S_{u,v})
\end{equation}
 that
\begin{equation}\label{eq:assumption:intro:zeta}
  \Omega \lvert \chih \rvert_{\gs} \leq C_0\qquad \Omega \lvert \chibh \rvert \leq C_0\qquad \Omega \lvert \zeta \rvert \leq C_0  \tag{\emph{\textbf{BA:I}.iv}}
\end{equation}

\bigskip
The assumptions (\emph{\textbf{BA:I}}) are $\mathrm{L}^\infty$-bounds on $S_{u,v}$. They already allow us to prove that the $\mathrm{L}^2(\Sigma_r)$- norm of the Weyl curvature decays; see \eqref{def:Sigma:r:intro} for the definition of the spacelike hypersurface $\Sigma_r$, and Section~\ref{sec:global:redshift} for proof of the decay of the Weyl curvature flux through $\Sigma_r$. However, here we seek to prove decay of the Weyl curvature in $\mathrm{L}^4(S_{u,v})$. This requires us to make  additional assumptions on the \emph{derivatives of the connection coefficients}. These assumptions, schematically bounds on
\begin{equation}
  \lVert \nabla \Gamma \rVert_{\mathrm{L}^4(S_{u,v})}  \tag{\emph{\textbf{BA:II}}}
\end{equation}
are too numerous and technical to state conveniently here and will instead be collected under the label (\emph{\textbf{BA:II}}) below.
They will allow us to prove that all tangential derivatives of $W$ to $\Sigma_r$ decay in $\mathrm{L}^2(\Sigma_r)$.

Finally, several assumptions will be needed  to deduce the decay rates of the  energy   and for the application of the elliptic theory in Sections~\ref{sec:electromagnetic}, \ref{sec:sobolev}. Most importantly,
\begin{equation}\label{eq:assumption:intro:Omega}
  C_0^{-1}\, r \leq \Omega \leq C_0\,r \tag{\emph{\textbf{BA:III}.i}}
\end{equation}
and the remaining assumptions will be collected under the label (\emph{\textbf{BA:III}}).

The complete list of assumptions is given collectively in Appendix~\ref{sec:BA}.

\bigskip
\emph{We emphasize that none of the assumptions make explicit reference to the Schwarzschild de Sitter geometry, and capture only some of the features that the relevant ``nearby'' geometries  have in common.} We expect the results of this paper to be an integral part of a general existence theorem for the problem discussed above which in particular would provide as a Corollary non-trivial examples of spacetimes satisfying these assumptions.

In the coordinates $(u,v;\vartheta^1,\vartheta^2)$ thus introduced\footnote{The coordinates $(\vartheta^1,\vartheta^2)$ are chosen arbitrarily on a domain of $S=\mathcal{C}\cap\overline{\mathcal{C}}$, and then transported first along the null geodesics generating $\overline{\mathcal{C}}$, and then those of $\Cb_v$.} the spacetime metric takes the form
\begin{equation}\label{eq:g:intro:double}
  g=-2\Omega^2\bigl(\ud u\otimes\ud v+\ud v\otimes \ud u\bigr)+\gs_{AB}\bigl(\ud \vartheta^A-b^A\ud v\bigr)\otimes\bigl(\ud\vartheta^B-b^B\ud v\bigr)
\end{equation}
In the special case that $b^A=0$, $\gs=r^2\gammao$, and $\Omega$ is independent of $\vartheta^A$, the metric reduces to the spherically symmetric form \eqref{eq:g:intro:spherical}. In Section~\ref{sec:schwarzschild:de:sitter} we will discuss  various specific choices of the functions $u$, $v$ on $\mathcal{Q}$ for the Schwarzschild de Sitter metric, and the associated values of the connection coefficients above.

In Section~\ref{sec:areal} we will discuss yet another form of the metric adapted to the decomposition relative to the level sets of the area radius
\begin{equation}\label{def:Sigma:r:intro}
  \Sigma_r=\Bigl\{ q\in\mathcal{R}: r(q)=r \Bigr\}\,.
\end{equation}
By \eqref{eq:assumption:intro:tr} these hypersurfaces are always \emph{spacelike}, and the area radius plays the role of a time-function:
\begin{equation} \label{eq:g:intro:r}
  g=-\phi^2\ud r^2+\gb{r}
\end{equation}
where $\gb{r}$ denotes the induced Riemannian metric on $\Sigma_r$.\footnote{Here we choose for simplicity coordinates on $\Sigma_r$ which are transported along the normal geodesics, which justifies the absence of a shift term.} 
While this form of the metric is advantageous in some parts of this paper,\footnote{In Section~\ref{sec:energy:estimates} we will consider the ``energy flux'' of the Weyl curvature through $\Sigma_r$, and the form \eqref{eq:g:intro:r} is particularly well adapted to the application of the \emph{co-area formula} for integration on domains foliated by $\Sigma_r$. Also the Sobolev inequality of Section~\ref{sec:sobolev} is applied on the manifold $\Sigma_r$.} and of some importance for the discussion of the asymptotics, we have chosen the double null gauge as the underlying differential structure, because it will allow us to formulate all assumptions that fix the gauge  and specify the initial data only on $\mathcal{C}\cup\overline{\mathcal{C}}$. Moreover, unlike in ``non-local'' gauges such as ``maximal'' gauges which fix the mean curvature of the surfaces $\Sigma_r$, the double null gauge allows us to localise any argument to the domain of dependence of any subset of $\Sigma_r$; this will be of importance for the  recovery of the assumptions made in this Section.

\subsection{Main result}
\label{sec:intro:result}

We are interested in the ``cosmological region'' $\mathcal{R}$ to the future of the cosmological horizons $\mathcal{C}\cup\overline{\mathcal{C}}$; see Fig.~\ref{fig:Goursat}. Above we have introduced the 2-dimensional closed Riemannian manifolds $(S_{u,v},\gs)\subset\mathcal{R}$ as the intersections of ``ingoing'' and ``outgoing'' null hypersurfaces $\Cb_v$, and $C_u$, the leaves of a double null foliation of $\mathcal{R}$ discussed in Section~\ref{sec:intro:foliation}. Each sphere $S_{u,v}$ has area $4\pi r^2(u,v)$ as discussed in Section~\ref{sec:intro:basic}, and we shall now introduce a dimensionless $\mathrm{L}^p$- norm for tensors $\theta$ on the spheres:
\begin{equation}
  \nLp{\theta}:=\biggl(\frac{1}{4\pi r^2}\int_{S_{u,v}}\lvert \theta \rvert_{\gs}^p\dm{\gs}\biggr)^\frac{1}{p}
\end{equation}

\begin{theorem}
  Let $(\mathcal{M},g)$ be a $3+1$-dimensional Lorentzian manifold, $\mathcal{R}\subset\mathcal{M}$ a domain with past boundary $\mathcal{C}\cup\overline{\mathcal{C}}$, where $\mathcal{C}$ and $\overline{\mathcal{C}}$ are future geodesically complete null hypersurfaces intersecting in a sphere $S_\circ$ diffeomorphic to $\mathbb{S}^2$. 
Assume that $g$ can be expressed globally on $\mathcal{R}$ in the form \eqref{eq:g:intro:double} and  that the double null foliation of $\mathcal{R}$ satisfies the assumptions  (\textbf{BA:I}-\textbf{BA:III}) (as listed in Appendix~\ref{sec:BA}). 

Suppose the Weyl field $W$  is a solution to the Bianchi equations \eqref{eq:bianchi:W:intro} and  $W\in H^1(\Sigma_{r_\circ})$ where $\Sigma_{r_\circ}\subset\mathcal{R}$ is a level set of $r$.
Then $W$ obeys %the conformal Weyl curvature $W$ of $g$ obeys
\begin{equation}\label{eq:thm:W}%  \lVerts W \rVerts_{\mathrm{L}^{4}(S_{u,v})}
  \nLq{W} \leq \frac{C}{r^{3}} \lVert W \rVert_{\mathrm{H}^1(\Sigma_{r_\circ})}\qquad \text{: on }\mathcal{R}
  \end{equation}
  where  $C>0$ only depends on the constants in (\textbf{BA:I}-\textbf{BA:III}).
\end{theorem}

% Recall that a lower bound on the decay of the Weyl curvature is given by \eqref{eq:weyl:sds:intro} from the situation in Schwarzschild de Sitter.

\begin{remark}
  The main estimate could alternatively be stated as \[\|W\|_{\mathrm{H}^1(\Sigma_r)}^2 := \int_{\Sigma_r}|W|^2+r^2|\nablab W|\lesssim  \frac{1}{r^3}\|W\|_{\mathrm{H}^1(\Sigma_{r_\circ})}^2\,.\]
    The estimate \eqref{eq:thm:W} then follows from the Sobolev embedding $W_1^2(\Sigma_r)\hookrightarrow L^6(\Sigma_r)$ which we discuss in Section~\ref{sec:sobolev} in the present setting under the assumptions (\textbf{BA:I}-\textbf{BA:III}). The Sobolev inequality on $\Sigma_r$ is derived from the isoperimetric Sobolev inequality on the spheres $S_{u,v}$, in the course of which we show that the trace operator from $W_1^2(\Sigma_r)\to L^4(S_{u,v})$ is bounded, which explains the decay statement is in $\mathrm{L}^4$.
\end{remark}

\begin{remark}
Recall that for the example of Schwarzschild de Sitter spacetime the Weyl curvature satisfies \eqref{eq:weyl:sds:intro}.
In terms of the decay \emph{rate} the Theorem thus states the optimal decay of \emph{all} components of the conformal curvature.\footnote{In Schwarzschild de Sitter there is really only \emph{one} non-vanishing component of the Weyl curvature. In this result all components of the Weyl curvature are on an equal footing, and at present there is no indication of any ``peeling'' behavior. This is consistent with the expectations formulated in \cite{ashtekar:PRL}, and we will revisit this point in Section~\ref{sec:electromagnetic} in the context of the electro-magnetic decomposition of the Weyl curvature.} In terms of \emph{regularity}, we do not estimate the curvature in $\mathrm{L}^\infty(S_{u,v})$ (which would require assumptions on the connection coefficients at second order of differentiablity). However, it suffices to control the curvature merely in $\mathrm{L}^4(S_{u,v})$ to obtain a existence for the Einstein equations \eqref{eq:eve:lambda:intro}.\footnote{This is an insight gained in \cite{bieri}, which has provided a significant simplification of the original proof of the non-linear stability of Minkowski space in \cite{ch:kl}. It eliminates in particular the need to work with \emph{second} derivatives of the curvature, but comes of course at the cost of less refined asymptotics; see \cite{bieri} and discussion therein.
  
  Further to the discussion of the conformal constraint equations in footnote~\ref{footnote:conformal:constraints}, we remark that the present results are at the level of $D\in \mathrm{H^1(\mathcal{I}^+)}$. Note however that the Theorem holds under minimal regularity assumptions on $\Omega$, which plays the role of the conformal factor. In view of Sections~2.2 \& 7.3 in \cite{schlue:optical} we expect to close the estimates with $\Omega$, $D\Omega$, $\Db\Omega$ merely in $\mathrm{W}_2^4(S_{u,v})$.
}
\end{remark}

\begin{remark}\label{remark:localised}
  A \emph{localised version} of the result holds: Let $\Sigma=\Sigma_{r_0}\cap\{u\leq u_0\}\cap \{v\leq v_0\}$ be a ``segment'' of a level set $r=r_0$, and $W\in H^1(\Sigma)$, then \eqref{eq:thm:W} holds for all $S=S_{u,v}$ \emph{in the domain of dependence of $\Sigma$}, namely $u\leq u_0$, $v\leq v_0$.
\end{remark}

In this paper we do not yet give an existence proof of solutions to \eqref{eq:eve:lambda:intro} in $\mathcal{R}$.
The theorem establishes what we expect to be the first part of a larger \emph{bootstrap}, or \emph{continuous induction} argument. The task of the second part is to establish the converse, namely that \emph{assuming}  $\mathrm{L}^4$-bounds on the Weyl curvature, we seek to \emph{prove} the $\mathrm{L}^\infty$, and $\mathrm{L}^4$-estimates (\emph{\textbf{BA}}) on the connection coefficients, \emph{under suitable assumptions on the characterisitc initial data}. In a simplified and restricted --- yet very instructive --- setting, this is achieved in \cite{schlue:optical}; see Section~\ref{sec:optical} for further comments.

\begin{remark}\label{remark:ivp}
  The assumptions on the initial data in the Theorem are made on the spacelike hypersurface $\Sigma_{r_0}$. In the context of the non-linear problem described on page~\pageref{quote:stability:sds:goursat} geometric initial data is prescribed on the characteristic null hypersurfaces $\mathcal{C}\cup\overline{\mathcal{C}}$.
Proceeding similarly to \CCh{2} we can derive from the null constraint equations --- and now using the exponential decay of the energy density as defined in \Ceq{2.70} --- the precise behaviour of all geometric along the cosmological horizons, and we expect that analogously to the scalar case --- see Section~4.2 in \cite{glw} --- the ``local redshift effect'' (due to the positive ``surface gravity'') can be exploited to prove a local existence result up to $\Sigma_{r_0}$, such that all decay properties towards $\iota^+$ are inherited. Besides a direct application to the Weyl curvature, the above theorem may also be applied to Weyl fields corresponding to suitable renormalisations of the Weyl curvature. These constructions are independent of the results in this paper.
\end{remark}

\subsection{Comments on the proof}
\label{sec:intro:proof}

The proof is largely a treatment of the Bianchi equations \eqref{eq:bianchi:W:intro} for the Weyl curvature:
\begin{equation}\label{eq:W:Bianchi}
  \nabla_\alpha W^\alpha_{\phantom{\alpha}\beta\gamma\delta}=0
\end{equation}

In analogy to Maxwell's theory, one can construct an energy-momentum tensor $Q(W)$, the Bel-Robinson tensor
\begin{equation}\label{eq:bel:robinson}
  Q[W]_{\alpha\beta\gamma\delta}= W_{\alpha\rho\gamma\sigma}W_{\beta\phantom{\rho}\delta}^{\phantom{\beta}\rho\phantom{\delta}\sigma}+\ld W_{\alpha\rho\gamma\sigma}\ld W_{\beta\phantom{\rho}\delta}^{\phantom{\beta}\rho\phantom{\delta}\sigma}
\end{equation}
which has the well-known properties --- see for example \CKCh{7.1} --- that it is \emph{positive} when evaluated on any causal future-directed vectors at a point (Proposition~4.2 in \cite{ch:kl:lin}), and \emph{divergence-free} if $W$ is a solution to \eqref{eq:W:Bianchi} (Proposition~4.4 in \cite{ch:kl:lin}):
\begin{equation}\label{eq:Q:W:div}
  \nabla^\alpha Q[W]_{\alpha\beta\gamma\delta}=0\,. 
\end{equation}
This allows us define \emph{energy currents} with the help of ``multiplier vectorfields'' $X,Y,Z$,
\begin{equation}\label{eq:current:P}
  P[W]^{(X,Y,Z)}_\alpha=-Q[W]_{\alpha\beta\gamma\delta}X^\beta Y^\gamma Z^\delta
\end{equation}
which give rise to \emph{energy identities}; see derivations in Section~\ref{sec:energy:identity}.
An important example in the context of this paper are energy identities on space-time domains bounded by the level sets $\Sigma_r$:
\begin{equation}
  \mathcal{D}=\bigcup_{r_1\leq r \leq r_2}\Sigma_r
\end{equation}
The usefulness of the resulting identity, 
\begin{equation}\label{eq:energy:id:intro}
  \int_{\Sigma_{r_2}}Q(n,X,Y,Z)\dm{\gb{r_2}}+  \int_{\mathcal{D}}\divergence (-P)[W]^{(X,Y,Z)}\dm{g} =   \int_{\Sigma_{r_1}}Q(n,X,Y,Z)\dm{\gb{r_1}}
\end{equation}
where $n$ denotes the (future directed) unit normal to $\Sigma_r$,
 depends crucially on the properties of the ``divergence term'', or ``bulk term'' on $\mathcal{D}$.
By \eqref{eq:Q:W:div} the integrand is given purely by a contraction of $Q(W)$ with the ``deformation tensor'' of the multiplier vectorfields:
\begin{equation}
  {}^{(X)}\pi=\mathcal{L}_X g
\end{equation}
In fact, since $Q(W)$ is \emph{trace-free} (and symmetric) in all indices, only the ``trace-free part'' of these tensors appear:
\begin{equation}
  {}^{(X)}\hat{\pi}={}^{(X)}\pi-\frac{1}{4}g\tr{}^{(X)}\pi
\end{equation}

\bigskip
In Section~\ref{sec:redshift:vectorfield} we construct a multiplier vectorfield $M$ and prove that the associated divergence terms have \emph{a sign} which yields by \eqref{eq:energy:id:intro} a \emph{monotone} energy. In fact, we will use
\begin{equation}\label{eq:M:intro}
  M=\frac{1}{2}\frac{1}{\Omega}\Bigl(\Lbh+\Lh\Bigr)
\end{equation}
to prove in Section~\ref{sec:global:redshift} that the current 
\begin{equation}\label{eq:current:P:1}
  P[W]^{M}_\alpha:=P[W]^{(M,M,M)}_\alpha
\end{equation}
has the following positivity property: 
\begin{quote}
  \itshape Suppose the connection coefficients satisfy the assumptions (\textbf{BA:I}). Then  for any solution $W$ to \eqref{eq:W:Bianchi}:
  \begin{equation}\label{ineq:redshift:intro}
     \int_{\mathcal{D}}\divergence (-P)[W]^{M}\dm{g} \geq \int_{r_1}^{r_2}\frac{6}{r}\int_{\Sigma_r}Q[W](n,M,M,M)\dm{\gb{r_1}} \ud r
  \end{equation}

\end{quote}
The existence of such a positive current is obviously intimately related to the assumption \eqref{eq:assumption:intro:tr},\footnote{The assumption that \emph{both} null expansions are positive fails in $\mathcal{S}$, where $\tr\chib<0$. Similarly in the study of \emph{asymptotically flat} black hole exteriors, with $\Lambda=0$, one has $\tr\chib<0$. In these settings the construction of \emph{positive currents} is much more subtle, and in particular sensitive to the presence of ``trapped'' null geodesics; see \cite{dhr:linear} and references therein.} and the numerical prefactor in the above inequality is important, because it will directly translate into the decay rate of the energy associated to $M$.

This part of the argument is in close analogy to my treatment of linear waves on Kerr de Sitter cosmologies in \cite{glw}; cf.~discussion in Section~\ref{sec:redshift:vectorfield}. The inequality \eqref{ineq:redshift:intro} lends itself to the interpretation that $M$ captures the classical \emph{redshift effect} in the cosmological region,\footnote{Much like the redshift observed in black hole spacetimes, the effect is due to the presence of horizons: An observer $A$ who stays away from a black hole perceives a signal sent from an observer $B$ who crosses the \emph{event horizon} as redshifted, because $B$ leaves the past of $A$ in finite proper time; see e.g.~\cite{dr:clay}. In the cosmological region, where all observers are drifting away from each other (due to the \emph{expansion} of space), each observer has its \emph{own} cosmological horizon, which all other observers (in his past) cross in finite proper time. %(These horizons are not to be confused with the null hypersurfaces $\mathcal{C}$, which is really the ``cosmological horizon'' of the central black hole.) 
This effect is already present in the de Sitter cosmology, where each time-like geodesic has a cosmological horizon with positive surface gravity;  see \cite{glw, schlue:optical}.}   in the language of ``compatible currents''; we will thus often refer to $M$ as the ``global redshift vectorfield''.

In the context of the linear wave equation on Schwarzschild de Sitter spacetimes, the treatment of \emph{higher order} energies, and pointwise estimates is a trivial extension of the ``global redshift estimate'' because the tangent space to $\Sigma_r$ is in this case spanned by Killing vectorfields. Indeed the commutation of \eqref{eq:wave:intro} with the generators of the spherical isometries of $S_{u,v}$, and of the translational isometry $T$ along $\Sigma_r$ immediately gives the desired higher order energy estimates in \cite{glw}. In the present context, however, this approach is not very fruitful, because even if we were to construct generators $\Omega_{(i)}$ and $T$ of ``spherical'' and ``translational'' actions on $\Sigma_r$, these actions cannot be expected to generate \emph{asymptotic symmetries} (as in \cite{ch:kl}; because unlike in the asymptotically flat case future null infinity may not possess \emph{any} symmetries).

In our approach then we use the global redshift vectorfield also as a \emph{commutator}.\footnote{This is an idea due to Dafermos and Rodnianski that already appeared in the study of the linear wave equation on Schwarzschild spacetimes, related to the redshift effect on the event horizon; see \cite{dr:clay}. Similarly to Theorem~3.2 therein one might also consider \emph{multiple} commutations by $N$, with the aim of obtaining $\overline{\nabla}^k W\in\mathrm{L}^2(\Sigma_r)$, but in view of the complexity that we encounter already at the level of \emph{one} commutation --- see Section~\ref{sec:commuted:redshift} --- we do not expect that a ``higher order global redshift effect'' can be established easily. Note that a \emph{second} commutation by $N$, and a suitable version of \eqref{ineq:commuted:redshift:intro} for ${\tilde{\mathcal{L}}}_N^2 W$, would give pointwise estimates for $W\in \mathrm{L}^\infty$.}

In general, the conformal properties of the Bianchi equations --- see e.g.~\CCh{12.1} --- allow us to define a ``modified Lie derivative'' $\MLie{X}W$ with respect to any vectorfield $X$ of a solution $W$ to  \eqref{eq:W:Bianchi} which satisfies the \emph{inhomogeneous} Bianchi equations
\begin{equation}\label{eq:Bianchi:MLie:W}
  \nabla^\alpha\bigl(\MLie{X}W\bigr)_{\alpha\beta\gamma\delta}={}^{(X)}J(W)_{\beta\gamma\delta}
\end{equation}
where ${}^{(X)}J(W)$ is a ``Weyl current'' which can be expressed in terms of contractions of $\pih{X}$ with $\nabla W$, and contractions of $\nabla\pih{X}$ with $W$; see e.g.~\CProp{12.1}. In the presence of an inhomogeneity in the Bianchi equations, the divergence-free property of the associated Bel-Robinson tensor fails, and \eqref{eq:Q:W:div} is replaced by
\begin{equation}
\begin{split}
    \nabla^\alpha Q[\MLie{X}W]_{\alpha\beta\gamma\delta}=& W_{\beta\phantom{\mu}\delta\phantom{\nu}}^{\phantom{\beta}\mu\phantom{\delta}\nu}{}^{(X)}J_{\mu\gamma\nu}(W)+ W_{\beta\phantom{\mu}\gamma\phantom{\nu}}^{\phantom{\beta}\mu\phantom{\delta}\nu}{}^{(X)}J_{\mu\delta\nu}(W)\\&+\ld W_{\beta\phantom{\mu}\delta\phantom{\nu}}^{\phantom{\beta}\mu\phantom{\delta}\nu}{}^{(X)}J^\ast_{\mu\gamma\nu}(W)+\ld W_{\beta\phantom{\mu}\gamma\phantom{\nu}}^{\phantom{\beta}\mu\phantom{\delta}\nu}{}^{(X)}J^\ast_{\mu\delta\nu}(W)\,.
  \end{split}
\end{equation}
Consequently the energy identity derived from the current
\begin{equation}\label{eq:P:M:X:intro}
  P[\MLie{X}W]^{M}_\alpha=-Q[\MLie{X}W]_{\alpha\beta\gamma\delta}M^\beta M^\gamma M^\delta
\end{equation}
contains an additional ``divergence term'' of the form
\begin{equation}\label{eq:div:X:intro}
  \int_{\mathcal{D}}(\divergence Q[\MLie{X}W])(M,M,M)\dm{g}\,.
\end{equation}
A major difficulty in the proof lies in showing that this contribution to the ``bulk term'' can \emph{also} be arranged to have a sign.

As already indicated above, one could choose $X=M$, but this choice does not succeed.

\begin{remark}
While  it is possible to show that (under suitable assumptions) ``at the highest order of derivatives'' the current \eqref{eq:P:M:X:intro} has the following positivity property:
\begin{multline}
  \int_{\mathcal{D}}(\divergence Q[\MLie{M}W])(M,M,M)\dm{g}\geq \\ \geq \int_{r_1}^{r_2}\frac{4}{r}\int_{\Sigma_r}Q[\MLie{M}W](n,M,M,M)\dm{\gb{r}} \ud r+\int_{\mathcal{D}}\mathcal{E}\dm{g}
\end{multline}
However, the ``lower order terms'' contained in the ``error'' $\mathcal{E}$ --- which are on the level  of $W$ and can in principle be controlled by the energy associated to $P^M[W]$ --- do not decay fast enough towards $\mathcal{I}^+$, for this energy estimate to give the \emph{rate of decay} of the energy associated to $P^M[\MLie{M}W]$ that would be required (in the application of the Sobolev inequality) to prove the $\mathrm{L}^4$-bound stated in the theorem.
\end{remark}

The treatment of the ``commutation'', or ``first order energy'' thus becomes the most complex part of this paper, because it requires us not only to find a sign in the divergence, but also to exhibit \emph{various cancellations in the lower order terms}.\footnote{The cancellations occur as a result of a delicate choice of the commutation vectorfield $N$, rather than particularly strict geometric assumptions. In other words, if the spacetime satisfies the assumptions (\textbf{BA:I,II}) introduced below (see also Appendix~\ref{sec:BA})  then a commutator can be construced as in Section~\ref{sec:commuted:redshift}, and thus the cancellations are stable within this class of spacetimes.} This is the reason why we are forced to compute very carefully most terms contained in \eqref{eq:div:X:intro}, including the signs and prefactors.\footnote{In problems where this circle of ideas --- the treatment of Bianchi equations using energies constructed from the Bel-Robinson tensor --- has been applied, these terms are usually treated as ``error terms''. Notably for the ``stability of Minkowski space'' they have only been written out schematically in \CKCh{8}. Although they are also treated as error terms in proof of the ``formation of black holes'' in \cite{ch:blue}, their precise nature, and ``scaling'', is much more important therein, and we take great advantage in this paper of the fact that at least in special cases precise algebraic expressions for the ``divergence terms'' have already been provided in \CCh{12-14}.}

It turns out that a suitable choice of a commutation vectorfield is given by
\begin{equation}
  X=\Omega^2 M_q
\end{equation}
where
\begin{equation}
  M_q=\frac{1}{2}\frac{1}{\Omega}\Bigl(q\Lbh+q^{-1}\Lh\Bigr)\,,\qquad q= \sqrt{\frac{\overline{\Omega\tr\chi}}{\overline{\Omega\tr\chib}}}\,.
\end{equation}
One can think of the map $M\mapsto M_q$ as induced by a Lorentz transformation with the effect of aligning $M_q$ with the normal $n$ to $\Sigma_r$. The final commutation vectorfield $X$, subsequently denoted by $N$,  is then obtained from $M_q$ by scaling with the weight $\Omega^2$.

In Section~\ref{sec:commuted:redshift} we shall prove the following:
\begin{quote}
  \itshape Suppose the connection coefficients satisfy the assumptions (\textbf{BA:I}) and (\textbf{BA:II}). Then there exists a constant $C>0$, such that for any solution $W$ to \eqref{eq:W:Bianchi}:
  \begin{multline}\label{ineq:commuted:redshift:intro}
     \int_{\mathcal{D}}\divergence (-P)[\MLie{N}W]^{M_q}\,\dm{g} \geq \int_{r_1}^{r_2}\frac{6}{r}\int_{\Sigma_r}Q[\MLie{N}W](n,M_q,M_q,M_q)\dm{\gb{r_1}} \ud r\\
     -\int_{r_1}^{r_2}\frac{C}{r}\int_{\Sigma_r}\frac{1}{\Omega} Q[W](n,M_q,M_q,M_q)\dm{\gb{r}}\ud r\\
     -\int_{r_1}^{r_2}\frac{C}{r} \int_{\Sigma_r}\frac{1}{\Omega^2}\lvert \nablab W \rvert ^2\dm{\gb{r}}\ud r
  \end{multline}
\end{quote}
In this estimate we have achieved that the ``error'' (second term on the r.h.s.) is small --- because $\Omega^{-1}$ is small --- and controlled by the energy associated to $P[\MLie{N}W]^{M_q}$ and $P[W]^{M_q}$, up to terms which only involve \emph{tangential derivatives} to $\Sigma_r$ (third term). 
Next we prove that the latter can be controlled by the first two terms, but it is again a non-trivial statement that in this estimate no ``lower order terms'' appear, which would obstruct the ``redshift'' gained with the positivity of the first term on the r.h.s.~of \eqref{ineq:commuted:redshift:intro}.

The ``electromagnetic decomposition'' of a Weyl field $W$ relative to $\Sigma_r$ --- much like  the decomposition of the Faraday tensor $F$ in electric and magnetic fields $E$, and $H$ relative to a given frame of reference --- recast the Bianchi equations in a system akin to  Maxwell's equations:
\begin{subequations}\label{eq:Maxwell:intro}
  \begin{gather}
    \divb E=H\wedge k
\qquad \widehat{\mathcal{L}_nH}+\curlb E+\frac{1}{2}k\times H=\nablab\log\phi\wedge E\\
    \divb H= -E\wedge k\qquad \widehat{\mathcal{L}_nE}-\curlb H+\frac{1}{2}k\times E=-\nablab\log\phi\wedge H
  \end{gather}
\end{subequations}
In Section~\ref{sec:electromagnetic} we derive an elliptic estimate for this Hodge system on $\Sigma_r$ \footnote{The analysis of these systems has been developed in some generality in Chapter 2-4 of \cite{ch:kl}.} that  allows us to control all tangential derivatives to $\Sigma_r$ by the energy associated to $P[\MLie{N}W]^{(M_q)}$. Here it is essential that the commutator vectorfield $N$ has been aligned with the normal $n$ to $\Sigma_r$, and carries a weight that leads again to exact cancellations with the ``lower order terms'' (on the level of the second fundamental form $k$ of $\Sigma_r$) present in \eqref{eq:Maxwell:intro}.

\begin{quote}
  \itshape Suppose the assumptions (\textbf{BA:I}) and (\textbf{BA:III}) hold. Then there exists a constant $C>0$ such that for all solutions to \eqref{eq:Maxwell:intro},
  \begin{multline}
    \int_{\Sigma_r} \frac{1}{\Omega^2}\lvert \nablab W\rvert^2 \dm{\gb{r}}\leq C \int_{\Sigma_r}\frac{1}{\Omega} Q[\MLie{N}W](n,M_q,M_q,M_q)\dm{\gb{r}}\\
    +C\int_{\Sigma_r}\frac{1}{\Omega}Q[W](n,M_q,M_q,M_q)\dm{\gb{r}}
  \end{multline}

\end{quote}

In conclusion, we obtain under the assumptions (\emph{\textbf{BA:I-III}}), stated with a slight abuse of notation and freely using \eqref{eq:assumption:intro:Omega}, that
\begin{equation}
  \int_{\Sigma_r}\lvert W \rvert^2 + r^2 \lvert \nablab W \rvert^2 \dm{\gb{r}}\lesssim \frac{1}{r^{3}}
\end{equation}
which then easily implies the statement of the theorem by a Sobolev trace inequality on $\Sigma_r$ which we discuss in Section~\ref{sec:sobolev}.

\subsection{Further comments on the assumptions}
\label{sec:optical}

We outline briefly how some of the assumptions (\emph{\textbf{BA:I}}) made in this paper are recovered and refer to \cite{schlue:optical} for a more elaborate discussion in a simplified setting. Recall the task is here to prove the estimates $(\emph{\textbf{BA}})$ under suitable assumptions on the initial data, using the bounds on the Weyl curvature established in this paper. In \cite{schlue:optical} this is carried out in a drastically ``simplified'' setting where the Weyl curvature not only decays, but vanishes identically.\footnote{This in turn is motivated by the insight that some of the main difficulties in obtaining a global double null foliation with the desired properties already occur in de Sitter space, and are less related to the specific decay properties of the Weyl curvature.} It is also ``restricted'' in the sense that the initial data on one of the cosmological horizons is shear-free.\footnote{This restriction is of some interest for the full problem, too: While in the formulation of page~\pageref{quote:stability:sds:goursat} non-trivial data is posed on both cosmological horizons (and this is certainly the relevant setting for the non-linear stability problem) the physically relevant setting --- namely the specific setting in which physical meaning can be ascribed to the asymptotic quantities --- is the special case when a \emph{no incoming radiation condition} is imposed on $\overline{\mathcal{C}}^+$; see also~\cite{ashtekar:I}.}

Among the most important assumptions are \eqref{eq:assumption:intro:tr}, \eqref{eq:assumption:intro:omega} and \eqref{eq:BA:I:vi:e}. Note that while $2\omega$ and $\Omega\tr\chi$ are both linearly growing in $r$, their difference is assumed to be bounded. To prove this, one considers the difference of the propagation equations for $2\omega$, and $\Omega\tr\chi$, namely the equations for $\Lb(2\omega)$ and $\Lb(\Omega\tr\chi)$ which are at the level of curvature, and one finds  using the Einstein equations \eqref{eq:eve:lambda:intro} that\footnote{We already see that the Weyl curvature on the right hand side is \emph{not} the slowest decaying term.}
\begin{equation}\label{eq:Lb:omega:trchi}
    \Lb\bigl(2\omega-\Omega\tr\chi\bigr)   = -2\Omega^2\Bigl(2\rho[W]-\frac{1}{2}\bigl(\chib,\chi)+\divs\eta+\lvert \eta\rvert^2 -2(\eta,\etab)+\lvert\etab\rvert^2+\frac{\Lambda}{3}\Bigr) \\
    %= -2\Omega^2\Bigl(3\rho[W]-\bigl(\chibh,\chih)+K+\divs\eta+\lvert \eta\rvert^2 -2(\eta,\etab)+\lvert\etab\rvert^2\Bigr)
\end{equation}
Now in \cite{schlue:optical} we define the \emph{mass aspect function} by
    \begin{equation}\label{eq:mu}
        \mu :=-\rho[W]+\frac{1}{2}(\chih,\chibh)-\divs\eta= K-\divs\eta+\frac{1}{4}\tr\chi\tr\chib-\frac{\Lambda}{3}
    \end{equation}
where in the second equality we have used the Gauss equation \eqref{eq:gauss} which relates the Gauss curvature $K$ of $(S_{u,v},\gs)$ to the ambient Weyl curvature. Therefore
\begin{equation}
      \Lb\bigl(2\omega-\Omega\tr\chi\bigr)   = -2\Omega^2\Bigl(\rho[W]-\mu+\lvert \eta\rvert^2 -2(\eta,\etab)+\lvert\etab\rvert^2-\frac{1}{4}\tr\chi\tr\chib+\frac{\Lambda}{3}\Bigr)\,,
    \end{equation}
    thus reducing the boundedness of $2\omega-\Omega\tr\chi$ to decay properties of $\rho[W]$, and $\mu$, and also $\eta$, $\etab$, and $(\Lambda/3-\tr\chi\tr\chib/4)$.
The mass aspect function plays an equally central role as in \cite{ch:kl}; (cf.~Section~5.3 in \cite{schlue:optical}): It satisfies a ``good'' propagation equation in the sense that one can hope to prove that $\mu=\mathcal{O}(r^{-3})$, provided various final gauge choices are satisfied.
Once such a bound on $\mu$ is obtained, the idea is to view the definition of \eqref{eq:mu} as part of an elliptic system on $S_{u,v}$ for the torsion $\eta$:
\begin{subequations}
\begin{align}
  \divs\eta=&-\rho[W]+\frac{1}{2}(\chih,\chibh)-\mu\\
  \curls\eta=&-\frac{1}{2}\chih\wedge\chibh+\sigma[W]
\end{align}
\end{subequations}
The assumption \eqref{eq:BA:I:vi:e} on $\eta$ can then be recovered using elliptic estimates --- which hold under sufficient control on the isoperimetric constant of $(S_{u,v},\gs)$ --- and the established bounds on the Weyl curvature; (cf.~Section 6, 7 in \cite{schlue:optical}).

A key remaining difficulty in the recovery is the above mentioned ``final gauge choice'', which needs to be addressed in order to be able to choose various boundary values for the propagation equations of the structure coefficients. More precisely, given a spacetime domain $\mathcal{D}=\cup_{r_1\leq r\leq r_2}\Sigma_r=\cup_{r_1\leq r(u,v)\leq r_1}S_{u,v}$ it involves the geometric construction of a family of new spheres $S_{u,v}'$ in $\mathcal{D}$, such that $\mu$, $\tr\chi$, and $\tr\chi$ assume specific values; for a solution to this problem in a different setting see \cite{KS:GCM}.

\subsection{Relation to earlier work}
\label{sec:intro:earlier}

We have already mentioned the work of Hintz and Vasy on the non-linear stability of the Kerr-de Sitter family in the black hole exterior on the domain bounded by the cosmological horizon \cite{hi:vasy:stability}.\footnote{While the result is very relevant for this paper, the ideas of their proof are entirely different: In the stationary region $\mathcal{S}$ a ``resonance expansion'' is available for solutions to the linearized equations, and the authors succeed in relating all growing modes to pure gauge solutions; cf.~Section~4 in \cite{hi:vasy:stability}. } Another important result to be mentioned in this context is the work of Friedrich on the stability of the de Sitter spacetime \cite{friedrich:desitter}. We will discuss here briefly its relevance to the stability problem for \emph{Schwarzschild-}de Sitter cosmologies. Finally we will mention the work of Ringstr{\"o}m \cite{ringstrom:invent}, and Rodnianski and Speck \cite{speck:stabilize,speck:FLRW:euler,rod:speck:FLRW:euler}.

In \cite{friedrich:desitter} Friedrich proved that the future development of Cauchy data on $\mathbb{S}^3$ is \emph{geodesically complete}, provided the initial data is ``a small perturbation'' of the datum induced by the de Sitter solution. Now with regard to the initial data induced by a Schwarzschild de Sitter solution, a Cauchy hypersurface $\Sigma$ as in Fig.~\ref{fig:Cauchy} \emph{cannot} be expressed as a ``perturbation'' of de Sitter data. However, a \emph{truncation} $[\Sigma]$ of $\Sigma$ \emph{away from the event horizons} could be viewed as a perturbation of a suitable ``segment'' of de Sitter data, at least for small mass $0<m\ll 1/(3\sqrt{\Lambda})$; see Fig.~\ref{fig:desitter:stability} (right). It is plausible that the resulting data on $[\Sigma]\simeq [\delta,\pi-\delta]\times\mathbb{S}^2$ can be glued to the ``spherical caps'' $[0,\delta]\times\mathbb{S}^2$ of an $\mathbb{S}^3$, to obtain an admissible initial data set for \cite{friedrich:desitter}; see Fig.~\ref{fig:desitter:stability} (left).\footnote{The author is not aware of paper where this procedure is carried out in detail, but one would expect that the outlined truncation/gluing steps can be achieved using known techniques; see e.g.~\cite{isenberg:pollack}.} The result of Friedrich would then yield a geodesically complete spacetime, which agrees with the future development of $\Sigma$ \emph{on the domain of dependence} of $[\Sigma]$; see Fig.~\ref{fig:desitter:stability} (shaded). Thus \cite{friedrich:desitter} could be used to show the stability of regions $\mathcal{D}_F$ realized as the past of spatially compact segments $\Sigma_c^+$ of $\mathcal{I}^+$,
\begin{equation}
  \mathcal{D}_F\subset I^-(\Sigma^+_c)\qquad \Sigma_c^+\subset \mathcal{I}^+\simeq \mathbb{R}\times\mathbb{S}^2
\end{equation}
Notably, such an argument cannot achieve a stability statement ``in a neighborhood of time-like infinity $\iota^+$;'' see Fig.~\ref{fig:desitter:stability}.

\begin{figure}
  \centering
  \includegraphics{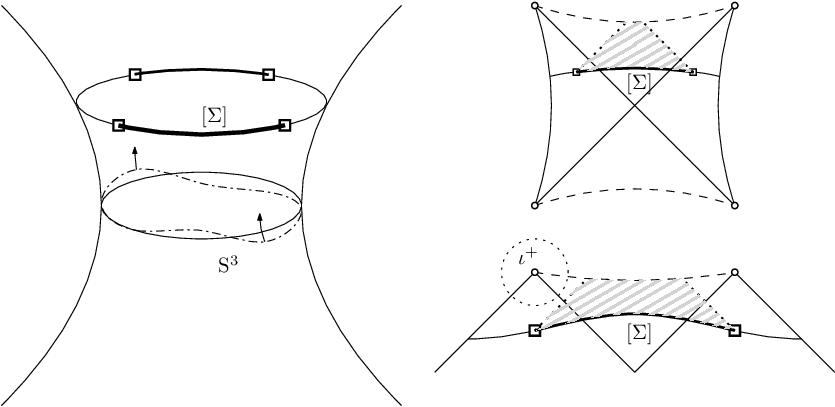}
  \caption{\emph{Bottom right:} Truncation $[\Sigma]$ of Cauchy hypersurface $\Sigma$ away from the event horizons, and domain of dependence of $[\Sigma]$ (shaded). \emph{Top right:} $[\Sigma]$ viewed as ``perturbation'' of a segment of a hypersurface in de Sitter spacetime. \emph{Left:} Sketch of initial data for Friedrich's theorem (dash-dotted); Initial data corresponding to $[\Sigma]$ as a subset of $\mathbb{S}^3$ (bold, between ``truncation spheres'' depicted as squares).}
  \label{fig:desitter:stability}
\end{figure}

Another approach to implement the conformal method in the Schwarzschild de Sitter  setting has been pursued in \cite{valiente}. There it is found that initial data for the conformal field equations on a $\Sigma_r$ hypersurface can only be constructed under strong fall off assumptions to trivial data which amount to excluding the points $\iota^+$ from the analysis. This still yields a ``localised'' result in the sense of Remark~\ref{remark:localised}, and gives an interesting discussion of the geometric data on $\mathcal{I}^+$.

While the results in \cite{friedrich:desitter} are closely related to the conformal properties of \eqref{eq:eve:lambda:intro}, and achieve a global existence result by a reduction to a ``local in time'' problem, Ringstr\"om  provided   a treatment of the ``Einstein-non-linear scalar field system'' --- which includes the Einstein vacuum equations with positive cosmological constant as a special case --- that reproves the results in \cite{friedrich:desitter}, without resorting to a ``conformal compactification'', and without specific reference to the topology of the initial data \cite{ringstrom:invent}. In fact, the set-up in \cite{ringstrom:invent} exploits a causal feature of ``accelerated expansion'' already evident from the Penrose diagram of de Sitter spacetime, cf.~Fig.~\ref{fig:desitter:stability} (right): Consider a spacelike Cauchy hypersurface $\Sigma$ in de Sitter spacetime, let $[\Sigma]\subset\mathcal{R}$ be a truncation of $\Sigma$ contained in the expanding region, and $p\in [\Sigma]$. Choose $R>0$ such that $B_{4R}(p)\subset[\Sigma]$, then the future of $B_R(p)$ is contained in the domain of dependence of $B_{4R}(p)$, $I^+(B_R(p))\subset \mathcal{R}\setminus I^+(\Sigma\setminus B_{4R}(p))$, provided $\Sigma$ is ``at sufficiently late time'', e.g.~ if $\min_{[\Sigma]} r\gg \rc$. This allows Ringstr\"om to prove ``global in time'' results, from ``local in space'' assumptions on the inital data, which cover in particular perturbations of the de Sitter solution, but are not restricted to the $\mathbb{S}^3$ topology. Notably, the ``asymptotic expansions'' of Theorem~2 in \cite{ringstrom:invent} show the existence of ``asymptotic functional degrees of freedom'', namely that the solution converges to a metric which after rescaling by the expected behavior in time differs from the rescaled de Sitter metric, even at the leading order parametrized by a free ``profile'' function. \footnote{Also \cite{ringstrom:invent} stops short of relating the ``asymptotic profiles'' to the details of the initial data, which is largely an open problem; see also \cite{rendall:asymptotics}.}

This paper does not yet give a full global existence theorem for solutions to \eqref{eq:eve:lambda:intro} on the level of \cite{friedrich:desitter,ringstrom:invent,hi:vasy:stability}. It does however accomplish what one expects to be an essential step towards the stability of the expanding region of \emph{Schwarzschild-} de Sitter cosmologies:
We show that the Weyl curvature decays under sufficiently general assumptions --- roughly corresponding to Part~II of the original proof of the non-linear stability of Minkowski space \cite{ch:kl}. %\footnote{In \cite{bieri} --- which has weakened the assumptions in \cite{ch:kl}, and simplified some aspects of the proof, in particular the need to work with \emph{second} derivatives of the curvature --- Bieri presented several ideas which will allow us to prove a full existence result based on the bounds of the Weyl curvature obtained in this paper.}

The underlying decay mechanism --- namely the expansion of spacetime --- has also played a prominent role in the work of Speck on Friedman-Lema{\^\i}tre-Robertson-Walker cosmologies: They were shown to be future stable in \cite{rod:speck:FLRW:euler, speck:FLRW:euler, had:speck:FLRW:dust} as solutions to the Euler-Einstein system, and it was observed in particular that the ``de Sitter''-like expansion prevents the formation of shocks in relativistic fluids (with linear barotropic equation of state) \cite{speck:stabilize}. For stiff fluids Rodnianski and Speck also showed stable ``big bang'' singularity formulation in the past \cite{rod:speck:bang, rod:speck:bang:linear}. Some elements of their proof --- in particular the existence of a monotone energy at the level of the commuted equations, the resulting smallness of the Weyl curvature, and the functional degrees of freedoms associated to all possible ``end states'' --- bear some resemblance to the approach pursued in this paper. \footnote{Albeit in \cite{rod:speck:bang} a constant mean curvature gauge is used ``to synchronize  the singularity'', while in this work we rely on a double null gauge to ``parametrize null infinity''.}

Finally while we do use \cite{ch:blue} as our primary reference throughout for quoting various formulas related to the double null formalism, many of the central propositions in particular related to the construction of currents of course first appeared in \cite{ch:kl} and \cite{ch:kl:lin}.

%%% Local Variables:
%%% mode: latex
%%% TeX-master: "weyl"
%%% End:

\begin{quote}
  \textbf{Acknowledgements.} I would like to thank Mihalis Dafermos for drawing my attention to this problem during my Ph.D.~and for his continued encouragement and support.
  I would also like to thank Abhay Ashtekar, and Lydia Bieri for many stimulating discussions at a conference at the Tsinghua Sanya International Mathematics Forum, China, in January 2016  and Gustav Holzegel for several useful comments following a talk at Imperial College London in May 2016.
  The author gratefully acknowledges the support of \emph{ERC consolidator Grant 725589 EPGR}, \emph{ERC advanced grant 291214 BLOWDISOL}, in the years 2016-18, as well as the support of the \emph{Fondation Sciences Math\'ematiques de Paris} in 2015/16.
\end{quote}

%\section{Geometry of the de Sitter solution}
%\label{sec:desitter:geometry}

%\input{desitter-geometry-basic}

%\subsection{Double null foliations of de Sitter}
%\label{sec:desitter:double}

%\input{desitter-geometry-double}

\section{Einstein's equations with cosmological constant}
\label{sec:eve:lambda}

The Einstein vacuum equations with positive cosmological constant $\Lambda>0$ are
\begin{equation}
  \label{eq:eve:lambda}
  \Ric(g)=\Lambda g
\end{equation}

\subsection{Weyl curvature}
\label{sec:weyl}

Presently we shall focus on the conformal decomposition of the curvature tensor of a $3+1$-dimensional spacetime manifold, which plays an important role in this context.

Recall the \emph{Schouten} tensor 
\begin{equation}
  P_{\alpha\beta}=\frac{1}{2}\Bigl[\Ric_{\alpha\beta}-\frac{R}{6}g_{\alpha\beta}\Bigr]
\end{equation}
where $R$ denotes the scalar curvature; see e.g.~\cite{fefferman:conformal}.
We observe that for any solution to \eqref{eq:eve:lambda} the Schouten tensor is simply
\begin{equation}
  P_{\alpha\beta}=\frac{\Lambda}{6}g_{\alpha\beta}
\end{equation}

The \emph{Weyl} curvature $W$, in general, is defined by \cite{fefferman:conformal}
\begin{equation}
  W_{\alpha\beta\mu\nu}=R_{\alpha\beta\mu\nu}+\Bigl[P_{\beta\mu}g_{\alpha\nu}+P_{\alpha\nu}g_{\beta\mu}-P_{\beta\nu}g_{\alpha\mu}-P_{\alpha\mu}g_{\beta\nu}\Bigr]
\end{equation}
which for solutions to \eqref{eq:eve:lambda} then reduces to:
\begin{equation}\label{eq:weyl:lambda}
  W_{\alpha\beta\mu\nu}=R_{\alpha\beta\mu\nu}+\frac{\Lambda}{3}\Bigl[g_{\alpha\nu} g_{\beta\mu}-g_{\alpha\mu}g_{\beta\nu}\Bigr]
\end{equation}
Note that $W$ has the same algebraic symmetries as the curvature tensor $R$, and in addition is totally trace-free.
We shall thus proceed in Section~\ref{sec:weyl:decomposition} with the null decompositon of the \emph{Weyl} curvature.

\subsection{Null decomposition of the Weyl curvature}
\label{sec:weyl:decomposition}

We have already referred to the symmetries of the Weyl curvature. 
We note that the Weyl curvature \eqref{eq:weyl:lambda} is a ``Weyl field'' in the sense of \CCh{12}:
It is anti-symmetric in the first two and last two indices, and satisfies the cyclic identity:
\begin{equation}
  W_{\alpha[\beta\mu\nu]}=0
\end{equation}
Moreover, the Weyl curvature satisfies the trace conditon:
\begin{equation}
  g^{\alpha\mu}W_{\alpha\beta\mu\nu}=\Ric_{\beta\nu}+\frac{\Lambda}{3}\Bigl[g_{\beta\nu}-4g_{\beta\nu}\Bigr]=0
\end{equation}

The dual of $W$ is defined by (as we know, left and right duals coincide)
\begin{equation}
  W_{\alpha\beta\mu\nu}^\ast=\frac{1}{2}W_{\alpha\beta}^{\phantom{\alpha\beta}\gamma\delta}\epsilon_{\gamma\delta\mu\nu}
\end{equation}

There are 10 algebraically independent components of a Weyl field.
Let $(e_A:A=1,2;e_3,e_4)$ be an orthonormal null frame field.
Then the 2-covariant  tensorfields
\begin{equation}
  \alphab_{AB}[W]=W_{A3B3} \quad \alpha_{AB}[W]=W_{A4B4}
\end{equation}
account for 2 components each, because they are symmetric and trace-free:
\begin{equation}
  g^{AB}\alphab_{AB}=g^{AB}W_{A3B3}=g^{\mu\nu}W_{\mu 3\nu 3}=0
\end{equation}
Also the 1-forms
\begin{equation}
  \betab_A[W]=\frac{1}{2}W_{A334}\quad \beta_A[W]=\frac{1}{2}W_{A434}
\end{equation}
account for 2 components each, which leaves us with 2 functions
\begin{equation}
  \rho[W]=\frac{1}{4}W_{3434}\quad \sigma[W] \epsilon_{AB}=\frac{1}{2}W_{AB34}
\end{equation}

Note that with \eqref{eq:weyl:lambda} we have
\begin{gather}
  \alphab_{AB}[W]=R_{A3B3}\qquad \alpha_{AB}[W]=R_{A4B4}\\
  \betab_A[W]=R_{A334}\qquad \beta_A[W]=R_{A434}\\
  \rho[W]=\frac{1}{4}R_{3434}+\frac{\Lambda}{3}\qquad \sigma[W]\epsilon_{AB}=\frac{1}{2} R_{AB34}
\end{gather}
Here we used $g_{33}=g_{3A}=g_{44}=g_{4A}=0$, and $g_{34}=-2$.

Thus the only component that differs from the corresponding null decompositon of the curvature tensor $R$ (which is \emph{only} a Weyl field in the case $\Lambda=0$) is $\rho$. %We shall thus drop the notation $[W]$ from the null components, as the discrepancy only lies in one component which we will keep in mind.

Note that $\sigma[W]$ can equally be defined by
\begin{equation}
  \begin{split}
  \sigma[W]&=\rho[{}^\ast W]=\frac{1}{4}{}^\ast W_{3434}=\frac{1}{4}W_{34}^{\phantom{34}\alpha\beta}\epsilon_{\alpha\beta34}\\
  &=\frac{1}{4}W_{34}^{\phantom{34}12}+\frac{1}{4}W_{34}^{\phantom{34}21}(-1)=\frac{1}{2}R_{3412}
  \end{split}
\end{equation}

The remaining components of $W$ are expressed as, cf. \Ceq{12.34},
\begin{subequations}\label{eq:W:remaining:components}
\begin{align}
  W_{A3BC}&=g_{AB}\betab_C[W]-g_{AC}\betab_B[W]=\frac{1}{2}g_{AB}R_{C334}-\frac{1}{2}g_{AC}R_{B334}\\
  W_{A4BC}&=-g_{AB}\beta_C[W]+g_{AC}\beta_B[W]=-\frac{1}{2}g_{AB}R_{C434}+\frac{1}{2}g_{AC}R_{B434}\\
  W_{A3B4}&=-\rho[W]g_{AB}+\sigma[W]\epsilon_{AB}=-\frac{1}{4}R_{3434}\, g_{AB}-\frac{\Lambda}{3}g_{AB}+\frac{1}{2}R_{AB34}\\
  W_{ABCD}&=-\rho[W]\epsilon_{AB}\epsilon_{CD}=-\Bigl(\frac{1}{4}R_{3434}+\frac{\Lambda}{3}\Bigr)\bigl(g_{AC}g_{BD}-g_{AD}g_{BC}\bigr)
\end{align}
\end{subequations}

We also use the notation $\ld\alphab$, $\ld \betab$ for the left duals of $\alphab$ and $\alpha$, respectively, which are related to null decomposition of $\ld W$ according to \Ceq{12.35}:
\begin{equation}
  \ld \alphab(W)=\alphab(\ld W)\qquad \ld\betab (W)=\betab(\ld W)\,.
\end{equation}
In terms of an orthonormal frame $(e_A:A=1,2)$:
\begin{equation}
  \ld \alphab_{AB}= \epsilons_{A}^{\sharp C}\alpha_{CB}\qquad \ld \betab_A=\epsilons_A^{\sharp B}\beta_B 
\end{equation}
where $2\epsilons_{AB}=\epsilon_{AB34}$. Analogous definitions apply to the left duals of $\alpha$, and $\beta$. We note that as $\alphab$, and $\alpha$, the 2-covariant tensorfields $\ld\alphab$, and $\ld \alpha$ are symmetric and trace-free. 

\subsection{Bianchi identities}
\label{sec:bianchi}

Recall that in general the curvature tensor $R$ satisfies the Bianchi identities:
\begin{equation}\label{eq:bianchi}
  \nabla_{\mu}R^\alpha_{\phantom{\alpha}\beta\nu\lambda}+\nabla_\nu R^\alpha_{\phantom{\alpha}\beta\lambda\mu}+\nabla_\lambda R^\alpha_{\phantom{\alpha}\beta\mu\nu}=0
\end{equation}
which, by setting $\alpha=\nu$ and summing, yields the contracted Bianchi identies:
\begin{equation}
  \nabla_\alpha R^\alpha_{\phantom{\alpha}\beta\lambda\mu}=\nabla_\lambda R_{\mu\beta}-\nabla_\mu R_{\lambda\beta}
\end{equation}
Schematically, these equations say
\begin{equation}
  \divergence \Riem =\curl \Ric
\end{equation}
But here, of course, for any solution to the vacuum equations with positive cosmological constant, $\Ric(g)=\Lambda g$, and
\begin{equation}
  \nabla_\lambda g_{\mu\beta}=0
\end{equation}
by metric compatibility of the connection. Thus, as in the case $\Lambda=0$,
\begin{equation}
  \nabla_\alpha R^{\alpha}_{\phantom{\alpha}\beta\nu\lambda}=0
\end{equation}
This implies now that for a solution to \eqref{eq:eve:lambda} also the Weyl curvature is divergence free:
\begin{gather}
  W^\alpha_{\phantom{\alpha}\beta\mu\nu}=R^\alpha_{\phantom{\alpha}\beta\mu\nu}+\frac{\Lambda}{3}\Bigl[\delta^\alpha_{\nu} g_{\beta\mu}-\delta^\alpha_{\mu}g_{\beta\nu}\Bigr]\label{eq:weyl:eve:lambda:raised}\\
  \nabla_\alpha W^\alpha_{\phantom{\alpha}\beta\mu\nu}=\frac{\Lambda}{3}\Bigl[\nabla_\nu g_{\beta\nu}-\nabla_\mu g_{\beta\nu}\Bigr]=0\label{eq:bianchi:weyl:div}
\end{gather}
With the same formula, \eqref{eq:weyl:eve:lambda:raised}, we see that the ``Bianchi identity'' \eqref{eq:bianchi} is also true for the Weyl curvature:
\begin{equation}
  \nabla_{\mu}W^\alpha_{\phantom{\alpha}\beta\nu\lambda}+\nabla_\nu W^\alpha_{\phantom{\alpha}\beta\lambda\mu}+\nabla_\lambda W^\alpha_{\phantom{\alpha}\beta\mu\nu}=0
\end{equation}
or, for short, the homogeneous Bianchi equations hold:
\begin{equation}\label{eq:bianchi:weyl:hom}
  \nabla_{[\mu}W_{\nu\lambda]\phantom{\alpha}\beta}^{\phantom{\nu\lambda]}\alpha}=0
\end{equation}
This is consistent with general principles, according to which the equation \eqref{eq:bianchi:weyl:hom}, for the Weyl field $W$, also written as
\begin{equation}
  D W =0
\end{equation}
is \emph{equivalent} to
\begin{equation}
  D {}^\ast W=0\,,
\end{equation}
by the symmetries of a Weyl field, which are precisely the equations \eqref{eq:bianchi:weyl:div}.

Now we are in the situation where we have a Weyl field $W$ satisfying %the ``inhomogeneous'' Bianchi equations
\eqref{eq:bianchi:weyl:div}:
\begin{equation}
  \label{eq:bianchi:weyl:inhom}
  \nabla_\alpha  W^\alpha_{\phantom{\alpha}\beta\mu\nu}=0
\end{equation}
Therefore by \CProp{12.4} the null decomposition of \eqref{eq:bianchi:weyl:inhom} takes precisely the form given therein.
In other words, the Bianchi equations are \emph{verbatim} those of the vacuum equations in the case $\Lambda=0$, with the understanding that the null components refer to the null decomposition of the Weyl curvature tensor.%, which really only differs from the decomposition of the Riemann curvature in the single component $\rho$.

%%% Local Variables:
%%% mode: latex
%%% TeX-master: "weyl"
%%% End:

\subsection{Double null gauge}
\label{sec:null}

We follow the conventions of \CCh{1}, for the definition of the double null foliation, and all associated geometric quantities.

We have already introduced  in Section~\ref{sec:intro:foliation} the \emph{optical functions} $u$, $v$ as solutions to the eikonal equations \eqref{eq:eikonal} such that the surfaces of intersection of the level sets of $u$, $v$,  \eqref{eq:spheres}, are spheres diffeomorphic to $\mathbb{S}^2$. Null geodesic normals $\Lbp$, and $\Lp$ are introduced as in \eqref{eq:Lp}, whose components in any coordinate system are
\begin{equation}
  \label{eq:Lp:components}
  {\Lp}^\mu=-2(g^{-1})^{\mu\nu}\partial_\nu u \qquad   {\Lbp}^\mu=-2(g^{-1})^{\mu\nu}\partial_\nu v
\end{equation}
and with their help we have introduced the null lapse function $\Omega$ in \eqref{eq:null:lapse}. Note that with
\begin{equation}\label{eq:L}
  \Lb =\Omega^2\Lbp\qquad L=\Omega^2\Lp
\end{equation}
we have
\begin{subequations}
  \begin{gather}
    \Lb u=1\qquad \Lb v=0\\
    L u=0\qquad L v=1
  \end{gather}
\end{subequations}
and
\begin{equation}
  g(L,\Lb)=-2\Omega^2
\end{equation}
Moreover, local coordinates on $S_{u,v}$ are introduced as in \CCh{1.4}: We choose coordinates $(\vartheta^1,\vartheta^2)$ on $S_{0,0}$, which are then transported first along the geodesics generated by $\Lbp$ on $\Cb$, and then along the null geodesics generated by $\Lp$ on $C_u$. In this ``canonical coordinate system'' the metric takes the form \eqref{eq:g:intro:double}.

\subsubsection{Area radius} 
\label{sec:area:radius}

Recall that we have already introduced in \eqref{eq:area:intro} the area radius $r(u,v)$ of $S_{u,v}$.
Since, by definition
\begin{equation}\label{eq:r:area}
  4\pi r^2(u,v)=\int_{S_{u,v}}\dm{\gs}=\int_{S_{u,0}}\Phi_v^\ast\dm{\gs}
\end{equation}
where $\Phi_v$ is the 1-parameter group generated by $L$, we have
\begin{subequations}
\begin{equation}
  D (4\pi r^2)=\int_{S_{u,s}}\Omega\tr\chi\dm{\gs}\\
\end{equation}
where by definition $Df=Lf$ for any function $f$, and thus
\begin{gather}
  Dr=\frac{r}{2}\overline{\Omega\tr\chi}\label{eq:Dr}\\
  \Db r=\frac{r}{2}\overline{\Omega\tr\chib}
\end{gather}
\end{subequations}
where $\overline{\cdot}$ denotes the average of a function on the sphere
\begin{equation}
  \overline{f}(u,v) :=\frac{1}{4\pi r^2(u,v)}\int_{S_{u,v}}f(u,v,\vartheta^1,\vartheta^2)\dm{\gs}(\vartheta^1,\vartheta^2)\,.
\end{equation}

\subsubsection{Optical structure coefficients}

We have already introduced the normalised null normals $(\Lbh,\Lh)$ in \eqref{eq:Lh}, and the null structure coefficients $\omegabh$, and $\omegah$ in \eqref{eq:omegah}. The null normals $(\Lb,L)$ defined in \eqref{eq:L} then satisfy
\begin{equation}
  \label{eq:omega}
  \nabla_L L=2\omega L\qquad \nabla_{\Lb}\Lb =2\omegab\Lb
\end{equation}
where
\begin{subequations}
\begin{gather}\label{eq:D:log:Omega}
  \omega=D\log \Omega\qquad \omegab=\Db\log\Omega\\
  \omegah=\frac{1}{\Omega}\omega\qquad  \omegabh=\frac{1}{\Omega}\omegab\,.
\end{gather}
\end{subequations}

The null second fundamental forms $\chi$ and $\chib$ are defined in \eqref{eq:chi}, and its trace-free parts are:
\begin{equation}
  \chih=\chi-\frac{1}{2}\gs\tr\chi\qquad \chibh=\chib-\frac{1}{2}\gs\tr\chib\,.
\end{equation}
We frequently adopt the musical notation (${}^\sharp$) to indicate ``raising'' an index with $\gs$, for instance:
\begin{equation}
  (\chih^{\sharp\sharp})^{AB}=(\gs^{-1})^{AC}(\gs^{-1})^{BD}\chih_{CD}
\end{equation}

We record here for future reference the Gauss equation of the embedding of $S_{u,v}$ in the spacetime, which relates the Gauss curvature $K$ of $(S_{u,v},\gs)$ to (the component $\rho$ of) the  ambient Weyl curvature $W$:
\begin{equation}\label{eq:gauss}
    K+\frac{1}{4}\tr\chi\tr\chib-\frac{\Lambda}{3}=-\rho[W]+\frac{1}{2}(\chih,\chibh)
\end{equation}

In addition to the torsion $\zeta$ defined in \eqref{eq:zeta}, one may also define a notion of torsion with respect to the null geodesic normals:
\begin{equation}
  \eta(X)=\frac{\Omega^2}{2}g(\nabla_X\Lp,\Lbp)\qquad   \etab(X)=\frac{\Omega^2}{2}g(\nabla_X\Lbp,\Lp)\qquad :X\in T_{S_{u,v}}
\end{equation}
which is related to $\zeta$ via
\begin{equation}
  \eta=\zeta+\ds\log\Omega\qquad \etab=-\zeta+\ds\log\Omega\,.
\end{equation}

With the above notation for the structure coefficients we can then express the frame relations as follows:
\begin{subequations}\label{eq:normalised:frame:relations}
\begin{gather}
  \nabla_X\Lh=-\zeta(X)\Lh+\chi^\sharp\cdot X\,,\qquad \nabla_X\Lbh=\zeta(X)\Lbh+\chib^\sharp\cdot X\qquad : X\in \mathrm{T}S_{u,v}\\
  \nabla_{\Lh}\Lbh=-\omegah\Lbh+2\etab^\sharp\,,\qquad \nabla_{\Lbh}\Lh=-\omegabh\Lh+2\eta^\sharp\,, 
\end{gather}
\end{subequations}
where again $\eta^\sharp$ and $\etab^\sharp$ denote the ($S_{u,v}$-tangent) vectorfields corresponding to the 1-forms $\eta$ and $\etab$, respectively: $\etab^{\sharp A}=\gs^{AB}\etab_B$, $\eta^{\sharp A}=\gs^{AB}\eta_B$, where $e_A:A=1,2$ is an arbitrary basis for $\mathrm{T}S_{u,v}$, and moreover
\begin{equation}
  \chi^\sharp\cdot e_A= \chi_A^{\sharp C} e_C= (\gs^{-1})^{CB}\chi_{AB}e_C
\end{equation}

For detailed discussion of the formulas see  \Ceq{1.69, 1.79, 1.86, 1.87}.

%%% Local Variables:
%%% mode: latex
%%% TeX-master: "weyl"
%%% End:

\subsection{Areal time function}
\label{sec:areal}

Besides the double null gauge, which is particularly suited for the characteristic initial value problem, other gauges typically involve the choice of a \emph{time function}. 
This concept appears here naturally in the form of the area radius which is increasing towards to future --- this is one manifestation of the expansion of the cosmological region. However, while we do use the decomposition of the Einstein equations relative to a given time-function, we do not impose an equation on its level sets, such as in \cite{ch:kl}, or \cite{rod:speck:bang,rod:speck:bang:linear}; here the time function is chosen once the double null foliation is fixed.

Given an ``areal time function'' $r$,
we define
\begin{equation}\label{eq:def:V}
  V^\mu=-g^{\mu\nu}\partial_\nu r
\end{equation}
and the associated lapse function by
\begin{equation}\label{eq:def:phi}
  \phi=\frac{1}{\sqrt{-g(V,V)}}
\end{equation}
Then the unit normal to the level sets of $r$, $\Sigma_r$ is 
\begin{equation}
  n=\phi V
\end{equation}
%\begin{equation}
  % S:=\phi^2 V=\phi n
%\end{equation}

In the following it will be useful to express these in terms of quantities associated to the double null foliation:

\begin{lemma}\label{lemma:normal:r}
  The lapse function of the foliation by level sets $\Sigma_r$ is
  \begin{equation}\label{eq:phi}
    \phi=\frac{2}{r}\frac{\Omega}{\sqrt{\overline{\Omega\tr\chi}\,\overline{\Omega\tr\chib}}}
  \end{equation}
  and the normal $n$ to each leaf $\Sigma_r$ is given by
  \begin{equation}\label{eq:normal:q}
    n=\frac{1}{2}\Bigl(q\Lbh+q^{-1} \Lh)\qquad \text{where } q:=\sqrt{\frac{\overline{\Omega\tr\chi}}{\overline{\Omega\tr\chib}}}
  \end{equation}
 
\end{lemma}

\begin{proof}
%   In general,
%   \begin{equation}
%   n=\phi V
% \end{equation}
% where $V$ is the gradient vectorfield, and $\phi$ is the lapse,
% \begin{equation}
%   V^\mu=-g^{\mu\nu}\partial_\nu r\qquad \phi=\frac{1}{\sqrt{-g(V,V)}}
% \end{equation}
Here we need explicit expressions for the components of the inverse:
\begin{equation}
  \begin{split}
    g^{-1} &=-\frac{1}{2}\Lbh\otimes \Lh-\frac{1}{2}\Lh\otimes \Lbh+\sum_{A=1}^2 E_A\otimes E_A\\
    &=-\frac{1}{2}\frac{1}{\Omega^2}\Lb\otimes L-\frac{1}{2}\frac{1}{\Omega^2} L\otimes \Lb+\sum_{A=1}^2E_A\otimes E_A
  \end{split}
\end{equation}
so in particular
\begin{equation}
  g^{uv}=-\frac{1}{2}\frac{1}{\Omega^2}\qquad g^{uA}=-\frac{1}{2}\frac{1}{\Omega^2}b^A\qquad g^{vA}=0\,.
\end{equation}
This yields
\begin{gather}
  V^u=-g^{uv}\partial_v r=\frac{1}{4}\frac{1}{\Omega^2}r\overline{\Omega\tr\chi}\qquad V^v=\frac{1}{4}\frac{1}{\Omega^2}r\,\overline{\Omega\tr\chib}\\
  V^A=-g^{Au}\partial_u r=\frac{1}{4}\frac{1}{\Omega^2}b^A r\,\overline{\Omega\tr\chib}
\end{gather}
or 
\begin{equation}
  V=\frac{r}{4}\frac{1}{\Omega}\Bigl(\overline{\Omega\tr\chi}\,\Lbh+\overline{\Omega\tr\chib}\, \Lh\Bigr)
\end{equation}
This implies
\begin{equation}\label{eq:gVV}
  g(V,V)=-\frac{r^2}{4}\frac{1}{\Omega^2}\,\overline{\Omega\tr\chi}\,\overline{\Omega\tr\chib}
\end{equation}
and thus the statement of the Lemma.

\end{proof}

\subsubsection{Induced metric}

Let us discuss here the metric on $\Sigma_r$, in particular as $r$ tends to infinity.
On $\Sigma_r$ we may use $(u,\vartheta^1,\vartheta^2)$ as coordinates. Recall from Lemma~\ref{lemma:normal:r} the expression for the normal to $\Sigma_r$,
and that in general $\gb{r}$, the induced metric on $\Sigma_r$ is given by
\begin{equation}
  \gb{r}(X,Y)=g(\Pi X,\Pi Y)\qquad \Pi X=X+g(n,X)n\,.
\end{equation}

\begin{lemma}\label{lemma:g:induced}
  The metric on $\Sigma_r$ in $(u,\vartheta^1,\vartheta^2)$ coordinates, takes the form
  \begin{equation}
    \gb{r}=q^{-2}\Omega^2\ud u^2+\gs_{AB}\ud \vartheta^A\ud\vartheta^B
  \end{equation}
  and the volume form on $\Sigma_r$ is
  \begin{equation}
    \dm{\gb{r}}=q^{-1}\Omega\,\ud u\wedge\dm{\gs}=q^{-1}\Omega\sqrt{\det\gs}\:\ud u \wedge \ud \vartheta^1\wedge \ud\vartheta^2
  \end{equation}

\end{lemma}
\begin{proof}
  Since
  \begin{equation}
    \partial_u=\Lb=\Omega\Lbh
  \end{equation}
  we have
  \begin{equation}
    \Pi \partial_u=\frac{1}{2}\Lb-\frac{1}{2}q^{-2}L
  \end{equation}
and so
\begin{equation}
  (\gb{r})_{uu}=q^{-2}\Omega^2\,.
\end{equation}
Moreover,
\begin{gather}
  (\gb{r})_{u\vartheta^A}=g(\Pi\frac{\partial}{\partial u},\frac{\partial}{\partial \vartheta^A})=0\\
  (\gb{r})_{\vartheta^A\vartheta^B}=\gs_{AB}
\end{gather}

\end{proof}

\begin{remark}
  There appears no ``shift'' in the induced metric, because with our present choice the angular coordinates are Lie transported along the ingoing null geodesics.
\end{remark}

\subsubsection{Second fundamental forms}

Following the discussion of the first fundamental form,  $\gb{r}$, we now turn to the second fundamental form $k_r$  of $\Sigma_r$.

Recall the \emph{Codazzi equations}:
\begin{equation}
  (\nablab_Xk)(Y,Z)-(\nablab_Y k)(X,Z)=R(Z,n,X,Y)
\end{equation}
where $X,Y,Z$ are tangent to $\Sigma_r$.

We use a ``convenient frame'': $(E_0=n,E_i)$ where $g(E_i,E_j)=\delta_{ij}$ and $[\phi n,E_i]=0$.\footnote{Note the frame  is \emph{not} ``Fermi transported''.} Then %(Careful: orthonormal frame on cylinder?)
\begin{equation}
  g_{00}=-1\quad g_{i0}=0 \quad g_{ij}=(\gb{r})_{ij}
\end{equation}
\begin{equation}
  \frac{\partial \gbb_{ij}}{\partial r}=2\phi k_{ij}
\end{equation}
and
\begin{subequations}
  \begin{gather}
    \nablab_i k_{jm}-\nablab_j k_{im}=R_{m0ij}\label{eq:codazzi:ij}\\
    \nablab^m k_{jm}-\nablab_j\tr k=\Ric_{0j}=\Lambda g_{0j}=0
  \end{gather}
\end{subequations}

Moreover the \emph{Gauss equations} are:
\begin{subequations}\label{eq:gauss:sigma}
\begin{gather}
  \Rb_{minj}+k_{mn}k_{ij}-k_{mj}k_{ni}=R_{minj}\\
  \overline{\Ric}_{ij}+\tr k k_{ij}-k_{i}^{\phantom{i}m}k_{mj}=\Ric_{ij}+R_{0i0j}\\
  \Rb+(\tr k)^2-\lvert k\rvert^2=R+2\Ric_{00}=2\Lambda
\end{gather}
\end{subequations}

The ``acceleration of the normal lines'' is given by  $\nabla_n n=\nablab\log\phi$. In particular if $\phi$ is constant on $\Sigma_r$ the normal lines are geodesics parametrized by arc length.
The second variation equation reads in the above frame:
\begin{subequations}
\begin{gather}
  \frac{\partial k_{ij}}{\partial r}=\nablab^2_{ij}\phi+\phi\Bigl\{-R_{i0j0}+k_{i}^{\phantom{i}m}k_{mj}\Bigr\}\\
  \frac{\partial k_{ij}}{\partial r}=\nablab^2_{ij}\phi+\phi\Bigl\{  -\overline{\Ric}_{ij}-\tr k k_{ij}+2k_{i}^{\phantom{i}m}k_{mj}+\Lambda g_{ij}\Bigr\}
\end{gather}
\end{subequations}
and therefore
\begin{equation}
  \begin{split}
    \frac{\partial \tr k}{\partial r}=&\frac{\partial}{\partial r}\Bigl(\bigl(\overline{g}^{-1}\bigr)^{ij}k_{ij}\Bigr)=\bigl(\overline{g}^{-1}\bigr)^{ij}\frac{\partial k_{ij}}{\partial r}-2\phi k^{ij}k_{ij}\\
    =&\laplaceb\phi-\phi\Bigl( \Rb+(\tr k)^2-3\Lambda\Bigr)=\laplaceb\phi+\phi\Bigl(\Lambda-\lvert k\rvert^2\Bigr)\,.
  \end{split}
\end{equation}

Finally we note the associated connection coefficients:
\begin{subequations}
  \begin{gather}
    \Gamma_{00}^0= 0\quad \Gamma_{00}^i=\nablab^i\log \phi\\
    \Gamma_{i0}^0=0\quad \Gamma_{i0}^j=k_i^j\\
    \Gamma_{0i}^0=\nablab_i\log\phi\quad \Gamma_{0i}^j=k_i^j\\
    \Gamma_{ij}^0=k_{ij}\quad \Gamma_{ij}^m=\overline{\Gamma}_{ij}^m
  \end{gather}
\end{subequations}

For a detailed derivation of these formulas see for example Chapter~4 in \cite{schlue:gr}.

%%% Local Variables:
%%% mode: latex
%%% TeX-master: "weyl"
%%% End:

\section{Schwarzschild de Sitter cosmology}
\label{sec:schwarzschild:de:sitter}

In this Section we briefly discuss some aspects of the geometry of the Schwarzschild de Sitter solution \cite{kottler:sds,weyl:sds}. Its global geometry --- as depicted in the Penrose diagram of Fig.~\ref{fig:penrose:schwarzschild:de:sitter} ---  has already been discussed in Section~3 of \cite{glw}; cf.~\cite{gibbons:hawking}.

Here we are mainly interested in the values of the structure coefficients for different choices of double null foliations, which has partly motivated our assumptions in Section~\ref{sec:intro:foliation}. We restrict ourselves to the cosmological region $\mathcal{R}$, and \emph{spherically symmetric} foliations. %\footnote{Some features of \emph{non-}spherically symmetric gauge transformations have already been discussed for the de Sitter solution in Section~\ref{sec:desitter:geometry}; the difficulties that were highlighted there --- in particular the existence of double null foliations whose spheres of intersection do \emph{not} foliate correctly null infinity --- are not specific to the de Sitter solution, and occur equally for the Schwarzschild de Sitter metric which shares its leading order asymptotics. We do not revisit these  aspects of non-spherically symmetric foliations here, but it is important to keep in mind that we chose the de Sitter geometry mostly for convenience, specifically exploiting its embedding in a higher dimensional Minkowski space to find explicit solutions to the eikonal equation.}

The formulas derived in this section are not used in the remainder of this paper, but they give the reader the opportunity to familiarize themselves with the Schwarzschild de Sitter solution in double null gauge. We discuss in particular the gauge freedom in the class of spherically symmetric foliations,\footnote{For a broader discussion of \emph{non-}spherically symmetric foliations in this setting see \cite{schlue:optical}.} from which the reader can see that our assumptions do not single out a specific choice of double null coordinates.

\subsection{General properties}

The Schwarzschild de Sitter spacetime is a \emph{spherically symmetric} solution to \eqref{eq:eve:lambda:intro}, and distinguishes itself from de Sitter solution by the presence of a mass $m>0$. The manifold is $\mathcal{Q}\times\mathbb{S}^2$, and the metric $g$ takes the form \eqref{eq:g:intro:spherical}. Moreover --- as we have seen in Section~\ref{sec:intro:foliation} --- in double null coordinates the metric takes the form \eqref{eq:g:intro:double}, which simply reduces to
\begin{equation}\label{eq:g:spherical:null}
  g=-4\Omega^2\ud u\ud v +r^2\gammao
\end{equation}
The mass $m$, representing the ``mass energy contained in a sphere'' $S_{u,v}$,  can be defined unambiguously in spherical symmetry as a function $m:\mathcal{Q}\to [0,\infty)$ satisfying
\begin{equation}\label{eq:m:def}
  1-\frac{2m(u,v)}{r}-\frac{\Lambda r^2}{3}=-\frac{1}{\Omega^2}\frac{\partial r}{\partial u}\frac{\partial r}{\partial v}
\end{equation}
In \emph{vacuum},  the Einstein equations \eqref{eq:eve:lambda:intro} then imply that $m$ is a constant, (which parametrizes this 1-parameter family of solutions.) This allows us further to pass from the unknown $r:\mathcal{Q}\to (0,\infty)$ to the ``Regge-Wheeler coordinate''
\begin{equation}\label{eq:rs:def}
  \rs=\int\frac{\ud r}{1-\frac{2m}{r}-\frac{\Lambda r^2}{3}}
\end{equation}
which by virtue of \eqref{eq:eve:lambda:intro} satisfies the simple p.d.e.
\begin{equation}\label{eq:rs:pde}
  \frac{\partial^2 \rs}{\partial u \partial v}=0
\end{equation}
The various double null coordinates discussed below can be thought of as different choices of functions $f$, $g$ appearing in the general solution $\rs(u,v)=f(u)+g(v)$  of \eqref{eq:rs:pde}, and constants of integration in \eqref{eq:rs:def}.

Let us also note that the polynomial in $r$ on the l.h.s.~of \eqref{eq:m:def} has three real distinct roots $\overline{r_{\mathcal{C}}},\rh,\rc$ provided $0<m<1/(3\sqrt{\Lambda})$, the two positive ones $\rh$ and $\rc$ coinciding with the event, and cosmological horizons $\mathcal{H}$, and $\mathcal{C}$, respectively, (where $\partial_u r=0$, or $\partial_vr=0$, by the equation). In the following we are only interested in charts covering the cosmological region, and horizons, namely the domain $r\geq\rc$.

\subsection{Eddington-Finkelstein gauge}
\label{sec:eddington}

In ``Eddington-Finkelstein'' coordinates we choose
\begin{equation}\label{eq:rs}
  \rs=u+v=-\int_r^{\infty}\frac{1}{\frac{\Lambda r^2}{3}+\frac{2m}{r}-1}\ud r
\end{equation}
and thus cover $\mathcal{R}^+$ by
\begin{equation}
  \mathcal{R}^+=I^+(\mathcal{C}^+\cup\overline{\mathcal{C}}^+) = \Bigl\{ (u,v,\vartheta^1,\vartheta^2) : u+v < 0 \Bigr\}\,.
\end{equation}
 Note that the cosmological horizons $\mathcal{C}^+\cup\overline{\mathcal{C}}^+$ at $\rs=-\infty$, are \emph{not} covered by this chart, but strictly only its future; moreover future null infinity $\mathcal{I}^+$ can be identified with the surface $u+v=0$.
In these coordinates the metric takes the form
\begin{equation}
  g=-4\Bigl(\frac{\Lambda r^2}{3}+\frac{2m}{r}-1\Bigr)\ud u\ud v+r^2\gammao_{AB}\ud \vartheta^A\vartheta^B
\end{equation}
and we note specifically
\begin{equation}
    \frac{\partial r}{\partial u}=\frac{\partial r}{\partial v}=\Omega^2=\frac{\Lambda}{3}r^2+\frac{2m}{r}-1
\end{equation}
With the definitions of the null normals $(\Lbh,\Lh)$ of Section~\ref{sec:intro:foliation} it is then straight-forward to verify that\footnote{We caution that this frame is not regular on the horizon, so the values of $\tr\chi$, and $\tr\chi$ have no meaning on $\mathcal{C}\cup\overline{\mathcal{C}}$. }
\begin{subequations}
\begin{gather}
  \omega=\frac{\Lambda}{3}r-\frac{1}{2}\frac{2m}{r^2}\label{eq:omegac}\\
  \omegah=\frac{1}{\Omega}\partial_v\log\Omega=\frac{1}{2}\frac{1}{\Omega}\Bigl(2\frac{\Lambda}{3}r-\frac{2m}{r^2}\Bigr)\longrightarrow\sqrt{\frac{\Lambda}{3}}\quad(r\to\infty)\\
  \omegabh=\omegah
\end{gather}
\end{subequations}
and
\begin{subequations}\label{eq:chic}
\begin{gather}
  \chi_{AB}=g(\nabla_{e_A}\Lh,e_B)=\frac{\Omega}{r}\gs_{AB}\\
  \chih=0\qquad \tr\chi=\frac{2}{r}\Omega\longrightarrow 2\lambdasq \qquad (r\to\infty)\\
  \chibh=0\qquad\tr\chib=\frac{2}{r}\Omega
\end{gather}
\end{subequations}

The Gauss equation \eqref{eq:gauss} now allows us to calculate the $\rho$ component of the Weyl curvature: Since the spheres $S_{u,v}$ are \emph{round}, we have $K=r^{-2}$, and we obtain with \eqref{eq:chic} that
\begin{equation}\label{eq:rhoc}
  \rho[W]=\frac{\Lambda}{3}-K-\frac{1}{4}\tr\chi\tr\chib=\frac{1}{r^2}\Bigl(\frac{\Lambda}{3}r^2-1-\Omega^2\Bigr)=-\frac{2m}{r^3}\,.
\end{equation}
(This shows in particular that the mass $m\neq 0$ is the obstruction to conformal flatness.)
Moreover, since  $\alpha[W]=\alphab[W]=\beta[W]=\betab[W]=\zeta[W]=0$ in spherical symmetry symmetry, and thus also $\sigma[W]=0$, %by \eqref{eq:curl:etab}.
we have also proven \eqref{eq:weyl:sds:intro}.
In summary , $\rho[W]$, $\omegah$, $\omegabh$, and $\tr\chi$, $\tr\chib$ are the only non-vanishing null structure components for the Schwarzschild de Sitter solution.

\subsection{Gauge transformations and ``regular'' coordinates}

The choice of null coordinates in Section~\ref{sec:eddington} has a shortcoming: the coordinates do not extend to the cosmological horizons. 
While Eddington-Finkelstein coordinates provide a natural notion of ``retarded and advanced time'',   we will now discuss coordinates which extend beyond the cosmological horizons.
The following discussion highlights in particular the gauge dependence of the structure coefficients, and is relevant for the dynamical problem.

\subsubsection{Kruskal coordinates}

Let us denote the Eddington-Finkelstein coordinates of Section~\ref{sec:eddington} by $(\us,\vs)$. Then ``Kruskal coordinates'' $(\uk,\vk)$ are obtained with the following transformation:
\begin{equation}
  \uk=e^{2\kc\us}\qquad \vk=e^{2\kc\vs}
\end{equation}
In these coordinates $r(u,v)$ is implicitly given by
\begin{equation}
  \uk\vk=(r-\rc)(r-\rh)^{-\alphah}(r+\rcv)^{-\alphacv}
\end{equation}
where $\alphah,\alphacv>0$ are positive exponents (depending on $\Lambda$, $m$) satisfying $\alphah+\alphacv=1$; cf.~(3.16) in \cite{glw}.
In particular, in these coordinates the cosmological horizons $\overline{\mathcal{C}}^+$, and $\mathcal{C}^+$, are at $\uk=0$, and $\vk=0$, respectively, and the future boundary $r=\infty$ lies on the \emph{hyperbola} $\uk\vk=1$. The metric takes the form \eqref{eq:g:spherical:null} where $\Omega$ is \emph{non-degenerate} on the $\mathcal{C}\cup\overline{\mathcal{C}}$; in fact
\begin{equation}
  \Omega^2_{\mathcal{K}}=\frac{1}{4}\frac{\Lambda}{3}\kc^{-2}\frac{(r-\rh)^{1+\alphah}(r+\rcv)^{1+\alphacv}}{r}\longrightarrow\frac{1}{4}\frac{\Lambda}{3}\kc^{-2}r^2\qquad (r\to\infty)
\end{equation}
and
\begin{equation}
  \frac{\partial r}{\partial \uk}=2\kc\Omega^2_{\mathcal{K}}\vk\qquad\frac{\partial r}{\partial \vk}=2\kc\Omega^2_{\mathcal{K}}\uk
\end{equation}
where $\kc>0$ is the surface gravity of the cosmological horizons; see Section~3 of \cite{glw} for derivations.
It is then straight-forward to calculate that \emph{in this gauge},
\begin{subequations}
\begin{gather}
  \chi_{AB}=g(\nabla_{e_A}\Lh,e_B)=\frac{1}{\Omega_{\mathcal{K}}}\frac{1}{r}\partial_{\vk} r\: g_{AB}=\frac{2}{r}\kc\Omega_{\mathcal{K}} \uk\: g_{AB}\,,\\
  \chih=0\qquad\tr\chi=4\kc\frac{\Omega_{\mathcal{K}} \uk}{r}\\
  \tr\chi=0 \quad\text{on } \uk=0\,,\quad 
  \tr\chi\to 2\sqrt{\frac{\Lambda}{3}}\uk\qquad\text{ as } r\to\infty\,.
\end{gather}
\end{subequations}
Similarly,
\begin{subequations}
\begin{gather}
  \chibh=0\qquad\tr\chib=4\kc\frac{\Omega_{\mathcal{K}}\vk}{r}\\
  \tr\chib=0 \quad\text{on } \vk=0\,,\qquad
  \tr\chib\to 2\sqrt{\frac{\Lambda}{3}}\vk\quad\text{ as } r\to\infty\,.
\end{gather}
\end{subequations}
We also calculate
\begin{subequations}
\begin{equation}
    \omegah=\frac{1}{\Omega_{\mathcal{K}}}\frac{\partial}{\partial \vk}\log\Omega_{\mathcal{K}}=\frac{1}{4\kc\Omega_{\mathcal{C}} \vk}\Bigl[2\frac{\Lambda}{3}r-\frac{2m}{r^2}-2\kc\Bigr]
\longrightarrow\frac{1}{\vk}\sqrt{\frac{\Lambda}{3}}\qquad \text{as } r\to\infty  
\end{equation}
and $\omegah=0$ on $\uk=0$.
Similarly,
\begin{equation}
  \omegabh=\frac{1}{4\kc\Omega_{\mathcal{K}} \uk}\Bigl[2\frac{\Lambda}{3}r-\frac{2m}{r^2}-2\kc\Bigr]\to\frac{1}{\uk}\sqrt{\frac{\Lambda}{3}}\qquad \text{as } r\to\infty
\end{equation}
\end{subequations}

\subsubsection{``Initial data'' gauge}

We give an example of a double null system which retains ``retarded time'' $u$ of ``Eddington-Finkelstein type'' along $\mathcal{C}^+$, and ``advanced time'' $v$ of ``Eddington-Finkelstein type'' along $\overline{\mathcal{C}}^+$, yet is regular at the past horizons. It is trivially obtained by ``patching'' the above coordinate systems, but its features are worth studying, because it mimics a suitable gauge choice for the characteristic initial value problem.

Let us define
\begin{equation}
  \underline{u}=
  \begin{cases}
    u_{\mathcal{K}}-1 & \us<0\\
    2\kc\us &\us>0
  \end{cases}
  \qquad
    \underline{v}=
  \begin{cases}
    v_{\mathcal{K}}-1 & \vs<0\\
    2\kc\vs &\vs>0
  \end{cases}
\end{equation}
Then
\begin{equation}
  \stackrel{\mathcal{Q}}{g}=
  \begin{cases}
    -4\Omega^2_{\mathcal{K}}\ud \uu\ud \vv & :\us<0,\vs<0\\
    -4\Omega^2_{\mathcal{K}}\uk\ud\uu\ud \vv=:-4\Omega_{(u)}^2\ud\uu\ud\vv &:\us>0,\vs<0\\
     -4\Omega^2_{\mathcal{K}}\vk\ud\uu\ud \vv=:-4\Omega_{(v)}^2\ud\uu\ud\vv &:\vs>0,\us<0\\
  \end{cases}
\end{equation}
where
\begin{subequations}
\begin{gather}
  \Omega_{(u)}^2=\Omega_{\mathcal{K}}^2\uk=\frac{1}{4\kc^2}\frac{1}{v}\Bigl[\frac{\Lambda r^2}{3}+\frac{2m}{r}-1\Bigr]\\
  \Omega_{(v)}^2=\Omega^2_{\mathcal{K}}\vk=\frac{1}{4\kc^2}\frac{1}{u}\Bigl[\frac{\Lambda r^2}{3}+\frac{2m}{r}-1\Bigr]
\end{gather}
\end{subequations}
  This means that \emph{in this gauge}, along $\mathcal{C}^+$, for $\us>0$,  the null lapse behaves like
  \begin{equation}
    \Omega\rvert_{\mathcal{C}^+}=\frac{1}{2\kc}\cdot c(\Lambda,m)\cdot e^{\kc\us}
  \end{equation}
  and along the null infinity, for $\us>0$,
  \begin{equation}
    \Omega^2\rvert_{\mathcal{I}^+}=\frac{1}{4\kc^2}\cdot e^{2\kc\us}\cdot\lambdaq r^2
  \end{equation}
%  So the behavior is very different from Eddington-Finkelstein coordiantes, even at the leading order.

Let us calculate the structure coefficients in the region $\us>0$; (the region $\vs>0$ is entirely analogous).
In the same way as in Section~\ref{sec:eddington}, we find, relative to the normalised frame, 
\begin{equation}
  \Lbh=\frac{1}{\Omega_{(u)}}\frac{\partial}{\partial \uu}\qquad \Lh = \frac{1}{\Omega_{(u)}}\frac{\partial}{\partial \vv}
\end{equation}
that
\begin{subequations}
\begin{equation}
  \begin{split}
    \chib_{AB}&=\frac{1}{\Omega_{(u)}}\frac{1}{r}\partial_{\uu}r\, g_{AB}=\frac{1}{\Omega_{(u)}}\frac{1}{r}\frac{1}{2\kc}\Omega_\ast^2\,g_{AB}\\
    &=\Omega_{(u)}\frac{2\kc}{r}v_{\mathcal{K}}\,g_{AB}=\frac{1}{\Omega_{(u)}}\frac{1}{2\kc}\frac{1}{r}\Bigl(\frac{\Lambda}{3}r^2+\frac{2m}{r}-1\Bigr)g_{AB}
  \end{split}  
\end{equation}
or
\begin{equation}
  \chibh=0\qquad\tr\chib=\frac{1}{\Omegau}\frac{1}{\kc r}\Bigl(\frac{\Lambda}{3}r^2+\frac{2m}{r}-1\Bigr)
\end{equation}
so that
\begin{gather}
  \tr\chib\rvert_{\mathcal{C}^+}=0\\
  \tr\chib\longrightarrow 2\cdot e^{-\kc\us}\cdot\sqrt{\frac{\Lambda}{3}}\qquad (r\to\infty)
\end{gather}
\end{subequations}
Similarly, we find
\begin{subequations}
\begin{gather}
  \chi_{AB}=\frac{1}{\Omega_{(u)}}\frac{1}{r}\partial_{\vk}r g_{AB}=\frac{2\kc}{r}\Omega_{(u)} g_{AB}\\
  \chih=0\qquad\tr\chi=\frac{4\kc}{r}\Omega_{(u)}
\end{gather}
in particular
\begin{gather}
  \tr\chi\rvert_{\mathcal{C}^+}=c(\Lambda,m)\cdot e^{\kc \us}\\
  \tr\chi\longrightarrow 2\lambdasq \cdot e^{\kc\us}\qquad(r\to\infty)
\end{gather}
\end{subequations}
It remains to calculate the values of $\omegah$, $\omegabh$ in this gauge. We find
\begin{subequations}
\begin{gather}
  \omegabh=\frac{1}{2}\frac{1}{\Omegau^3}\partial_{\uu}\Omegau^2=\frac{1}{4\kc\Omegau}\Bigl(2\frac{\Lambda}{3}r-\frac{2m}{r^2}\Bigr)\\
  \omegah=\frac{1}{2}\frac{1}{\Omegau^3}\frac{\partial\Omega^2_{\mathcal{K}}}{\partial \vk}\uk=\kc\Omegau\Bigl(\frac{1+\alphah}{r-\rh}+\frac{1+\alphacv}{r+\rcv}-\frac{1}{r}\Bigr)
\end{gather}
\end{subequations}
and in particular
\begin{subequations}
\begin{gather}
  \omegabh\rvert_{\mathcal{C}^+}=c(\Lambda,m)e^{-\kc\us}\,,\qquad 
  \omegabh\longrightarrow\lambdasq\cdot e^{-\kappa\us}\quad(r\to\infty)\\
  \omegah\rvert_{\mathcal{C}^+}=c(\Lambda,m)\cdot e^{\kc\us}\,,\qquad 
  \omegah\longrightarrow\lambdasq\cdot e^{\kc\us}\quad (r\to\infty)
\end{gather}
\end{subequations}

\subsubsection{Gauge invariance}

In view of the assumptions on the structure coefficients outlined in Section~\ref{sec:intro:foliation}, we discuss the gauge -dependence and -invariance of the relevant quantities for the Schwarzschild-de Sitter example.

In Table~\ref{tab:gauges} we summarize the asymptotics towards null infinity of the values of the connection coefficients for the Schwarzschild de Sitter metric in the gauges discussed above.

\begin{table}[ht]
  \centering
  \begin{tabular}[c]{c|c|c|c}
  Gauge & \emph{Edd.-Finkelstein} & \emph{Kruskal} & ``\emph{Initial data}'' \\ \hline & & & \\
    $\Omega^2$ & $\lambdaq r^2$ & $\frac{1}{4}\lambdaq\frac{1}{\kc^2} r^2$    & $\frac{1}{4}\lambdaq \frac{1}{\kc^2} e^{2\kc\us} r^2$ \\ & & & \\
    $ \tr\chi  $ & $2\lambdasq$ & $2\lambdasq e^{2\kc \us}$ & $2\lambdasq e^{\kc \us}$ \\
    $ \tr\chib  $ & $2\lambdasq$ & $2\lambdasq e^{2\kc \vs}$ & $2\lambdasq e^{-\kc \us}$ \\ & & &\\
    $\omegah$ & $\lambdasq$ & $\lambdasq e^{-2\kappa \vs} $ &  $\lambdasq e^{\kc \us}$ \\
    $\omegabh$ & $\lambdasq$ & $\lambdasq e^{-2\kc\us}$ & $\lambdasq e^{-\kc\us}$ \\ & & &\\
    $q$ & $1$ & $e^{2\kc \us}$ & $e^{\kc \us}$
  \end{tabular}
  \caption{Schwarzschild de Sitter values in Eddington-Finkelstein, Kruskal and ``Initial data'' gauges, asymptotically towards null infinity.}
  \label{tab:gauges}
\end{table}

Note that each quantity, $\Omega$, $\omegah$, $\omegabh$, $\tr\chi$, $\tr\chib$, has the same asymptotics  in $r$ (\emph{towards} null infinity) in all gauges, but different behavior in $\us$ \emph{along} null infinity. In particular note that $\Omega$ differs by a prefactor even at the leading order.

Nonetheless ---  since $\lim_{r\to\infty} (\us+\vs)\rvert_{\Sigma_r}=0$ ---  we have in all gauges
\begin{equation}
  \lim_{r\to\infty}\frac{1}{4}\tr\chi\tr\chib=\lambdaq
\end{equation}
in agreement with the Gauss equation \eqref{eq:gauss}. 
Moreover, we have in all gauges,
\begin{equation}
  \lim_{r\to\infty} \lvert 2\omegah -\tr\chi \rvert = 0 \qquad   \lim_{r\to\infty} \lvert 2\omegabh -\tr\chib \rvert = 0 \,.
\end{equation}
and
\begin{equation}
  \lim_{r\to\infty} \lvert q \tr\chib - q^{-1}\tr\chi \rvert =0
\end{equation}
In fact, we have here in all gauges
\begin{equation}
  \lim_{r\to\infty} q \tr\chib = \lim_{r\to\infty} q^{-1}\tr\chi=2\lambdasq\,.
\end{equation}

%%% Local Variables:
%%% mode: latex
%%% TeX-master: "weyl"
%%% End:

%\subsection{Gauge transformations}

%\input{gauge-transformation}

%\subsection{Example}

%\input{gauge-invariance-example}

\section{Global redshift effect}
\label{sec:energy:estimates}

This Section contains a central part of this paper: We will prove a non-trivial bound for the Weyl curvature in spacetimes that satisfy our assumptions.
This is achieved by means of energy estimates for the Bianchi equations, which are recalled in some generality in Section~\ref{sec:energy:identity}.
In Section~\ref{sec:redshift:vectorfield} we will construct a suitable ``multiplier vectorfield'' whose associated energy is ``redshifted'', or ``damped'' in a fashion that is related to the expansion of the spacetime. This approach will then be further developed in Section~\ref{sec:commuted:redshift} to obtain also bounds on the derivatives of the Weyl curvature.

\subsection{Energy identity}
\label{sec:energy:identity}

Let us recall that the conformal curvature tensor $W$ satisfies  the contracted Bianchi equations \eqref{eq:W:Bianchi}.
The Bel-Robinson tensor $Q(W)$ defined by \eqref{eq:bel:robinson}
is symmetric and trace-free \emph{in all indices}; cf.~\CProp{12.5}.
Moreover it is non-negative when evaluated on future-directed causal vectors.

We have ---  by \CProp{12.6}, as a consequence of \eqref{eq:W:Bianchi} --- that $Q(W)$ is divergence-free, see \eqref{eq:Q:W:div}.
In particular, if we define the energy current $P$ associated to $W$ as in \eqref{eq:current:P}
then it follows from \eqref{eq:Q:W:div} that
\begin{multline}
  \nabla^\alpha P[W]_\alpha^{(X,Y,Z)}=\\=-\frac{1}{2}Q[W]_{\alpha\beta\gamma\delta}{}^{(X)}\pi^{\alpha\beta} Y^\gamma Z^\delta-\frac{1}{2}Q[W]_{\alpha\beta\gamma\delta} X^\beta {}^{(Y)}\pi^{\alpha\gamma} Z^\delta-\frac{1}{2}Q[W]_{\alpha\beta\gamma\delta}X^\beta Y^\gamma {}^{(Z)}\pi^{\alpha\delta}
\end{multline}
In view of the trace-free property of $Q[W]$ it is actually only the trace-free part of the deformation tensor that enters here:
\begin{gather}
  {}^{(X)}\hat{\pi}={}^{(X)}\pi-\frac{1}{4}(\tr {}^{(X)}\pi)\:g\\
  {}^{(X)}\pi=\mathcal{L}_X g \quad {}^{(X)}\pi_{\mu\nu}=\nabla_\mu X_\nu+\nabla_\nu X_\mu\\
  Q[W]_{\alpha\beta\gamma\delta}{}^{(X)}\pi^{\alpha\beta} Y^\gamma Z^\delta=Q[W]_{\alpha\beta\gamma\delta}{}^{(X)}\hat{\pi}^{\alpha\beta} Y^\gamma Z^\delta
\end{gather}
Using also the symmetry with respect to any index we finally obtain:
\begin{equation}\label{eq:div:P:W}
  \nabla^\alpha P[W]_\alpha^{(X,Y,Z)}=-\frac{1}{2}Q[W]_{\alpha\beta\gamma\delta}\Bigl({}^{(X)}\hat{\pi}^{\alpha\beta} Y^\gamma Z^\delta+{}^{(Y)}\hat{\pi}^{\alpha\beta} X^\gamma Z^\delta+{}^{(Z)}\hat{\pi}^{\alpha\beta} X^\gamma Y^\delta\Bigr)
\end{equation}

%In \CCh{12.4} Christodoulou derives the energy identity by hand, in a way particularly suited to regions with null boundaries. I will instead follow ideas from his ``Mathematical Problems in GR I'' book, see Chapter 4.2 therein. We need a more explicit derivation here, that is not restricted to the case where the multipliers $X$, $Y$, $Z$ are conformal Killing vectorfields, such that ${}^{(X)}\hat{\pi}=0$.

We defined $P$ as a 1-form. Let $\ld P$ be the dual of the corresponding vectorfield $P^\sharp$,  which is a 3-form: 
\begin{equation}
  \ld P_{\nu\kappa\lambda}=P^\mu \epsilon_{\mu\nu\kappa\lambda}
\end{equation}
Here $\epsilon=\ud\mu_{g}$ is the volume form of $g$.
The exterior derivative of $\ld P$ is a 4-form, and hence must be proportional to the volume form:
\begin{equation}
  \ud \ld P= f\ud \mu_{g}
\end{equation}
Moreover,%\footnote{I was very confused with getting the overall signs right, but as written I think it is now correct. The first sign, arising by solving $\omega=f\epsilon$ for $f$, yielding $f=-\ld \omega$, arises because $\epsilon^{1234}\epsilon_{1234}=-1$, see beginning of Chapter 7 in Stability of Minkowski. Now the second change of sign, comes from our abuse of notation, namley the easily overlooked fact that the dual we defined of $P$ is actually a \emph{right dual}: $P^\ast_{\nu\kappa\lambda}=P_\mu \epsilon^\mu_{\phantom{\mu}\nu\kappa\lambda}$, which differs by a sign from the \emph{left dual}: $\ld P_{\nu\kappa\lambda}=\epsilon_{\nu\kappa\lambda}^{\phantom{\nu\kappa\lambda}\mu}P_\mu$. For the left dual, I now trust Proof 8.1.1 on page 210 of the Stability of Minkowski, where it is stated that $\ld(\ud \ld P)=\nabla^\mu P_\mu$. Therefore $\ld \ud P^\ast=-\ld \ud \ld P=-\nabla^\mu P_\mu$, which explaines the second identity.}
\begin{equation}
  f=-\ld \ud \ld P=\nabla^\mu P_\mu
\end{equation}
so we can conclude
\begin{equation}\label{eq:d:ld:P}
  \ud \ld P=\nabla^\mu P_\mu\:\dm{g}\,.
\end{equation}
This implies, that integrated on any spacetime region $\mathcal{D}$, we have by virtue of Stokes theorem,
\begin{equation}
  \int_{\partial\mathcal{D}}\ld P=\int_{\mathcal{D}}\ud \ld P=\int_{\mathcal{D}}\nabla^\mu P_\mu \dm{g}
\end{equation}

\begin{figure}[tb]
  \centering
  \includegraphics{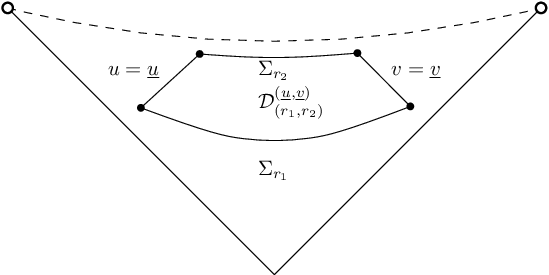}
  \caption{Spacetime domain used in energy identities}
  \label{fig:energy:id}
\end{figure}

Let the domain $\mathcal{D}$ be as in Figure~\ref{fig:energy:id}, namely
\begin{equation}\label{eq:D:uvr}
  \mathcal{D}^{(\underline{u},\underline{v})}_{(r_1,r_2)}=\bigcup_{r_1\leq r \leq r_2}\Sigma_r \bigcap \{u\leq \underline{u}\} \bigcap \{v\leq \underline{v}\}
\end{equation}
so that
\begin{equation}
  \partial \mathcal{D}=\Sigma_{r_1}^c\cup C_{\underline{u}}^c \cup \Sigma_{r_2}^c \cup \underline{C}_{\underline{v}}^c
\end{equation}
where superscript $c$ denotes that these surfaces are appropriately ``capped''.
We have
\begin{equation}\label{eq:flux:ld:P}
  \int_{\Sigma_{r_2}^c} \ld P=\int_{\Sigma_{r_2}^c} P^n\dm{\gb{r_2}}=\int_{\Sigma_{r_2}^c} -g(P,n)\dm{\gb{r_2}}=\int_{\Sigma_{r_2}^c} Q[W](n,X,Y,Z)\dm{\gb{r_2}}
\end{equation}
where $n$ is the unit normal to $\Sigma_r$; note that in the boundary integrals arising in Stokes theorem, the normal is always outward pointing, in particular it will have the opposite sign on the past boundary $\Sigma_{r_1}$\,.
For the null boundaries we recall first from \Ceq{1.204-6} that
\begin{subequations}
\begin{gather}
  \epsilon=\dm{g}=2\Omega^2\sqrt{\det\gs}\:\ud u\wedge\ud v\wedge\ud\vartheta^1\wedge\ud\vartheta^2\\
  \dm{\gs}=\sqrt{\det\gs}\ud\vartheta^1\wedge\ud\vartheta^2
\end{gather}
\end{subequations}
and then calculate, using that $(u,\vartheta^1,\vartheta^2)$ are coordinates on $\Cb$,
\begin{equation}
  \begin{split}
  \int_{\underline{C}_{\underline{v}}^c} \ld P&= -\int \ld P(\partial_u,\partial_{\vartheta^1},\partial_{\vartheta^2})\ud u\ud \vartheta^1\ud\vartheta^2    \\
  &=-\int P^\mu \epsilon_{\mu u \vth^1 \vth^2} \ud u\ud \vartheta^1\ud\vartheta^2=\int P^v \epsilon_{uv\vth^1\vth^2}\ud u\ud \vartheta^1\ud\vartheta^2\\
  &= \int 2\Omega^2 P^v  \sqrt{\det\gs}\ud u\ud \vartheta^1\ud\vartheta^2=-\int g(P,\partial_u)\sqrt{\det\gs}\ud u\ud \vartheta^1\ud\vartheta^2\\
  &= \int Q[W](\partial_u,X,Y,Z) \ud u\dm{\gs}=\int \Omega Q[W](\Lbh,X,Y,Z) \ud u\dm{\gs}
  \end{split}
\end{equation}
where we used that
\begin{gather}
  P^\sharp=P^u\partial_u+P^v\partial_v+P^A\partial_{\vth^A}\\
  g(P,\partial_u)=-2\Omega^2 P^v\,.
\end{gather}
On the null hypersurface $C_u$ we have to be more careful because
\begin{equation}
  g(P,\partial_v)=-2\Omega^2 P^u+P^v\gs_{AB}b^Ab^B-\gs_{AB}P^A b^B
\end{equation}
However,
\begin{equation}
  g(P,b^A\partial_{\vth^A})=-P^v \gs_{AB}b^A b^B+\gs_{AB}P^Bb^A
\end{equation}
and so
\begin{equation}
  g(P,\partial_v+b^A\partial_{\vth^A})=-2\Omega^2 P^u
\end{equation}
Therefore,  similarly
\begin{equation}\label{eq:flux:ld:P:null}
  \begin{split}
    \int_{C_u} \ld P &= \int \ld P_{v\vth^1\vth^2}\ud v\ud\vth^2\vth^2=\int 2\Omega^2 P^u \sqrt{\det\gs}\ud v\ud \vth^1\vth^2\\
    &=-\int g(P,\partial_v+b^A\partial_{\vth^A})\ud v\dm{\gs}=\int \Omega Q[W](\Lh,X,Y,Z)\ud v\dm{\gs}
  \end{split}
\end{equation}

To summarize we have proven the following:

\begin{proposition}\label{prop:energy:identity}
  Let $(\mathcal{M},g)$ be a spacetime, $g$ be globally expressed in double null gauge on a domain $\mathcal{R}\subset\mathcal{M}$, such that the level sets of the area radius $r$ are \emph{spacelike} on $\mathcal{R}$.
  Moreover let $\mathcal{D}$ be a domain of the form \eqref{eq:D:uvr};% bounded in the past by the surface $\Sigma_{r_1}$ and to the future by $\Sigma_{r_2}$, and the null hypersurfaces $C_{\ub}$ and $\Cb_{\vb}$
 see~Fig.~\ref{fig:energy:id}. Then for any Weyl field $W$ satisfying the Bianchi equations
\begin{equation}
  \nabla^\alpha W_{\alpha\beta\gamma\delta}=J_{\beta\gamma\delta}
\end{equation}
we have
  \begin{multline}\label{eq:energy:identity}
    \int\ud v\int_{S_{\ub,v}}\dm{\gs}\: \Omega Q[W](\Lh,X,Y,Z)+\int_{\Sigma_{r_2}^c}\: Q[W](n,X,Y,Z)\dm{\gb{r_2}}\\
    +\int_{\mathcal{D}} \bigl(\divergence Q(W))(X,Y,Z) \dm{g}    +\int\ud u\int_{S_{u,\vb}}\:\dm{\gs} \Omega Q[W](\Lbh,X,Y,Z) \\
+\int_{\mathcal{D}} K^{(X,Y,Z)}[W] \dm{g}=\int_{\Sigma_{r_1}^c} Q[W](n,X,Y,Z)\dm{\gb{r_1}}
  \end{multline}
  where $Q[W]$ denotes the Bel-Robinson tensor of  $W$, and
\begin{equation}\label{eq:K:XYZ}
   K^{(X,Y,Z)}[W]=\frac{1}{2}Q[W]_{\alpha\beta\gamma\delta}\Bigl({}^{(X)}\hat{\pi}^{\alpha\beta} Y^\gamma Z^\delta+{}^{(Y)}\hat{\pi}^{\alpha\beta} X^\gamma Z^\delta+{}^{(Z)}\hat{\pi}^{\alpha\beta} X^\gamma Y^\delta\Bigr)\,.
\end{equation}

\end{proposition}

\begin{proof}
  By \eqref{eq:d:ld:P},
  \begin{equation}
    \int_\mathcal{D} \nabla^\mu P_\mu \dm{g}=\int_{\partial\mathcal{D}}\ld P=\int_{\Sigma_{r_2}^c}\ld P
    +\int_{C_{\ub}^c}\ld P+\int_{\Cb_{\vb}^c}\ld P-\int_{\Sigma_{r_1}^c}\ld P
  \end{equation}
  and we can insert the expressions \eqref{eq:flux:ld:P} and \eqref{eq:flux:ld:P:null} for the energy flux from above.

If $J=0$, or $W$ is the \emph{conformal} Weyl field, then by \eqref{eq:div:P:W}
  \begin{equation}
    -\nabla^\alpha P[W]_\alpha^{(X,Y,Z)}=K^{(X,Y,Z)}[W]\,
  \end{equation}
More generally, if $J\neq 0$, then the divergence contains the additional term
  \begin{equation}
    -\nabla^\alpha P[W]_\alpha^{(X,Y,Z)}=(\divergence Q(W))(X,Y,Z)+K^{(X,Y,Z)}[W]\,.
  \end{equation}

\end{proof}

%%% Local Variables:
%%% mode: latex
%%% TeX-master: "weyl"
%%% End:

\subsection{Global redshift vectorfield}
\label{sec:redshift:vectorfield}

We define
\begin{equation}\label{eq:M}
  M=\frac{1}{2}\frac{1}{\Omega}\Bigl(\Lbh+\Lh\Bigr)\,.
\end{equation}
Note $M$ is time-like future-directed, and the associated energy flux \eqref{eq:flux:ld:P} is positive. Its crucial property however is that also the associated divergence \eqref{eq:d:ld:P} has a sign and bounds the energy flux, which lends it the name of a ``redshift vectorfield''.

\begin{remark}
  The choice \eqref{eq:M} is motivated by the form of the ``global redshift vectorfield'' used in our treatment of linear waves on Schwarzschild de Sitter cosmologies in \cite{glw}.
Therein we introduced 
\begin{equation}
  M=\frac{1}{r}\frac{\partial }{\partial r}
\end{equation}
relative to coordinates $(t,r,\vartheta^1,\vartheta^2)$ such that in the cosmological region the metric takes the form \eqref{eq:g:intro:r} with
\begin{equation}
  \phi=\frac{1}{\sqrt{\lambdaf}}\,,
\end{equation}
and
\begin{equation}
  \gb{r}=\Bigl(\lambdaf\Bigr)\ud t^2+r^2\gammao_{AB}\ud \vartheta^A\ud \vartheta^B\,.
\end{equation}

Alternatively, using the gradient vectorfield $V$ of $r$ introduced in \eqref{eq:def:V}, this vectorfield can be expressed as
\begin{equation}\label{eq:V:r}
  M=\frac{1}{r}\phi^2 V\,,
\end{equation}
and in the coordinates introduced in Section~\ref{sec:eddington}, 
\begin{equation}\label{eq:V:uv}
  V=\frac{1}{2}\Bigl(\frac{\partial}{\partial \us}+\frac{\partial}{\partial \vs}\Bigr)
\end{equation}
which implies
\begin{equation}\label{eq:M:uv}
  M=\frac{\phi^2}{2r}\Bigl(\frac{\partial}{\partial \us}+\frac{\partial}{\partial \vs}\Bigr)
  =\frac{1}{2}\frac{1}{r}\frac{1}{\lambdaf}\Bigl(\frac{\partial}{\partial \us}+\frac{\partial}{\partial \vs}\Bigr)\,.
\end{equation}
In fact, as discussed in \VS{4.1} it is equivalent to use the vectorfield
\begin{equation}\label{eq:Mp:uv}
  M^\prime=\frac{\partial}{\partial r}
  =\frac{1}{2}\frac{1}{\Omega^2}\Bigl(\frac{\partial}{\partial \us}+\frac{\partial}{\partial \vs}\Bigr)\,,
\end{equation}
which takes a remarkably simple form, and coincides precisely with \eqref{eq:M}.

\end{remark}

\subsubsection{Fluxes}

We will derive an energy identity associated to the multiplier vectorfield $M$ on a domain  foliated by level sets of the area radius $r$. Let us first look at the energy flux of a current constructed from $M$ through a surface $\Sigma_r$.
Recall here Lemma~\ref{lemma:normal:r} concerning the normal to $\Sigma_r$.

\begin{lemma}\label{lemma:flux:M}
Let $M$ denote the vectorfield \eqref{eq:M} and  $P^M$  the current
\begin{equation}
  P^M[W]=P[W]^{(M,M,M)}\,.
\end{equation}
Then
\begin{multline}\label{eq:P:prime}
  \int_{\Sigma_r} \ld P^M=\int_{\Sigma_r}\frac{1}{(2\Omega )^3}\Bigl\{q\lvert\alphab[W]\rvert^2+2(3q+q^{-1})\lvert \betab[W] \rvert^2\\+6(q+q^{-1})\bigl(\rho[W]^2+\sigma[W]^2\bigr)+2(q+3q^{-1})\lvert\beta[W]\rvert^2+q^{-1}\lvert \alpha[W] \rvert^2\Bigr\}\dm{\gb{r}}
\end{multline}
and
\begin{equation}
  \int_{C_u}\ld P^M=\int\ud v\int_{S_{u,v}}\dm{\gs}\frac{1}{4}\frac{1}{\Omega^2}\Bigl[2\lvert\betab\rvert^2+6\rho^2+6\sigma^2+6\lvert\beta\rvert^2+\lvert\alpha\rvert^2\Bigr]
\end{equation}
  
\end{lemma}
\begin{proof}
  This follows immediately from \CLemma{12.2}.
\end{proof}

%\todo{Add remark about scaling of energy at infinity; see previous version}

\subsubsection{Deformation Tensor}

Next we calculate the components of the trace-free part of the deformation tensor of $M$, which enter the expression for $K^{(M,M,M)}$ in \eqref{eq:K:XYZ}.

\begin{lemma}\label{lemma:deformation:M}
  The null components of the trace-free part $\hat{\pi}$ of the deformation tensor of~$M$ are given by
\begin{subequations}
  \begin{gather}
    \hat{\pi}^{\Lbh\Lbh}=\frac{\omegah}{\Omega}\quad \hat{\pi}^{\Lbh\Lh}=\frac{1}{8}\frac{1}{\Omega}\Bigl(\tr\chib+\tr\chi\Bigr)
    \quad \hat{\pi}^{\Lh\Lh}=\frac{\omegabh}{\Omega}\\
    \hat{\pi}^{\Lbh A}=-\frac{1}{\Omega}\etab^{\sharp A}\qquad \hat{\pi}^{\Lh A}=-\frac{1}{\Omega}\eta^{\sharp A}\\
    \hat{\pi}^{AB}=\frac{1}{\Omega}\Bigl(\hat{\chib}^{\sharp\sharp AB}+\hat{\chi}^{\sharp\sharp AB}+\frac{1}{4}\bigl(\tr\chi+\tr\chib\bigr)(g^{-1})^{AB}\Bigr)
  \end{gather}
\end{subequations}

\end{lemma}

\begin{proof}
  In view of the frame relations \eqref{eq:omegah} and \eqref{eq:normalised:frame:relations} we have
\begin{subequations}
  \begin{gather}
    \nabla_{\Lbh}M=\frac{1}{\Omega}\Bigl(\eta^\sharp-\omegabh\Lh\Bigr)\\
    \nabla_{\Lh}M=\frac{1}{\Omega}\Bigl(\etab^\sharp-\omegah\Lbh\Bigr)\\
    \nabla_AM=\frac{1}{2}\frac{1}{\Omega}\Bigl(\chib_A^\sharp-\etab_A\Lbh+\chi_A^\sharp-\eta_A\Lh\Bigr)
  \end{gather}
\end{subequations}
and thus
\begin{subequations}
\begin{gather}
  \pi_{\Lbh\Lbh}=2g(\nabla_{\Lbh}M,\Lbh)=\frac{4}{\Omega}\omegabh\\
  \pi_{\Lh\Lh}=2 g(\nabla_{\Lh}M,\Lh)=\frac{4}{\Omega}\omegah\\
  \pi_{\Lbh\Lh}=0\\
  \pi_{\Lbh e_A}=g(\nabla_{\Lbh}M,e_A)+g(\Lbh,\nabla_AM)=\frac{2}{\Omega}\eta_A\\
  \pi_{\Lh e_A}=\frac{2}{\Omega}\etab_A\\
  \pi_{e_Ae_B}=g(\nabla_AM,e_A)+g(e_A,\nabla_BM)=\frac{1}{\Omega}\Bigl(\chib_{AB}+\chi_{AB}\Bigr)
\end{gather}
\end{subequations}
and
\begin{equation}
  \tr\pi=\frac{1}{\Omega}\Bigl(\tr\chib+\tr\chi\Bigr)
\end{equation}
Therefore
\begin{subequations}
  \begin{gather}
  \hat{\pi}_{\Lbh\Lbh}=\frac{4}{\Omega}\omegabh\qquad  \hat{\pi}_{\Lh\Lh}=\frac{4}{\Omega}\omegah\\
  \hat{\pi}_{\Lbh\Lh}=\frac{1}{2}\frac{1}{\Omega}\Bigl(\tr\chib+\tr\chi\Bigr)\\
  \hat{\pi}_{\Lbh e_A}=\frac{2}{\Omega}\eta_A\qquad
  \hat{\pi}_{\Lh e_A}=\frac{2}{\Omega}\etab_A\\
  \hat{\pi}_{e_Ae_B}=\frac{1}{\Omega}\Bigl(\chibh_{AB}+\chih_{AB}+\frac{1}{4}\bigl(\tr\chi+\tr\chib\bigr)g_{AB}\Bigr)    
  \end{gather}
\end{subequations}
or alternatively
\begin{subequations}
  \begin{gather}
    \hat{\pi}^{\Lbh\Lbh}=\frac{1}{4}\hat{\pi}_{\Lh\Lh}=\frac{\omegah}{\Omega}\qquad \hat{\pi}^{\Lh\Lh}=\frac{\omegabh}{\Omega}\\
    \hat{\pi}^{\Lbh\Lh}=\frac{1}{8}\frac{1}{\Omega}\Bigl(\tr\chib+\tr\chi\Bigr)\\
    \hat{\pi}^{\Lbh A}=-\frac{1}{\Omega}\etab^{\sharp A}\qquad \hat{\pi}^{\Lh A}=-\frac{1}{\Omega}\eta^{\sharp A}\\
    \hat{\pi}^{AB}=\frac{1}{\Omega}\Bigl(\hat{\chib}^{\sharp\sharp AB}+\hat{\chi}^{\sharp\sharp AB}+\frac{1}{4}\bigl(\tr\chi+\tr\chib\bigr)(\gs^{-1})^{AB}\Bigr)\,.
  \end{gather}
\end{subequations}

\end{proof}

Now we can calculate
\begin{equation}\label{eq:KM:sum:pre}
  \begin{split}
     K^M&:=K^{(M,M,M)}[W]=\frac{3}{2}Q[W]_{\alpha\beta\gamma\delta}{}^{(M)}\hat{\pi}^{\alpha\beta} M^\gamma M^\delta  \\
     &=\frac{3}{8}\frac{1}{\Omega^2}Q[W]_{\alpha\beta\gamma\delta}{}^{(M)}\hat{\pi}^{\alpha\beta} \Lbh^\gamma \Lbh^\delta  +\frac{3}{4}\frac{1}{\Omega^2}Q[W]_{\alpha\beta\gamma\delta}{}^{(M)}\hat{\pi}^{\alpha\beta} \Lbh^\gamma \Lh^\delta  \\
     &\qquad +\frac{3}{8}\frac{1}{\Omega^2}Q[W]_{\alpha\beta\gamma\delta}{}^{(M)}\hat{\pi}^{\alpha\beta} \Lh^\gamma \Lh^\delta  
  \end{split}
\end{equation}
which involves the null components of $Q[W]$. These are quadratic expressions in the Weyl curvature, which are given by \CLemma{12.2}. In particular, we have
\begin{subequations}
\begin{gather}
  Q[W](\Lbh,\Lbh,\Lbh,\Lbh)=2|\alphab[W]|^2\qquad \quad  Q[W](\Lh,\Lh,\Lh,\Lh)=2|\alpha[W]|^2\\  Q[W](\Lh,\Lbh,\Lbh,\Lbh)=4|\betab[W]|^2\qquad Q[W](\Lbh,\Lh,\Lh,\Lh)=4|\beta[W]|^2\\
  Q(W)(\Lbh,\Lbh,\Lh,\Lh)=4\bigl(\rho[W]^2+\sigma[W]^2\bigr)
\end{gather}
\end{subequations}

\begin{lemma}\label{lemma:K:M}
  With $M$ defined by \eqref{eq:M}, we have
  \begin{equation}
    K^M[W]=K_+[W]+K_-[W]
  \end{equation}
where
\begin{multline}\label{eq:K:plus}
  \frac{\Omega^3}{3}K_+[W]=\frac{1}{4}\omegah\lvert\alphab[W]\rvert^2+\omegah\lvert\betab[W]\rvert^2+\frac{1}{2}(\omegabh+\omegah)(\rho[W]^2+\sigma[W]^2)+\omegabh\lvert\beta[W]\rvert^2+\frac{1}{4}\omegabh\lvert\alpha[W]\rvert^2\\
  +\frac{1}{2}\bigl(\tr\chi+\tr\chib\bigr)\Bigl(\frac{1}{2}\lvert\betab[W]\rvert^2+\rho[W]^2+\sigma[W]^2+\frac{1}{2}\lvert\beta[W]\rvert^2\Bigr)
\end{multline}
and
\begin{multline}\label{eq:K:minus}
  \frac{\Omega^3}{3}K_-[W]=\alphab[W](\etab^\sharp,\betab^\sharp[W])-\alpha[W](\eta^\sharp,\beta^\sharp[W])\\
  +\rho[W]\bigl((2\etab^\sharp+\eta^\sharp)\cdot\betab[W]-(2\eta+\etab^\sharp)\cdot\beta[W]\bigr)+\sigma[W]\bigl((2\etab^\sharp+\eta^\sharp)\cdot\ld\betab[W]+(2\eta^\sharp+\etab^\sharp)\cdot\ld\beta[W]\bigr)\\
  -(\chibh+\chih)(\betab[W]^\sharp,\beta[W]^\sharp)+\frac{1}{4}\rho[W](\chibh+\chih,\alphab[W]+\alpha[W])+\frac{1}{4}\sigma[W](\chibh+\chih,\ld \alphab[W]-\ld \alpha[W])
\end{multline}
\end{lemma}

\begin{proof}
Using the expressions for the null components of the Bel-Robinson tensor listed in \CLemma{12.2} we note first that
\begin{subequations}
\begin{gather}
  \etab^{\sharp A}Q_{3A33}=-4\alphab(\etab^\sharp,\betab^\sharp)\\
  \eta^{\sharp A} Q_{4A33}=-4\rho\eta^\sharp\cdot\betab-4\sigma\eta^\sharp\cdot\ld\betab\\
  \bigl(\chibh^{\sharp\sharp AB}+\chih^{\sharp\sharp AB}\bigr) Q_{AB33}=2\rho(\chibh+\chih,\alphab)+2\sigma(\chibh+\chih,\ld\alphab)\\
  (g^{-1})^{AB} Q_{AB33}=4\lvert\betab\rvert^2
\end{gather}
%\begin{multline}    \bigl(\chib^{\sharp\sharp AB}+\chi^{\sharp\sharp AB}\bigr) Q_{AB33}=\\=2(\tr\chib+\tr\chi)\lvert\betab[W]\rvert^2+2(\rho[W]-\rhoc)(\chibh+\chih,\alphab[W])+2\sigma[W](\chibh+\chih,\ld \alphab[W]) \end{multline}
\end{subequations}
and hence
\begin{multline}
  \Omega\:Q_{\alpha\beta\gamma\delta}\,\hat{\pi}^{\alpha\beta}\Lbh^\gamma\Lbh^\delta=
  2\omegah\lvert\alphab\rvert^2+4\omegabh(\rho^2+\sigma^2)+2(\tr\chib+\tr\chi)\lvert\betab\rvert^2\\
  +8\alphab(\etab^\sharp,\betab^\sharp)+8\rho\eta^\sharp\cdot\betab+8\sigma\eta^\sharp\cdot\ld\betab
  +2\rho(\chibh+\chih,\alphab)+2\sigma(\chibh+\chih,\ld \alphab)\,.
\end{multline}

Similarly
\begin{multline}
  \Omega\:Q_{\alpha\beta\gamma\delta}\,\hat{\pi}^{\alpha\beta}\Lh^\gamma\Lh^\delta=
  2\omegabh\lvert\alpha\rvert^2+4\omegah(\rho^2+\sigma^2)  +2(\tr\chib+\tr\chi)\lvert\beta\rvert^2\\
  -8\alpha(\eta^\sharp,\beta^\sharp)-8\rho\etab^\sharp\cdot\beta+8\sigma\etab^\sharp\cdot\ld\beta
+2\rho(\chibh+\chih,\alpha)-2\sigma(\chibh+\chih,\ld \alpha)\,.
\end{multline}
Moreover, by \CLemma{12.2} we have
\begin{subequations}
\begin{gather}
  \etab^{\sharp A}Q_{3A34}=-4\rho\etab^{\sharp}\cdot \betab-4\sigma\etab^\sharp\cdot \ld\betab\\
  \eta^{\sharp A} Q_{4A34}=4\rho\eta^\sharp\cdot\beta-4\sigma\eta^{\sharp A}\cdot\ld\beta\\
  \Bigl(\chibh^{\sharp\sharp AB}+\chih^{\sharp\sharp AB}\Bigr)Q_{AB34}=-2\Bigl(\chibh+\chih,\betab[W]\otimesh\beta[W]\Bigr)\label{eq:chi:sharpsharp:beta:otimesh}\\
  (g^{-1})^{AB}Q_{AB34}=4\Bigl(\rho^2+\sigma^2\Bigr)
\end{gather}
\end{subequations}
where $\betab\otimesh\beta$ denotes the symmetric trace-free 2-covariant tensorfield:
\begin{equation}
  (\betab\otimesh\beta)_{AB}=\betab_A\beta_B+\betab_B\beta_A-(\betab,\beta)\gs_{AB}\,.
\end{equation}
In particular we can write on the right hand side of \eqref{eq:chi:sharpsharp:beta:otimesh}
\begin{equation}
  % \Bigl(\chibh+\chih,\betab[W]\otimesh\beta[W]\Bigr)=
  \Bigl(\chibh^{\sharp\sharp AB}+\chih^{\sharp\sharp AB}\Bigr)\bigl(\betab[W]\otimesh\beta[W]\bigr)_{AB}=2\chibh(\betab[W]^\sharp,\beta[W]^\sharp)+2\chih(\betab[W]^\sharp,\beta[W]^\sharp)\,.
\end{equation}
Therefore,
\begin{multline}
  \Omega Q_{\alpha\beta\gamma\delta}\hat{\pi}^{\alpha\beta}\Lbh^\gamma \Lh^\delta=
  4\omegah\lvert\betab\rvert^2+4\omegabh\lvert\beta\rvert^2+2\bigl(\tr\chi+\tr\chib\bigr)\Bigl(\rho^2+\sigma^2\Bigr)\\
  +8\rho[W]\etab^{\sharp}\cdot \betab+8\sigma\etab^\sharp\cdot \ld\betab
  -8\rho\eta^\sharp\cdot\beta+8\sigma\eta^{\sharp A}\cdot\ld\beta
  -4(\chibh+\chih)(\betab^\sharp,\beta^\sharp)
\end{multline}

Summing up these contributions according to \eqref{eq:KM:sum:pre} yields the statement of the Lemma.
\end{proof}

\begin{remark}
  We already see that $K_+[W]$ is manifestly positive if $\tr\chi>0$, $\tr\chib>0$, and $\omegah>0$, $\omegabh>0$, as it will be the case under our assumptions.
  Indeed the assumption \eqref{eq:assumption:intro:omega} ensures that $2\omegah$ is close to $\tr\chi$, and $2\omegabh$ is close to $\tr\chib$, the latter of which are positive by \eqref{eq:assumption:intro:tr}. In fact, the difference asymptotically tends to zero  because $\Omega$ is comparable to $r$ by \eqref{eq:assumption:intro:Omega}.
\end{remark}

\subsubsection{Lorentz Transformations}
\label{sec:lorentz}

We will see  in Section~\ref{sec:global:redshift} that the simple choice \eqref{eq:M} for $M$  suffices to obtain  the desired energy estimate for the Weyl curvature $W$. However, it turns out in Section~\ref{sec:commuted:redshift} that a more refined choice is necessary to obtain an estimates for higher order energies.

The required adjustment amounts to co-aligning the vectorfield with the normal to $\Sigma_r$.
This can be achieved by formally  keeping  exactly the same definition of $M$, \emph{but} changing the null frame that is used in \eqref{eq:M}.
The fact that we can keep this simple definition in terms of another null frame will be computationally very advantageous.

A simple Lorentz transformation is given by:
\begin{equation}\label{eq:lorentz}
  \Lbh\mapsto a \Lbh\qquad \Lh\mapsto a^{-1}\Lh
\end{equation}
for some function $a>0$.
\begin{remark}
Let us give a heuristic discussion for the choice of $a$ which aligns  $M$ with $n$.
Regarding the normal, we have by Lemma~\ref{lemma:normal:r}, and taking $r\to\infty$,
\begin{equation}
  \phi n=\frac{1}{r}\frac{\Omega}{\overline{\Omega\tr\chib}}\Lbh+\frac{1}{r}\frac{\Omega}{\overline{\Omega\tr\chi}}\to\frac{1}{r}\Bigl(\frac{1}{\tr\chib}\Lbh+\frac{1}{\tr\chi}\Lh\Bigr)
\end{equation}
where we have neglected  asymptotic deviations from spherical averages.
Since by the Gauss equation \eqref{eq:gauss},
\begin{equation}\label{eq:tr:chi:a}
  \frac{1}{4}\tr\chi\tr\chib\to\lambdaq
\end{equation}
we can expect that for some function $a_\chi$, as $r\to\infty$,
\begin{equation}
  \tr\chi\to 2\lambdasq a_\chi\qquad \tr\chib\to2\lambdasq a_\chi^{-1}
\end{equation}
Now we see clearly that the function $a_\chi$ appearing in the asymptotics of $\tr\chi$, and $\tr\chib$, is the required rescaling of the null vectors in \eqref{eq:lorentz}.
More precisely,  with $a=a_\chi$, we can expect that $M$, formally given by \eqref{eq:M} relative to a frame resulting from the Lorentz transformation \eqref{eq:lorentz}, satifies asymptotically
\begin{equation}
  M\to \phi n\,.
\end{equation}
\end{remark}

For any function $a>0$, let us denote by
\begin{equation}
  e_3=a\Lbh\qquad e_4=a^{-1}\Lh
\end{equation}
and $e_A:A=1,2$ an arbitrary frame on the spheres. Note that both $(\Lbh,\Lh;e_A)$ and $(e_3,e_4;e_A)$ are null frames.
Here the frame $(\Lbh,\Lh;e_A)$ is the null frame derived from the double null coordiantes, and we continue to denote by $\omegabh,\omegah,\eta,$ etc.~the associated structure coefficients. Now for the frame $(e_3,e_4;e_A)$ we find the following connection coefficients:
\begin{subequations}\label{eq:frame:a}
  \begin{gather}
    \nabla_3e_3=\bigl(a\omegabh+\Lbh a\Bigr)e_3\qquad \nabla_4 e_4=\bigl(a^{-1}\omegah+\Lh a^{-1}\Bigr)e_4\\
    \nabla_3 e_4=2\eta^\sharp-\bigl(a\omegabh-a^2\Lbh a^{-1}\bigr)e_4\qquad \nabla_4 e_3=2\etab^\sharp-\bigl(a^{-1}\omegah-a^{-2}\Lh a\bigr)e_3\\
    \nabla_A e_3=a\chib_A^{\sharp B}e_B+\bigl(\zeta_A+\ds_A\log a\bigr)e_3\qquad \nabla_A e_4=a^{-1}\chi_A^{\sharp B}e_B-\bigl(\zeta_A+\ds_A\log a\bigr)e_4\\
    \nabla_3 e_A=\nablas_3e_A +\eta_Ae_3\qquad \nabla_4e_A=\nablas_4 e_A+\etab_A e_4\\
    \nabla_A e_B=\nablas_A e_B+\frac{1}{2}a^{-1}\chi_{AB}e_3+\frac{1}{2}a\chib_{AB}e_4
  \end{gather}
\end{subequations}
Here we use the notation:
\begin{equation}
  \nablas_3e_A=\Pi\nabla_{\Lbh}e_A\qquad   \nablas_4e_A=\Pi\nabla_{\Lh}e_A
\end{equation}
where \begin{equation}\label{eq:Pi}
  \Pi X=X+\frac{1}{2}g(X,\Lbh)\Lh+\frac{1}{2}g(X,\Lh)\Lbh
\end{equation}
is the projection to subspace of the tangent space spanned by $e_A:A=1,2$, namely the tangent space of the spheres.

Also note:
  \begin{equation}\label{eq:D:log:Omega:a}
    e_3\log\Omega=a\omegabh\qquad e_4\log\Omega=a^{-1}\omegah
  \end{equation}
In other words, the Lorentz transformation \eqref{eq:lorentz} induces the following transformations of the structure coefficients:
\begin{subequations}\label{eq:lorentz:structure}
  \begin{gather}
    \chib\mapsto a\chib\qquad \chi\mapsto a^{-1}\chi\\
    \zeta\mapsto \zeta+\ds\log a\qquad \eta\mapsto \eta+\ds\log a\qquad \etab\to\etab-\ds\log a\\
    \omegabh\mapsto a\omegabh+\frac{1}{\Omega}\Db a\qquad \omegah\mapsto a^{-1}\omegah+\frac{1}{\Omega}D a^{-1}
  \end{gather}
\end{subequations}
Recall here the notation $\Db a=\Lb a$, and $D a^{-1}=L a^{-1}$ introduced in Section~\ref{sec:area:radius}.

\begin{lemma}\label{lemma:deformation:M:a}
  For any function $a=a(u,v,\vartheta^1,\vartheta^2)>0$, let $M_a$ denote the vectorfield
  \begin{equation}
    \label{eq:M:a}
    M_a=\frac{1}{2}\frac{1}{\Omega}\bigl(a\Lbh+a^{-1}\Lh\bigr)=\frac{1}{2}\frac{1}{\Omega}\bigl(e_3+e_4\bigr)
  \end{equation}
  Then with respect to the null frame $(e_3,e_4;e_A)$ the null components of the deformation tensor of $M_a$ are:
\begin{subequations}
  \begin{gather}
    \pih{M_a}_{33}=\frac{4}{\Omega}a\omegabh+\frac{2}{\Omega}\Lbh a\qquad \pih{M_a}_{44}=\frac{4}{\Omega}a^{-1}\omegah+\frac{2}{\Omega}\Lh a^{-1}\\
    \pih{M_a}_{34}=\frac{1}{2}\frac{1}{\Omega}\bigl(a\tr\chib+a^{-1}\tr\chi\bigr)-\frac{1}{2}\frac{1}{\Omega}\bigl(\Lbh a+\Lh a^{-1}\bigr)\\
    \pih{M_a}_{3A}=\frac{1}{\Omega}\bigl(2\eta_A+\ds_A\log a\bigr)\qquad \pih{M_a}_{4A}=\frac{1}{\Omega}\bigl(2\etab_A-\ds_A\log a\bigr)\\
    \pih{M_a}_{AB}=\frac{1}{\Omega}\Bigl(a\chibh_{AB}+a^{-1}\chih_{AB}+\frac{1}{4}\bigl(a\tr\chib+a^{-1}\tr\chi-\Lbh a-\Lh a^{-1}\bigr)\gs_{AB}\Bigr)
  \end{gather}
\end{subequations}
where $\chib,\chi,\eta,\etab,\zeta,\omegabh,\omegah$ are the structure coefficients associated to the null frame $(\Lbh,\Lh;e_A)$.
\end{lemma}
\begin{proof}
  Analogously to Lemma~\ref{lemma:deformation:M} we first compute
\begin{subequations}
  \begin{gather}
    \nabla_3 M_a=\frac{1}{\Omega}\bigl(\eta^\sharp-a\omegabh e_4\bigr)+\frac{1}{2}\frac{1}{\Omega}(\Lbh a)\bigl(e_3-e_4\bigr)\\
    \nabla_4M_a=\frac{1}{\Omega}\bigl(\etab^\sharp-a^{-1}\omegah e_3\bigr)+\frac{1}{2}\frac{1}{\Omega}(\Lh a^{-1})(e_4-e_3)\\
    \nabla_AM_a=\frac{1}{2}\frac{1}{\Omega}\bigl(a\chib^{\sharp B}_A+a^{-1}\chi^{\sharp B}_A\bigr)e_B-\frac{1}{2}\frac{1}{\Omega}\bigl(\etab_A-\ds_A\log a\bigr)e_3-\frac{1}{2}\frac{1}{\Omega}\bigl(\eta_A+\ds_A\log a\bigr)e_4
  \end{gather}
\end{subequations}
using the frame relations \eqref{eq:frame:a} and \eqref{eq:D:log:Omega:a},
and then infer
\begin{subequations}
  \begin{gather}
    \pi_{33}=\frac{4}{\Omega}a\omegabh+\frac{2}{\Omega}\Lbh a\qquad \pi_{44}=\frac{4}{\Omega}a^{-1}\omegah+\frac{2}{\Omega}\Lh a^{-1}\\
    \pi_{34}=-\frac{1}{\Omega}\Lbh a-\frac{1}{\Omega}\Lh a^{-1}\\
    \pi_{3A}=\frac{1}{\Omega}\bigl(2\eta_A+\ds_A\log a\bigr)\qquad \pi_{4A}=\frac{1}{\Omega}\bigl(2\etab_A-\ds_A\log a\bigr)\\
    \pi_{AB}=\frac{1}{\Omega}\bigl(a\chib_{AB}+a^{-1}\chi_{AB}\bigr)
  \end{gather}
\end{subequations}
and
\begin{equation}
  \tr\pi=\frac{1}{\Omega}\bigl(a\tr\chib+a^{-1}\tr\chi+\Lbh a+\Lh a^{-1}\bigr)
\end{equation}
which gives the formulas of the Lemma  for the trace free part of $\pi$.
\end{proof}

We continue to denote by $(\alphab[W],\alpha[W],\betab[W],\beta[W],\rho[W],\sigma[W])$ the null components of $W$ with respect to the null frame $(\Lbh,\Lh;e_A)$. 
To avoid confusion, we will explicitly denote the null components of $W$ with respect to the null frame $(e_3,e_4;e_A)$ by $(\alphab_a[W],\alpha_a[W],\betab_a[W],\beta_a[W],\rho_a[W],\sigma_a[W])$, and note that
\begin{subequations}\label{eq:lorentz:curvature}
  \begin{gather}
    \alphab_a[W]=a^2\alphab[W]\qquad \alpha_a[W]=a^{-2}\alpha[W]\\
    \betab_a[W]=a\betab[W]\qquad \beta_a[W]=a^{-1}\beta[W]\\
    \rho_a[W]=\rho[W]\qquad \sigma_a[W]=\sigma[W]\label{eq:lorentz:rho:sigma}
  \end{gather}
\end{subequations}

\bigskip

Note that if we take $a:=q$, where $q$ is the quotient appearing in \eqref{eq:normal:q}, then the normal takes the simple form
\begin{equation}\label{eq:n:q}
  n=\frac{1}{2}\bigl(e_3+e_4\bigr);
\end{equation}
this, of course, is the purpose of introducing the frame $(e_3,e_4;e_A)$ in the first place.

Moreover, it follows immediately from \eqref{eq:flux:ld:P} with \eqref{eq:M:a} and \eqref{eq:n:q}  that
\begin{multline}\label{eq:flux:Ma}
  \int_{\Sigma_r}\ld P^{M_q}[W]=\int_{\Sigma_r}Q[W](n,M_q,M_q,M_q)\dm{\gb{r}}\\%=\int_{\Sigma_r}\frac{1}{(2\Omega)^3}\Bigl[q^4\lvert\alphab[W]\rvert^2+8q^2\lvert\betab[W]\rvert^2+12\rho[W]^2+12\sigma[W]^2+8q^{-2}\lvert\beta[W]\rvert^2+q^{-4}\lvert\alpha[W]\rvert^2\Bigr]\dm{\gb{r}}\\
=\int_{\Sigma_r}\frac{1}{(2\Omega)^3}\Bigl[\lvert\alphab_q[W]\rvert^2+8\lvert\betab_q[W]\rvert^2+12\rho[W]^2+12\sigma[W]^2+8\lvert\beta_q[W]\rvert^2+\lvert\alpha_q[W]\rvert^2\Bigr]\dm{\gb{r}}
\end{multline}
and from \eqref{eq:flux:ld:P:null} with \eqref{eq:M:a} that
\begin{multline}
  \int_{C_u}\ld P^{M_q}[W]=\int\ud v\int_{S_{u,v}}\dm{\gs}\Omega Q[W](\Lh,M_q,M_q,M_q)=\\
  %=\int\ud v\int_{S_{u,v}}\dm{\gs}\frac{q}{(2\Omega)^2}\Bigl[2q^2\lvert\betab[W]\rvert^2+6\rho[W]^2+6\sigma[W]^2+6q^{-2}\lvert\beta[W]\rvert^2+q^{-4}\lvert\alpha[W]\rvert^2\Bigr]\\
  =\int\ud v\int_{S_{u,v}}\dm{\gs}\frac{q}{(2\Omega)^2}\Bigl[2\lvert\betab_q[W]\rvert^2+6\rho[W]^2+6\sigma[W]^2+6\lvert\beta_q[W]\rvert^2+\lvert\alpha_q[W]\rvert^2\Bigr]\,.
\end{multline}
These fluxes should be compared with Lemma~\ref{lemma:flux:M} where the corresponding fluxes associated to $M$ were stated.

Let us prove the analogue of Lemma~\ref{lemma:K:M}:
\begin{lemma}\label{lemma:K:M:a}
  Let $M_a$ be defined as in \eqref{eq:M:a}. Then
  \begin{equation}
    K^{M_a}[W]=K_+^a[W]+K_-^a[W]
  \end{equation}
  where
  \begin{multline}
    \frac{(2\Omega)^3}{3}K_+^a[W]=\bigl(2a^{-1}\omegah+\Lh a^{-1}\bigr)a^4\lvert\alphab\rvert^2+2(a\tr\chib+a^{-1}\tr\chi+4a^{-1}\omegah-\Lbh a+\Lh a^{-1})a^2\lvert\betab\rvert^2\\+\bigl(4a \omegabh+4a\tr\chib-2\Lbh a+4a^{-1}\omegah+4a^{-1}\tr\chi-2\Lh a^{-1}\bigr)(\rho^2+\sigma^2)\\
      +2\bigl(a\tr\chib+a^{-1}\tr\chi+4a\omegabh+\Lbh a-\Lh a^{-1}\bigr)a^{-2}\lvert\beta\rvert^2 +\bigl(2a\omegabh+\Lbh a\bigr)a^{-4}\lvert\alpha\rvert^2
  \end{multline}
and
\begin{multline}
  \frac{\Omega^3}{3}K_-^a[W]=  a^3\alphab(2\etab^\sharp-\nablas\log a,\betab^\sharp)+a(\rho\betab+\sigma\ld\betab,2\eta+\nablas\log a)
  +\frac{1}{4}a^2(a\chibh+a^{-1}\chih,\rho\alphab+\sigma\ld \alphab)\\
  +2 a(\rho\betab+\sigma\ld\betab,2\etab-\ds\log a)-2a^{-1}(\rho\beta-\sigma\ld\beta,2\eta+\ds\log a)
  -(a\chibh+a^{-1}\chih)(\betab^\sharp,\beta^\sharp)\\
  -a^{-3}\alpha(2\eta^\sharp+\nablas\log a,\beta^\sharp)-a^{-1}(\rho\beta-\sigma\ld\beta,2\etab-\nablas\log a)
+\frac{1}{4}a^{-2}(a\chibh+a^{-1}\chih,\rho\alpha-\sigma\ld \alpha)
\end{multline}

\end{lemma}

\begin{remark}
  Note that the formula for $K_+^a$ reduces to \eqref{eq:K:plus} when $a=1$. 
Moreover, the formulas in Lemma~\ref{lemma:K:M:a} are obtained from those in Lemma~\ref{lemma:K:M} with the replacements
\begin{subequations}\label{eq:lorentz:structure:effective}
  \begin{gather}
    \omegabh\mapsto a\omegabh+\frac{1}{2}\Lbh a\qquad \omegah\mapsto a^{-1}\omegah+\frac{1}{2}\Lh a^{-1}\\
    \tr\chib+\tr\chi\mapsto a\tr\chib+a^{-1}\tr\chi-\Lbh a-\Lh a^{-1}\\
    \etab\mapsto\etab-\ds\log a\qquad \eta\mapsto\eta+\ds\log a\\
    \chibh\mapsto a\chibh\qquad \chih\mapsto a^{-1}\chih
  \end{gather}
\end{subequations}
%(they agree with the correspondence \eqref{eq:lorentz:structure} up to terms involving derivatives of $a$).
\end{remark}

\begin{proof}
As in \eqref{eq:KM:sum:pre} we have
  \begin{equation}\label{eq:KM:sum:pre:q}
  \begin{split}
     K^{M_a}&=\frac{3}{8}\frac{1}{\Omega^2}Q[W]_{\alpha\beta\gamma\delta}{}^{(M)}\hat{\pi}^{\alpha\beta} e_3^\gamma e_3^\delta  +\frac{3}{4}\frac{1}{\Omega^2}Q[W]_{\alpha\beta\gamma\delta}{}^{(M)}\hat{\pi}^{\alpha\beta} e_3^\gamma e_4^\delta  \\
     &\qquad +\frac{3}{8}\frac{1}{\Omega^2}Q[W]_{\alpha\beta\gamma\delta}{}^{(M)}\hat{\pi}^{\alpha\beta} e_4^\gamma e_4^\delta  
  \end{split}
\end{equation}

The statement of the Lemma then follows from the following contributions:
\begin{multline}
  \Omega\:Q_{\alpha\beta\gamma\delta}\,\hat{\pi}^{\alpha\beta}e_3^\gamma e_3^\delta=\\
  =\bigl(2a^{-1}\omegah+\Lh a^{-1}\bigr)a^4\lvert\alphab\rvert^2+2(a\tr\chib+a^{-1}\tr\chi-\Lbh a-\Lh a^{-1})a^2\lvert\betab\rvert^2
+\bigl(4a \omegabh+2\Lbh a\bigr)(\rho^2+\sigma^2)\\
  +8a^3\alphab(2\etab^\sharp-\nablas\log a,\betab^\sharp)+8a(\rho\betab+\sigma\ld\betab,2\eta+\nablas\log a)
  +2a^2(a\chibh+a^{-1}\chih,\rho\alphab+\sigma\ld \alphab)
\end{multline}
\begin{multline}
  \Omega Q_{\alpha\beta\gamma\delta}\hat{\pi}^{\alpha\beta}e_3^\gamma e_4^\delta=\\
  =\bigl(4a^{-1}\omegah+2\Lh a^{-1}\bigr)a^2\lvert\betab\rvert^2+2\bigl(a\tr\chib+a^{-1}\tr\chi-\Lbh a-\Lh a^{-1}\bigr)\bigl(\rho^2+\sigma^2\bigr)+\bigl(4a\omegabh+2\Lbh a\bigr)a^{-2}\lvert\beta\rvert^2\\
  +8 a(\rho\betab+\sigma\ld\betab,2\etab-\ds\log a)-8a^{-1}(\rho\beta-\sigma\ld\beta,2\eta+\ds\log a)
  -4(a\chibh+a^{-1}\chih)(\betab^\sharp,\beta^\sharp)
\end{multline}
\begin{multline}
  \Omega\:Q_{\alpha\beta\gamma\delta}\,\pi^{\alpha\beta}\Lh^\gamma\Lh^\delta=\\
  =\bigl(2a\omegabh+\Lbh a\bigr)a^{-4}\lvert\alpha\rvert^2  +2\bigl(a\tr\chib+a^{-1}\tr\chi-\Lbh a-\Lh a^{-1}\bigr)a^{-2}\lvert\beta\rvert^2 +\bigl(4a^{-1}\omegah+2\Lh a^{-1}\bigr)(\rho^2+\sigma^2)\\
  -8a^{-3}\alpha(2\eta^\sharp+\nablas\log a,\beta^\sharp)-8a^{-1}(\rho\beta-\sigma\ld\beta,2\etab-\nablas\log a)
+2a^{-2}(a\chibh+a^{-1}\chih,\rho\alpha-\sigma\ld \alpha)
\end{multline}

\end{proof}

%%% Local Variables:
%%% mode: latex
%%% TeX-master: "weyl"
%%% End:

\subsection{Global redshift estimate}
\label{sec:global:redshift}

In this Section we will show that the energy on $\Sigma_r$ associated to the current $P^M[W]$ decays uniformly in $r$. The decay mechanism lies in the expansion of the spacetime --- as manifested in our assumptions, in particular $\tr\chi>0$, and $\tr\chib>0$ --- and results in the positivity of the $K^M$:
In  Section~\ref{sec:redshift:vectorfield} we have proven that $K_+^M$ is positive, and in this section we will show that $K_-^M$ is an error which can be absorbed in $K_+^M$.

In order to exploit the positivity of $K_+[W]$ in the bulk term --- in comparison to the flux terms associated to $ P^M[W]$ ---  we need a version of the coarea formula: We foliate the spacetime domain $\mathcal{D}$ by the level sets of $r(u,v)=c$, and first note that we have already calculated the normal separation of the leaves in Lemma~\ref{lemma:normal:r}:
\begin{equation}\label{eq:phi}
  \phi=\frac{1}{\sqrt{-g(V,V)}}=\frac{2}{r}\frac{\Omega}{\sqrt{\overline{\Omega\tr\chi}\,\overline{\Omega\tr\chib}}}
\end{equation}
Therefore we have, for any function $f$,
\begin{equation}\label{eq:coarea:r}
  \int_{\mathcal{D}}f\dm{g}=\int\ud r\int_{\Sigma_r}\phi f\dm{\gb{r}}=\int\ud r\int_{\Sigma_r}\frac{2}{r}\frac{\Omega f}{\sqrt{\overline{\Omega\tr\chi}\,\overline{\Omega\tr\chib}}}\dm{\gb{r}}
\end{equation}

%We now compare, taking into account the lapse, the positive quantity $K_+$ with the energy flux of $P$.

Will first demonstrate the positivity of $K^M=K^M_++K^M_-$ under assumptions that are explicitly stated in the following Lemma~\ref{lemma:K:plus:energy} and Lemma~\ref{lemma:K:plus}. These assumptions are slightly more general than the assumptions (\textbf{BA:I}) which we discuss below to prove the positivity of $K^{M_q}$ in Lemma~\ref{lemma:K:plus:q}.

\begin{lemma}\label{lemma:K:plus:energy}
  Let the null structure coefficients satisfy 
  \begin{subequations}\label{eq:global:redshift:assumption}
    \begin{gather}
      \tr\chi >0\qquad \tr\chib >0\\
      \lvert 2\omega-\Omega \tr\chi\rvert\leq \epsilon\, \Omega \tr\chi
      \qquad \lvert 2\omegab- \Omega \tr\chib \rvert\leq \epsilon\, \Omega \tr\chib\\
      \lvert\Omega\tr\chi-\overline{\Omega\tr\chi}\rvert\leq \epsilon\,\overline{\Omega\tr\chi} \qquad       \lvert\Omega\tr\chib-\overline{\Omega\tr\chib}\rvert\leq \epsilon\,\overline{\Omega\tr\chib}
    \end{gather}
  \end{subequations}
  for some $\epsilon>0$.
  Then
  \begin{equation}
    \int_{\mathcal{D}} K_+[W] \dm{g}\geq  6\int\frac{\ud r}{r}\int_{\Sigma_r}\,(1-\epsilon)^2\:\ld P^M[W]\,.
  \end{equation}

\end{lemma}

\begin{proof}
Consider the expression \eqref{eq:K:plus} for $K_+$ in Lemma~\ref{lemma:K:M}.
We compare the coefficients to each curvature term with the corresponding coefficients in the expression \eqref{eq:P:prime} for the curvature flux $\ld P^M$ in Lemma~\ref{lemma:flux:M}.
We begin with $\alpha$, $\alphab$, $\beta$, $\betab$:
\begin{multline*}
  2\biggl[\frac{3}{4}\omegah\lvert\alphab[W]\rvert^2+3\omegah\lvert\betab[W]\rvert^2+\frac{3}{4}\bigl(\tr\chi+\tr\chib\bigr)\Bigr(\lvert\betab[W]\rvert^2+\lvert\beta[W]\rvert^2\Bigr)+3\omegabh\lvert\beta[W]\rvert^2+\frac{3}{4}\omegabh\lvert\alpha[W]\rvert^2\biggr]\\
  \geq\frac{6(1-\epsilon)}{8}\biggl[\tr\chi\lvert\alphab[W]\rvert^2+2\Bigl(3\tr\chi+\tr\chib\Bigr)\lvert\betab[W]\rvert^2+2\Bigl(\tr\chi+3\tr\chib\Bigr)\lvert\beta[W]\rvert^2+\tr\chib\lvert\alpha[W]\rvert^2\biggr]
\end{multline*}
and continue with $\rho$, $\sigma$:
\begin{equation*}
  2\frac{3}{2}\Bigl[\omegabh+\omegah+\tr\chi+\tr\chib\Bigr]\bigl(\rho[W]^2+\sigma[W]^2\bigr)\geq\frac{6(1-\epsilon)}{8}6\Bigl[\tr\chi+\tr\chib\Bigr]\bigl(\rho[W]^2+\sigma[W]^2\bigr)
\end{equation*}
We add up these two inequalties, and in view of the formula \eqref{eq:phi} for the lapse function, it then follows immediately from Lemma~\ref{lemma:K:M} and Lemma~\ref{lemma:flux:M} that
  \begin{equation*}
    \phi K_+\dm{\gb{r}}\geq \frac{6(1-\epsilon)^2}{r}\:\ld P^M
  \end{equation*}
and therefore by the co-area formula the statement of the Lemma follows.

\end{proof}

It remains to estimate the error terms occuring in Lemma~\ref{lemma:K:M}.

\begin{lemma}\label{lemma:K:plus}
  Let the structure coefficients satisfy
\begin{subequations}\label{eq:global:redshift:error:assumption}
\begin{gather}
  4 \lvert \etab \rvert_{\gs} \leq \epsilon \omegah \quad   4 \lvert \eta \rvert_{\gs} \leq \epsilon \omegabh\\
  4 \lvert \etab \rvert_{\gs} \leq \epsilon (\tr\chi+\tr\chib) \qquad 4 \lvert \eta \rvert_{\gs} \leq \epsilon (\tr\chi+\tr\chib)\\
  4 \lvert \chibh \rvert_{\gs} \leq \epsilon \omegah \quad 4 \lvert \chih \rvert_{\gs} \leq \epsilon \omegah\\
  4 \lvert \chibh \rvert_{\gs} \leq \epsilon \omegabh \quad 4 \lvert \chih \rvert_{\gs} \leq \epsilon \omegabh
\end{gather}
\end{subequations}
 for some $\epsilon>0$.
Then
\begin{equation}
   \lvert K_-\rvert  \leq \epsilon K_+\,.
\end{equation}

\end{lemma}

\begin{proof}
  Immediate from the expressions \eqref{eq:K:plus} and \eqref{eq:K:minus} obtained in Lemma~\ref{lemma:K:M}.
\end{proof}

% A completely analogous estimate can be derived for the vectorfield $M_q$ under the assumptions \eqref{eq:global:redshift:assumption}, and an additional assumption on the derivatives of of $q$:
% \begin{equation}\label{eq:assumption:D:q}
%   \lvert D\log q \rvert \leq \epsilon\, \Omega \tr\chi\quad\quad \lvert \Db\log q\rvert \leq \epsilon\, \Omega \tr\chib
% \end{equation}

 We will state the redshift estimate for $M_q$ under the stronger assumptions  \eqref{eq:assumption:intro:tr}, \eqref{eq:assumption:intro:omega}, \eqref{eq:assummption:intro:average},  \eqref{eq:assumption:intro:zeta}, introduced in Section~\ref{sec:intro:foliation}, which correspond to the choice $\epsilon = C_0\Omega^{-1}$ in Lemma~\ref{lemma:K:plus:energy}, \ref{lemma:K:plus}:
\begin{equation}\label{eq:BA:I} \tag{\emph{\textbf{BA:I}.i-iv}}
  \begin{split}
  \tr\chi>0 \quad&\quad \tr\chib>0\\
   \lvert 2\omega - \Omega \tr\chi \rvert \leq C_0 \tr\chi \quad&\quad  \lvert 2\omegab - \Omega \tr\chib \rvert \leq C_0 \tr\chib \\
  \lvert \Omega \tr\chi - \overline{ \Omega\tr\chi } \rvert \leq  C_0\Omega^{-1} \overline{\Omega \tr\chi } \quad&\quad   \lvert \Omega \tr\chib - \overline{ \Omega\tr\chib } \rvert \leq  C_0\Omega^{-1} \overline{\Omega \tr\chib } \\
  \Omega \lvert \chih \rvert_{\gs} \leq C_0 \tr\chi\quad & \quad \Omega \lvert \chibh \rvert \leq C_0 \tr\chib 
\end{split}
\end{equation}
and the additional assumptions
\begin{gather}\label{eq:BA:I:v} 
  \lvert D\log q \rvert \leq C_0 \tr\chi \qquad \lvert \Db \log q\rvert \leq C_0\tr\chib \tag{\emph{\textbf{BA:I}.v}}\\
  \Omega \lvert \eta \rvert + \Omega \lvert \etab \rvert +   \Omega\lvert \ds\log q\rvert \leq C_0 (q\tr\chib+q^{-1}\tr\chi) \tag{\emph{\textbf{BA:I}.vi}}
\end{gather}
\begin{remark}
With these assumptions --- together with \eqref{eq:assumption:intro:Omega} recalled below ---  the error as quantified in Lemma~\ref{lemma:K:plus:q} is of one order $r$ smaller than the error arising from the assumptions in  Lemma~\ref{lemma:K:plus:energy} and Lemma~\ref{lemma:K:plus}. This is ultimately the reason the decay rate of the Weyl curvature is $r^{-3}$ as opposed to $r^{-3+\epsilon}$.  
\end{remark}

Let us denote for simplicity by
\begin{equation}
  P^q[W]:= P^{(M_q,M_q,M_q)}\,,\qquad K^q[W]:=K^{(M_q,M_q,M_q)}[W]\,.
\end{equation}

\begin{lemma}\label{lemma:K:plus:q}
Let the structure coefficients satisfy (\textbf{BA:I}) for some $C_0>0$.
Then
\begin{equation}
      \phi K_+^q[W]\dm{\gb{r}}\geq \frac{6}{r} \bigl(1-2C_0\Omega^{-1}\bigr)^2 \:\ld P^q[W]
\end{equation}
Moreover, for some constant $C>0$,
\begin{equation}\label{eq:K:q:error}
  \phi \lvert K_-^q[W] \rvert \leq C C_0 \Omega^{-1} \phi K_+^q[W]\,.
\end{equation}

\end{lemma}
\begin{proof}
Note first that under the assumptions on the structure coefficients
 \begin{gather*}
    \Omega \lvert \Lbh q\rvert\leq C_0 q\tr\chib\qquad \Omega \lvert \Lh q^{-1}\rvert \leq C_0 q^{-1}\tr\chi\\
    2 q^{-1}\omegah+\Lh q^{-1}\geq (1-2C_0\Omega^{-1}) q^{-1}\tr\chi\\
    q\tr\chib+q^{-1}\tr\chi+4q^{-1}\omegah-\Lbh q+\Lh q^{-1}\geq (1-C_0\Omega^{-1})q\tr\chib+3 (1-C_0\Omega^{-1}) q^{-1}\tr\chi\displaybreak[0]\\
    4q \omegabh+4q\tr\chib-2\Lbh q+4q^{-1}\omegah+4q^{-1}\tr\chi-2\Lh q^{-1}\geq (6-4 C_0\Omega^{-1})q\tr\chib +(6-4 C_0\Omega^{-1})q^{-1}\tr\chi\\
    q\tr\chib+q^{-1}\tr\chi+4q\omegabh+\Lbh q-\Lh q^{-1} \geq 3(1-C_0\Omega^{-1}) q\tr\chib+(1-C_0\Omega^{-1})q^{-1}\tr\chi\\
    2q\omegabh+\Lbh q \geq (1-2 C_0\Omega^{-1})q\tr\chib
  \end{gather*}

Hence by Lemma~\ref{lemma:K:M:a}:
\begin{multline*}
      \frac{(2\Omega)^3}{3}K_+^q[W]\geq (1-2C_0\Omega^{-1}) q^{-1}\tr\chi q^4\lvert\alphab\rvert^2+2(1-C_0\Omega^{-1})\bigl(q\tr\chib+3 q^{-1}\tr\chi\bigr)q^2\lvert\betab\rvert^2\\
+(6-4C_0\Omega^{-1})\bigl(q\tr\chib +q^{-1}\tr\chi\bigr)(\rho^2+\sigma^2)\\
      +2(1-C_0\Omega^{-1})\bigl(3 q\tr\chib+q^{-1}\tr\chi\bigr)q^{-2}\lvert\beta\rvert^2 +(1-2C_0\Omega^{-1})q\tr\chib q^{-4}\lvert\alpha\rvert^2
\end{multline*}
Secondly note that by the definition of $q$,
  \begin{gather*}
    \phi\, q^{-1}\tr\chi = \frac{2}{r}\frac{\Omega\tr\chi}{\overline{\Omega\tr\chi}}\geq \frac{2}{r}(1-C_0\Omega^{-1})\\
    \phi\, q\tr\chib =\frac{2}{r}\frac{\Omega\tr\chib}{\overline{\Omega\tr\chib}}\geq \frac{2}{r}(1-C_0\Omega^{-1})
  \end{gather*}
Therefore
\begin{equation*}
      \frac{1}{3} \phi K_+^q[W]\geq \frac{2}{r}\frac{(1-2C_0\Omega^{-1})^2}{(2\Omega)^3}\biggl[  q^4\lvert\alphab\rvert^2+ 8q^2\lvert\betab\rvert^2+12(\rho^2+\sigma^2)+8q^{-2}\lvert\beta\rvert^2 + q^{-4}\lvert\alpha\rvert^2\biggr]
\end{equation*}
and thus in comparison to \eqref{eq:flux:Ma},
\begin{equation*}
  \phi K_+^q \dm{\gb{r}} \geq (1-2C_0\Omega^{-1})^2\,\frac{6}{r}\,\ld P^{M_q}
\end{equation*}

The final bound \eqref{eq:K:q:error} follows by inspection of the formulas given in Lemma~\ref{lemma:K:M:a}.
\end{proof}

Now by  \eqref{eq:assumption:intro:Omega} we   assume
\begin{equation}
  \frac{C_0}{r}\geq \Omega^{-1}\geq \frac{1}{C_0 r}\,;
\end{equation}
this uniform bound of $\Omega$ on $\Sigma_r$ is necessary for the sharp decay rate  in the following Proposition.
We will state the main conclusion for the energy current associated to $M_q$, but this result of course also holds for the energy associated to $M$ and the same proof applies.

\begin{proposition}\label{prop:global:redshift:rough}
Assume (\textbf{BA:I}) and \eqref{eq:assumption:intro:Omega} for some  $C_0>0$.
Then there exists a constant $C(r_0)$ such that for any solution $W$ to \eqref{eq:W:Bianchi} with initial data such that
\begin{equation}\label{eq:global:redshift:data}
  D^q[W]=\int_{\Sigma_{r_0}}\ld P^q[W]<\infty
\end{equation}
we have
\begin{equation}\label{eq:global:redshift:Sigmar}
  r^{6}\int_{\Sigma_{r}}\ld P^q[W]\leq C( r_0 ) D^q[W]\,,\qquad (r\geq r_0)\,.
\end{equation}
\end{proposition}

\begin{proof}
  Apply Proposition~\ref{prop:energy:identity} to the energy current $P^q[W]$ to obtain the inequality
  \begin{equation*}
  \int_{\Sigma_{r_2}^c}\ld P^q[W] + \int_{\mathcal{D}_{(r_1,r_2)}} K^q[W]\dm{g}\leq \int_{\Sigma_{r_1}^c}\ld P^q[W]
\end{equation*}
for any $r_2>r_1>r_0$. Here we used that the flux terms through the null hypersurfaces $C_{\underline{u}}^c$ and $\underline{C}_{\underline{v}}^c$ as given by \eqref{eq:flux:ld:P:null} are nonnegative due to the positivity properties of $Q(W)$ (see Section~\ref{sec:intro:proof} and \CProp{12.5}) which allows us to drop the first and the fourth term on the left hand side of \eqref{eq:energy:identity}:
\begin{equation}\label{eq:flux:capped}
    \int_{C_u} \ld P^q =\int \Omega Q[W](\Lh,M_q,M_q,M_q)\ud v\dm{\gs}\geq 0\qquad \int_{\underline{C}_{\underline{v}}^c} \ld P^q\geq 0
\end{equation}
Moreover the third term on the left hand side of \eqref{eq:energy:identity} vanishes by \eqref{eq:Q:W:div}.
Let us choose  $r_0$  sufficiently large, so that
\begin{equation*}
  CC_0\Omega^{-1}< 1\qquad (r\geq r_0)
\end{equation*}
Then by the co-area formula \eqref{eq:coarea:r}, and by Lemma~\ref{lemma:K:plus:q},
\begin{multline*}
  \int_{\mathcal{D}_{(r_1,r_2)}} K^q[W]\dm{g} = \int_{r_1}^{r_2}\ud r \int_{\Sigma_r^c}\phi K^q[W] \dm{\gb{r}}\\\geq  \int_{r_1}^{r_2}\ud r \int_{\Sigma_r^c}(1-C C_0\Omega^{-1})\phi K_+^q[W] \dm{\gb{r}}\geq 6  \int_{r_1}^{r_2} \frac{\ud r}{r}\bigl(1-\frac{CC_0^2}{r}\bigr)^3\int_{\Sigma_r^c} \ld P^q[W] 
\end{multline*}
which implies the inequality
  \begin{equation*}
  \int_{\Sigma_{r_2}^c}\ld P^q[W] + 6 \int_{r_1}^{r_2}\ud r\frac{1}{r}\int_{\Sigma_r^c} \ld P^q[W]\leq \int_{\Sigma_{r_1}^c}\ld P^q[W]+ C \int_{r_1}^{r_2}\ud r\frac{1}{r^2}\int_{\Sigma_r^c} \ld P^q[W]\,.
\end{equation*}
A Gronwall-type argument then implies the statement of the Proposition, see Section~\ref{sec:gronwall}.

Note that  in view \eqref{eq:global:redshift:data},  and again  using \eqref{eq:flux:capped}  we can pass from the estimate on $\mathcal{D}^{(\underline{u},\underline{v})}_{(r_1,r_2)}$ defined in \eqref{eq:D:uvr} to the unbounded domain $\cup_{r_1\leq r\leq r_2}\Sigma_r$ by taking the limits $\underline{u}\to\infty$, and $\underline{v}\to\infty$.

\end{proof}

%\todo{Possibly include equivalent statement for weighted null flux; see previous version}

\begin{remark}
  An equivalent statement can also be derived for a weighted null flux. 
\end{remark}

\section{First order redshift}
\label{sec:commuted:redshift}

The aim of this Section is derive an energy estimate for $\nabla W$, similar to the redshift estimate for $W$ in Section~\ref{sec:global:redshift}.
This is achieved by commuting the Bianchi equations \eqref{eq:W:Bianchi} with a vectorfield $X$, which yields an inhomogeneous equation of the form \eqref{eq:Bianchi:MLie:W} for the modified Lie derivative $\MLie{X}W$. The strategy here is to choose $X$ to be future-directed time-like, in fact colinear with the normal to $\Sigma_r$, and to derive a redshift estimate for solutions $\MLie{X}W$ to \eqref{eq:Bianchi:MLie:W}, which can then be used to control all derivatives $\nablab W$ \emph{tangential} to $\Sigma_r$. This last step relies on an elliptic estimate in the context of the electro-magnetic decomposition of $W$ with respect to $\Sigma_r$, which we will discuss separately in Section~\ref{sec:electromagnetic}.

A natural choice of the commutator  would be
\begin{equation}
  M_q=\frac{1}{2\Omega}\bigl(e_3+e_4)
\end{equation}
where
\begin{equation}
  e_3=q\Lbh\qquad e_4=q^{-1}\Lh\,;
\end{equation}
namely the ``aligned'' redshift vectorfield of Section~\ref{sec:redshift:vectorfield}.
The task is then to exhibit a positivity property of 
\begin{equation}\label{eq:div:M:q:term}
  \bigl(\divergence Q(\MLie{M_q}W)\bigr)(M_q,M_q,M_q)\,,
\end{equation}
which appears as an additional term in the energy identity of Proposition~\ref{prop:energy:identity}, for solutions to the inhomogeneous Bianchi equations.
While \eqref{eq:div:M:q:term} \emph{does} have a sign  in the highest order terms $\nabla W$, the lower order terms at the level of $W$ still form an obstruction to the required decay \emph{rate} of the energy associated to $\MLie{M_q}W$ on $\Sigma_r$.
We choose instead as commutator vectorfield:\footnote{One reason to expect that this vectorfield produces the correct lower order terms can already be inferred without computing the terms in \eqref{eq:div:M:q:term}: The modified Lie derivative is an expression of the form \eqref{eq:MLie:W:N}. Given control on $\MLie{N}W$, we only obtain control on all tangential derivatives by the elliptic estimate of Section~\ref{sec:electromagnetic}, if the lower order term in the expression for $\MLie{N}W$ match precisely the lower order terms in the ``Maxwell equations''  \eqref{eq:maxwell} of the electro-magnetic decomposition on $\Sigma_r$, cf.~Lemma~\ref{lemma:maxwell:sources}. This gives a  condition on $\tr\pid{N}$ which can motivate the rescaling by $\Omega^2$ in \eqref{eq:N}.}
\begin{equation}
  \label{eq:N}
  N=\Omega^2 M_q
\end{equation}

\begin{remark}
  In this section $n$  will denote one of the components of the deformation tensor of $N$ to be defined in \eqref{eq:nb:n}, and \emph{not} the unit norm to $\Sigma_r$.
\end{remark}

\begin{remark}
  Note that $N$ is orthogonal to the spheres $S_{u,v}$. So while the scalars $q$ and $\Omega$ determine its direction and magnitude in the plane $\langle \Lbh,\Lh\rangle \subset \mathrm{T}\mathcal{M}$ spanend by $(\Lbh,\Lh)$, the plane itself always satisfies $\langle \Lbh,\Lh\rangle^\perp\subset\mathrm{T}S_{u,v}$. Its orientation thus depends on the spheres $S_{u,v}$ whose embedding in $\mathcal{M}$ is characterised by the assumptions $(\textbf{BA})$.
\end{remark}

Now by the classical \CProp{12.1} the Weyl field $\MLie{N}W$,
\begin{multline}\label{eq:MLie:W:N}
  \MLie{N}W_{\alpha\beta\gamma\delta}=\mathcal{L}_NW_{\alpha\beta\gamma\delta}-\frac{1}{8}\tr\pid{N}\,W_{\alpha\beta\gamma\delta}\\
  -\frac{1}{2}\Bigl[\pih{N}_{\alpha}^{\phantom{\alpha}\mu}W_{\mu\beta\gamma\delta}+\pih{N}_{\beta}^{\phantom{\beta}\mu}W_{\alpha\mu\gamma\delta}+\pih{N}_{\gamma}^{\phantom{\gamma}\mu}W_{\alpha\beta\mu\delta}+\pih{N}_{\delta}^{\phantom{\delta}\mu}W_{\alpha\beta\gamma\mu}\Bigr]
\end{multline}
satisfies the equation
\begin{equation}\label{eq:Bianchi:MLie:W}
  \nabla^\alpha\bigl(\MLie{N}W\bigr)_{\alpha\beta\gamma\delta}={}^{(N)}J(W)_{\beta\gamma\delta}
\end{equation}
where the Weyl current $\MJ{N}{W}$ is detailed below in \eqref{eq:MJN}.
Therefore, according to \CLemma{12.3} we have the important formula
\begin{multline}\label{eq:divQ:M:null}
    8\Omega^3 \bigl(\divergence Q(\MLie{N}W)\bigr)(M_q,M_q,M_q) =\phantom{+} 4(\alphabt_q,\Thetabt_q)+8(\betabt_q,\Xibt_q)\\
    +3\Bigl(8\rhot_q\bigl(\Lambdabt_q+\Lambdat_q\bigr)-8\sigmat_q\bigl(\Kbt_q-\Kt_q\bigr)-8(\betabt_q,\Ibt_q)+8(\betat_q,\It_q)\Bigr)\\
    +4(\alphat_q,\Thetat_q)-8(\betat_q,\Xit_q)
\end{multline}
where $\alphabt_q,\betabt_q$,$\rhot_q,\sigmat_q$,$\betat_q,\alphat_q$ refer to the null components of $\MLie{N}W$,
and $\Xibt_q,\Xit_q$, $\Lambdabt_q,\Lambdat_q$, $\Kbt_q,\Kt_q$, $\Thetabt_q,\Thetat_q$ refer to the null components of $\MJ{N}{W}$.

\begin{remark}
Throughout this Section we use a null decomposition with respect to the null frame $(e_3,e_4;e_A)$.
However the null structure coefficents are still associated to the frame coming from the foliation $(\Lbh,\Lh;e_A)$.
To avoid confusion we append a subscript $q$ to any null components decomposed relative to $(e_3,e_4;e_A)$.
In particular in reference to \eqref{eq:divQ:M:null} we have
  \begin{gather*}
    (\alphabt_q)_{AB}:=(\alphab_q[\MLie{N}W])_{AB}:=\MLie{N}W(e_A,e_3,e_B,e_3)\quad (\alphat_q)_{AB}:=\MLie{N}W(e_A,e_4,e_B,e_4)\\ (\betabt_q)_A:=(\betab_q[\MLie{N}W])_{A}:=\frac{1}{2}\MLie{N}W(e_A,e_3,e_3,e_4)\quad (\betat_q)_A:=\frac{1}{2}\MLie{N}W(e_A,e_4,e_3,e_4)\\ \rhot_q:=\frac{1}{4}(\MLie{N}W)(e_3,e_4,e_3,e_4) \quad\sigmat_q:=\frac{1}{4}\epsilons^{AB}(\MLie{N}W)(e_A,e_B,e_3,e_4)
  \end{gather*}
where in fact $\rhot_q=\rhot$ and $\sigmat_q=\sigmat$ by \eqref{eq:lorentz:rho:sigma}, and
\begin{gather*}
(\Xibt_q)_A:=(\Xib_q[{}^{(N)} J(W)])_A:=\frac{1}{2}{}^{(N)} J(W)_{33A}\quad (\Xit_q)_A:=\frac{1}{2}{}^{(N)} J(W)_{44A}\\ \Lambdabt_q:=\Lambdab_q[{}^{(N)}J(W)]:=\frac{1}{4}{}^{(N)}J(W)_{343}\quad \Lambdat_q:=\frac{1}{4}{}^{(N)}J(W)_{434}\displaybreak[0]\\ 
 \Kbt_q:=\Kb_q[{}^{(N)}J(W)]:=\frac{1}{4}\epsilons^{AB}{}^{(N)}J(W)_{3AB}\quad \Kt_q:=\frac{1}{4}\epsilons^{AB}{}^{(N)}J(W)_{4AB}\\ 
(\Thetabt_q)_{AB}=\Thetab(\MJ{N}{W})_{AB}\quad (\Thetat_q)_{AB}=\Theta(\MJ{N}{W})_{AB}\,,
  \end{gather*}
cf.~\CCh{12.2}.
\end{remark}

According to \CProp{12.1} we have that the Weyl current $\MJ{N}{W}$ in \eqref{eq:Bianchi:MLie:W} is given by
\begin{equation}\label{eq:MJN}
  \MJ{N}{W}={}^{(N)}J^1(W)+{}^{(N)}J^2(W)+{}^{(N)}J^3(W)
\end{equation}
where, cf.~\Ceq{14.5},
\begin{subequations}\label{eq:J:123}
  \begin{gather}
    {}^{(N)}J^1(W)_{\beta\gamma\delta}=\frac{1}{2}{}^{(N)}\hat{\pi}^{\mu\nu}\nabla_\nu W_{\mu\beta\gamma\delta}\\
    {}^{(N)}J^2(W)_{\beta\gamma\delta}=\frac{1}{2}p_\mu W^{\mu}_{\phantom{\mu}\beta\gamma\delta}\\
    {}^{(N)}J^3(W)_{\beta\gamma\delta}=\frac{1}{2}\Bigl(d_{\mu\beta\nu}W^{\mu\nu}_{\phantom{\mu\nu}\gamma\delta}+d_{\mu\gamma\nu}W^{\mu\phantom{\beta}\nu}_{\phantom{\mu}\beta\phantom{\nu}\delta}+d_{\mu\delta\nu}W^{\mu\phantom{\beta\gamma}\nu}_{\phantom{\mu}\beta\gamma\phantom{\nu}}\Bigr)
  \end{gather}
\end{subequations}
and\footnote{In this paper we denote the Weyl current defined in \eqref{def:d} by $d$, and not $q$ as in \cite{ch:blue}, because the notation $q$ is already used in Lemma~\ref{lemma:normal:r}.}
\begin{subequations}
  \begin{gather}
    p_\mu:= {}^{(N)}p_\mu:= \nabla^\alpha{}^{(N)}\,\hat{\pi}_{\alpha\mu}\\
    d_{\alpha\beta\gamma}:={}^{(N)}d_{\alpha\beta\gamma}:=\nabla_\beta{}^{(N)}\hat{\pi}_{\gamma\alpha}-\nabla_\gamma{}^{(N)}\hat{\pi}_{\beta\alpha}+\frac{1}{3}\Bigl(p_\beta g_{\alpha\gamma}-p_\gamma g_{\alpha\beta}\Bigr)\label{def:d}
  \end{gather}
\end{subequations}

\begin{remark}
  Note that only the part $J^1(W)$ contains terms $\nabla W$, and thus only the null decomposition of $J^1$ may contain $\alphabt,\alphat$,$\betabt,\betat$,$\rhot,\sigmat$. We sometimes refer to these as ``principal terms''. In the first place it suffices then to look at the components of $J^1$ to show the presence of positive quadratic terms in $\nabla W$ in  \eqref{eq:divQ:M:null}. Its null components are calculated in \CLemma{14.1}; in fact, the formulas in \CLemma{14.1} are only true in the special case that two components of the deformation tensor $n=\nb=0$ (defined below) vanish, and we will discuss the general case in Section~\ref{sec:weyl:currents}.
\end{remark}

\begin{remark}  The discussion of the ``lower order'' terms --- involving $W$ --- then requires the inclusion of the parts $J^2$, and $J^3$.   The null decompositions of $J^2$ and $J^3$ are detailed in \CCh{14.3-4}, cf.~\CLemma{14.2}.
 We emphasize that these parts cannot be treated separately, because cancellations appear accross the expressions for $J^1$, $J^2$, and $J^3$. The relevant terms are  collected under the label $X$ --- in fact, $X:=X^1+X^{2+3}$ where $X^1$ is defined in \eqref{eq:X:Y:1} and refers to the relevant terms appearing in $J^1$, and $X^{2+3}$ is defined in \eqref{eq:X:Y:2:3} and refers to the relevant terms coming from the currents $J^2$ and $J^3$ --- and the cancellations are exhibited in Section~\ref{sec:J:cancel}. We also refer to Remark~\ref{remark:prefactor:cancel} below --- where we discuss specific terms contributing to these cancellations --- for further explanation of the choice of \eqref{eq:N:q}.
\end{remark}

We adopt the notation of \CCh{8.1} for the components of the deformation tensor, i.e.
\begin{subequations}\label{eq:pih:components}
\begin{gather}
      i:={}^{(N)}i:={}^{(N)}\hat{\pi}\!\!\!/\\
 \underline{m}:={}^{(N)}\underline{m}_q:={}^{(N)}\hat{\pi}\!\!\!/_3\qquad  m:={}^{(N)}m_q:={}^{(N)}\hat{\pi}\!\!\!/_4\\
    \underline{n}:={}^{(N)}\underline{n}_q:={}^{(N)}\hat{\pi}_{33}\qquad  n:={}^{(N)}n_q:={}^{(N)}\hat{\pi}_{44}\label{eq:nb:n}\\
    j:={}^{(N)}\hat{\pi}_{34}
\end{gather}
\end{subequations}
and calculate here in particular the values for the commutator vectorfield $N$; recall also our results from Lemma.~\ref{lemma:deformation:M:a}.

\begin{lemma}\label{lemma:deformation:N}
  The components of the deformation tensor of
  \begin{equation}\label{eq:N:q}
    N=\frac{\Omega}{2}\bigl(e_3+e_4)
  \end{equation}
  are
\begin{subequations}\label{eq:ijk:N}
  \begin{gather}
    i=\Omega\bigl(q\chibh+q^{-1}\chih\bigr)+\frac{\Omega}{4}\bigl(q\tr\chib+q^{-1}\tr\chi-2q\omegabh-2q^{-1}\omegah-\Lbh q-\Lh q^{-1} \bigr)\gs\\
    \underline{m}=\Omega\bigl(2\zeta+\ds\log q\bigr)\qquad m=-\Omega\bigl(2\zeta+\ds\log q\bigr)\\
    \underline{n}=2\Omega \Lbh q\qquad  n=2\Omega\Lh q^{-1}\\
    j=\frac{\Omega}{2}\Bigl(q\tr\chib+q^{-1}\tr\chi-2q\omegabh-2q^{-1}\omegah-\Lbh q-\Lh q^{-1}\Bigr)
  \end{gather}
\end{subequations}

\end{lemma}

\begin{proof}
  Recall that we have already calculated the deformation tensor of $M_q$ in  Lemma.~\ref{lemma:deformation:M:a}.
Moreover,
\begin{subequations}
\begin{gather*}
  {}^{(N)}\pi_{34}=-2\bigl(e_3+e_4\bigr)\Omega+\Omega^2\,{}^{(M_q)}\pi_{34}\\
  {}^{(N)}\pi_{33}=-4e_3\Omega+\Omega^2\,{}^{(M_q)}\pi_{33}\qquad
  {}^{(N)}\pi_{44}=-4e_4\Omega+\Omega^2\,{}^{(M_q)}\pi_{44}\\
  {}^{(N)}\pi_{3A}=-2e_A\Omega+\Omega^2\,{}^{(M_q)}\pi_{3A}\qquad   {}^{(N)}\pi_{4A}=-2e_A\Omega+\Omega^2\,{}^{(M_q)}\pi_{4A}\\
  {}^{(N)}\pi_{AB}=\Omega^2\,{}^{(M_q)}\pi_{AB}
\end{gather*}
\end{subequations}
and 
\begin{equation*}
  \tr{}^{(N)}\pi=2\bigl(e_3+e_4\bigr)\Omega+\Omega^2\tr{}^{(M_q)}\pi
\end{equation*}
Now recall from \eqref{eq:D:log:Omega} that
\begin{gather*}
  e_3\Omega=q\Db\log \Omega=q\omegab\qquad e_4\Omega=q^{-1}\omega\\
    \tr{}^{(M_q)}\pi=\frac{1}{\Omega}\bigl(q\tr\chib+q^{-1}\tr\chi+\Lbh q+\Lh q^{-1}\bigr)
\end{gather*}
to infer that
\begin{equation*}
  \tr{}^{(N)}\pi= \Omega\Bigl( 2q\omegabh +2q^{-1}\omegah+q\tr\chib+q^{-1}\tr\chi+\Lbh q+\Lh q^{-1}\Bigr)
\end{equation*}
Moreover, using the results of Lemma~\ref{lemma:deformation:M:a}, 
\begin{gather*}
  \pih{N}_{34}=-\bigl(q\omegab+q^{-1}\omega\bigr)+\Omega^2\pih{M_q}_{34}=\frac{\Omega}{2}\Bigl(q\tr\chib+q^{-1}\tr\chi-2q\omegabh-2q^{-1}\omegah-\Lbh q-\Lbh q^{-1}\Bigr)\\
  \pih{N}_{33}=2\Omega\Lbh q\qquad \pih{N}_{44}=2\Omega\Lh q^{-1}\\
  \pih{N}_{3A}=\Omega\bigl(2\eta_A+\ds_A\log q-2\ds_A\log \Omega\bigr)=\Omega\bigl(2\zeta_A+\ds_A\log q\bigr)\\
  \pih{N}_{4A}=\Omega\bigl(2\etab_A-\ds_A\log q-2\ds_A\log \Omega\bigr)=-\Omega\bigl(2\zeta_A+\ds_A\log q\bigr)\\
  \pih{N}_{AB}=\Omega\bigl(q\chibh+q^{-1}\chih\bigr)+\frac{\Omega}{4}\bigl(q\tr\chib+q^{-1}\tr\chi-2q\omegabh-2q^{-1}\omegah-\Lbh q-\Lh q^{-1}\bigr){\gs}_{AB}
\end{gather*}
\end{proof}

%\todo{Insert here remarks from earlier version.}

Given that in this Section we work mainly with the null decomposition with respect to $(e_3,e_4;e_A)$ we will state here for convenience the form of the Bianchi equations relative to this frame.

\begin{proposition}\label{prop:bianchi:q}
  The Bianchi equations decomposed in the frame $(e_3,e_4;e_A)$ read as in \CProp{12.4} with the following replacements:
\begin{subequations}\label{eq:replacements}
\begin{gather}
  D\to q^{-1}D\qquad \Db\to q\Db\\
  \alpha\to\alpha_q\quad\beta\to\beta_q\quad\betab\to\betab_q\quad\alphab\to\alphab_q\\
    \chi\to q^{-1}\chi\qquad\chib\to q\chib  \qquad   \zeta\to \zeta+\ds\log q\\
    \omegabh\to q\omegabh+\Omega^{-1}\Db q\qquad \omegah\to q^{-1}\omegah+\Omega^{-1}D q^{-1}
\end{gather}
\end{subequations}

\end{proposition}

\begin{proof}
  First derive the analogous formulas to \CProp{1.1}. Since these are derived using Leibniz rule and the frame relations \Ceq{1.175}, it is clear that if instead the relations \eqref{eq:frame:a} are used, they are formally obtained with replacements
  \begin{gather*}
        \chi\to q^{-1}\chi\qquad\chib\to q\chib  \qquad   \zeta\to \zeta+\ds\log q\\
    \omegabh\to q\omegabh+\Omega^{-1}\Db q\qquad \omegah\to q^{-1}\omegah+\Omega^{-1}D q^{-1}
  \end{gather*}
  and 
  \begin{equation*}
      \alpha\to\alpha_q\quad\beta\to\beta_q\quad\betab\to\betab_q\quad\alphab\to\alphab_q
  \end{equation*}
  because the null decomposition now refers to the frame $(e_3=q\Lbh, e_4=q^{-1}\Lh;e_A)$.
Therefore \Ceq{1.190} holds for the null decomposition of $W$ with respect to this frame, and with the above replacements.
It remains to check that, by \Ceq{1.191,1.193} for any 1-form $\xi$, and any 2-form $\theta$,
\begin{gather*}
  q^{-1}D\xi=q^{-1}\Omega\bigl(\nablas_{\Lh}\xi+\chi^\sharp\cdot\xi\bigr)=\Omega\bigl(\nablas_4\xi+q^{-1}\chi^\sharp\cdot\xi\bigr)\\
  q^{-1}D\theta=q^{-1}\Omega\bigl(\nablas_{\Lh}\theta+\chi\times\theta+\theta\times\chi\bigr)=\Omega\bigl(\nablas_4\theta+q^{-1}\chi\times\theta+q^{-1}\theta\times\chi\bigr)
\end{gather*}
(similarly for the conjugate equations) and thus the above replacements of the structure coefficients are consistent with the replacements
\begin{equation*}
    D\to q^{-1}D\qquad \Db\to q\Db\,.
\end{equation*}
\end{proof}

\subsection{Commutations}

In this section we compute the null decomposition of the Weyl field $\MLie{N}W$.
These computations are carried out using Leibniz rule for the Lie derivative
\begin{multline}\label{eq:Lie:Leibniz}
  (\mathcal{L}_NW)(e_\alpha,e_\beta,e_\gamma,e_\delta)=N(W(e_\alpha,e_\beta,e_\gamma,e_\delta))-W([N,e_\alpha],e_\beta,e_\gamma,e_\delta)-W(e_\alpha,[N,e_\beta],e_\gamma,e_\delta)\\-W(,e_\alpha,e_\beta,[N,e_\gamma],e_\delta)-W(,e_\alpha,e_\beta,e_\gamma,[N,e_\delta])
\end{multline}
which is also used to define $\Lies_N\alphab$, $\Lies_N\alpha$, $\Lies_N\betab$, $\Lies_N\beta$, for example:
\begin{equation}\label{eq:Lies:alpha}
  (\Lies_N\alpha)_{AB}=N(\alpha_{AB})-\alpha(\Pi[N,e_A],e_B)-\alpha(e_B,\Pi[N,e_B])
\end{equation}
where the projection $\Pi$ to the spheres is defined in \eqref{eq:Pi}.

\begin{lemma}\label{lemma:null:decomposition:LieMW}
  The null components of $\MLie{N}W$ are given by
  \begin{subequations}
    \begin{multline}
            \alpha_q[\MLie{N}W]=\hat{\Lies}_{N} \alpha_q+\frac{\Omega}{8}\bigl(6q\omegabh+6q^{-1}\omegah-q\tr\chib-q^{-1}\tr\chi+7\Lbh q-\Lh q^{-1}\bigr)\alpha_q\\-\Omega\bigl(2\zeta-\ds\log q\bigr)\otimesh\beta_q
    \end{multline}
    \begin{multline}
           \alphab_q[\MLie{N}W]=\hat{\Lies}_{N} \alphab_q+\frac{\Omega}{8}\bigl(6q^{-1}\omegah+6q\omegabh-q\tr\chib-q^{-1}\tr\chi+7\Lh q^{-1}-\Lbh q\bigr)\alphab_q\\-\Omega\bigl(2\zeta+\ds\log q\bigr)\otimesh\betab_q
    \end{multline}
    \begin{multline}
            \beta_q[\MLie{N}W]=\Lies_{N}\beta_q-\frac{1}{2}\hat{i}^\sharp\cdot\beta_q+\frac{\Omega}{8}\bigl(6q\omegabh+6q^{-1}\omegah+q\tr\chib+q\tr\chi+\Lbh q+\Lh q\bigr)\beta_q\\-\frac{3}{4}\Omega\bigl(2\zeta+\ds\log q\bigr)\rho-\frac{3}{4}\Omega\bigl(2\ld\zeta+\ld\ds\log q\bigr)\sigma+\frac{1}{4}\Omega\alpha^\sharp_q\cdot \bigl(2\zeta+\ds\log q\bigr)
    \end{multline}
    \begin{multline}
            \betab_q[\MLie{N}W]=\Lies_{N}\betab_q-\frac{1}{2}\hat{i}^\sharp\cdot\betab_q+\frac{\Omega}{8}\bigl(6q\omegabh+6q^{-1}\omegah+q\tr\chib+q\tr\chi+\Lbh q+\Lh q\bigr)\betab_q\\-\frac{3}{4}\Omega\bigl(2\zeta+\ds\log q\bigr)\rho+\frac{3}{4}\Omega\bigl(2\ld\zeta+\ld\ds\log q\bigr)\sigma+\frac{1}{4}\Omega\alphab^\sharp_q\cdot\bigl(2\zeta+ \ds\log q\bigr)
    \end{multline}
    \begin{multline}
            \rho[\MLie{N}W]=N\rho+\frac{3}{8}\Omega\bigl(2q\omegabh +2q^{-1}\omegah+q\tr\chib+q^{-1}\tr\chi+\Lbh q+\Lh q^{-1}\bigr)\rho\\+\frac{\Omega}{2}(2\zeta+\ds\log q,\beta_q+\betab_q)
    \end{multline}
    \begin{multline*}
      \sigma[\MLie{N}W]=N\sigma+\frac{3}{8}\Omega\Bigl(2q\omegabh+2q^{-1}\omegah+q\tr\chib+q^{-1}\tr\chi+\Lbh q+\Lh q^{-1}\Bigr)\sigma\\
    +\frac{\Omega}{2}\bigl(2\ld\zeta+\ld\ds\log q,\beta_q-\betab_q\bigr)
    \end{multline*}
  \end{subequations}

\end{lemma}

\begin{lemma}\label{lemma:commute:M:frame}
  The commutation relations of the vectorfield $N$ as given by \eqref{eq:N:q} with the frame are
  \begin{subequations}
    \begin{gather}
      [N,e_A]=\frac{\Omega}{2}\Pi [e_3+e_4,e_A]-\frac{\Omega}{2}\ds_A\log q \bigl(e_3-e_4\bigr)\\
      [N,e_3]=-\frac{1}{2}\bigl(\mb-m\bigr)^{\sharp C}e_C+\Omega\ds\log q^{\sharp C}e_C+\frac{1}{4}\nb e_4-\frac{\Omega}{2}\bigl(q^{-1}\omegah+q\omegabh+\Lh q^{-1}\bigr)e_3\\
      [N,e_4]=\frac{1}{2}\bigl(\mb-m\bigr)^{\sharp C}e_C-\Omega\ds\log q^{\sharp C}e_C  -\frac{\Omega}{2}\bigl(q\omegabh+\Lbh q+q^{-1}\omegah\bigr)e_4+\frac{1}{4}n e_3
    \end{gather}
  \end{subequations}

\end{lemma}

\begin{proof}
  It follows from the frame relations \eqref{eq:frame:a} that
  \begin{subequations}
    \begin{gather*}
      [e_3,e_4]=\nabla_3e_4-\nabla_4e_3=2\eta^\sharp-2\etab^\sharp-\bigl(q\omegabh+\Lbh q\bigr)e_4+\bigl(q^{-1}\omegah+\Lh q^{-1}\bigr)e_3\\
      [e_3,e_A]=\nablas_3e_A-q\chib_A^\sharp+\bigl(\ds_A\log\Omega-\ds_A\log q\bigr)e_3\\
      [e_4,e_A]=\nablas_4e_A-q^{-1}\chi_A^\sharp+\bigl(\ds_A\log\Omega+\ds_A\log q\bigr)e_4
    \end{gather*}
  \end{subequations}
  and thus

    \begin{gather*}
      [N,e_3]=-\Omega\eta^\sharp+\Omega\etab^\sharp+\frac{\Omega}{2}\Lbh q\, e_4-\frac{\Omega}{2}\bigl(q^{-1}\omegah+q\omegabh+\Lh q^{-1}\bigr)e_3\\
      [N,e_4]=\Omega\eta^\sharp-\Omega\etab^\sharp      -\frac{\Omega}{2}\bigl(q\omegabh+\Lbh q+q^{-1}\omegah\bigr)e_4+\frac{\Omega}{2}\Lh q^{-1}\,e_3\\
      [N,e_A]=\frac{\Omega}{2}\bigl(\nablas_3e_A+\nablas_4 e_A-q\chib_A^\sharp-q^{-1}\chi_A^\sharp\bigr)-\frac{\Omega}{2}\ds_A\log q\,(e_3-e_4)
    \end{gather*}

The stated formulas then follow with \eqref{eq:ijk:N}.

\end{proof}

\begin{proof}[Proof of Lemma~\ref{lemma:null:decomposition:LieMW}]
Let us first compute 
  $\mathcal{L}_N W(\cdot,e_4,\cdot,e_4)$, %$\mathcal{L}_N W(\cdot,e_3,\cdot,e_3)$, $\mathcal{L}_N W(\cdot,e_4,e_3,e_4)$, $\mathcal{L}_N W(\cdot,e_3,e_3,e_4)$, and $\mathcal{L}_N W(e_3,e_4,e_3,e_4)$,
  where the slots denoted by a $\cdot$ are to be evaluated on the frame $e_A:A=1,2$ on $S_{u,v}$.  Using \eqref{eq:Lie:Leibniz} we have
  \begin{multline*} (\mathcal{L}_NW)(e_A,e_4,e_B,e_4)=N(W(e_A,e_4,e_B,e_4))-W([N,e_A],e_4,e_B,e_4)-W(e_A,[N,e_4],e_B,e_4)\\-W(,e_A,e_4,[N,e_B],e_4)-W(,e_A,e_4,e_B,[N,e_4])
  \end{multline*}
In view of the commutation formula for $[N,e_B]$ from Lemma~\ref{lemma:commute:M:frame} we obtain
\begin{multline*}
  N(W(e_A,e_4,e_B,e_4))-W([N,e_A],e_4,e_B,e_4)-W(,e_A,e_4,[N,e_B],e_4)=\\
  =N((\alpha_q)_{AB})-W(\frac{\Omega}{2}\Pi[e_3+e_4,e_A],e_4,e_B,e_4)-W(e_A,e_4,\frac{\Omega}{2}\Pi[e_3+e_4,e_B],e_4)\\
  +\frac{\Omega}{2}\ds_A\log q \,W(e_3-e_4,e_4,e_B,e_4)+\frac{\Omega}{2}\ds_B\log q \,W(e_A,e_4,e_3-e_4,e_4)\\
  =(\Lies_{N}\alpha_q)_{AB}+\Omega \ds_A\log q (\beta_q)_B+\Omega \ds_B\log q(\beta_q)_A
\end{multline*}
where in the last equality we have used the definition \eqref{eq:Lies:alpha} and that $(\Omega/2)\Pi[e_3+e_4,e_A]=\Pi[N,e_A]$.
Moreover using the commutation formula for $[N,e_4]$ from Lemma~\ref{lemma:commute:M:frame} we obtain
\begin{multline*}
  -W(e_A,[N,e_4],e_B,e_4)-W(e_A,e_4,e_B,[N,e_4])=\\
  =-\bigl(\frac{1}{2}(\underline{m}-m)^\sharp-\Omega\ds\log q^{\sharp}\bigr)^C \bigl( W(e_A,e_C,e_B,e_4)+ W(e_A,e_4,e_B,e_C)\bigr)\\
  +\Omega\bigl(q\omegabh+\Lbh q+q^{-1}\omegah\bigr)W(e_A,e_4,e_B,e_4)
        -\frac{1}{4} n W(e_A,e_3,e_B,e_4)-\frac{1}{4} n W(e_A,e_4,e_B,e_3)\\
        =-2\Omega\zeta^{\sharp C}\bigl(-2\gs_{AB}(\beta_q)_C+\gs_{BC}\beta_A+\gs_{AC}(\beta_q)_B\bigr)
        +\Omega\bigl(q\omegabh+\Lbh q+q^{-1}\omegah\bigr)(\alpha_q)_{AB}+\frac{1}{2}n \rho \gs_{AB}
\end{multline*}
where we have used the formula for $\underline{m}-m$ from Lemma~\ref{lemma:deformation:N} as well as the formulas \eqref{eq:W:remaining:components} for the null components of a Weyl field.
In view of the formula \Ceq{1.182} for the $\otimesh$ product of two $S_{u,v}$ 1-forms, which we apply here in form
\begin{equation*}
  \Omega\ds_A\log q\,\beta_B+\Omega\ds_B\log q\,\beta_A=\Bigl(\ds\log q\otimesh\beta)_{AB}+\bigl(\Omega\ds\log q,\beta)\gs_{AB}
\end{equation*}
we then obtain the following formula for $(\mathcal{L}_NW)_{A4B4}$:
  \begin{multline*}
            (\mathcal{L}_{N}W)(\cdot,e_4,\cdot,e_4)=\Lies_{N}\alpha_q+\Omega\bigl(q\omegabh+q^{-1}\omegah+\Lbh q\bigr)\alpha_q+\frac{1}{2}n\rho\gs\\-\Omega\bigl(\eta-\etab-\ds\log q\bigr)\otimesh\beta_q+\Omega\bigl(\eta-\etab-ds\log q,\beta\bigr)\gs
          \end{multline*}

          Similarly we compute:
  \begin{multline*}
            (\mathcal{L}_{N}W)(\cdot,e_3,\cdot,e_3)=\Lies_{N}\alphab_q+\Omega\bigl(q^{-1}\omegah+q\omegabh+\Lh q^{-1}\bigr)\alphab_q+\frac{1}{2}\nb\rho\gs\\
            -\Omega\bigl(\eta-\etab+\ds\log q\bigr)\otimesh\betab+\Omega(\eta-\etab+\ds\log q,\betab)\gs
  \end{multline*}
  \begin{multline*}
    (\mathcal{L}_{N}W)(\cdot,e_4,e_3,e_4)=2 \Lies_{N}\beta_q+\Omega\bigl(3q\omegabh+3q^{-1}\omegah+\Lbh q+\Lh q^{-1}\Bigr)\beta_q\\
    -\Omega\bigl(2\ds\log q+\eta-\etab) \rho-3\Omega\:\ld(\eta-\etab+\ds\log q)\sigma-\frac{n}{2}\betab_q+\Omega\alpha^\sharp\cdot (\eta-\etab+\ds\log q)
  \end{multline*}
  \begin{multline*}
    (\mathcal{L}_{N}W)(\cdot,e_3,e_3,e_4)=2 \Lies_{N}\betab_q+\Omega\bigl(3q\omegabh+3q^{-1}\omegah+\Lbh q+\Lh q^{-1}\Bigr)\betab_q\\
    -\Omega\bigl(\ds\log q+\eta-\etab) \rho+3\Omega\:\ld(\eta-\etab+\ds\log q)\sigma-\frac{\nb}{2}\beta_q+\Omega\alphab^\sharp\cdot (\eta-\etab+\ds\log q)
  \end{multline*}
  \begin{multline*}
        (\mathcal{L}_{N}W)(e_3,e_4,e_3,e_4)=4\mathcal{L}_{N}\rho+4\Omega\bigl(2q\omegabh+2q^{-1}\omegah+\Lbh q+\Lh q^{-1}\bigr)\rho\\+4\Omega(\eta-\etab+\ds\log q)^\sharp\cdot\beta_q+4\Omega(\eta-\etab+\ds\log q)^\sharp\cdot\betab_q
  \end{multline*}
  \begin{multline*}
    \epsilons^{AB}\mathcal{L}_{N}W_{AB34}=4N \sigma+2\Omega\Bigl( 2q\omegabh +2q^{-1}\omegah+q\tr\chib+q^{-1}\tr\chi+\Lbh q+\Lh q^{-1}\Bigr)\sigma \\
    +2\Omega\bigl( \ds\log q +\eta-\etab,\ld\betab_q\bigr)-2\Omega\bigl( \ds\log q +\eta-\etab,\ld\beta_q\bigr)
  \end{multline*}
In the last formula we also used \Ceq{12.46}.

We adopt the notation  of \Ceq{12.48} for the terms appearing in the second line of the formula \eqref{eq:MLie:W:N} for the modified Lie derivative, namely we write
  \begin{equation*}
    (\MLie{N}W)_{\alpha\beta\gamma\delta}=(\mathcal{L}_{N}W)_{\alpha\beta\gamma\delta}-\frac{1}{2}{}^{(N)}[W]_{\alpha\beta\gamma\delta}-\frac{1}{8}\tr{}^{(N)}\pi\,W_{\alpha\beta\gamma\delta}
  \end{equation*}
where
\begin{equation*}
  {}^{(N)}[W]_{\alpha\beta\gamma\delta}=\pih{N}_{\alpha}^{\phantom{\alpha}\mu}W_{\mu\beta\gamma\delta}+\pih{N}_{\beta}^{\phantom{\beta}\mu}W_{\alpha\mu\gamma\delta}+\pih{N}_{\gamma}^{\phantom{\gamma}\mu}W_{\alpha\beta\mu\delta}+\pih{N}_{\delta}^{\phantom{\delta}\mu}W_{\alpha\beta\gamma\mu}\,.
\end{equation*}

  We then proceed similarly to \Ceq{12.49} and find that:
  \begin{align*}
    {}^{(N)}[W]_{A4B4}=&(\hat{i},\alpha_q)\gs_{AB}-2 m^C(\beta_q)_C\gs_{AB}+n\rho\gs_{AB}\\
    {}^{(N)}[W]_{A3B3}=&(\hat{i},\alphab_q)\gs_{AB}+2\mb^C(\betab_q)_C\gs_{AB}+\nb\rho\gs_{AB}\\
    {}^{(N)}[W]_{A434}=&-m_A\rho+3\ld m_A\sigma+2\hat{i}_A^{\sharp B}(\beta_q)_B-2j(\beta_q)_A-n(\betab_q)_A+(\alpha_q)_A^{\sharp B}\mb_B\\
    {}^{(N)}[W]_{A334}=&\mb_A\rho+3\ld \mb_A\sigma+2\hat{i}_A^{\sharp B}(\betab_q)_B-2j\betab_A-\nb(\beta_q)_A-(\alphab_q)_A^{\sharp B}\cdot m_B\\
    {}^{(N)}[W]_{3434}=&-8j\rho+4\mb^\sharp\cdot\beta_q-4m^\sharp\cdot\betab_q\\
    \epsilons^{AB}{}^{(N)}[W]_{AB34}=&0
  \end{align*}
The formulas given in the statement of the Lemma then follow.\qedhere

\end{proof}

\subsection{Weyl Currents}
\label{sec:weyl:currents}

  The components of the first order Weyl current
  \begin{equation}
    {}^{(X)}J^1(W)_{\beta\gamma\delta}=\frac{1}{2}{}^{(X)}\hat{\pi}^{\mu\nu}\nabla_\nu W_{\mu\beta\gamma\delta}
  \end{equation}
  have been calculated for a general commutation vectorfield $X$, and are presented \emph{in the special case $n=\nb=0$}, and \emph{relative to the frame} $(\Lbh,\Lh;e_A)$ in \CLemma{14.1}.
Here and in the following Lemma $i$, $\mb$, $m$, $\nb$, $n$, and $j$ refer to the null decomposition of the deformation tensor ${}^{(X)}\hat{\pi}$ of a general commutation vectorfield $X$ as given in \eqref{eq:pih:components} in the case $X=N$.
Note that $\tr i=j$ holds generally because $\tr \hat{\pi}=0$.

In the following Lemma we list the formulas for the components of ${}^{(X)}J^1(W)$ in the general case, $n\neq \nb \neq 0$,  decomposed relative to the frame $(e_3,e_4;e_A)$. These formulas can be inferred from the expressions in \CLemma{14.1} using the replacements \eqref{eq:lorentz:structure}.

\begin{lemma}\label{lemma:J:M:null}
The null components of the Weyl current $J^1(W)$ relative to the frame $(e_3,e_4;e_A)$ are given by

  \begin{subequations}
  \begin{equation}
    \begin{split}
       4\Xi^1_A= -\frac{1}{2}j&\Bigl\{ q^{-1}\Omega^{-1}( D\beta_q)_A+(\divs\alpha_q)_A -q^{-1}\chi_A^B(\beta_q)_B\\&-q^{-1}\tr\chi(\beta_q)_A-\bigl(q^{-1}\omegah+\frac{1}{\Omega}D q^{-1}\bigr)(\beta_q)_A+\alpha_A^B\bigl(2\zeta_B+2\ds_B\log q-\etab_B\bigr)\Bigr\}\\
+\frac{1}{2}m^B&\Bigl\{2\nablas_B(\beta_q)_A+q\Omega^{-1}(\Dbh\alpha_q)_{AB}-q\chib_B^C(\alpha_q)_{AC}-q\tr\chib(\alpha_q)_{AB}+2\bigl(q\omegabh+\frac{1}{\Omega}\Db q\bigr)(\alpha_q)_{AB}\\&+2\bigl((\zeta+\ds\log q-2\eta)\otimesh\beta_q\bigr)_{AB}-3q^{-1}\bigl(\chi_{AB}\rho+\ld\chi_{AB}\sigma\bigr)\Bigr\}\\
+\frac{1}{2}\mb^B&\Bigl\{q^{-1}\Omega^{-1}(\Dh\alpha_q)_{AB}-q^{-1}\tr\chi(\alpha_q)_{AB}-2\bigl(q^{-1}\omegah+\frac{1}{\Omega}D q^{-1}\bigr)(\alpha_q)_{AB}\Bigr\}\\
-\hat{i}^{BC}&\Bigl\{\nablas_C(\alpha_q)_{AB}-q^{-1}(\chi\otimesh\beta_q)_{CAB}+2\bigl(\zeta+\ds\log q\bigr)_C(\alpha_q)_{AB}\Bigr\}\\
+\frac{1}{2}n&\Bigl\{-q\Omega^{-1}(\Db\beta_q)_A+q\chib_A^B(\beta_q)_B-\bigl(q\omegabh+\frac{1}{\Omega}\Db q\bigr)(\beta_q)_A+3\eta_A\rho+3\ld\eta_A\sigma\Bigr\}
    \end{split}
  \end{equation}
Moreover
    \begin{gather}
      4\Xib^1=*+\frac{1}{2}\nb\Bigl\{q^{-1}\Omega^{-1}D\betab_q-q^{-1}\chi^\sharp\cdot\betab_q+\bigl(q^{-1}\omegah+\frac{1}{\Omega}D q^{-1}\bigr)\betab_q+3\etab\rho-3\ld\etab\sigma\Bigr\}\\
      4\Lambdab^1=*+\frac{1}{2}\nb\Bigl\{q^{-1}\Omega^{-1}D\rho-2\etab^\sharp\cdot\beta_q\Bigr\}\\
      4\Lambda^1=*+\frac{1}{2}n\Bigl\{q\Omega^{-1}\Db\rho+2\eta^\sharp\cdot\betab_q\Bigr\}\displaybreak[0]\\
      4\Kb^1=*+\frac{1}{2}\nb\Bigl\{-q^{-1}\Omega^{-1}D\sigma-2\etab\wedge\beta_q\Bigr\}\\
      4K^1=*+\frac{1}{2}n\Bigl\{q\Omega^{-1}\Db\sigma+2\eta\wedge\betab_q\Bigr\}\displaybreak[0]\\
      4\Ib^1=*+\frac{1}{2}n\Bigl\{-q\Omega^{-1}\Db\betab_q+q\chib^\sharp\cdot\betab_q+\bigl(q\omegabh+\frac{1}{\Omega}\Db q\bigr)\betab_q-\eta^\sharp\cdot\alphab_q\Bigr\}\\
      4I^1=*+\frac{1}{2}\nb\Bigl\{q^{-1}\Omega^{-1}D\beta_q-q^{-1}\chi^\sharp\cdot\beta_q-\bigl(q^{-1}\omegah+\frac{1}{\Omega}D q^{-1}\bigr)\beta_q-\etab^\sharp\cdot\alpha_q\Bigr\}\\
      4\Thetab^1=*+\frac{1}{2}n\Bigl\{q\Omega^{-1}\Dbh\alphab_q-q\tr\chib\alphab_q-2\bigl( q\omegabh+\frac{1}{\Omega}\Db q\bigr)\alphab_q\Bigr\}\\
      4\Theta^1=*+\frac{1}{2}\nb\Bigl\{q^{-1}\Omega^{-1}\Dh\alpha_q-q^{-1}\tr\chi\alpha_q-2\bigl(q^{-1}\omegah+\frac{1}{\Omega}D q^{-1}\bigr)\alpha_q\Bigr\}
    \end{gather}
where $*$ denotes the corresponding terms listed in \CLemma{14.1} (pages 446-451) with the formal replacements of \eqref{eq:replacements}.
  \end{subequations}

\end{lemma}

\begin{proof}
  Consider for example
  \begin{equation*}
    \Lambdab^1=\frac{1}{4} J_{343}^1(W)=\frac{1}{2}\frac{1}{4}\hat{\pi}^{\mu\nu}\nabla_\nu W_{\mu 343}
  \end{equation*}
  Let us write out the terms which have $j$ as a common factor. This can arise either from $(\mu\nu)=(43)$, or $(\mu\nu)=(AB)$, because $\tr i=j$. Indeed
  \begin{equation*}
    \frac{1}{2}\hat{\pi}^{43}\nabla_3 W_{4 343}=\frac{1}{2}\frac{1}{4}j\Bigl[  e_3\bigl( 4\rho \bigr)-2W(\nabla_3e_4,e_3,e_4,e_3)-2W(e_4,\nabla_3e_3,e_4,e_3) \Bigr]
  \end{equation*}
  Now inserting the frame relations \eqref{eq:frame:a} (with $a=q$) we see that these differ only in the coefficients from those used in \Ceq{1.175}, and can be obtained from the latter using the replacements \eqref{eq:lorentz:structure} with ($a=q$). For definiteness,
  \begin{gather*}
    -2W(\nabla_3e_4,e_3,e_4,e_3)=8\eta^\sharp\cdot\betab_q+8\bigl(q\omegabh-q^2\Lbh q^{-1}\bigr)\rho\\
    -2W(e_4,\nabla_3e_3,e_4,e_3)=-8\bigl(q\omegabh+\Lbh q\bigr)\rho
  \end{gather*}
and thus
  \begin{equation*}
    \frac{1}{2}\hat{\pi}^{43}\nabla_3 W_{4 343}=\frac{1}{2}j\Bigl[  q \Omega^{-1} \Db \rho +2\eta^\sharp\cdot\betab_q\Bigr]
  \end{equation*}
  Furthermore we have the following contribution:
  \begin{equation*}
    \frac{1}{2}\hat{\pi}^{AB}\nabla_B W_{A 343}=\frac{1}{2}\hat{i}^{AB}\nabla_BW_{A343}+\frac{1}{4} j g^{AB}\nabla_B W_{A343}
  \end{equation*}
while
\begin{multline*}
  g^{AB}\nabla_B W_{A343}=-g^{AB}e_B\bigl( 2\betab_A\bigr)-g^{AB}W(\nabla_B e_A,e_3,e_4,e_3)-g^{AB}W(e_A,\nabla_Be_3,e_4,e_3)\\-g^{AB}W(e_A,e_3,\nabla_Be_4,e_3)-g^{AB}W(e_A,e_3,e_4,\nabla_B e_3)
\end{multline*}
and again inserting the frame relations \eqref{eq:frame:a} (with $a=q$) gives the same result as inserting \Ceq{1.175} followed by the replacements \eqref{eq:lorentz:structure}.
For definiteness,
\begin{gather*}
  -g^{AB}W(\nabla_B e_A,e_3,e_4,e_3)=+2\gs^{AB}\nablas_B e_A^C(\betab_q)_C-2 q\tr\chib \rho\\
  -g^{AB}W(e_A,\nabla_Be_3,e_4,e_3)=2q\gs^{AB}\chib_B^{\sharp C}\sigma\epsilons_{AC}+2\gs^{AB}\bigl(\zeta_B+\ds_B\log q\bigr)(\betab_q)_A\\
  -g^{AB}W(e_A,e_3,\nabla_Be_4,e_3)=-q^{-1}\gs^{AB}\chi_B^{\sharp C}(\alphab_q)_{AC}-2\gs^{AB}\bigl(\zeta_B+\ds_B\log q\bigr)(\betab_q)_A\\
  -g^{AB}W(e_A,e_3,e_4,\nabla_B e_3)=q\gs^{AB}\chib_B^{\sharp C}\bigl(-\rho\gs_{AC}+\sigma\epsilons_{AC}\bigr)+2\gs^{AB}\bigl(\zeta_B+\ds_B\log q\bigr)(\betab_q)_A
\end{gather*}
and thus
\begin{equation*}
  g^{AB}\nabla_B W_{A343}=-2\gs^{AB}\nablas_B(\betab_q)_A-3 q\tr\chib \rho-q^{-1}(\chih,\alphab_q)+3q\tr\ld\chib\sigma+2\bigl(\zeta+\ds\log q,\betab_q\bigr)
\end{equation*}
In conclusion, we have calculated that the terms in $\Lambdab$ \emph{which come with a factor $j$} are given by
\begin{equation*}
 4 \Lambdab^1\doteq \frac{1}{2}j\Bigl[  q \Omega^{-1} \Db \rho -\frac{3}{2} q\tr\chib \rho+\bigl(\zeta+2\eta+\ds\log q,\betab_q\bigr)
-\divs\betab_q-\frac{1}{2}q^{-1}(\chih,\alphab_q)\Bigr]
\end{equation*}
This coincides precisely with the formula given in \CLemma{14.1} (pages 448-449) modulo the replacements indicated in the statement of this Lemma.

Similarly for all other components.
\end{proof}

\subsection{Positive Current}
\label{sec:positive}

Finally in this section we will analyse in all detail the terms appearing in \eqref{eq:divQ:M:null} for the commutation vectorfield $N$ defined in \eqref{eq:N}.  As explained in Section~\ref{sec:intro:foliation}, and as it is clear from \eqref{eq:J:123}, a quantitative bound on $\divergence Q[\MLie{N}W]$ can only be obtained under additional assumptions on $\nabla \pid{N}$.

In addition to (\emph{\textbf{BA:I}.i-vi}) we assume
\begin{gather}\label{eq:BA:I:vii}
       q\tr\chib\leq C_0\qquad q^{-1}\tr\chi\leq C_0 \tag{\emph{\textbf{BA:I}.vii}}\\
           \Omega  \lvert q\tr\chib -q^{-1}\tr\chi \rvert\leq C_0  \bigl( q\tr\chib+q^{-1}\tr\chi\bigr) \tag{\emph{\textbf{BA:I}.viii}}
\end{gather}
and to deal with several ``borderline'' terms we strengthen (\emph{\textbf{BA:I}.iv}) and (\emph{\textbf{BA:I}.vi}) to
\begin{gather} 
      \Omega^2 \lvert  \chibh \rvert \leq C_0 \tr\chib \qquad \Omega^2 \lvert \chih \rvert \leq C_0 \tr\chi \tag{\emph{\textbf{BA:I}.iv}${}^\prime$} \label{eq:BA:I:iv:e}\\
  \Omega^2 \lvert \eta \rvert + \Omega^2 \lvert \etab \rvert + \Omega^2 \lvert \ds\log q\rvert\leq C_0   (q\tr\chib+q^{-1}\tr\chi)\tag{\emph{\textbf{BA:I}.vi}${}^\prime$} \label{eq:BA:I:vi:e}
\end{gather}
Furthermore we assume
\begin{gather}
  \lvert D (\Omega\chibh) \rvert \leq C_0 \tr\chi\tr\chib \qquad   \lvert \Db (\Omega\chih) \rvert \leq C_0 \tr\chi\tr\chib \tag{\emph{\textbf{BA:II}.i}}\label{eq:BA:II:i}\\
  \lvert D (\Omega\chih) \rvert \leq C_0 \tr\chi\tr\chi \qquad \lvert \Db (\Omega\chibh) \rvert \leq C_0 \tr\chib\tr\chib \notag\\
  \Omega \lvert \nablas (\Omega\chibh) \rvert \leq C_0\tr\chib\qquad \Omega \lvert \nablas (\Omega\chih) \rvert \leq C_0 \tr\chi \tag{\emph{\textbf{BA:II}.ii}}
\end{gather}
 and (\emph{\textbf{BA:II}.iii-viii}) below.

The main conclusion is that under these assumptions the divergence \eqref{eq:divQ:M:null} is \emph{positive} up to a sufficiently fast decaying error:

\begin{proposition}
  \label{prop:div:N}

Assume (\textbf{BA:I}) --- including \eqref{eq:BA:I:iv:e} and \eqref{eq:BA:I:vi:e} --- and (\textbf{BA:II}) hold for some $C_0>0$.
Then there is a constant $C>0$,  such  that for all solutions $W$ to \eqref{eq:W:Bianchi},
\begin{multline}\label{eq:div:N}
 \phi \bigl(\divergence Q(\MLie{N}W)\bigr)(M_q,M_q,M_q) \geq \\
  \geq-\frac{C}{r}\frac{C_0}{\Omega}\Bigl[Q[\MLie{N}W](n,M_q,M_q,M_q)+Q[W](n,M_q,M_q,M_q)\Bigr] -\frac{C}{r}\frac{C_0}{\Omega}\Omega^2\Ps^q
\end{multline}
where
\begin{equation}\label{eq:Ps}
  \Ps^q:=\frac{1}{(2\Omega)^3}\Bigl[\lvert \nablas \alphab_q \rvert^2 + \lvert \nablas \betab_q \rvert^2 + \lvert \nablas\rho \rvert^2 + \lvert \nablas \sigma \rvert^2 + \lvert \nablas \beta_q\rvert^2 +\lvert \nablas\alpha_q\rvert^2\Bigr]
\end{equation}
\end{proposition}

We will give the proof of this Proposition in Sections~\ref{sec:J:1}-\ref{sec:J:cancel}.

\bigskip

\subsubsection{$J^1$}
\label{sec:J:1}

\begin{lemma}\label{lemma:div:N:one}
Assume that (\textbf{BA:I}.i-iii,v,vii,viii) hold, i.e.~for some $C_0>0$, 
\begin{subequations}\label{eqs:prop:div:N:assumptions:X}
\begin{gather*}
      \tr\chi >0\qquad \tr\chib >0 \\
       \lvert 2\omega-\Omega \tr\chi\rvert\leq C_0 \tr\chi
      \qquad  \lvert 2\omegab-\Omega\tr\chib\rvert\leq C_0 \tr\chib \\
      \lvert\Omega\tr\chi-\overline{\Omega\tr\chi}\rvert\leq C_0\Omega^{-1}\,\overline{\Omega\tr\chi}\qquad       \lvert\Omega\tr\chib-\overline{\Omega\tr\chib}\rvert\leq C_0\Omega^{-1}\,\overline{\Omega\tr\chib} \displaybreak[0]\\
     \lvert D \log q\rvert \leq C_0 \tr\chi \qquad \lvert \Db \log q \rvert \leq C_0 \tr\chib \\%\label{eq:assumption:D:logq}\\
     q\tr\chib\leq C_0\qquad q^{-1}\tr\chi\leq C_0\\
     \Omega \lvert q\tr\chib -q^{-1}\tr\chi \rvert\leq C_0 \bigl( q\tr\chib+q^{-1}\tr\chi\bigr)
\end{gather*}
\end{subequations}
and moreover that (\textbf{BA:I}.iv,vi) hold,\footnote{Note specifically that this Lemma does \emph{not} require the stricter version (\emph{\textbf{BA:I}.iv${}^\prime$,vi${}^\prime$)} to hold.}   i.e. for some $C_0>0$,
\begin{subequations} \label{eqs:div:N:assumptions:Y:relaxed}
\begin{gather*}
    \Omega \lvert  \chibh \rvert \leq C_0 \tr\chib \qquad \Omega \lvert \chih \rvert \leq C_0 \tr\chi\\
  \Omega \lvert \eta \rvert + \Omega \lvert \etab \rvert + \Omega \lvert \ds\log q\rvert\leq C_0   \bigl( q\tr\chib+q^{-1}\tr\chi\bigr)\,.
\end{gather*}
\end{subequations}

Then,
\begin{multline}\label{eq:div:N:one}
 \phi \Bigl[ \bigl(\divergence Q(\MLie{N}W)\bigr)(M_q,M_q,M_q) \Bigr]^1\geq  \frac{\phi}{(2\Omega)^3}\bigl[X^1+Y^1\bigr]\\
  -\frac{C}{r}\frac{1}{\Omega}Q[\MLie{N}W](n,M_q,M_q,M_q)  -\frac{C}{r}\frac{1}{\Omega}Q[W](n,M_q,M_q,M_q) -\frac{C}{r}\frac{1}{\Omega}\Bigl[\Omega^2\Ps^q\Bigr]
\end{multline}
where
\begin{subequations}\label{eq:X:Y:1}
\begin{gather}
  X^1:=X_\Lambda^1+X_K^1+X_{\Xib}^1+X_{\Xi}^1+X_{\Thetab}^1+X_{\Theta}^1\\
  Y^1:=X_\Lambda^1+Y_K^1+Y_{\Xib}^1+Y_{\Xi}^1+Y_{\Thetab}^1+Y_{\Theta}^1
\end{gather}
\end{subequations}
and $X_\Lambda^1$, $X_K^1$, $X_{\Xib}^1$, $X_{\Xi}^1$, $X_{\Thetab}^1$, $X_{\Theta}^1$, are given by \eqref{eq:X:Lambda:one}, \eqref{eq:X:K:one}, \eqref{eq:X:Xib:one}, \eqref{eq:X:Xi:one}, \eqref{eq:X:Thetab:one}, \eqref{eq:X:Theta:one}, respectively, while $Y_\Lambda^1$, $Y_K^1$, $Y_{\Xib}^1$, $Y_{\Xi}^1$, $Y_{\Thetab}^1$, $Y_{\Theta}^1$, are given by \eqref{eq:Y:Lambda:one}, \eqref{eq:Y:K:one}, \eqref{eq:Y:Xib:one}, \eqref{eq:Y:Xi:one}, \eqref{eq:Y:Thetab:one}, \eqref{eq:Y:Theta:one}, respectively.
\end{lemma}

\begin{proof}
  We use Lemma~\ref{lemma:J:M:null}, and \CLemma{14.1} as discussed above, to write out the null components of the current ${}^{(N)}J^1(W)$.

We begin with
\begin{equation*}
  \Lambdabt_q^1 := \frac{1}{4}{}^{(N)}J^1[W](e_3,e_4,e_3)\qquad \Lambdat_q^1 := \frac{1}{4}{}^{(N)}J^1[W](e_4,e_3,e_4)
\end{equation*}

In the first place we are only interested in terms whose factors are either $j$, or $n$, $\nb$, and ignore all other terms with factors $m$, $\mb$, and $\hat{i}$; thus in the following formula we set

\begin{subequations}
\begin{equation}
  m=0\,,\qquad\mb=0\,,\qquad \hat{i}=0\,.
\end{equation}

By Lemma~\ref{lemma:J:M:null},

  \begin{multline}\label{eq:LambdatLambdabt}
    8\Lambdat^1_q+8\Lambdabt^1_q=j\Bigl\{q^{-1}\Omega^{-1}D\rho+\divs\beta_q-\frac{3}{2}q^{-1}\tr\chi\rho+(\zeta+\ds\log q-2\etab,\beta_q)-\frac{1}{2}q(\chibh,\alpha_q)\\+q\Omega^{-1}\Db\rho-\divs\betab_q-\frac{3}{2}q\tr\chib\rho+(\zeta+\ds\log q+2\eta,\betab_q)-\frac{1}{2}q^{-1}(\chih,\alphab_q)\Bigr\}\\
    +\nb\Bigl\{q^{-1}\Omega^{-1}D\rho-2\etab^\sharp\cdot\beta_q\Bigr\}
     +n\Bigl\{q\Omega^{-1}\Db\rho+2\eta^\sharp\cdot\betab_q\Bigr\}\displaybreak[0]\\
    =j\Bigl\{\frac{2}{\Omega}\bigl(q^{-1}D\rho+q\Db\rho\bigr)-4(\etab,\beta_q)+4(\eta,\betab_q)\Bigr\}  \\
    +\nb\Bigl\{q^{-1}\Omega^{-1}D\rho-2\etab^\sharp\cdot\beta_q\Bigr\}
     +n\Bigl\{q\Omega^{-1}\Db\rho+2\eta^\sharp\cdot\betab_q\Bigr\}
   \end{multline}
where we have used the Bianchi equations in the form of Prop.~\ref{prop:bianchi:q}.
We symmetrize the terms with coefficients in $\nb$, and $n$, and use commutation Lemma~\ref{lemma:null:decomposition:LieMW} to obtain:
\begin{multline}   \label{eq:LambdatLambdabt:sym}
8\Lambdat^1_q+8\Lambdabt^1_q=\Bigl(j+\frac{\nb+n}{4}\Bigr)\Bigl\{\frac{2}{\Omega}\bigl(q^{-1}D\rho+q\Db\rho\bigr)-4(\etab,\beta_q)+4(\eta,\betab_q)\Bigr\} \\
-\frac{\nb-n}{2}\Bigl\{q^{-1}\Omega^{-1}D\rho-2\etab^\sharp\cdot\beta_q-q\Omega^{-1}\Db\rho-2\eta^\sharp\cdot\betab_q\Bigr\}\displaybreak[0]\\
=\Bigl(j+\frac{n+\nb}{4}\Bigr)\Bigl\{\frac{4}{\Omega}\rhot-\mathbf{\frac{3}{2}}\bigl(2q\omegabh+2q^{-1}\omegah+q\tr\chib+q^{-1}\tr\chi+\Lbh q+\Lh q^{-1}\bigr)\rho\\-2(\eta+\etab+\ds\log q,\beta_q)+2(\eta+\etab-\ds\log q,\betab_q)\Bigr\}\\
-\frac{\nb-n}{2}\Bigl\{-\frac{3}{2}\bigl(q^{-1}\tr\chi-q\tr\chib\bigr)\rho+\divs\beta_q+\divs\betab_q\\+(\zeta+\ds\log q,\beta)-(\zeta+\ds\log q,\betab)-\frac{1}{2}q(\chibh,\alpha_q)+\frac{1}{2}q^{-1}(\chih,\alphab_q)\Bigr\}
  \end{multline}
In this symmetrization,  we have on one hand gained that the sum  yields an additional \emph{positive} term, while on the other hand by the Bianchi equations of Prop.~\ref{prop:bianchi:q} the difference leaves us with terms only involving \emph{angular} derivatives:
  \begin{multline}
    q^{-1}\Omega^{-1}D\rho-2\etab^\sharp\cdot\beta_q-q\Omega^{-1}\Db\rho-2\eta^\sharp\cdot\betab_q=\\=-\frac{3}{2}\bigl(q^{-1}\tr\chi-q\tr\chib\bigr)\rho+\divs\beta_q+\divs\betab_q\\+(\zeta+\ds\log q,\beta)-(\zeta+\ds\log q,\betab)-\frac{1}{2}q(\chibh,\alpha_q)+\frac{1}{2}q^{-1}(\chih,\alphab_q)
  \end{multline}
\end{subequations}
\begin{remark}\label{remark:prefactor:cancel}
  With $M_q$ as commutator the prefactor that appears in bold in \eqref{eq:LambdatLambdabt:sym} would be different, and this term would fail to cancel with similar terms appearing in $\Lambdat^2+\Lambdabt^2$ discussed below. The cancellations that \emph{do} occur with $N$ --- defined by \eqref{eq:N} and expressed as in \eqref{eq:N:q} ---  as a commutator are discussed in Section~\ref{sec:J:cancel}.
\end{remark}

In the second instance we consider all terms with factors $m$, $\mb$, and $\hat{i}$, and set 
\begin{subequations}
\begin{equation}
  j=0\,,\qquad n=0\,,\qquad \nb =0\,,
\end{equation}
in the following formula:

  \begin{multline}
    8\Lambdat^1=-\Bigl(m,\Omega^{-1}q\Db\beta_q-q\chib^\sharp\cdot\beta+\bigl(q\omegabh+\Omega^{-1}\Db q\bigr)\beta_q-3\eta\rho-3\ld\eta\sigma\Bigr)\\
    -2\Bigl(m,\ds\rho+q^{-1}\chi^\sharp\cdot\betab_q-q\chib^\sharp\cdot\beta_q\Bigr)\\
    -2\Bigl(\mb,\Omega^{-1}q^{-1}D\beta_q-q^{-1}\chi^\sharp\cdot\beta_q-\bigl(q^{-1}\omegah+\Omega^{-1}D q^{-1}\bigr)\beta_q-\etab^\sharp\cdot\alpha_q\Bigr)\\
    +\Bigl(\hat{i},\nablas\otimesh\beta_q+\bigl(\zeta+\ds\log q\bigr)\otimesh\beta_q-3\bigl(q^{-1}\chih\rho+q^{-1}\ld\chih\sigma\bigr)-\frac{1}{2}q\tr\chib\alpha_q\Bigr)\displaybreak[0]\\
    =-\Bigl(m,-2q\tr\chib \beta_q+q^{-1}\tr\chi\betab_q+3\ds\rho-\ld\ds\sigma+4q^{-1}\chih^\sharp\cdot\betab_q-2q\chibh^\sharp\cdot\beta_q\Bigr)\\
    -2\Bigl(\mb,-2q^{-1}\tr\chi\beta_q+\divs\alpha_q+2\bigl(\zeta+\ds\log q\bigr)^\sharp\cdot\alpha_q\Bigr)\\
    +\Bigl(\hat{i},\nablas\otimesh\beta_q+\bigl(\zeta+\ds\log q\bigr)\otimesh\beta_q-3\bigl(q^{-1}\chih\rho+q^{-1}\ld\chih\sigma\bigr)-\frac{1}{2}q\tr\chib\alpha_q\Bigr)
  \end{multline}
  \begin{multline}
    8\Lambdabt^1=-\Bigl(\mb,-\Omega^{-1}q^{-1}D\betab_q+q^{-1}\chi^\sharp\cdot\betab_q-\bigl(q^{-1}\omegah+\Omega^{-1}D q^{-1}\bigr)\betab_q-3\etab\rho+3\ld\etab\sigma\Bigr)\\
    -2\Bigl(\mb,\ds\rho-q\chib^\sharp\cdot\beta_q+q^{-1}\chi^\sharp\cdot\betab_q\Bigr)\\
    -2\Bigl(m,-\Omega^{-1}q\Db\betab_q+q\chib^\sharp\cdot\beta_q+\bigl(q\omegab+\Omega^{-1}\Db q\bigr)\betab_q-\eta^\sharp\cdot\alphab_q\Bigr)\\
    +\Bigl(\hat{i},-\nablas\otimesh\betab+\bigl(\zeta+\ds\log q\bigr)\betab_q-3\bigl(q\chibh\rho-q\ld\chibh\sigma\bigr)-\frac{1}{2}q^{-1}\tr\chi\alphab_q\Bigr)\displaybreak[0]\\
    =-\Bigl(\mb,2q^{-1}\tr\chi\betab_q-q\tr\chib\beta_q+3\ds\rho-\ld\ds\sigma-4\chibh^\sharp\cdot\beta_q+2q^{-1}\chi^\sharp\cdot\betab_q\Bigr)\\
    -2\Bigl(m,2q\tr\chib\betab_q+\divs\alphab_q-2\bigl(\zeta+\ds\log q\bigr)^\sharp\cdot \alphab_q\Bigr)\\
    +\Bigl(\hat{i},-\nablas\otimesh\betab+\bigl(\zeta+\ds\log q\bigr)\betab_q-3\bigl(q\chibh\rho-q\ld\chibh\sigma\bigr)-\frac{1}{2}q^{-1}\tr\chi\alphab_q\Bigr)
  \end{multline}
where we used the Bianchi equations in the second step to eliminate $\Db\beta_q$, $D\beta_q$, $D\betab_q$, and $\Db\betab_q$ in terms of derivatives tangential to the spheres.

Let us employ here already that by Lemma~\ref{lemma:deformation:N},
\begin{equation}
  \mb=-m
\end{equation}
to conclude
  \begin{multline}
    8\Lambdat^1_q+8\Lambdabt^1_q
    =\Bigl(q\tr\chib-4q^{-1}\tr\chi\Bigr)\bigl(m, \beta_q\bigr)+\Bigl(-q^{-1}\tr\chi+4q\tr\chib\Bigr)\bigl(\mb,\betab_q\bigr)\\-\Bigl(m,4q^{-1}\chih^\sharp\cdot\betab_q-2q\chibh^\sharp\cdot\beta_q+4\chibh^\sharp\cdot\beta_q-2q^{-1}\chi^\sharp\cdot\betab_q\Bigr)\\
    -2\Bigl(\mb,\divs\alpha_q+2\bigl(\zeta+\ds\log q\bigr)^\sharp\cdot\alpha_q\Bigr)-2\Bigl(m,\divs\alphab_q-2\bigl(\zeta+\ds\log q\bigr)^\sharp\cdot \alphab_q\Bigr)\\
    +\Bigl(\hat{i},\nablas\otimesh\beta_q+\bigl(\zeta+\ds\log q\bigr)\otimesh\beta_q-3\bigl(q^{-1}\chih\rho+q^{-1}\ld\chih\sigma\bigr)-\frac{1}{2}q\tr\chib\alpha_q\Bigr)\\
-\Bigl(\hat{i},\nablas\otimesh\betab-\bigl(\zeta+\ds\log q\bigr)\betab_q+3\bigl(q\chibh\rho-q\ld\chibh\sigma\bigr)+\frac{1}{2}q^{-1}\tr\chi\alphab_q\Bigr)
  \end{multline}

\end{subequations}

In particular we have the following contributions to the divergence \eqref{eq:divQ:M:null}:

\begin{subequations}

\begin{equation}\label{eq:div:Lambda:one}
    \boxed{3\cdot8\rhot\bigl(\Lambdabt^1+\Lambdat^1\bigr) = 6\frac{2}{\Omega}\Bigl(  j + \frac{1}{4}(\nb+n)\Bigr) \rhot^2+X_\Lambda^1+Y_\Lambda^1+Q_\Lambda^1+R_\Lambda^1}
  \end{equation}

Note here that
  \begin{gather}
    \lvert j \rvert \leq \frac{1}{2} q \lvert \Omega \tr\chib-2\omegab\rvert+\frac{1}{2} q^{-1}\lvert \Omega \tr\chi-2\omega\rvert+\frac{1}{2}\lvert \Db q\rvert +\frac{1}{2}\lvert D q^{-1}\rvert\leq C\bigl(q\tr\chib+q^{-1}\tr\chi\bigr)\\
   \lvert \frac{\nb+n}{4} \rvert =  \frac{1}{2}\lvert \Db q \rvert + \frac{1}{2} \lvert D q^{-1}\rvert\leq \frac{C}{2}\bigl(q\tr\chib+q^{-1}\tr\chi\bigr)
  \end{gather}

  % \begin{remark}
  %   Note that for $M_q$ as a commutator we would have by Lemma~\ref{lemma:deformation:M:a},
  %   \begin{equation*}
  %     j+\frac{1}{4}(\nb+n)=\frac{1}{2\Omega}\Bigl(q\tr\chib+q^{-1}\tr\chi+2q\omegabh+2q^{-1}\omegah\Bigr)\geq 0
  %   \end{equation*}
  %   which explains the positive terms in Lemma~\ref{lemma:div:leading}.
  % \end{remark}

Further to the notation used in \eqref{eq:div:Lambda:one}, we have a quadratic error term
  \begin{multline}
    Q_\Lambda^1:=3\rhot \Bigl(j+\frac{n+\nb}{4}\Bigr)\Bigl\{-2(\eta+\etab+\ds\log q,\beta_q)+2(\eta+\etab-\ds\log q,\betab_q)\Bigr\}\\
    -3\rhot\frac{\nb-n}{4}\Bigl\{-\frac{3}{2}\bigl(q^{-1}\tr\chi-q\tr\chib\bigr)\rho+\divs\beta_q+\divs\betab_q\\+(\zeta+\ds\log q,\beta)-(\zeta+\ds\log q,\betab)-\frac{1}{2}q(\chibh,\alpha_q)+\frac{1}{2}q^{-1}(\chih,\alphab_q)\Bigr\}
  \end{multline}
and a term which we do not estimate but rather keep in its precise form:
\begin{equation}\label{eq:X:Lambda:one}
  X_\Lambda^1:=-\frac{9}{2}\Bigl(j+\frac{n+\nb}{4}\Bigr)\bigl(2q\omegabh+2q^{-1}\omegah+q\tr\chib+q^{-1}\tr\chi+\Lbh q+\Lh q^{-1}\bigr)\rho\rhot
\end{equation}
Finally we collect the following ``borderline error terms'' in
\begin{equation}\label{eq:Y:Lambda:one}
  Y_\Lambda^1:=   -\frac{3}{2}\rhot q\tr\chib  \bigl(\hat{i},\alpha_q\bigr)- \frac{3}{2}\rhot q^{-1}\tr\chi \bigl(\hat{i},\alphab_q\bigr)
    -9 q^{-1}\tr\chi \rhot \bigl(m, \beta_q\bigr)+9q\tr\chib \rhot \bigl(\mb,\betab_q\bigr)
\end{equation}
and call
\begin{equation}
  R_{\Lambda}^1:= \Bigl(\text{quadratic error term from case } j=\nb=n=0\Bigr) -  Y_\Lambda^1 
\end{equation}
In fact we can already estimate the quadratic error terms by:
\begin{multline}\label{eq:Q:Lambda:one}
  \lvert Q_\Lambda^1 \rvert\leq 12C \bigl(q\tr\chib+q^{-1}\tr\chi\bigr) \rhot \bigl( \lvert  \ds \log q \rvert+  \lvert \eta \rvert+\lvert \etab\rvert\bigr)\bigl( \lvert \beta_q \rvert + \lvert \betab_q \rvert\bigr)\\
+ \frac{3C}{2}\rhot\bigl(q\tr\chib+q^{-1}\tr\chi\bigr)\Bigl\{\frac{3}{2}\bigl|q^{-1}\tr\chi-q\tr\chib\bigr|\rho+ |\divs\beta_q|+|\divs\betab_q|\\+|\zeta+\ds\log q||\beta_q|+|\zeta+\ds\log q||\betab_q|+\frac{1}{2}|q\chibh||\alpha_q|+\frac{1}{2}|q^{-1}\chih||\alphab_q|\Bigr\}
\end{multline}

\begin{multline}
  \lvert R_\Lambda^1 \rvert \leq  3 |q\tr\chib-q^{-1}\tr\chi||m||\rhot|| \beta_q|+3|-q^{-1}\tr\chi+q\tr\chib||\mb||\rhot||\betab_q|\\+\frac{3|\rhot|}{\Omega}|m|\Bigl|4q^{-1}\Omega\chih^\sharp\cdot\betab_q-2q\Omega\chibh^\sharp\cdot\beta_q+4q\Omega \chibh^\sharp\cdot\beta_q-2q^{-1}\Omega\chi^\sharp\cdot\betab_q\Bigr|\\
    +6|\rhot||\mb|\Bigl|\divs\alpha_q+2\bigl(\zeta+\ds\log q\bigr)^\sharp\cdot\alpha_q\Bigr|+6|\rhot||m|\Bigl|\divs\alphab_q-2\bigl(\zeta+\ds\log q\bigr)^\sharp\cdot \alphab_q\Bigr|\\
    +3|\rhot||\hat{i}|\Bigl|\nablas\otimesh\beta_q+\bigl(\zeta+\ds\log q\bigr)\otimesh\beta_q-3\bigl(q^{-1}\chih\rho+q^{-1}\ld\chih\sigma\bigr)\Bigr|\\
+3|\rhot||\hat{i}|\Bigl|\nablas\otimesh\betab-\bigl(\zeta+\ds\log q\bigr)\betab_q+3\bigl(q\chibh\rho-q\ld\chibh\sigma\bigr)\Bigr|
\end{multline}

Note here also that
\begin{gather}
      |\hat{i}|=\Omega\bigl(q|\chibh|+q^{-1}|\chih|\bigr)\leq C_0\bigl(q\tr\chib+q^{-1}\tr\chi\bigr)\\
    |\underline{m}|=|m|=\Omega\bigl(2|\zeta|+|\ds\log q|\bigr) \leq C_0\bigl(q\tr\chib+q^{-1}\tr\chi\bigr)
\end{gather}

\end{subequations}

Let us now turn to the remaining components:
\begin{gather*}
  \Kbt^1_q := \frac{1}{4}\epsilons^{AB}{}^{(N)}J^1[W](e_3,e_A,e_B)\qquad \Kt^1_q := \frac{1}{4}\epsilon^{AB}{}^{(N)}J^1[W](e_4,e_A,e_B)\\
  (\Ibt_q^1)_A := \frac{1}{2}{}^{(N)}J^1[W](e_4,e_3,e_A)\qquad   (\It_q^1)_A := \frac{1}{2}{}^{(N)}J^1[W](e_3,e_4,e_A)\\
  (\Xibt^1_q)_A := \frac{1}{2}{}^{(N)}J^1[W](e_3,e_3,e_A)\qquad   (\Xit^1_q)_A := \frac{1}{2}{}^{(N)}J^1[W](e_4,e_4,e_A)
\end{gather*}
and
\begin{gather*}
  (\Thetabt_q^1)_{AB} := \frac{1}{2}\bigl( (J\!\!\!\!/_3)_{AB}+(J\!\!\!\!/_3)_{BA}-\gs_{AB}\tr (J\!\!\!\!/_3)_{AB}\bigr)\\
  (\Thetat_q^1)_{AB} := \frac{1}{2}\bigl( (J\!\!\!\!/_4)_{AB}+(J\!\!\!\!/_4)_{BA}-\gs_{AB}\tr (J\!\!\!\!/_4)_{AB}\bigr)
\end{gather*}
where
\begin{equation*}
  (J\!\!\!\!/_3)_{AB} := {}^{(N)}J^1[W](e_A,e_3,e_B)\qquad   (J\!\!\!\!/_4)_{AB} := {}^{(N)}J^1[W](e_A,e_4,e_B)\,.
\end{equation*}

In the first instance we set
\begin{equation}
  m=0\,,\qquad\mb=0\,,\qquad \hat{i}=0\,.
\end{equation}

  \begin{multline}
    8\Kt^1_q-8\Kbt^1_q=j\Bigl\{q^{-1}\Omega^{-1}D\sigma-\curls\beta_q-\frac{3}{2}q^{-1}\tr\chi\sigma-(\zeta+\ds\log q-2\etab)\wedge\beta_q+\frac{1}{2}q\chibh\wedge\alpha_q\\
    +q\Omega^{-1}\Db\sigma-\curls\betab_q-\frac{3}{2}q\tr\chib\sigma+(\zeta+\ds\log q+2\eta)\wedge\betab_q-\frac{1}{2}q^{-1}\chih\wedge\alphab_q\Bigr\}\\
      +n\Bigl\{q\Omega^{-1}\Db\sigma+2\eta\wedge\betab_q\Bigr\}-\nb\Bigl\{-q^{-1}\Omega^{-1}D\sigma-2\etab\wedge\beta_q\Bigr\}\displaybreak[0]\\
    =\Bigl(j+\frac{n+\nb}{4}\Bigr)\Bigl\{\frac{2}{\Omega}\bigl(q\Db\sigma+q^{-1}D\sigma)+4\etab\wedge\beta_q+4\eta\wedge\betab_q\Bigr\}      \\
     +\frac{n-\nb}{2}\Bigl\{q\Omega^{-1}\Db\sigma+2\eta\wedge\betab_q-q^{-1}\Omega^{-1}D\sigma-2\etab\wedge\beta_q\Bigr\}\displaybreak[0]\\
=\Bigl(j+\frac{n+\nb}{4}\Bigr)\Bigl\{\frac{4}{\Omega}\sigmat-\frac{3}{2}\Bigl(2q\omegabh+2q^{-1}\omegah+q\tr\chib+q^{-1}\tr\chi+\Lbh q+\Lh q^{-1}\Bigr)\sigma\\+2(3\etab-\eta-\ds\log q)\wedge\beta_q+2(3\eta-\etab+\ds\log q)\wedge\betab_q\Bigr\}\\
%     +\frac{n-\nb}{2}\Bigl\{q\Omega^{-1}\Db\sigma+2\eta\wedge\betab_q-q^{-1}\Omega^{-1}D\sigma-2\etab\wedge\beta_q\Bigr\}\\
+\frac{n-\nb}{2}\Bigl\{-\frac{3}{2}\bigl(q\tr\chib-q^{-1}\tr\chi\bigr)\sigma-\curls\betab+\curls\beta\\+\bigl(\zeta+\ds\log q\bigr)\wedge(\betab_q+\beta_q)-\frac{1}{2}q^{-1}\chih\wedge\alphab_q-\frac{1}{2}q\chibh\wedge\alpha_q\Bigr\}
  \end{multline}

  \begin{multline}
    8\Xibt^1_q-8\Ibt^1_q=j\Bigl\{q\Omega^{-1}\Db\betab_q-\divs\alphab_q-q\chib^\sharp\cdot\betab_q-q\tr\chib\betab_q-\bigl(q\omegabh+\Omega^{-1}\Db q\bigr)\betab_q+\alphab_q^\sharp\cdot(2\zeta+2\ds\log q+\eta)\\
    +q^{-1}\Omega^{-1}D\betab_q-\ds\rho+\ld\ds\sigma-q^{-1}\tr\chi\betab_q+\bigl(q^{-1}\omegah+\Omega^{-1}D q^{-1}\bigr)\betab_q+3\etab\rho-3\ld\etab\sigma-q^{-1}\chi^\sharp\cdot\betab_q+2q\chibh\cdot\beta_q\Bigr\}\\
    +\nb\Bigl\{q^{-1}\Omega^{-1}D\betab_q-q^{-1}\chi^\sharp\cdot\betab_q+\bigl(q^{-1}\omegah+D q^{-1}\bigr)\betab_q+3\etab\rho-3\ld\etab\sigma\Bigr\}\\-n\Bigl\{-q\Omega^{-1}\Db\betab_q+q\chib^\sharp\cdot\betab_q+\bigl(q\omegabh+\Omega^{-1}\Db q\bigr)\betab_q-\eta^\sharp\cdot\alphab_q\Bigr\}\displaybreak[0]\\
    =\Bigl(j+\frac{n+\nb}{4}\Bigr)\Bigl\{\frac{2}{\Omega}q\Db\betab-2q\chib^\sharp\cdot\betab_q-2\bigl(q\omegabh+\Omega^{-1}\Db q\bigr)\betab_q+2\eta^\sharp\cdot\alphab_q\\
    +\frac{2}{\Omega}q^{-1}D\betab+2\bigl(q^{-1}\omegah+\Omega^{-1}Dq^{-1}\bigr)\betab_q+6\etab\rho-6\ld\etab\sigma-2q^{-1}\chi^\sharp\cdot\betab_q\Bigr\}+jq\tr\chib\betab_q\\
    +\frac{\nb-n}{2}\Bigl\{q^{-1}\Omega^{-1}D\betab_q-q^{-1}\chi^\sharp\cdot\betab_q+\bigl(q^{-1}\omegah+D q^{-1}\bigr)\betab_q+3\etab\rho-3\ld\etab\sigma\\-q\Omega^{-1}\Db\betab_q+q\chib^\sharp\cdot\betab_q+\bigl(q\omegabh+\Omega^{-1}\Db q\bigr)\betab_q-\eta^\sharp\cdot\alphab_q\Bigr\}\displaybreak[0]\\
    =\Bigl(j+\frac{n+\nb}{4}\Bigr)\Bigl\{\frac{4}{\Omega}\betabt_q-\frac{1}{2}\bigl(10 q\omegabh+2 q^{-1}\omegah+3q\tr\chib+3q^{-1}\tr\chi+5 \Lbh q-3\Lh q\bigr)\betab_q\\+3\bigl(\eta+\etab+\ds\log q\bigr)\rho-3\ld\bigl(\eta+\etab+\ds\log q\bigr)\sigma+\alphab^\sharp_q\cdot\bigl(\eta+\etab- \ds\log q\bigr)\Bigr\}+j\,q\tr\chib\betab_q\\
    +\frac{\nb-n}{2}\Bigl\{\bigl(-q^{-1}\tr\chi+2q\tr\chib\bigr)\betab_q-\ds\rho+\ld\ds\sigma+\divs\alphab_q+2q \chibh^\sharp\cdot\beta-2\bigl(\zeta+\ds\log q\bigr)^\sharp\cdot\alphab_q\Bigr\}
  \end{multline}
 
 \begin{multline}
    -16\Ibt^1_q=j\Bigl\{\frac{2}{\Omega}q^{-1}D\betab-2\ds\rho+2\ld\ds\sigma-2q^{-1}\tr\chi\betab+2\bigl(q^{-1}\omegah+\Omega^{-1}D q^{-1}\bigr)\betab_q\\+6\etab\rho-6\ld\etab\sigma-2q^{-1}\chi\cdot\betab_q+4q\chibh\cdot\beta_q\Bigr\}
    -2n\Bigl\{-\Omega^{-1}q\Db\betab+q\chib^\sharp\cdot\betab_q+\bigl(q\omegabh+\Omega^{-1}\Db q\bigr)\betab_q-\eta^\sharp\cdot\alphab_q\Bigr\}\\
    =4j\Bigl\{\frac{1}{\Omega}q^{-1}D\betab+\bigl(q^{-1}\omegah+D q^{-1}\bigr)\betab_q+3\etab\rho-3\ld\etab\sigma-q^{-1}\chi^\sharp\cdot\betab_q\Bigr\}\\+2n\Bigl\{\Omega^{-1}q\Db\betab-q\chib^\sharp\cdot\betab-\bigl(q\omegabh+\Omega^{-1}\Db q\bigr)\betab_q+\eta^\sharp\cdot\alphab_q\Bigr\}\displaybreak[0]\\
    =\bigl(2j+n\bigr)\Bigl\{\frac{2}{\Omega}\betabt_q-\frac{1}{4}\bigl(10 q\omegabh+2 q^{-1}\omegah+3q\tr\chib+3q^{-1}\tr\chi+5\Lbh q-3\Lh q^{-1}\bigr)\betab_q\\
+\frac{3}{2}\bigl(\eta+\etab+\ds\log q\bigr)\rho-\frac{3}{2}\ld\bigl(\eta+\etab+\ds\log q\bigr)\sigma+\frac{1}{2}\alphab^\sharp_q\cdot\bigl(\eta+\etab- \ds\log q\bigr)\Bigr\}\\
+\bigl(2j-n\bigr)\Bigl\{\bigl(2q\tr\chib-q^{-1}\tr\chi\bigr)\betab_q-\ds\rho+\ld\ds\sigma+\divs\alphab_q+2q \chibh^\sharp\cdot\beta-2\bigl(\zeta+\ds\log q\bigr)^\sharp\cdot\alphab_q\Bigr\}
  \end{multline}

  \begin{multline}
    8\It^1_q-8\Xit^1_q=j\Bigl\{\Omega^{-1}q\Db\beta+\ds\rho+\ld\ds\sigma-q\tr\chib\beta_q+\bigl(q\omegabh+\Omega^{-1}\Db q\bigr)\beta_q-3\eta\rho-3\ld\eta\sigma-q\chib^\sharp\cdot\beta_q+2q^{-1}\chi^\sharp\cdot\betab_q\\
    +\Omega^{-1}q^{-1}D\beta_q+\divs\alpha-q^{-1}\chi^\sharp\cdot\beta_q-q^{-1}\tr\chi\beta_q-\bigl(q^{-1}\omegah+D q^{-1}\bigr)\beta_q+(2\zeta+2\ds\log q-\etab)^\sharp\cdot\alpha_q\Bigr\}\\
+\nb\Bigl\{\Omega^{-1}q^{-1}D\beta_q-q^{-1}\chi^\sharp\cdot\beta_q-\bigl(q^{-1}\omegah+\Omega^{-1}D q^{-1}\bigr)\beta_q-\etab^\sharp\cdot\alpha_q\Bigr\}\\-n\Bigl\{-\Omega^{-1}q\Db\beta_q+q\chib^\sharp\cdot\beta_q-\bigl(q\omegabh+\Omega^{-1}\Db q\bigr)\beta_q+3\eta\rho+3\ld\eta\sigma\Bigr\}\displaybreak[0]\\
    =\Bigl(j+\frac{\nb+n}{4}\Bigr)\Bigl\{\frac{2}{\Omega}q\Db\beta+2\bigl(q\omegabh+\Omega^{-1}\Db q\bigr)\beta_q-6\eta\rho-6\ld\eta\sigma-2q\chib^\sharp\cdot\beta_q\\+\frac{2}{\Omega}q^{-1}D\beta_q-2q^{-1}\chi^\sharp\cdot\beta_q-2\bigl(q^{-1}\omegah+\Omega^{-1}Dq^{-1}\bigr)\beta_q-2\etab^\sharp\cdot\alpha_q\Bigr\}+jq^{-1}\tr\chi\beta_q\\
+\frac{\nb-n}{2}\Bigl\{\Omega^{-1}q^{-1}D\beta_q-q^{-1}\chi^\sharp\cdot\beta_q-\bigl(q^{-1}\omegah+\Omega^{-1}D q^{-1}\bigr)\beta_q-\etab^\sharp\cdot\alpha_q\\-\Omega^{-1}q\Db\beta_q+q\chib^\sharp\cdot\beta_q-\bigl(q\omegabh+\Omega^{-1}\Db q\bigr)\beta_q+3\eta\rho+3\ld\eta\sigma\Bigr\}\displaybreak[0]\\
    =\Bigl(j+\frac{\nb+n}{4}\Bigr)\Bigl\{\frac{4}{\Omega}\betat_q-\frac{1}{2}\bigl(2q\omegabh+10q^{-1}\omegah+3q\tr\chib+3q\tr\chi-\Lbh q+5\Lh q\bigr)\beta_q\\-3\bigl(\eta+\etab-\ds\log q\bigr)\rho-3\ld\bigl(\eta+\etab-\ds\log q\bigr)\sigma-\alpha^\sharp_q\cdot \bigl(\eta+\etab+\ds\log q\bigr)\Bigr\}+jq^{-1}\tr\chi\beta_q\\
+\frac{\nb-n}{2}\Bigl\{\bigl(q\tr\chib-2 q^{-1}\tr\chi\bigr)\beta_q+\divs\alpha_q-\ds\rho-\ld\ds\sigma\\+2\bigl(\zeta+\ds\log q\bigr)^\sharp\cdot\alphab_q-2q^{-1}\chih^\sharp\cdot\betab_q\Bigr\}
  \end{multline}

  \begin{multline}
    16\It^1_q=j\Bigl\{2\Omega^{-1}q\Db\beta+2\ds\rho+2\ld\ds\sigma-2q\tr\chib\beta_q+2\omegabh\beta_q-6\eta\rho-6\ld\eta\sigma-2q\chib^\sharp\cdot\beta_q+4q^{-1}\chih^\sharp\cdot\betab_q\Bigr\}\\
    +2\nb\Bigl\{\Omega^{-1}q^{-1}D\beta-q^{-1}\chi^\sharp\cdot\beta_q-\bigl(q^{-1}\omegah+\Omega^{-1} D q^{-1}\bigr)\beta_q-\etab^\sharp\cdot\alpha_q\Bigr\}\\
    =4j\Bigl\{\frac{1}{\Omega}q\Db\beta+\bigl(q\omegabh+\Omega^{-1}\Db q\bigr)\beta_q-q\chib^\sharp\cdot\beta_q-3\eta\rho-3\ld\eta\sigma\Bigr\}\\
    +2\nb\Bigl\{\Omega^{-1}q^{-1}D\beta-q^{-1}\chi^\sharp\cdot\beta_q-\bigl(q^{-1}\omegah+\Omega^{-1} D q^{-1}\bigr)\beta_q-\etab^\sharp\cdot\alpha_q\Bigr\}\displaybreak[0]\\
    =\Bigl(\frac{2j+\nb}{2}\Bigr)\Bigl\{\frac{4}{\Omega}\betat_q-\frac{1}{2}\bigl(2q\omegabh+10q^{-1}\omegah+3q\tr\chib+3q\tr\chi-\Lbh q+5\Lh q\bigr)\beta_q\\-3\bigl(\eta+\etab-\ds\log q\bigr)\rho-3\ld\bigl(\eta+\etab-\ds\log q\bigr)\sigma-\alpha^\sharp_q\cdot \bigl(\eta+\etab+\ds\log q\bigr)\Bigr\}\\
+\bigl(2j-\nb\bigr)\Bigl\{\bigl(q\tr\chib-2 q^{-1}\tr\chi\bigr)\beta_q+\divs\alpha_q-\ds\rho-\ld\ds\sigma+2\bigl(\zeta+\ds\log q\bigr)^\sharp\cdot\alphab_q-2q^{-1}\chih^\sharp\cdot\betab_q\Bigr\}
  \end{multline}

  \begin{multline}
    4\Thetabt^1_q=\frac{1}{2}j\Bigl\{\Omega^{-1}q^{-1}\Dh\alphab_q-\nablas\otimesh\betab_q-\frac{3}{2}q^{-1}\tr\chi\alphab+2\bigl(q^{-1}\omegah+\Omega^{-1}D q^{-1}\bigr)\alphab_q\\+(\zeta+\ds\log q+4\etab)\otimesh\betab_q-3q\chibh\rho+3q\ld\chibh\sigma\Bigr\}
    +\frac{1}{2}n\Bigl\{\Omega^{-1}q\Dbh\alphab_q-q\tr\chib\alphab_q-2\bigl(q\omegabh+\Omega^{-1}D q^{-1}\bigr)\alphab_q\Bigr\}\\
    =j\Bigl\{\frac{1}{\Omega}q^{-1}\Dh\alphab-q^{-1}\tr\chi\alphab_q+2\bigl(q^{-1}\omegah+\Omega^{-1}D q^{-1}\bigr)\alphab_q+4\etab\otimesh\betab\Bigr\}\\+\frac{1}{2}n\Bigl\{\Omega^{-1}q\Dbh\alphab_q-q\tr\chib\alphab_q-2\bigl(q\omegabh+\Omega^{-1}\Db q\bigr)\alphab_q\Bigr\}\displaybreak[0]\\
    =\frac{2j+n}{4}\Bigl\{\frac{2}{\Omega}\alphabt_q-\frac{1}{4}\bigl(-2q^{-1}\omegah+14q\omegabh+3q\tr\chib+3q^{-1}\tr\chi-\Lh q^{-1}+7\Lbh q\bigr)\alphab_q\\+2\bigl(\eta+\etab+\ds\log q\bigr)\otimesh\betab_q\Bigr\}\\
+\frac{2j-n}{4}\Bigl\{-\frac{2}{\Omega}\alphabt_q+\frac{1}{4}\bigl(-2q^{-1}\omegah+14q\omegabh+3q\tr\chib+3q^{-1}\tr\chi-\Lh q^{-1}+7\Lbh q\bigr)\alphab_q\\-q^{-1}\tr\chi\alphab_q-2\bigl(\eta+3\etab\bigr)\otimesh\betab_q-2\nablas\otimesh\betab_q-6q\chibh \rho+6\ld\chibh \sigma\Bigr\}
  \end{multline}

  \begin{multline}
    4\Thetat^1_q=\frac{1}{2}j\Bigl\{\Omega^{-1}q\Dbh\alpha_q+\nablas\otimesh\beta_q-\frac{3}{2}q\tr\chib\alpha_q+2\bigl(q\omegabh+\Omega^{-1}\Db q\bigr)\alpha_q\\+(\zeta+\ds\log q-4\eta)\otimesh\beta_q-3q^{-1}\chih\rho-3q^{-1}\ld\chih\sigma\Bigr\}\\+\frac{1}{2}\nb\Bigl\{\Omega^{-1}q^{-1}\Dh\alpha-q^{-1}\tr\chi\alpha_q-2\bigl(q^{-1}\omegah+\Omega^{-1}D q^{-1}\bigr)\alpha_q\Bigr\}\displaybreak[0]\\=j\Bigl\{\frac{1}{\Omega}q\Dbh\alpha-q\tr\chib\alpha_q+2\bigl(q\omegabh+\Omega^{-1}\Db q\bigr)\alpha_q-4\eta\otimesh\beta_q\Bigr\}\\+\frac{1}{2}\nb\Bigl\{\Omega^{-1}q^{-1}\Dh\alpha_q-q^{-1}\tr\chi\alpha_q-2\bigl(q^{-1}\omegah+\Omega^{-1}D q^{-1}\bigr)\alpha_q\Bigr\}\displaybreak[0]\\
    =\frac{2j+\nb}{4}\Bigl\{\frac{2}{\Omega}\alphat_q-\frac{1}{4}\bigl(-2q\omegabh+14q^{-1}\omegah+3q\tr\chib+3q^{-1}\tr\chi-\Lbh q+7\Lh q^{-1}\bigr)\alpha_q\\-2\bigl(\eta+\etab-\ds\log q\bigr)\otimesh\beta_q\Bigr\}\\
+\frac{2j-\nb}{4}\Bigl\{-\frac{2}{\Omega}\alphat_q+\frac{1}{4}\bigl(-2q\omegabh+14q^{-1}\omegah+3q\tr\chib+3q^{-1}\tr\chi-\Lbh q+7\Lh q^{-1}\bigr)\alpha_q\\+2\bigl(\eta+\etab-\ds\log q\bigr)\otimesh\beta_q-q\tr\chib\alpha_q+2\nablas\otimesh\beta_q+2(\zeta+\ds\log q)\otimesh\beta-6\chih\rho-6\ld\chih\sigma\Bigr\}
  \end{multline}

In the second instance we set 
\begin{equation*}
  j=0\,,\qquad n=0\,,\qquad \nb =0\,,
\end{equation*}
but we shall not list these expressions here  and instead collect all terms with factors $m$, $\mb$, and $\hat{i}$, directly under the label $R^1$, and $Y^1$ below.
In view of Lemma~\ref{lemma:J:M:null} the algebraic expressions for the null components of $J^1$ involving factors in $m$, $\mb$, and $\hat{i}$ can be read off verbatim from \CLemma{14.1}. We employ immediately Prop.~\ref{prop:bianchi:q} to eliminate  derivatives of $W$, $W$ in null directions, $DW$, $\Db W$ in favor of angular derivatives $\nablas W$.
In the extreme cases where derivatives appear that cannot be directly eliminated using the Bianchi equations --- such as $\Dbh\alphab$ in $\Xibt^1$ --- we also use Lemma~\ref{lemma:null:decomposition:LieMW} to rewrite these in terms of $\MLie{N}W$:
\begin{multline}
  \frac{q}{\Omega}\Dbh\alphab-q\tr\chib\alphab_q-2(\omegabh+\Lbh q)\alphab_q=\\
  =\frac{2}{\Omega}\alphabt-\frac{1}{4}\bigl(6q^{-1}\omegah+6q\omegabh-q\tr\chib-q^{-1}\tr\chi+7\Lh q^{-1}-\Lbh q\bigr)\alphab_q\\-q^{-1}\Omega^{-1}\Dh\alphab-q\tr\chib\alphab_q-2(\omegabh+\Lbh q)\alphab_q+2\bigl(2\zeta+\ds\log q\bigr)\otimesh\betab_q\\
= \frac{2}{\Omega}\alphabt-\frac{1}{4}\bigl(6q^{-1}\omegah+6q\omegabh+3q\tr\chib+q^{-1}\tr\chi+7\Lh q^{-1}-\Lbh q\bigr)\alphab_q\\+\nablas\otimesh\betab+\bigl(4\etab+3\zeta+\ds\log q)\otimesh\betab_q+3q\chibh\rho-3q\ld\chibh\sigma
\end{multline}

\smallskip

In conclusion:

\begin{subequations}
\begin{equation}\label{eq:div:K:one}
    \boxed{3(-8\sigmat)\bigl(\Kbt^1-\Kt^1\bigr)=6  \frac{2}{\Omega}\Bigl(j+\frac{n+\nb}{4}\Bigr) \sigmat^2+X_K^1+Q_K^1+Y_K^1+R_K^1}
  \end{equation}

  \begin{equation}\label{eq:X:K:one}
    X_K^1:=-\frac{9}{2}\Bigl(j+\frac{n+\nb}{4}\Bigr)\Bigl(2q\omegabh+2q^{-1}\omegah+q\tr\chib+q^{-1}\tr\chi+\Lbh q+\Lh q^{-1}\Bigr)\sigmat\sigma
  \end{equation}

  \begin{multline}\label{eq:Q:K:one}
Q_K^1:=3\sigmat\Bigl(j+\frac{n+\nb}{4}\Bigr)\Bigl\{2(3\etab-\eta-\ds\log q)\wedge\beta_q+2(3\eta-\etab+\ds\log q)\wedge\betab_q\Bigr\}\\
+3\sigmat\frac{n-\nb}{2}\Bigl\{-\frac{3}{2}\bigl(q\tr\chib-q^{-1}\tr\chi\bigr)\sigma-\curls\betab+\curls\beta\\+\bigl(\zeta+\ds\log q\bigr)\wedge(\betab_q+\beta_q)-\frac{1}{2}q^{-1}\chih\wedge\alphab_q-\frac{1}{2}q\chibh\wedge\alpha_q\Bigr\}
  \end{multline}

  \begin{equation}\label{eq:Y:K:one}
    Y_K^1 := \frac{3}{2}q\tr\chib \sigmat \hat{i}\wedge\alpha-\frac{3}{2}q^{-1}\tr\chi\sigmat \hat{i}\wedge\alphab_q-9q\tr\chib\sigmat m\wedge\betab_q+9q^{-1}\tr\chi\sigmat m\wedge \beta_q
  \end{equation}

  \begin{multline}
    R_K^1 := 3\sigmat m\wedge\Bigl\{\bigl(q^{-1}\tr\chi-q\tr\chib\bigr)\betab_q+\bigl(q^{-1}\tr\chi-q\tr\chib\bigr) \beta_q+2\ds\rho-2q\chibh^\sharp\cdot\beta_q+2q^{-1}\chih^\sharp\cdot\betab_q\Bigr\}\\
-2\divs\alpha-2\divs\alphab-2(\eta-2\zeta-2\ds\log q)^\sharp\cdot\alpha_q-2(\etab+2\zeta+2\ds\log q)^\sharp\cdot\alpha_q\Bigr\}\\
-6\sigmat\Bigl(\mb,\etab^\sharp\cdot \ld\alpha_q+\eta^\sharp\cdot\ld\alphab_q\Bigr)\\
-3\sigmat\Bigl(\hat{i},\nablas\otimesh\ld\beta+\nablas\otimesh\ld\betab+(\zeta+\ds\log q)\otimesh(\ld\beta_q-\ld\betab_q)+3(q\ld\chibh-q^{-1}\ld\chih)\rho+3(q\chibh+q^{-1}\chih)\sigma\Bigr)
  \end{multline}

\end{subequations}

\begin{subequations}

  \begin{equation}\label{eq:div:Xib:one}
\boxed{    8(\betabt_q,\Xibt^1-\Ibt^1) -16(\betabt_q,\Ibt^1) = 2 \frac{2}{\Omega} \Bigl(j+\frac{n+\nb}{4}+\frac{2j+n}{2}\Bigr) \lvert \betabt_q \rvert^2 + X_{\Xib}^1+ Q_{\Xib}^1 +Y_{\Xib}^1+R_{\Xib}^1}
  \end{equation}

  \begin{multline}\label{eq:X:Xib:one}
    X_{\Xib}^1:=-\frac{1}{2}\Bigl(j+\frac{n+\nb}{4}+\frac{2j+n}{2}\Bigr)\bigl(10 q\omegabh+2 q^{-1}\omegah+3q\tr\chib+3q^{-1}\tr\chi+5 \Lbh q-3\Lh q\bigr)(\betabt_q,\betab_q)\\
+j\,q\tr\chib(\betabt_q,\betab_q)+\Bigl(\frac{\nb-n}{2}+2j-n\Bigr)\bigl(2q\tr\chib-q^{-1}\tr\chi\bigr)(\betabt_q,\betab_q)
  \end{multline}

  \begin{multline}\label{eq:Q:Xib:one}
    Q_{\Xib}^1:=\Bigl(j+\frac{n+\nb}{4}+\frac{2j+n}{2}\Bigr)\Bigl(\betabt_q,3\bigl(\eta+\etab+\ds\log q\bigr)\rho-3\ld\bigl(\eta+\etab+\ds\log q\bigr)\sigma+\alphab^\sharp_q\cdot\bigl(\eta+\etab- \ds\log q\bigr)\Bigr)\\
+\Bigl(\frac{\nb-n}{2}+2j-n\Bigr)\Bigl(\betabt_q,-\ds\rho+\ld\ds\sigma+\divs\alphab_q+2q \chibh^\sharp\cdot\beta-2\bigl(\zeta+\ds\log q\bigr)^\sharp\cdot\alphab_q\Bigr)
  \end{multline}

  \begin{multline}\label{eq:Y:Xib:one}
    Y_{\Xib}^1 := 6q\tr\chib\,\betabt^A\hat{i}_A^{\phantom{A}B}\beta_B +\frac{1}{2}q^{-1}\tr\chi\alphab_{AB}\betabt^A\mb^B+3q\tr\chib (\betabt,\mb) \rho-3q\tr\chib(\betabt,\ld\mb)\sigma\\
-\frac{1}{4}\bigl(6q^{-1}\omegah+6q\omegabh+3q\tr\chib+q^{-1}\tr\chi+7\Lh q^{-1}-\Lbh q\bigr)\alphab_{AB}\betabt^A m^B
  \end{multline}

  \begin{multline}
    R_{\Xib}^1 := \frac{2}{\Omega}\alphabt_{AB}\betabt^A m^B+\betabt^A\mb^B\Bigl\{4\nablas_B\betab_A-q^{-1}\chih_B^{\phantom{B}C}\alphab_{AC}+2((\zeta+\ds\log q+2\etab)\otimesh\betab)_{AB}\\
-2(\nablas\otimesh\betab)_{AB}-((4\etab-\zeta-\ds\log q)\otimesh\beta)_{AB}-\bigl((4\etab+3\zeta+\ds\log q)\otimesh\betab\bigr)_{AB}\Bigr\}\\
-2\betabt^{A}\hat{i}^{BC}\Bigl\{\nablas_C\alphab_{AB}+(q\chibh-3q^{-1}\chih)_{CA}\betab_B+(q\chibh+3q^{-1}\chih)_{CB}\betab_A-(q\chibh-3q^{-1}\chih)_C^{\phantom{C}D}\betab_D\gs_{AB}\\-2(\zeta+\ds\log q)_C\alphab_{AB}
-3q\chibh_{CA}\beta_B+3q\chibh_{CB}\beta_A+3\gs_{AB}q\chibh_C^{\phantom{C}D}\beta_D\Bigr\}\displaybreak[0]\\
+3(\betabt,\mb)\Bigl\{\frac{3}{2}(q\tr\chib-q^{-1}\tr\chi)\rho+\divs\beta+\divs\betab\\+(2\etab+\zeta+\ds\log q,\beta)-(2\eta-\zeta-\ds\log q,\beta)-\frac{1}{2}(q\chibh,\alpha)+\frac{1}{2}(q^{-1}\chih,\alphab)\Bigr\}\\
-3(\betabt,\ld\mb)\Bigl\{\frac{3}{2}(q\tr\chib-q^{-1}\tr\chi)\sigma-\curls\beta+\curls\betab\\-(2\etab+\zeta+\ds\log q)\wedge\beta+(2\eta-\zeta-\ds\log q)\wedge\betab+\frac{1}{2}q\chibh\wedge\alpha+\frac{1}{2}q^{-1}\chih\wedge\alphab\Bigr\}\\
-6\betabt^A\mb^B\Bigl\{\etab_A\beta_B-\etab_B\beta_A+\gs_{AB}(\etab,\beta)+\eta_A\betab_B-\eta_B\betab_A+\gs_{AB}(\eta,\betab)\Bigr\}\displaybreak[0]\\
+3\betabt^Am^B\Bigl\{2(\zeta+\ds\log q)_B\betab_A-q^{-1}\chih_B^{\phantom{B}C}\alphab_{AC}\Bigr\}
-6\betabt^A\hat{i}_A^B\ds_B\rho+6\betabt^A\hat{i}_A^{\phantom{A}B}\ds_B\sigma
  \end{multline}

\end{subequations}

\begin{subequations}

  \begin{equation}\label{eq:div:Xi:one}
\boxed{    8(\betat_q,\It^1-\Xit^1) +   16(\betat_q,\It^1) = 2 \frac{2}{\Omega} \Bigl(j+\frac{\nb+n}{4}+\frac{2j+\nb}{2}\Bigr)\lvert \betat_q \rvert^2 + X_\Xi^1  +Q_{\Xi}^1+Y_{\Xi}^1+R_{\Xi}^1 }
  \end{equation}

  \begin{multline}\label{eq:X:Xi:one}
    X_\Xi^1:=-\Bigl(j+\frac{\nb+n}{4}+\frac{2j+\nb}{2}\Bigr)\frac{1}{2}\bigl(2q\omegabh+10q^{-1}\omegah+3q\tr\chib+3q\tr\chi-\Lbh q+5\Lh q\bigr)(\betat_q,\beta_q)\\
+jq^{-1}\tr\chi(\betat_q,\beta_q)+\Bigl(\frac{\nb-n}{2}+2j-\nb\Bigr)\bigl(q\tr\chib-2 q^{-1}\tr\chi\bigr)(\betat_q,\beta_q)
  \end{multline}

  \begin{multline}\label{eq:Q:Xi:one}
    Q_{\Xi}^1:=\Bigl(j+\frac{\nb+n}{4}+\frac{2j+\nb}{2}\Bigr)\Bigl(\betat_q,-3\bigl(\eta+\etab-\ds\log q\bigr)\rho-3\ld\bigl(\eta+\etab-\ds\log q\bigr)\sigma-\alpha^\sharp_q\cdot \bigl(\eta+\etab+\ds\log q\bigr)\Bigr)\\
+\Bigl(\frac{\nb-n}{2}+2j-\nb\Bigr)\Bigl(\betat_q,\divs\alpha_q-\ds\rho-\ld\ds\sigma+2\bigl(\zeta+\ds\log q\bigr)^\sharp\cdot\alphab_q-2q^{-1}\chih^\sharp\cdot\betab_q\Bigr)
  \end{multline}

  \begin{multline}\label{eq:Y:Xi:one}
    Y_{\Xi}^1 := 6q^{-1}\tr\chi\betat^A\hat{i}_A^{\phantom{A}B}\betab_B  +\frac{1}{2}q\tr\chib\alpha_{AB}\betat^Am^B+3q^{-1}\tr\chi (\betat,m) \rho-3q^{-1}\tr\chi(\betat,\ld m)\sigma\\
-\frac{1}{4}\bigl(6q\omegabh+6q^{-1}\omegah+3q^{-1}\tr\chi+q\tr\chib+7\Lbh q-\Lh q^{-1}\bigr)\alpha_{AB}\betat^A \mb^B
  \end{multline}

%Just obtained by conjugation, to be checked
  \begin{multline}
    R_{\Xi}^1 :=\frac{2}{\Omega}\alphat_{AB}\betat^A \mb^B+ \betat^Am^B\Bigl\{4\nablas_B\beta_A-q\chibh_B^{\phantom{B}C}\alpha_{AC}-2((-\zeta-\ds\log q+2\eta)\otimesh\beta)_{AB}\\
-2(\nablas\otimesh\beta)_{AB}+((4\eta-\zeta-\ds\log q)\otimesh\betab)_{AB}+\bigl((4\eta+3\zeta+\ds\log q)\otimesh\beta\bigr)_{AB}\Bigr\}\\
-2\betat^{A}\hat{i}^{BC}\Bigl\{\nablas_C\alpha_{AB}+(q^{-1}\chih-3q\chibh)_{CA}\beta_B+(q^{-1}\chih+3q\chibh)_{CB}\beta_A-(q^{-1}\chih-3q\chibh)_C^{\phantom{C}D}\beta_D\gs_{AB}\\+2(\zeta+\ds\log q)_C\alpha_{AB}
-3q^{-1}\chih_{CA}\betab_B+3q^{-1}\chih_{CB}\betab_A+3\gs_{AB}q^{-1}\chih_C^{\phantom{C}D}\betab_D\Bigr\}\displaybreak[0]\\
+3(\betat,m)\Bigl\{-\frac{3}{2}(q\tr\chib-q^{-1}\tr\chi)\rho+\divs\betab+\divs\beta\\-(2\etab-\zeta-\ds\log q,\betab)+(2\etab+\zeta+\ds\log q,\beta)-\frac{1}{2}(q^{-1}\chih,\alphab)+\frac{1}{2}(q\chibh,\alpha)\Bigr\}\\
-3(\betat,\ld m)\Bigl\{-\frac{3}{2}(q\tr\chib-q^{-1}\tr\chi)\sigma-\curls\betab+\curls\beta\\+(2\eta-\zeta-\ds\log q)\wedge\betab-(2\etab+\zeta+\ds\log q)\wedge\beta+\frac{1}{2}q^{-1}\chih\wedge\alphab+\frac{1}{2}q\chibh\wedge\alpha\Bigr\}\\
-6\betat^Am^B\Bigl\{\eta_A\betab_B-\eta_B\betab_A+\gs_{AB}(\eta,\betab)+\etab_A\beta_B-\etab_B\beta_A+\gs_{AB}(\etab,\beta)\Bigr\}\displaybreak[0]\\
+3\betat^A\mb^B\Bigl\{-2(\zeta+\ds\log q)_B\beta_A-q\chibh_B^{\phantom{B}C}\alpha_{AC}\Bigr\}
+6\betat^A\hat{i}_A^B\ds_B\rho+6\betat^A\hat{i}_A^{\phantom{A}B}\ds_B\sigma
  \end{multline}

\end{subequations}

\begin{subequations}
  \begin{equation}\label{eq:div:Thetab:one}
\boxed{        4(\alphabt_q,\Thetabt^1) = \frac{n}{\Omega} \lvert \alphabt_q \rvert^2 +X_{\Thetab}^1  + Q_{\Thetab}^1+Y_{\Thetab}^1+R_{\Thetab}^1 }
  \end{equation}

  \begin{multline}\label{eq:X:Thetab:one}
    X_{\Thetab}^1 := -\frac{1}{4} \frac{n}{2} \bigl(-2q^{-1}\omegah+14q\omegabh+3q\tr\chib+3q^{-1}\tr\chi-\Lh q^{-1}+7\Lbh q\bigr)(\alphabt_q,\alphab_q)\\
-\frac{2j-n}{4} q^{-1}\tr\chi (\alphabt_q,\alphab_q)
  \end{multline}

  \begin{equation}\label{eq:Q:Thetab:one}
    Q_{\Thetab}^1 :=  n\Bigl(\alphabt_q,\bigl(\eta+\etab+\ds\log q\bigr)\otimesh\betab_q\Bigr)
-\frac{2j-n}{2}\Bigl(\alphabt_q,\bigl(\zeta+\ds\log q\bigr)\otimesh\betab_q+\nablas\otimesh\betab_q +3q\chibh \rho-3q\ld\chibh\sigma\Bigr)
  \end{equation}

  \begin{equation}\label{eq:Y:Thetab:one}
    Y_{\Thetab}^1 := -\frac{3}{2}q\tr\chib \bigl(\alphabt,\hat{i}\bigr)\rho+\frac{3}{2}q\tr\chib\bigl(\alphabt,\ld\hat{i}\bigr)\sigma-\frac{1}{4}q\tr\chib\Bigl(\alphabt,m\otimesh\betab_q\Bigr)
  \end{equation}

  \begin{multline}
    R_{\Thetab}^1 := -\frac{1}{2}\biggl(\alphabt,\mb\otimesh\Bigl(\bigl(q^{-1}\tr\chi-q\tr\chib\bigr)\betab_q+\ds\rho-\ld\ds\sigma-2q^{-1}\chih^\sharp\cdot\betab\Bigr)\biggr)\\
    -\frac{1}{2}\biggl(\alphabt,m\otimesh\Bigl(\divs\alphab_q+(\etab-2\zeta-2\ds\log q)^\sharp\cdot\alphab_q\Bigr)\biggr)\\
    -\frac{1}{2}\biggl(\alphabt,\etab\otimesh\Bigl(m^\sharp\cdot\alphab_q\Bigr)\biggr)+\frac{1}{2}\bigl(m,\eta+2\ds\log q)\bigl(\alphabt,\alphab_q\bigr)\displaybreak[0]\\
-\frac{1}{2}\alphabt^{AB}m^C\Bigl\{\nablas_C\alphab_{AB}+(q\chibh\otimesh\betab_q)_{CAB}\Bigr\}
-2\alphabt^{AB}\hat{i}_{A}^{\phantom{A}D}\nablas_D\betab_B
+\alphabt^{AB}\hat{i}^{CD}q^{-1}\chih_{CA}\alphab_{DB}\\
-\bigl(\hat{i},q^{-1}\chih\bigr)\bigl(\alphabt,\alphab_q)+\frac{1}{2}\Bigl(\alphabt,\bigl(\hat{i}^\sharp\cdot(\zeta+\ds\log q)\bigr)\otimesh\betab_q\Bigr)
  \end{multline}
\end{subequations}

\begin{subequations}
  \begin{equation}\label{eq:div:Theta:one}
\boxed{    4(\alphat_q,\Thetat^1) = \frac{\nb}{\Omega}\lvert \alphat_q \rvert^2 +X_\Theta^1  + Q_\Theta^1 +Y_{\Theta}^1+R_{\Theta}^1} 
  \end{equation}

  \begin{multline}\label{eq:X:Theta:one}
    X_\Theta^1 := -\frac{1}{4}\frac{\nb}{2}\bigl(-2q\omegabh+14q^{-1}\omegah+3q\tr\chib+3q^{-1}\tr\chi-\Lbh q+7\Lh q^{-1}\bigr)(\alphat_q,\alpha_q)\\
-\frac{2j-\nb}{4} q\tr\chib (\alphat_q,\alpha_q)
  \end{multline}

  \begin{equation}\label{eq:Q:Theta:one}
    Q_\Theta^1 := -\nb\Bigl(\alphat_q,\bigl(\eta+\etab-\ds\log q\bigr)\otimesh\beta_q\Bigr)
+\frac{2j-\nb}{2}\Bigl(\alphat_q,\nablas\otimesh\beta_q+(\zeta+\ds\log q)\otimesh\beta-3q^{-1}\chih\rho-3q^{-1}\ld\chih\sigma\Bigr)
  \end{equation}

  \begin{equation}\label{eq:Y:Theta:one}
    Y_{\Theta}^1 := -\frac{3}{2}q^{-1}\tr\chi \bigl(\alphat,\hat{i}\bigr)\rho+\frac{3}{2}q^{-1}\tr\chi\bigl(\alphat,\ld\hat{i}\bigr)\sigma+\frac{1}{4}q^{-1}\tr\chi\Bigl(\alphat,\mb\otimesh\beta_q\Bigr)
  \end{equation}

  \begin{multline}
    R_{\Theta}^1 := -\frac{1}{2}\biggl(\alphat,m\otimesh\Bigl(\bigl(q^{-1}\tr\chi-q\tr\chib\bigr)\beta_q+\ds\rho+\ld\ds\sigma+2q^{-1}\chih^\sharp\cdot\betab\Bigr)\biggr)\\
    -\frac{1}{2}\biggl(\alphat,\mb\otimesh\Bigl(-\frac{3}{2}q^{-1}\tr\chi\beta_q+\divs\alpha_q+(\etab+2\zeta+2\ds\log q)^\sharp\cdot\alpha_q\Bigr)\biggr)\\
    -\frac{1}{2}\biggl(\alphat,\eta\otimesh\Bigl(\mb^\sharp\cdot\alpha_q\Bigr)\biggr)+\frac{1}{2}\bigl(\mb,\etab-2\ds\log q)\bigl(\alphat,\alpha_q\bigr)\displaybreak[0]\\
-\frac{1}{2}\alphat^{AB}\mb^C\Bigl\{\nablas_C\alpha_{AB}+(q^{-1}\chih\otimesh\beta_q)_{CAB}\Bigr\}
+2\alphat^{AB}\hat{i}_{A}^{\phantom{A}D}\nablas_D\beta_B
+\alphat^{AB}\hat{i}^{CD}q\chibh_{CA}\alpha_{DB}\\
-\bigl(\hat{i},q\chibh\bigr)\bigl(\alphat,\alpha_q)+\frac{1}{2}\Bigl(\alphat,\bigl(\hat{i}^\sharp\cdot(\zeta+\ds\log q)\bigr)\otimesh\beta_q\Bigr)
  \end{multline}

\end{subequations}

The full divergence, $\divergence Q[\MLie{N}W](M_q,M_q,M_q)$, is given --- according to \eqref{eq:divQ:M:null} --- by the sum of \eqref{eq:div:Lambda:one}, \eqref{eq:div:K:one}, \eqref{eq:div:Xib:one}, \eqref{eq:div:Xi:one}, \eqref{eq:div:Thetab:one}, and \eqref{eq:div:Theta:one}. After multiplying by $\phi$, the sum of the  ``principal terms'' --- namely the first terms in each of the formulas --- is bounded by
\begin{multline*}
  \frac{\phi}{(2\Omega)^3}\biggl[\frac{2}{\Omega}\Bigl(  j + \frac{1}{4}(\nb+n)\Bigr) \bigl(2\lvert \betabt_q \rvert^2+6\rhot^2 +  6\sigmat^2 +2\lvert \betat_q \rvert^2\bigr)\\
+\frac{n}{\Omega} \lvert \alphabt_q \rvert^2+2 \frac{2}{\Omega} (2j+n) \lvert \betabt_q \rvert^2+2 \frac{2}{\Omega} (2j+\nb)\lvert \betat_q \rvert^2+\frac{\nb}{\Omega}\lvert \alphat_q \rvert^2\biggr]\leq\\
\leq \frac{C}{\Omega}\frac{2}{r}\frac{1}{(2\Omega)^3}\Bigl[|\alphabt|^2+|\betabt|^2+\rhot^2+\sigmat^2+|\betat|^2+|\alphat|^2\Bigr]
\leq \frac{1}{\Omega}\frac{C}{r} Q[\MLie{N}W](n,M_q,M_q,M_q)
\end{multline*}
where on the right hand side $n$ denotes the normal vector \eqref{eq:n:q}, and we used --- the lower bound was already used in the prooof of Lemma~\ref{lemma:K:plus:q} --- that 
\begin{subequations}
  \begin{gather}
    \phi\, q^{-1}\tr\chi = \frac{2}{r}\frac{\Omega\tr\chi}{\overline{\Omega\tr\chi}}\leq \frac{2}{r}(1+C_0\Omega^{-1})\\
    \phi\, q\tr\chib =\frac{2}{r}\frac{\Omega\tr\chib}{\overline{\Omega\tr\chib}}\leq \frac{2}{r}(1+C_0\Omega^{-1})
  \end{gather}
\end{subequations}
Moreover, it then follows from \eqref{eq:Q:Lambda:one}, \eqref{eq:Q:K:one}, \eqref{eq:Q:Xib:one}, \eqref{eq:Q:Xi:one}, \eqref{eq:Q:Thetab:one}, and \eqref{eq:Q:Theta:one} that
\begin{multline}\label{eq:J:1:error}
  \frac{\phi}{(2\Omega)^3}\Bigl[ |Q_\Lambda^1|+ |Q_\Lambda^1|+|Q_{\Xib}^1|^2+|Q_{\Xi}^1|^2+|Q_{\Thetab}|^2+|Q_{\Theta}|^2\Bigr]\\\leq \frac{C}{\Omega r} \Bigl(Q[\MLie{N}W](n,M_q,M_q,M_q)+ Q[W](n,M_q,M_q,M_q)\Bigr)\\
+\frac{C}{\Omega r}\frac{1}{(2\Omega)^3}\Bigl[\Omega^2|\nablas\alphab_q|^2+\Omega^2|\nablas\betab_q|^2+\Omega^2|\nablas\rho|^2+\Omega^2|\nablas\sigma|^2+\Omega^2|\nablas\beta_q|^2+\Omega^2|\alpha_q|^2\Bigr]
\end{multline}
Similarly for the terms $R_\Lambda^1$, $R_K^1$, $R_{\Xib}^1$, $R_{\Xi}^1$, $R_{\Thetab}^1$, $R_{\Theta}^1$.

\end{proof}

\subsubsection{$J^2$, and $J^3$}
\label{sec:J:23}

Here we analyse the currents $J^2$, and $J^3$. They contain ``lower order'' terms at the level of $W$, but also factors at the level of $\nabla \pi$.
We are interested in their precise structure to show cancellations with ``lower order'' terms from $J^1$, and control the remainder with assumptions on $\nabla \Gamma$.

Recall the notation
\begin{equation*}
  p_\beta=\nabla^\alpha\,\pih{N}_{\alpha\beta}\qquad
  p_3=p_\beta e_3^\beta\qquad p_4=p_\beta e_4^\beta\,,
\end{equation*}
and also that we denote by $d$ the 3-form
\begin{equation*}
  d_{\alpha\beta\gamma}=\nabla_\beta\, \pih{N}_{\gamma\alpha}-\nabla_\gamma\,\pih{N}_{\beta\alpha}+\frac{1}{3}\Bigl(p_\beta g_{\alpha\gamma}-p_\gamma g_{\alpha\beta}\Bigr)
\end{equation*}
which has the algebraic properties of a Weyl current, and can be decomposed into the $S_{u,v}$-1-forms $\Xib(d)$, $\Xi(d)$, the functions $\Lambdab(d)$, $\Lambda(d)$, $\Kb(d)$, $K(d)$, and the $S_{u,v}$ 2-forms $\Thetab(d)$, $\Theta(d)$; cf.~\Ceq{12.51-61}.

In order to control the derivatives of the deformation tensor of $N$ we introduce the following assumptions:
\begin{subequations}\label{eq:BA:II:remaining}
  \begin{gather}
     \lvert D\bigl(\Omega\tr\chib-2\omegab\bigr) \rvert \leq C_0 \tr\chi\tr\chib\qquad \lvert \Db\bigl(\Omega\tr\chi-2\omega\bigr) \rvert \leq C_0\tr\chib\tr\chi \tag{\emph{\textbf{BA:II}.iii}}\\
     \lvert D\bigl(\Omega\tr\chi-2\omega\bigr) \rvert  \leq C_0 \tr\chi \tr\chi\qquad  \lvert \Db\bigl(\Omega\tr\chib-2\omegab\bigr) \rvert \leq C_0 \tr\chib\tr\chib\notag\\
    \Omega \bigl\lvert \nablas(\Omega\tr\chi-2\omega) \bigr\rvert \leq C_0\tr\chi \qquad     \Omega \bigl\lvert \nablas (\Omega\tr\chib-2\omegab) \bigr\rvert \leq C_0\tr\chib\tag{\emph{\textbf{BA:II}.iv}}
  \end{gather}

  \begin{gather}
          \lvert D\Db\log q \rvert \leq C_0\tr\chi\tr\chib  \qquad  \lvert DD\log q \rvert \leq C_0\tr\chi\tr\chi\tag{\emph{\textbf{BA:II}.v}}\\
     \lvert \Db\Db\log q \rvert \leq C_0 \tr\chib\tr\chib \qquad \lvert \Db D\log q \rvert\leq C_0 \tr\chib\tr\chi\notag\\
         \lvert \Omega \nablas D \log q \rvert \leq C_0 \tr\chi \qquad \lvert \Omega \nablas \Db \log q \rvert \leq C_0 \tr\chib \tag{\emph{\textbf{BA:II}.vi}}
  \end{gather}

\begin{gather}
     \lvert D(\Omega\zeta) \rvert + \lvert D(\Omega\ds\log q\bigr) \rvert\leq C_0 \tr\chi\qquad \lvert \Db (\Omega \zeta ) \rvert + \lvert \Db(\Omega\ds\log q\bigr) \rvert\leq C_0 \tr\chib\tag{\emph{\textbf{BA:II}.vii}}\\
  \Omega | \nablas (\Omega\zeta) | + \Omega |\nablas (\Omega\ds\log q) |\leq C_0 \bigl(q\tr\chib+q^{-1}\tr\chi\bigr)\tag{\emph{\textbf{BA:II}.viii}}
\end{gather}
\end{subequations}
\begin{remark}
  In the context of the \emph{bootstrap argument} the assumptions (\emph{\textbf{BA:II}}) should  be formulated  in  the $\mathrm{L}^4(S_{u,v})$ norm, but we state them here for simplicity as $\mathrm{L}^\infty$ assumptions. While the scaling and weights are identical to the appropriate assumptions in $\mathrm{L}^4$, we will proceed here for convenience with estimating all quantities pointwise, as opposed to also considering integrations on $S_{u,v}$.
\end{remark}

\begin{lemma}\label{lemma:div:N:two:three}
Assume that (\textbf{BA:I}.i-viii) hold for some $C_0>0$.
Then,
\begin{multline}\label{eq:div:N:two:three}
 \phi \Bigl[ \bigl(\divergence Q(\MLie{N}W)\bigr)(M_q,M_q,M_q) \Bigr]^{2}+ \phi \Bigl[ \bigl(\divergence Q(\MLie{N}W)\bigr)(M_q,M_q,M_q) \Bigr]^{2}\dm{\gb{r}}\geq\\\geq  \frac{\phi}{(2\Omega)^3}\bigl[X^{2+3}+Y^{2+3}\bigr]\\  -\frac{C}{r}\frac{1}{\Omega}Q[\MLie{N}W](n,M_q,M_q,M_q)  -\frac{C}{r}\frac{1}{\Omega}Q[W](n,M_q,M_q,M_q) -\frac{C}{r}\frac{1}{\Omega}\Bigl[\Omega^2\Ps^q\Bigr]
\end{multline}
where
\begin{subequations} \label{eq:X:Y:2:3}
\begin{gather}
  X^{2+3}:=X_\Lambda^{2+3}+X_K^{2+3}+X_{\Xib}^{2+3}+X_{\Xi}^{2+3}+X_{\Thetab}^{2+3}+X_{\Theta}^{2+3}\\
  Y^{2+3}:=X_\Lambda^{2+3}+Y_K^{2+3}+Y_{\Xib}^{2+3}+Y_{\Xi}^{2+3}+Y_{\Thetab}^{2+3}+Y_{\Theta}^{2+3}
\end{gather}
\end{subequations}
and $X_\Lambda^{2+3}$, $X_K^{2+3}$, %$X_{\Xib}^{2+3}$, $X_{\Xi}^{2+3}$, 
$X_{\Thetab}^{2+3}$, $X_{\Theta}^{2+3}$, are given by \eqref{eq:X:Lambda:two}, \eqref{eq:X:K:two}, %\eqref{eq:X:Xib:two}, \eqref{eq:X:Xi:two}, 
\eqref{eq:X:Thetab:two}, \eqref{eq:X:Theta:two}, respectively, while $Y_\Lambda^{2+3}$, $Y_K^{2+3}$, %$Y_{\Xib}^{2+3}$, $Y_{\Xi}^{2+3}$, 
$Y_{\Thetab}^{2+3}$, $Y_{\Theta}^{2+3}$, are given by \eqref{eq:Y:Lambda:two}, \eqref{eq:Y:K:two}, %\eqref{eq:Y:Xib:two}, \eqref{eq:Y:Xi:two}, 
\eqref{eq:Y:Thetab:two}, \eqref{eq:Y:Theta:two}, respectively.

\end{lemma}

\begin{proof}
For the null components of ${}^{(N)}J^2[W]$,  we directly refer to \Ceq{14.98}:
  \begin{subequations}
\begin{gather}
  8\Lambdat^2+8\Lambdabt^2=-2(p_3+p_4)\rho+2(\ps,\beta_q-\betab_q) \label{eq:Lambdab:Lambda:two}\\
  8\Kt^2-8\Kbt^2=-2(p_4+p_3)\sigma-2\ps\wedge(\beta_q+\betab_q)\\
  4\Thetat^2=-p_3\alpha_q+\ps\otimesh\beta_q\\
  4\Thetabt^2=-p_4\alphab_q-\ps\otimesh\betab_q\\
  8(\Xibt^2-3\Ibt^2)=2p_3\betab_q-2\ps^\sharp\cdot\alphab_q-6p_4\betab_q-6\ps\rho+\ld\ps\sigma\\
  8(-\Xit^2+3\It^2)=2p_4\beta_q+2\ps^\sharp\cdot\alpha_q-6p_3\beta_q+6\ps\rho+6\ld\ps\sigma
\end{gather}

\end{subequations}

We do not write out the expressions for $p$ at this point, due to cancellations with contributions from $J^3$.

For the null components of ${}^{(N)}J^3[W]$ we have \CLemma{14.2}:

\begin{subequations}
\begin{multline}
  8\Lambdat^3+8\Lambdabt^3=\bigl(\Thetab(d),\alpha_q\bigr)+\bigl(\Theta(d),\alphab_q\bigr)-3\bigl(2\Lambda(d)+2\Lambdab(d)\bigr)\rho+3\bigl(2K(d)-2\Kb(d)\bigr)\sigma\\
  +2\bigl(\Xi(d),\betab_q\bigr)-2\bigl(\Xib(d),\beta_q\bigr)
\end{multline}
\begin{multline}
  8\Kt^3-8\Kbt^3=-\Thetab(d)\wedge \alpha_q+\Theta(d)\wedge\alphab_q-3\bigl(2\Lambda(d)+2\Lambdab(d)\bigr)\sigma-3\bigl(2K(d)-2\Kb(d)\bigr)\rho\\
  +2\Xi(d)\wedge \betab_q+2\Xib(d)\wedge \beta_q
\end{multline}
\begin{gather}
  4\Thetat^3 = 3 \Lambdab(d) \alpha_q - 3\Kb(d)\ld \alpha_q-6 I(d)\otimesh\beta+3\Theta(d)\rho+3\ld\Theta(d)\sigma\\
  4\Thetabt^3 = 3 \Lambda(d) \alphab_q  -  3 K(d)\ld\alphab_q + 6 \Ib(d)\otimesh\betab + 3\Thetab(d)\rho-3\ld\Thetab(d)\sigma
\end{gather}
\begin{multline}
  8\bigl(\Xibt^3-3\Ibt^3\bigr)=6(\Xi(d)-I(d))^\sharp\cdot \alphab_q-12\Lambdab(d)\betab_q-12\Kb(d)\ld\betab_q\\+6(\Xib(d)+3\Ib(d))\rho-6\ld(\Xib(d)+3\Ib(d))\sigma-12\Thetab(d)\cdot\beta_q^\sharp
\end{multline}
\begin{multline}
  8\bigl(-\Xit^3+3I^3\bigr)=+6(\Ib(d)-\Xib(d))^\sharp\cdot \alpha_q-12\Lambda(d)\beta_q-12K(d)\ld\beta_q\\
-6(\Xi(d)+3I(d))\rho-3\ld(\Xi(d)+3I(d))\sigma-12\Theta(d)\cdot\betab_q^\sharp
\end{multline}

\end{subequations}

Now most importantly,
\begin{gather*}
  2\Lambda(d)=\frac{1}{2}d_{434}=\frac{1}{2}\nabla_3\pih{N}_{44}-\frac{1}{2}\nabla_4\pih{N}_{34}+\frac{1}{3}p_4
  =-\nabla_4\pih{N}_{34}+g^{AC}\nabla_C\pih{N}_{A4}-\frac{2}{3}p_4\\
  2\Lambdab(d)=\frac{1}{2}d_{343}=-\nabla_3\pih{N}_{34}+g^{AC}\nabla_C\pih{N}_{A3}-\frac{2}{3}p_3
\end{gather*}
where we used that
\begin{gather*}
  p_4=\nabla^\alpha\pih{N}_{\alpha 4}=-\frac{1}{2}\nabla_4\pih{N}_{34}-\frac{1}{2}\nabla_3\pih{N}_{44}+\nabla^A\pih{N}_{A4}\\
  p_3=-\frac{1}{2}\nabla_4\pih{N}_{33}-\frac{1}{2}\nabla_3\pih{N}_{43}+\nabla^A\pih{N}_{A3}
\end{gather*}

Moreover,
\begin{gather*}
  2\Xi_A(d)=\Omega^{-1}q^{-1}Dm+q^{-1}\chi^\sharp\cdot m-\bigl(q^{-1}\omegah+\Lh q^{-1}\bigr)m-\nablas_A n-\bigl(\eta+\ds\log q\bigr)n\\
    2\Xib_A(d)=\Omega^{-1}q\Db\mb+q\chib^\sharp\cdot\mb-\bigl(q\omegabh+\Lbh q\bigr)\mb-\nablas_A\nb+\bigl(\eta+\ds\log q\bigr)\nb\\
    2\Theta(d)=2q^{-1}\Omega^{-1}\Dh\ih-q^{-1}\tr\chi\ih-\nablas\otimesh m+2q^{-1}\chih j+q\chibh n-\bigl(2\etab+\zeta+\ds\log q\bigr)\otimesh m\\
    2\Thetab(d)=2q\Omega^{-1}\Dbh\ih-\tr\chib\ih-\nablas\otimesh \mb+2q\chibh j+q^{-1}\chih \nb-\bigl(2\etab-\zeta-\ds\log q\bigr)\otimesh \mb\\
    2K(d)=\curls m+ q^{-1}\chih\wedge \ih+\bigl(\zeta+\ds\log q\bigr)\wedge m\\
    2\Kb(d)=\curls \mb+ q \chibh\wedge \ih-\bigl(\zeta+\ds\log q\bigr)\wedge \mb\\
    2I(d)=q^{-1}\Omega^{-1}D\mb-\bigl(q^{-1}\omegah+\Lh q^{-1}\bigr)\mb-\ds j-2\etab j-2\etab^\sharp\cdot\ih+\chib^\sharp\cdot m+\frac{2}{3}\ps\\
    2\Ib(d)=q\Omega^{-1}\Db m-\bigl(q\omegabh+\Lbh q\bigr)m-\ds j-2\eta j-2\eta^\sharp\cdot \ih+\chi^\sharp\cdot\mb+\frac{2}{3}\ps
\end{gather*}

Note also,
\begin{multline*}
  \ps=-\frac{1}{2}\bigl(q^{-1}\Omega^{-1}D\mb+q\Omega^{-1}\Db m\bigr)+\divs\hat{i}+\frac{1}{2}\ds j +\bigl(\etab+\eta\bigr)^\sharp\cdot \hat{i}+\bigl(\eta+\etab\bigr)j\\-\frac{1}{2}\bigl(q^{-1}\omegah+\Lh q^{-1}+q^{-1}\tr\chi\bigr)\mb-\frac{1}{2}\bigl(q\omegabh+\Lbh q+q\tr\chib\bigr)m
\end{multline*}

Since
\begin{gather*}
  \nabla_4\pih{N}_{34}=q^{-1}\Lh j\qquad \nabla_3\pih{N}_{34}=q\Lbh j\\
  g^{AC}\nabla_C\pih{N}_{A4}=\divs m-q^{-1}\tr\chi j-\frac{1}{2} q\tr\chib n\\
  g^{AC}\nabla_C\pih{N}_{A3}=\divs \mb-q\tr\chib j-\frac{1}{2} q^{-1}\tr\chi \nb
\end{gather*}
we obtain
\begin{multline}\label{eq:Lambdab:Lambda:three}
  8\Lambdat^3+8\Lambdabt^3=\Bigl[3\Omega^{-1}q^{-1}Dj+3\Omega^{-1}q\Db j-3\divs(m+\mb)\\
+3\bigl(q\tr\chib+q^{-1}\tr\chi\bigr)j+\frac{3}{2}q\tr\chib\,n+\frac{3}{2}q^{-1}\tr\chi\nb\Bigr]\rho+2(p_3+p_4)\rho\\
+\Bigl(\Omega^{-1}q^{-1}Dm+q^{-1}\chi^\sharp\cdot m-\bigl(q^{-1}\omegah+\Lh q^{-1}\bigr)m-\nablas_A n-\bigl(\eta+\ds\log q\bigr)n,\betab_q\Bigr)\\
-\Bigl(\Omega^{-1}q\Db\mb+q\chib^\sharp\cdot\mb-\bigl(q\omegabh+\Lbh q\bigr)\mb-\nablas_A\nb+\bigl(\eta+\ds\log q\bigr)\nb,\beta_q\Bigr)\displaybreak[0]\\
+\frac{1}{2}\bigl(2q\Omega^{-1}\Dbh\ih-\tr\chib\ih-\nablas\otimesh \mb+2q\chibh j+q^{-1}\chih \nb-\bigl(2\etab-\zeta-\ds\log q\bigr)\otimesh \mb,\alpha_q\bigr)\\
+\frac{1}{2}\bigl(2q^{-1}\Omega^{-1}\Dh\ih-q^{-1}\tr\chi\ih-\nablas\otimesh m+2q^{-1}\chih j+q\chibh n-\bigl(2\etab+\zeta+\ds\log q\bigr)\otimesh m,\alphab_q\bigr)\\
+3\Bigl(\curls m+ q^{-1}\chih\wedge \ih+\bigl(\zeta+\ds\log q\bigr)\wedge m-\curls \mb- q \chibh\wedge \ih+\bigl(\zeta+\ds\log q\bigr)\wedge \mb\Bigr)\sigma
\end{multline}

Note now that by adding \eqref{eq:Lambdab:Lambda:two} and \eqref{eq:Lambdab:Lambda:three} the term $2(p_3+p_4)\rho$ cancels.
Therefore --- let us only write the $\rho$ and $\sigma$ terms in the following formula --- we have:
\begin{multline*}
  8\bigl(\Lambdat^2+\Lambdat^3\bigr)+8\bigl(\Lambdabt^2+\Lambdabt^3\bigr)\doteq3\Bigl(j+\frac{n+\nb}{4}\Bigr)\bigl(q\tr\chib+q^{-1}\tr\chi\Bigr)\rho\\
  +\Bigl[3\Omega^{-1}q^{-1}D j+3\Omega^{-1}q\Db j+3\frac{n-\nb}{4}\bigl(q\tr\chib-q^{-1}\tr\chi\bigr)\Bigr]\rho
\end{multline*}

For the following recall also that
\begin{equation*}
  \mb+m=0\,.
\end{equation*}
In conclusion:
\begin{subequations}
\begin{equation}
  \boxed{3\cdot 8\rhot \Bigl(\Lambdabt^2+\Lambdabt^3+\Lambdat^2+\Lambdat^3\Bigr)=X_\Lambda^{2+3}+Y_\Lambda^{2+3}+Q_\Lambda^{2}+Q_{\Lambda}^{2,4}+Q_\Lambda^{3}+Q_{\Lambda}^{3,4}}
\end{equation}
where
\begin{equation}\label{eq:X:Lambda:two}
  X_{\Lambda}^{2+3}:=9\Bigl(j+\frac{n+\nb}{4}\Bigr)\bigl(q\tr\chib+q^{-1}\tr\chi\Bigr)\rho\rhot
\end{equation}
\begin{equation}\label{eq:Y:Lambda:two}
    Y_\Lambda^{2+3}:=3\rhot\bigl(q\Omega^{-1}\Dbh\ih,\alpha_q\bigr)-\frac{3}{2}\rhot q\tr\chib\bigl(\ih,\alpha_q\bigr)+3\rhot\bigl(q^{-1}\Omega^{-1}\Dh\ih,\alphab_q\bigr)-\frac{3}{2}\rhot q^{-1}\tr\chi\bigl(\ih,\alphab_q\bigr)
\end{equation}
and
\begin{multline*}
  Q_\Lambda^2=3\rhot\Bigl(2\bigl(\etab+\eta\bigr)^\sharp\cdot \hat{i}+2\bigl(\eta+\etab\bigr)j,\beta_q-\betab_q\Bigr)\\
    +3\rhot\Bigl(q^{-1}\omegah+\Lh q^{-1}+q^{-1}\tr\chi-q\omegabh-\Lbh q-q\tr\chib\Bigr)\bigl(m,\beta_q-\betab_q\bigr)
  \end{multline*}
\begin{equation*}  
Q_{\Lambda}^{2,4}=-\frac{3\rhot}{\Omega}\Bigl(q^{-1}D\mb+q\Db m-2\Omega\divs\hat{i}-\Omega\ds j,\beta_q-\betab_q\Bigr)
\end{equation*}

\begin{multline*}
    Q_{\Lambda}^{3}= 9\rhot\frac{n-\nb}{4}\bigl(q\tr\chib-q^{-1}\tr\chi\bigr)\rho\\+\frac{3\rhot}{2}\Bigl(q^{-1}\tr\chi -2q^{-1}\omegah-2\Lh q^{-1}\Bigr)\bigl(m,\betab_q\bigr)-\frac{3\rhot}{2}\Bigl(q\tr\chib-2q\omegabh-2\Lbh q\Bigr)\bigl(\mb,\beta_q\bigr)\\
+\frac{9\rhot}{2\Omega}\Bigl[D\log q\bigl(\Omega\tr\chib-2\omegab\bigr)-q^{-2}D\log q\bigl(\Omega\tr\chi-2\omega\bigr)-D\log q\Db\log q-q^{-2}D\log q D\log q\Bigr]\rho\\
+\frac{9\rhot}{2\Omega}\Bigl[q^2\Db \log q\bigl(\Omega\tr\chib-2\omegab\bigr)- \Db\log q\bigl(\Omega\tr\chi-2\omega\bigr)-q^2 \Db \log q\Db\log q-\Db \log qD\log q\Bigr]\rho\displaybreak[0]\\
+\frac{3\rhot}{\Omega}\Bigl(q^{-1}\Omega\chih^\sharp\cdot m-2 q^{-1}\Omega\nablas \log q D\log q-\Omega\bigl(\eta+\ds\log q\bigr)n,\betab_q\Bigr)\\
-\frac{3\rhot}{\Omega}\Bigl(q\Omega\chibh^\sharp\cdot\mb-2 q\Omega \nablas \log q \Db\log q+\Omega\bigl(\eta+\ds\log q\bigr)\nb,\beta_q\Bigr)\\
+\frac{3}{2}\rhot\bigl(2q\chibh j+q^{-1}\chih \nb-\bigl(2\etab-\zeta-\ds\log q\bigr)\otimesh \mb,\alpha_q\bigr)\\
+\frac{3}{2}\rhot\bigl(2q^{-1}\chih j+q\chibh n-\bigl(2\etab+\zeta+\ds\log q\bigr)\otimesh m,\alphab_q\bigr)\\
+9\rhot\Bigl( q^{-1}\chih\wedge \ih+\bigl(\zeta+\ds\log q\bigr)\wedge m- q \chibh\wedge \ih+\bigl(\zeta+\ds\log q\bigr)\wedge \mb\Bigr)\sigma
\end{multline*}
\begin{multline*}
  Q_{\Lambda}^{3,4}=\frac{9\rhot}{2\Omega} \Bigl[ D\bigl(\Omega\tr\chib-2\omegab\bigr)+q^{-2}D\bigl(\Omega\tr\chi-2\omega\bigr)- D\Db\log q+q^{-2} DD\log q\\
+q^2\Db\bigl(\Omega\tr\chib-2\omegab\bigr)+\Db\bigl(\Omega\tr\chi-2\omega\bigr)-q^2\Db\Db\log q+\Db D\log q\Bigr]\rho\\
+\frac{3\rhot}{\Omega}\Bigl(q^{-1}Dm+2 q^{-1}\Omega \nablas D\log q,\betab_q\Bigr)-\frac{3\rhot}{\Omega}\Bigl(q\Db\mb-2 q \Omega\nablas \Db\log q,\beta_q\Bigr)\\
+9\rhot\Bigl(\curls m-\curls \mb\Bigr)\sigma
-\frac{3}{2}\rhot\bigl(\nablas\otimesh \mb,\alpha_q\bigr)-\frac{3}{2}\rhot\bigl(\nablas\otimesh m,\alphab_q\bigr)
\end{multline*}

Here we used that
\begin{equation*}
  j=\frac{1}{2}q\bigl(\Omega\tr\chib-2\omegab\bigr)+\frac{1}{2}q^{-1}\bigl(\Omega\tr\chi-2\omega\bigr)-\frac{1}{2}q\Db\log q+\frac{1}{2}q^{-1}D\log q
\end{equation*}
\begin{multline*}
  q^{-1}Dj=\frac{1}{2}D\log q\bigl(\Omega\tr\chib-2\omegab\bigr)-\frac{1}{2}q^{-2}D\log q\bigl(\Omega\tr\chi-2\omega\bigr)\\-\frac{1}{2}D\log q\Db\log q-\frac{1}{2}q^{-2}D\log q D\log q\\
+\frac{1}{2} D\bigl(\Omega\tr\chib-2\omegab\bigr)+\frac{1}{2}q^{-2}D\bigl(\Omega\tr\chi-2\omega\bigr)-\frac{1}{2} D\Db\log q+\frac{1}{2}q^{-2} DD\log q
\end{multline*}
and similarly for $q\Db j$.
Also,
\begin{gather*}
  n=-2 q^{-1}D\log q\\
  \nablas n=2 q^{-1}\nablas \log q D\log q-2 q^{-1}\nablas D\log q
\end{gather*}
and similarly for $\nablas \nb$.

\end{subequations}

Analogously we conclude:
\begin{subequations}
\begin{equation}
  \boxed{3\cdot 8\sigmat \Bigl(\Kt^2+\Kt^3-\Kbt^2-\Kbt^3\Bigr)=X_K^{2+3}+Y_K^{2+3}+Q_K^{2}+Q_K^{2,4}+Q_K^{3}+Q_K^{3,4}}
\end{equation}
where
\begin{equation}\label{eq:X:K:two}
  X_K^{2+3}:=9\Bigl(j+\frac{n+\nb}{4}\Bigr)\bigl(q\tr\chib+q^{-1}\tr\chi\Bigr)\sigma\sigmat
\end{equation}
\begin{equation}
  \label{eq:Y:K:two}
  Y_K^{2+3} := -3\sigmat q\Omega^{-1}\Dbh\ih\wedge\alpha+\frac{3}{2}\sigmat q\tr\chib\ih\wedge\alpha+3\sigma q^{-1}\Omega^{-1}\Dh\ih\wedge\alphab-\frac{3}{2}\sigmat q^{-1}\tr\chi\ih\wedge\alphab
\end{equation}
The terms $Q_K^{2}$, $Q_K^{2,4}$, $Q_K^{3}$, $Q_K^{3,4}$, are of a similar structure as $Q_\Lambda^{2}$, $Q_\Lambda^{2,4}$, $Q_\Lambda^{3}$, $Q_\Lambda^{3,4}$, respectively.
\end{subequations}

With regard of the $\Thetab$, and $\Theta$ components, we compute:
\begin{gather*}
      -p_3+3\Lambdab(d)=\nabla_4\pih{N}_{33}-\frac{1}{2}\nabla_3\pih{N}_{34}-\frac{1}{2}\nabla^A\pih{N}_{A3}\\
      -p_4+3\Lambda(d)=\nabla_3\pih{N}_{44}-\frac{1}{2}\nabla_4\pih{N}_{34}-\frac{1}{2}\nabla^A\pih{N}_{A4}\\
      \nabla_4\pih{N}_{33}=q^{-1}\Omega^{-1} D \nb-4\etab^\sharp\cdot \mb+2\bigl(q^{-1}\omegah+\Lh q^{-1}\bigr)\nb\\
      \nabla_3\pih{N}_{44}=q\Omega^{-1}\Db n-4\eta^\sharp\cdot m+2\bigl(q\omegabh+\Lbh q\bigr)n
\end{gather*}
and 
\begin{multline*}
    \ps-6I(d)=-\ps-3q^{-1}\Omega^{-1}D\mb+3\bigl(q^{-1}\omegah+\Lh q^{-1}\bigr)\mb+3\ds j+6\etab j+6\etab^\sharp\cdot\ih-3q\chib^\sharp\cdot m\\
  =\frac{1}{2}\bigl(-5q^{-1}\Omega^{-1}D\mb+q\Omega^{-1}\Db m\bigr)-\divs\hat{i}+\frac{5}{2}\ds j -\bigl(-5\etab+\eta\bigr)^\sharp\cdot \hat{i}-\bigl(\eta-5\etab\bigr)j\\+\frac{1}{2}\bigl(7q^{-1}\omegah+7\Lh q^{-1}+q^{-1}\tr\chi\bigr)\mb+\frac{1}{2}\bigl(q\omegabh+\Lbh q-2q\tr\chib\bigr)m-3\chibh^\sharp\cdot m
\end{multline*}
\begin{multline*}
  -\ps+6\Ib(d)=\ps+3q\Omega^{-1}\Db m-3\bigl(q\omegabh+\Lbh q\bigr)m-3\ds j-6\eta j-6\eta^\sharp\cdot \ih+3q^{-1}\chi^\sharp\cdot\mb\\
  = -\frac{1}{2}\bigl(q^{-1}\Omega^{-1}D\mb-5q\Omega^{-1}\Db m\bigr)+\divs\hat{i}-\frac{5}{2}\ds j +\bigl(\etab-5\eta\bigr)^\sharp\cdot \hat{i}+\bigl(-5\eta+\etab\bigr)j\\-\frac{1}{2}\bigl(q^{-1}\omegah+\Lh q^{-1}-4q^{-1}\tr\chi\bigr)\mb-\frac{1}{2}\bigl(7q\omegabh+7\Lbh q+q\tr\chib\bigr)m+3\chih^\sharp\cdot\mb
\end{multline*}
Also,
\begin{gather*}
  \nabla_4\pih{N}_{34}=q^{-1}\Lh j\qquad \nabla_3\pih{N}_{34}=q\Lbh j\\
  g^{AC}\nabla_C\pih{N}_{A4}=\divs m-q^{-1}\tr\chi j-\frac{1}{2} q\tr\chib n\\
  g^{AC}\nabla_C\pih{N}_{A3}=\divs \mb-q\tr\chib j-\frac{1}{2} q^{-1}\tr\chi \nb
\end{gather*}

Therefore
\begin{subequations}
  \begin{equation}
\boxed{    4(\alphat_q,\Thetat^2+\Thetat^3) = X_\Theta^{2+3} +Y_\Theta^{2+3}+ Q_\Theta^{2+3} } 
  \end{equation}
where
  \begin{equation}\label{eq:X:Theta:two}
    X_\Theta^{2+3}:=\frac{1}{2}q\tr\chib j (\alphat_q,\alpha_q)+\frac{1}{4}\Bigl(q^{-1}\tr\chi+8 q^{-1}\omegah+8\Lh q^{-1}\Bigr)\nb (\alphat_q,\alpha_q)
  \end{equation}

  \begin{multline}\label{eq:Y:Theta:two}
    Y_\Theta^{2+3} := -\frac{3}{2}q^{-1}\tr\chi\Bigl(\rho\alphat_q-\sigma\ld\alphat,\ih\Bigr)+3\Bigl(\rho\alphat_q-\sigma\ld\alphat_q,q^{-1}\Omega^{-1}\Dh\ih\Bigr)\\
+\frac{1}{2}\bigl(7q^{-1}\omegah+7\Lh q^{-1}+q^{-1}\tr\chi\bigr)(\alphat,\mb\otimesh\beta)+\frac{1}{2}\bigl(q\omegabh+\Lbh q-2q\tr\chib\bigr) (\alphat,m\otimesh\beta)
  \end{multline}
and
  \begin{multline}
    Q_\Theta^{2+3} := \Bigl[q^{-1}\Omega^{-1} D \nb-4\etab^\sharp\cdot \mb-\frac{1}{2}q\Omega^{-1}\Db j-\frac{1}{2}\divs \mb\Bigr](\alphat_q,\alpha_q)\\
    -\frac{3}{2}\Bigl[\curls \mb+ q \chibh\wedge \ih-\bigl(\zeta+\ds\log q\bigr)\wedge \mb\Bigr](\alphat_q,\ld\alpha_q)
    +\Bigl(\alphat_q,\bigl(\ps-6I(d)\bigr)'\otimesh\beta_q\Bigr)\\
+\frac{3}{2}\Bigl(\rho\alphat_q-\sigma\ld\alphat_q,-\nablas\otimesh m-\bigl(2\etab+\zeta+\ds\log q\bigr)\otimesh m +2q^{-1}\chih j+q\chibh n\Bigr)
  \end{multline}
\end{subequations}
(Here $(\ps-6I)'$ denotes the above expression for $\ps-6I$ without the border line terms included in $Y_\Theta^{2+3}$ above.)

Similarly,

\begin{subequations}
  \begin{equation}
\boxed{        4(\alphabt_q,\Thetabt^2+\Thetabt^3) = X_{\Thetab}^{2+3} +Y_{\Thetab}^{2+3} + Q_{\Thetab}^{2+3} }
  \end{equation}

  \begin{equation}\label{eq:X:Thetab:two}
    X_{\Thetab}^{2+3} := \frac{1}{2} q^{-1}\tr\chi j(\alphabt_q,\alphab_q)+\frac{1}{4}\Bigl( q\tr\chib+8q\omegabh+8\Lbh q\Bigr) n(\alphabt_q,\alphab_q)
  \end{equation}

  \begin{multline}\label{eq:Y:Thetab:two}
    Y_{\Thetab}^{2+3} := -\frac{3}{2}\tr\chib\Bigl(\rho\alphabt+\sigma\ld\alphabt,\ih\Bigr)+3\Bigl(\rho\alphabt+\sigma\ld\alphabt,q\Omega^{-1}\Dbh\ih\Bigr)\\-\frac{1}{2}\bigl(q^{-1}\omegah+\Lh q^{-1}-4q^{-1}\tr\chi\bigr)\Bigl(\alphabt,\mb\otimesh\betab\Bigr) -\frac{1}{2}\bigl(7q\omegabh+7\Lbh q+q\tr\chib\bigr)\Bigl(\alphabt,m\otimesh\betab\Bigr)
  \end{multline}

  \begin{multline}
    Q_{\Thetab}^{2+3} := \Bigl[q\Omega^{-1}\Db n-4\eta^\sharp\cdot m-\frac{1}{2}q^{-1}\Lh j
-\frac{1}{2}  \divs m \Bigr](\alphabt_q,\alphab_q)\\
-\frac{3}{2}\Bigl[\curls m+ q^{-1}\chih\wedge \ih+\bigl(\zeta+\ds\log q\bigr)\wedge m\Bigr](\alphabt_q,\ld\alphab_q)+\Bigl(\alphabt,\bigl(6\Ib(d)-\ps\bigr)'\otimesh\betab\Bigr)\\
+\frac{3}{2}\Bigl(\rho\alphabt+\sigma\ld\alphabt,-\nablas\otimesh \mb+2q\chibh j+q^{-1}\chih \nb-\bigl(2\etab-\zeta-\ds\log q\bigr)\otimesh \mb\Bigr)
  \end{multline}

\end{subequations}

\end{proof}

\subsubsection{Proof of Prop.~\ref{prop:div:N}}
\label{sec:J:cancel}

The statement of Proposition~\ref{prop:div:N} follows from Lemma~\ref{lemma:div:N:one} and Lemma~\ref{lemma:div:N:two:three}.
Indeed by adding \eqref{eq:div:N:one} and \eqref{eq:div:N:two:three} we obtain
\begin{multline}
 \phi \bigl(\divergence Q(\MLie{N}W)\bigr)(M_q,M_q,M_q) \geq  \frac{\phi}{(2\Omega)^3}\bigl[X+Y\bigr]\\
  -\frac{C}{r}\frac{1}{\Omega}Q[\MLie{N}W](n,M_q,M_q,M_q)  -\frac{C}{r}\frac{1}{\Omega}Q[W](n,M_q,M_q,M_q) -\frac{C}{r}\frac{1}{\Omega}\Bigl[\Omega^2\Ps^q\Bigr]
\end{multline}
and it remains to control
\begin{equation}
  X:=X^1+X^{2+3}\,,\qquad Y:=Y^1+Y^{2+3}\,,
\end{equation}
where $X^1$, and $Y^1$ are defined in Lemma~\ref{lemma:div:N:one}, see \eqref{eq:X:Y:1}, and 
where $X^{2+3}$, $Y^{2+3}$ are defined in Lemma~\ref{lemma:div:N:two:three}, see \eqref{eq:X:Y:2:3}.
\emph{We will show that the terms in $X$ essentially cancel, while the terms in $Y$ are controlled by our assumptions.}

Consider first
\begin{subequations}
\begin{equation}
  X_\Lambda:=X_\Lambda^1+X_{\Lambda}^{2+3}\qquad
  X_K:=X_K^1+X_K^{2+3}
\end{equation}
We have by \eqref{eq:X:Lambda:one} and \eqref{eq:X:Lambda:two} as well as \eqref{eq:X:K:one} and \eqref{eq:X:K:two} that
\begin{equation}
\begin{split}
  X_\Lambda&=\frac{9}{2}\Bigl(j+\frac{n+\nb}{4}\Bigr)\Bigl[q\tr\chib+q^{-1}\tr\chi-2q\omegabh-2q^{-1}\omegah-\Lbh q-\Lh q^{-1}\bigr]\rho\rhot\\
  &=\frac{9}{\Omega}\Bigl(j+\frac{n+\nb}{4}\Bigr)j\rho\rhot\\
%\\    j&=\frac{\Omega}{2}\Bigl(q\tr\chib+q^{-1}\tr\chi-2q\omegabh-2q^{-1}\omegah-\Lbh q-\Lh q^{-1}\Bigr)
    X_K&=\frac{9}{\Omega}\Bigl(j+\frac{n+\nb}{4}\Bigr)j\sigmat\sigma
  \end{split}
\end{equation}
Thus 
\begin{equation}
  \phi \lvert X_\Lambda \rvert \leq  \frac{9}{\Omega}\frac{3C}{2}\phi\bigl(q\tr\chib+q^{-1}\tr\chi\bigr)\lvert j\rvert \rhot\rho \leq \frac{27}{\Omega}\frac{(2C)^3}{r}\rhot\rho 
\end{equation}
\end{subequations}
Moreover for 
\begin{subequations}
  \begin{equation}
    X_\Theta:=X_\Theta^1+X_\Theta^{2+3}\qquad X_{\Thetab}=X_{\Thetab}^1+X_{\Thetab}^{2+3}
  \end{equation}
we have by \eqref{eq:X:Thetab:one} and \eqref{eq:X:Thetab:two} as well as \eqref{eq:X:Theta:one} and \eqref{eq:X:Theta:two} that
  \begin{equation}
    \begin{split}
          X_\Theta&=\frac{\nb}{8}\bigl(2q\omegabh+2q^{-1}\omegah-q\tr\chib-q^{-1}\tr\chi+\Lbh q+9\Lh q^{-1}\bigr)(\alphat_q,\alpha_q)\\
          &=-\frac{\nb j}{4\Omega}(\alphat_q,\alpha_q)+\nb\Lh q^{-1}(\alphat_q,\alpha_q)\\
          X_{\Thetab}&=  \frac{n}{8} \bigl(2q^{-1}\omegah+2q\omegabh-q\tr\chib-q^{-1}\tr\chi+\Lh q^{-1}+9\Lbh q\bigr)(\alphabt_q,\alphab_q)\\
          &= -\frac{n j}{4\Omega} (\alphabt_q,\alphab_q) +n \Lbh q (\alphabt_q,\alphab_q)
%      j&=\frac{\Omega}{2}\Bigl(q\tr\chib+q^{-1}\tr\chi-2q\omegabh-2q^{-1}\omegah-\Lbh q-\Lh q^{-1}\Bigr)
    \end{split}
  \end{equation}
and thus
\begin{equation}
  \phi \lvert X_{\Theta} \rvert \leq \frac{1}{\Omega}\frac{C}{r}\lvert \alphat_q \rvert \lvert \alpha_q \rvert \qquad
  \phi \lvert X_{\Thetab} \rvert \leq \frac{1}{\Omega}\frac{C}{r}\lvert \alphabt_q \rvert \lvert \alphab_q \rvert
\end{equation}
\end{subequations}
Similarly for $X_{\Xib}$ and $X_{\Xi}$.

Let us now turn to the ``borderline error terms'' $Y$.\footnote{The appearance of such \emph{borderline error integrals} may be familiar already from \cite{ch:blue}. Indeed terms of precisely this form play a prominent role in Chapter~14.2, see e.g.~\Ceq{14.32, 14.34, 14.37}, which also require more precise estimates. }
We consider first
\begin{subequations}
\begin{equation}
Y_\Lambda:=Y_\Lambda^1+Y_\Lambda^{2+3}  
\end{equation}
which in view of \eqref{eq:Y:Lambda:one} and \eqref{eq:Y:Lambda:two} is given by
\begin{multline}
  Y_\Lambda=   -3\rhot q\tr\chib  \bigl(\hat{i},\alpha_q\bigr)- 3\rhot q^{-1}\tr\chi \bigl(\hat{i},\alphab_q\bigr)
    -9 q^{-1}\tr\chi \rhot \bigl(m, \beta_q\bigr)+9q\tr\chib \rhot \bigl(\mb,\betab_q\bigr)\\
    +3\rhot\bigl(q\Omega^{-1}\Dbh\ih,\alpha_q\bigr)+3\rhot\bigl(q^{-1}\Omega^{-1}\Dh\ih,\alphab_q\bigr)
\end{multline}
By \eqref{eq:BA:I:iv:e}, \eqref{eq:BA:I:vi:e} and \eqref{eq:BA:II:i} we estimate
\begin{gather*}
  \phi |\rhot q\tr\chib  \bigl(\hat{i},\alpha_q\bigr)| \leq |\rhot|\frac{4C_0}{r\Omega}\bigl(q\tr\chib+q^{-1}\tr\chi\bigr)|\alpha_q| \leq \frac{C_0}{\Omega}\frac{C}{r}\bigl[|\rhot|^2+|\alpha_q|^2\bigr]\\
\phi |q\tr\chib \rhot \bigl(\mb,\betab_q\bigr)|\leq  \frac{4}{r}\rhot\bigl(\Omega|\eta|+\Omega|\etab|+\Omega|\ds\log q|\bigr)|\betab|\leq \frac{C_0}{\Omega}\frac{C}{r}\bigl[\rhot^2+|\betab|^2\bigr]\\
\phi|\rhot\bigl(q\Omega^{-1}\Dbh\ih,\alpha_q\bigr)|\leq \frac{\phi}{\Omega}|\rhot| q|\Dbh\ih||\alpha_q|\leq\frac{|\rhot|}{\Omega}\phi \Bigl( q\tr\chib q\tr\chib+q\tr\chib q^{-1}\tr\chi\bigr)|\alpha_q|\leq\frac{C}{\Omega r}|\rhot||\alpha_q|
\end{gather*}
and thus 
\begin{equation}
  \frac{\phi}{(2\Omega)^3} |Y_\Lambda| \geq -\frac{C_0}{\Omega}\frac{C}{r}\bigl(Q[W](n,M_q,M_q,M_q)+Q[\MLie{N}W](n,M_q,M_q,M_q)\bigr)
\end{equation}
\end{subequations}

Next consider
\begin{subequations}
  \begin{equation}
    Y_\Theta:=Y_\Theta^1+Y_\Theta^{2+3}\qquad Y_{\Thetab}:=Y_{\Thetab}^1+Y_{\Thetab}^{2+3}\,.
  \end{equation}
By \eqref{eq:Y:Theta:one} and \eqref{eq:Y:Theta:two} as well as \eqref{eq:Y:Thetab:one} and \eqref{eq:Y:Thetab:two} we have
  \begin{multline}
    Y_{\Theta} = -3q^{-1}\tr\chi \bigl(\alphat,\hat{i}\bigr)\rho+3q^{-1}\tr\chi\bigl(\alphat,\ld\hat{i}\bigr)\sigma+\frac{3}{4}\bigl(q^{-1}\tr\chi+q\tr\chib\bigr)\Bigl(\alphat,\mb\otimesh\beta_q\Bigr)\\ 
    +3\Bigl(\rho\alphat_q-\sigma\ld\alphat_q,q^{-1}\Omega^{-1}\Dh\ih\Bigr)+\frac{1}{4}\bigl(q\tr\chib-2q\omegabh-2\Lbh q\bigr)\Bigl(\alphat,\mb\otimesh\beta_q\Bigr)
+\frac{7}{2}\bigl(q^{-1}\omegah+\Lh q^{-1}\bigr)(\alphat,\mb\otimesh\beta)
  \end{multline}
  \begin{multline}
    Y_{\Thetab} = -3q\tr\chib \bigl(\alphabt,\hat{i}\bigr)\rho+3 q\tr\chib\bigl(\alphabt,\ld\hat{i}\bigr)\sigma-\frac{3}{4}\bigl(q\tr\chib+q^{-1}\tr\chi\bigr)\Bigl(\alphabt,m\otimesh\betab_q\Bigr)\\
    +3\Bigl(\rho\alphabt+\sigma\ld\alphabt,q\Omega^{-1}\Dbh\ih\Bigr)-\frac{1}{4}\bigl(2q^{-1}\omegah+2\Lh q^{-1}-q^{-1}\tr\chi\bigr)\Bigl(\alphabt,\mb\otimesh\betab\Bigr) -\frac{7}{2}\bigl(q\omegabh+\Lbh q\bigr)\Bigl(\alphabt,m\otimesh\betab\Bigr)
  \end{multline}
and thus
\begin{gather}
  \phi |Y_\Theta | \leq \frac{C}{r} \frac{C_0}{\Omega}|\alphat|\bigl[|\rho|+|\sigma|+|\beta_q|\bigr]\\
  \phi |Y_{\Thetab} | \leq \frac{C}{r}\frac{C_0}{\Omega}|\alphabt|\bigl[|\rho|+|\sigma|+|\betab_q|\bigr]
\end{gather}

Similarly for $Y_{\Xib}$ and $Y_{\Xi}$. \qed

\end{subequations}

\subsection{Integral inequality}
\label{sec:gronwall}

We discuss briefly the implications of the integral inequality obtained with the help of Prop.~\ref{prop:div:N}.
The remaining error terms --- namely the terms involving angular derivatives in the second line of \eqref{eq:J:1:error} --- are the subject of Section~\ref{sec:electromagnetic}, where it will be shown that after integration on $\Sigma_r$ it is controlled by the \emph{same} error terms already introduced in \eqref{eq:div:N}, and thus does not alter the following discussion.

\subsubsection{Gronwall argument}
\label{sec:gronwall:argument}

\begin{lemma}
  Suppose $f:\mathbb{R}^+\to\mathbb{R}^+$ is a positive function that satisfies an integral inequality of the form
  \begin{equation}\label{eq:redshift:inequality:model}
  f(r_2)+\int_{r_1}^{r_2}\Bigl[\frac{\kappa(r)}{r}f(r) -\frac{C}{r}h(r)\Bigr]\ud r\leq f(r_1) \qquad \text{: for all } r_2>r_1\geq r_0
\end{equation}
where $C>0$ is a constant, $h:\mathbb{R}^+\to\mathbb{R}$ an arbitrary function, and $\kappa:\mathbb{R}^+\to\mathbb{R}^+$ is a positive function with the property that for some positive constants $0<\kappa_0<\kappa_1<\infty$:
\begin{equation}
  \kappa_0\leq \kappa(r)\leq \kappa_1
\end{equation}

Assume moreover that for some $K_1<\infty$, and $H_1<\infty$:
\begin{equation}\label{eq:H:K}
  K(r):=\int_{r_0}^r\frac{\kappa_1-\kappa(r)}{r}\ud r \leq K_1\,, \qquad H(r_1,r_0):=\int_{r_0}^{r_1}\lvert h(r)\rvert r^{\kappa_1-1}\ud r \leq H_1
\end{equation}

Then for some constant $C$, depending on $\kappa_1/\kappa_0$, $K_1$, $H_1$ and $f(r_0)$, we have that for all $r\geq 2r_0$:
\begin{equation}
  f(r)\leq \frac{C}{r^{\kappa_1}}\,.
\end{equation}
\end{lemma}

\begin{proof}
With
\begin{equation}
  F(r_2,r_1):= \int_{r_1}^{r_2}\Bigl[\frac{\kappa(r)}{r}f(r) -\frac{C}{r}h(r)\Bigr]\ud r
\end{equation}
we  obtain from the integral inequality that
\begin{multline*}
  \frac{\ud }{\ud r}\Bigl[F(r_2,r)r^{\kappa_1}\Bigr]_{r=r_1}=\frac{\kappa(r_1)}{r_1}\Bigl[-f(r_1)+\frac{\kappa_1}{\kappa(r_1)}F(r_2,r_1)+\frac{C}{\kappa(r_1)}h(r_1)\Bigr]r_1^{\kappa_1}\\
  \leq -\kappa(r_1) f(r_2) r_1^{\kappa_1-1}+\frac{\kappa_1-\kappa(r_1)}{r_1}F(r_2,r_1)r_1^{\kappa_1}+C h(r_1)r_1^{\kappa_1-1}
\end{multline*}
or
\begin{equation*}
  \frac{\ud }{\ud r}\Bigl[F(r_2,r)r^{\kappa_1}e^{-K(r)}\Bigr]_{r=r_1} \leq \Bigl[-\kappa_0 f(r_2) r_1^{\kappa_1-1}+C h(r_1)r_1^{\kappa_1-1}\Bigr]e^{-K(r_1)}
\end{equation*}
where $K(r)$ is defined by \eqref{eq:H:K} 
% \begin{equation*}
%   K(r)=\int_{r_0}^r\frac{\kappa_1-\kappa(r)}{r}\ud r\geq 0
% \end{equation*}
and we assume $K\leq K_1$ for some constant $0<K_1<\infty$.

Then we may integrate in $r_1$ on $[r_0,r_2]$
\begin{equation*}
  -F(r_2,r_0)r_0^{\kappa_1} e^{-K(r_0)}\leq -f(r_2)\frac{\kappa_0e^{-K_1}}{\kappa_1}\bigl(r_2^{\kappa_1}-r_0^{\kappa_1}\bigr)+C\int_{r_0}^{r_2}\lvert h(r)\rvert r^{\kappa_1-1}\ud r\,.
\end{equation*}
This implies the following upper bound on $f$:
\begin{equation*}
  \bigl(r^{\kappa_1}-r_0^{\kappa_1}\bigr)f(r)\leq \frac{\kappa_1 r_0^{\kappa_1} e^{K_1}}{\kappa_0}f(r_0)+\frac{C\kappa_1 e^{K_1}}{\kappa_0}H(r,r_0)
\end{equation*}
where $H$ is defined in \eqref{eq:H:K} and since $H\leq H_1<\infty$ is finite,  the estimate for $f$ follows.

\end{proof}

% \subsubsection{Applications}
% \label{sec:gronwall:app}

\begin{remark}[Application in Proof of Proposition~\ref{prop:global:redshift:rough}]
  In   Section~\ref{sec:global:redshift} we have seen \eqref{eq:redshift:inequality:model} to be true with
\begin{equation}
  f(r):=\int_{\Sigma_r}\ld P^q[W]\qquad \kappa(r) = 6-\frac{C}{r}\leq \kappa_1=6 \qquad h=0
\end{equation}
which  implies the statement of Proposition~\ref{prop:global:redshift:rough},
\begin{equation}\label{eq:gronwall:f}
  r^{6}f(r)\leq C r_0^{6} f(r_0)\,,%\qquad r\geq  r_0 \,,
\end{equation}
because
\begin{equation}
  K(r)\leq \int_{r_0}^\infty \frac{C}{r^2}\ud r<\infty\,.
\end{equation}
\end{remark}

\begin{remark}[Application in Proof of Prop.~\ref{prop:conclusion:energy}]
  In Section~\ref{sec:positive}  we  show that \eqref{eq:redshift:inequality:model} also holds for
\begin{subequations}
\begin{gather}
  f(r):=\int_{\Sigma_r}\ld P^q[\MLie{N}W]\\\kappa(r) = 6\Bigl(1-\frac{C_0^2}{r}\Bigr)^3-\frac{CC_0^2}{r}\leq \kappa_1=6 \qquad h(r)=\frac{C_0}{r}\int_{\Sigma_r}\ld P^q[W]
\end{gather}
\end{subequations}
Note that with these choices
\begin{equation}
  K(r)=\int_{r_0}^r\frac{\kappa_1-\kappa(r)}{r}\ud r\leq \int_{r_0}^{\infty}\frac{C}{r}\frac{C_0^2}{r}\ud r<\infty
\end{equation}
and using \eqref{eq:gronwall:f},
\begin{equation}
  H(r,r_0)\leq C_0 \int_{r_0}^{\infty} \int_{\Sigma_r}\ld P^q[W] r^4 \ud r <\infty\,.
\end{equation}
Therefore $r^{6}f(r)<\infty$ which is statement of Prop.~\ref{prop:conclusion:energy} below.
\end{remark}

\subsubsection{Conclusions}

We conclude this Section with the main estimate for the energy flux associated to the solution $\MLie{N}W$ of \eqref{eq:Bianchi:MLie:W}.
The proof will also make use of the main estimate of Section~\ref{sec:electromagnetic}.

\begin{proposition}\label{prop:conclusion:energy}
  Assume (\textbf{BA:I}), (\textbf{BA:II}) and (\textbf{BA:III}.i,ii) hold for some $C_0>0$.
  Then there exists a constant $C>0$ so that for all solutions $W$ to \eqref{eq:W:Bianchi}, we have
\begin{equation}
  r^{6}\int_{\Sigma_r}\ld P^q[\MLie{N}W]  + r^{6}\int_{\Sigma_r}\ld P^q[W]  \leq C \int_{\Sigma_{r_0}}\ld P^q[\MLie{N}W]  + C \int_{\Sigma_{r_0}}\ld P^q[W] 
\end{equation}
for all $r\geq r_0$.
\end{proposition}

\begin{proof}
We apply the energy identity of Prop.~\ref{prop:energy:identity} to the current $P^q[\MLie{N}W]$ on a domain \eqref{eq:D:uvr}.
The ``bulk term'' contains on one hand
\begin{equation*}
  \int_{\mathcal{D}}K^q[\MLie{N}W]\dm{g}\geq 6\int\frac{\ud r}{r}\int_{\Sigma_r}\bigl(1-C_0\Omega^{-1}\bigr)^3\ld P^q[\MLie{N}W]
\end{equation*}
where we applied Lemma~\ref{lemma:K:plus:q} --- which is independent of the ``Weyl field'' and only an algebraic property of the compatible current $P^q$ --- and on the other hand the ``divergence term'' 
\begin{equation*}
  \int_\mathcal{D}\bigl(\divergence Q[\MLie{N}W]\bigr)(M_q,M_q,M_q)\dm{g}=\int\ud r\int_{\Sigma_r}\phi\bigl( \divergence Q[\MLie{N}W] \bigr)(M_q,M_q,M_q)\dm{\gb{r}}
\end{equation*}
and by our Prop.~\ref{prop:div:N},
\begin{multline*}
\int_{\Sigma_r} \phi \bigl(\divergence Q[\MLie{N}W]\bigr)(M_q,M_q,M_q) \dm{\gb{r}}\geq \\
  \geq-\frac{C}{r}\int_{\Sigma_r}\frac{C_0}{\Omega}\Bigl[\ld P^q[\MLie{N}W]+\ld P^q[W]\Bigr] -\frac{C}{r}\int_{\Sigma_r}\frac{C_0}{\Omega}\Omega^2\Ps^q\dm{\gb{r}}\,.
\end{multline*}
Let us now invoke \eqref{eq:assumption:intro:Omega}, namely the assumption that $\Omega$ and $r$ are comparable.
Then by Corollary~\ref{cor:elliptic} below --- cf.~\eqref{eq:Ps} for the definition of $\Ps$, and Section~\ref{sec:electromagnetic} for the defintions of $E$, $H$ --- 
\begin{multline*}
  \int_{\Sigma_r}\frac{1}{\Omega}\Omega^2\Ps^q\dm{\gb{r}}\leq C_0\frac{1}{r}\int_{\Sigma_r}\frac{\Omega^2}{(2\Omega)^3}\Bigl[\lvert \nablas \alphab_q \rvert^2 + \lvert \nablas \betab_q \rvert^2 + \lvert \nablas\rho \rvert^2 + \lvert \nablas \sigma \rvert^2 + \lvert \nablas \beta_q\rvert^2 +\lvert \nablas\alpha_q\rvert^2\Bigr]\\
\leq C_0\frac{1}{r}\frac{1}{(2r)^3}\int_{\Sigma_r}|\nablab (rE) |^2+|\nablab (rH)|^2\dm{\gb{r}}\displaybreak[0]\\
\leq \frac{C_0}{r} \int_{\Sigma_r} \frac{1}{(2\Omega)^3}\Bigl[|\alphabt|^2+|\betabt|^2+\rhot^2+\sigmat^2+|\betat|^2+|\alphat|^2\Bigr]\dm{\gb{r}}\\+\frac{C_0}{r} \int_{\Sigma_r} \frac{1}{(2\Omega)^3}\Bigl[|\alphab|^2+|\betab|^2+\rho^2+\sigma^2+|\beta|^2+|\alpha|^2\Bigr]\dm{\gb{r}}\\
\leq \frac{C_0}{r}\int_{\Sigma_r}\ld P^q[\MLie{N}W]+\ld P^q[W]
\end{multline*}
Therefore
\begin{equation*}
\int_{\Sigma_r} \phi \bigl(\divergence Q[\MLie{N}W]\bigr)(M_q,M_q,M_q) \dm{\gb{r}}\geq 
  -\frac{C C_0^2}{r^2}\int_{\Sigma_r}\ld P^q[\MLie{N}W]+\ld P^q[W]
\end{equation*}
The energy
\begin{equation*}
  g(r) := \int_{\Sigma_r}\ld P^q[\MLie{N}W]
\end{equation*}
thus satisfies the integral inequality 
\begin{equation*}
  g(r_2)+\int_{r_1}^{r_2}\Bigl[\frac{\kappa(r)}{r} g(r)-\frac{C}{r^2}f(r)\Bigr] \ud r \leq g(r_1)
\end{equation*}
where
\begin{gather*}
  \kappa(r):= 6\Bigl(1-\frac{C_0^2}{r}\Bigr)^3-\frac{CC_0^2}{r}\qquad \kappa_1:=6\\
  f(r) := \int_{\Sigma_r}\ld P^q[W]\,.
\end{gather*}
Since also by Prop.~\ref{prop:global:redshift:rough}, $r^{6}f(r)\leq C r_0^{6}f(r_0)$, the argument of Sections~\ref{sec:gronwall:argument} applies, because here
\begin{gather*}
  H(r,r_0)=\int_{r_0}^r\frac{1}{r} f(r) r^5 \ud r \leq r_0^5 f(r_0)
\end{gather*}
and the statement of the Proposition follows.
\end{proof}

%%% Local Variables:
%%% mode: latex
%%% TeX-master: "weyl"
%%% End:

% That is permanently commented out
%\input{higher-order-q}

\section{Electromagnetic decomposition}
\label{sec:electromagnetic}

The electromagnetic formalism is a decomposition of the Weyl curvature --- and more generally of any ``Weyl field'' --- into ``electric'' and ``magnetic'' parts relative to a space-like hypersurface $\Sigma$.\footnote{One may think of the resulting fields as electric and magnetic parts in the frame of reference of an observer with 4-velocity $n$, the normal to $\Sigma$.} The decomposition with respect to $\Sigma_r$ --- which we use  in addition to the  null decomposition of $W$ introduced in Section~\ref{sec:null} --- occurs naturally in the redshift estimate of Sections~\ref{sec:global:redshift},~\ref{sec:commuted:redshift}, and is also necessary to control the remaining error that we have seen in Section~\ref{sec:positive}. The electromagnetic decomposition has been used in \cite{ch:kl}, which can serve as a reference for some of the results that we shall discuss.

In Section~\ref{sec:electro:bianchi} we discuss the from of the Bianchi equations relative to the electromagnetic decomposition, in Section~\ref{sec:electromagnetic:null} its relation to the null decomposition, and in Section~\ref{sec:electro:elliptic} an elliptic estimate for the resulting system on $\Sigma_r$.

\subsection{Bianchi equations}
\label{sec:electro:bianchi}

We consider a domain foliated by the spacelike level sets of the area radius $r$:
\begin{equation}
  \mathcal{D}=\bigcup \Sigma_r
\end{equation}
where each leaf $\Sigma_r$ is diffeomorphic to a cylinder $\mathbb{R}\times\mathbb{S}^2$.
We denote by $n$ the time-like unit normal to $\Sigma_r$, and by $\phi$ the associated lapse function of the foliation, cf.~Section~\ref{sec:areal}.

The \emph{electromagnetic decomposition} of a Weyl field $W$ consists of a pair of symmetric trace-less $\Sigma_r$-tangent $2$-tensors
\begin{subequations}\label{def:em}
\begin{gather}
E[W]:=ii_n[W]  \qquad H[W]:=ii_n[\ld W]\\
E[W](X,Y)=W(n,X,n,Y)\quad H[W](X,Y)=\ld W(n,X,n,Y)\,;
\end{gather}
\end{subequations}
we refer the reader to Section~4 of \cite{ch:kl:lin} for precise definitions ---  see in particular (4.4) therein --- and for a more detailed discussion of this decomposition. In particular, we recall that $E[W]$ and $H[W]$ \emph{together completely determine} the Weyl field $W$ as can be seen from the formula (4.3) in \cite{ch:kl:lin}. Furthermore it follows directly from (4.22) in \cite{ch:kl:lin} that with $E[W]$ and $H[W]$ as defined in \eqref{def:em}, and $n$ given by \eqref{eq:n:q}, we have
\begin{equation}
  |E[W]|^2+|H[W]|^2=\frac{1}{8}|\alphab_q|^2+|\betab_q|^2+\frac{3}{2}\rho^2+\frac{3}{2}\sigma^2+|\beta_q|^2+\frac{1}{8}|\alpha_q|^2\,.
\end{equation}

Following Chapter~7.2 in \cite{ch:kl} we will derive that $E$, and $H$, satisfy --- \emph{in the frame discussed in Section~\ref{sec:areal}} --- the following equations:
\begin{subequations}\label{eq:maxwell}
\begin{gather}
    \divergence E=H\wedge k \label{eq:div:E}\\
  \widehat{\mathcal{L}_nH}+\curl E=\nablab\log\phi\wedge E-\frac{1}{2}k\times H\label{eq:curl:E}\\
    \divergence H= -E\wedge k\label{eq:div:H}\\
    -\widehat{\mathcal{L}_nE}+\curl H=\nablab\log\phi\wedge H+\frac{1}{2}k\times E\label{eq:curl:H}
\end{gather}
where $\widehat{\mathcal{L}_nH}_{ij}=\mathcal{L}_nH_{ij}+\frac{2}{3}(k\cdot H){\gb{}}_{ij}$, $k$ is the second fundamental form of $\Sigma_r$, `$\times$' is the product defined in (4.4.17) in \cite{ch:kl} and `$\wedge$' in (7.2.2e-) therein. More explicitly, \eqref{eq:curl:E} can also be expressed as:
\begin{equation}
    \frac{1}{\phi}\frac{\partial H_{ij}}{\partial r}+(\curl E)_{ij}=(\nablab\log\phi\wedge E)_{ij}-\frac{1}{2}(H\times k)_{ij}+\frac{2}{3}(H\cdot k) \gbb_{ij}\label{eq:curl:E:long}
\end{equation}
\end{subequations}

Note that,   cf.~\CKeq{7.2.2},
\begin{subequations}
  \begin{gather*}
    \nabla_{k}W_{i0j0}=\nablab_k E_{ij}+\Bigl(\epsilon_{il}^{\phantom{il}m} H_{mj}-\epsilon_{jl}^{\phantom{jl}m}H_{mi} \Bigr)k^l_k\\
    \nabla_{k}\ld W_{i0j0}=\nablab_k H_{ij}-\Bigl(\epsilon_{il}^{\phantom{il}m} E_{mj}+\epsilon_{jl}^{\phantom{jl}m}E_{mi} \Bigr)k^l_k
  \end{gather*}
\end{subequations}
The equations  \eqref{eq:div:E} and \eqref{eq:div:H} then follow directly by taking the trace and using \eqref{eq:bianchi:weyl:inhom}.
Moreover\footnote{This formula differs from (7.2.2h) in \cite{ch:kl} because $\Sigma_r$ is \emph{not} maximal, and the frame \emph{not} Fermi transported.}
\begin{subequations}
\begin{multline*}
  \nabla_0 W_{kij0}=-\epsilon_{ki}^{\phantom{ki}l}\frac{1}{\phi}\frac{\partial H_{lj}}{\partial r}-\tr k \epsu{ki}{l} H_{lj}+\epsu{li}{m} k_k^l H_{mj}+\epsu{kl}{m} k_i^l H_{mj}+\epsu{ki}{m}k_j^l H_{ml}\\
  +(\nablab_k\log\phi) E_{ij}-(\nablab_i \log\phi )E_{kj}-(\nablab^l\log\phi) W_{kijl}
\end{multline*}
\begin{multline*}
  \nabla_k W_{i0j0}+\nabla_i W_{0kj0}=\epsilon_{ki}^{\phantom{ki}l}\frac{1}{\phi}\frac{\partial H_{lj}}{\partial r}+\tr k \epsu{ki}{l} H_{lj}-\epsu{li}{m} k_k^l H_{mj}-\epsu{kl}{m} k_i^l H_{mj}-\epsu{ki}{m}k_j^l H_{ml}\\
  -(\nablab_k\log\phi) E_{ij}+(\nablab_i \log\phi )E_{kj}+(\nablab^l\log\phi) W_{kijl}
\end{multline*}
\begin{equation*}
  2\epsu{n}{ki}\nabla_k W_{i0j0}=2\frac{1}{\phi}\frac{\partial H_{nj}}{\partial r}  -2\epsu{n}{ki}(\nablab_k\log\phi) E_{ij}-2\epsu{j}{lm}(\nablab_l\log\phi)E_{mn}
\end{equation*}
\end{subequations}
From (7.2.2f):
\begin{equation*}
  \epsu{n}{ki}\nabla_k W_{i0j0}+\epsu{j}{ki}\nabla_k W_{i0n0}=-2(\curl E)_{nj}-(H\times k)_{nj}+\frac{4}{3}(H\cdot k )g_{nj}
\end{equation*}
This yields \eqref{eq:curl:E:long}, and thus \eqref{eq:curl:E}. Similarly for \eqref{eq:curl:H}.

\subsection{Relation to null decomposition}
\label{sec:electromagnetic:null}

In this Section we discuss the relation between the electromagnetic and null decompositions.
In particular, when \eqref{eq:maxwell} is viewed as a Hodge system on $\Sigma_r$,
  \begin{subequations}\label{eq:maxwell:hodge}
    \begin{gather}
      \divergence E=\rho_E\qquad \curl E=\sigma_H\\
      \divergence H=\rho_H\qquad \curl H=\sigma_E
    \end{gather}
  \end{subequations}
we establish that the ``source terms'' coincide --- \emph{including the lower order terms} --- with components of the modified Lie derivative $\MLie{N}W$.
This is a non-trivial consequence of our choice of $N$ in \eqref{eq:N}, and is important for the application of the elliptic estimate of Section~\ref{sec:electro:elliptic}.

Recall first \[n=\frac{1}{2}(e_3+e_4)\,,\] and also define
\begin{equation}
  X=\frac{1}{2}\bigl(e_3-e_4\bigr)
\end{equation}
then \[e_3=n+X\qquad e_4=n-X\,.\]

We record, cf.~\CKeq{7.3.3e},
\begin{subequations}\label{eq:electro:null}
\begin{gather}
  E_{XX}=\rho\qquad H_{XX}=\sigma\\
  E_{AX}=\frac{1}{2}\betab_A+\frac{1}{2}\beta_A\qquad H_{AX}=\frac{1}{2}\ld\betab_A-\frac{1}{2}\ld\beta_A\\
  E_{AB}=\frac{1}{4}\alpha_{AB}+\frac{1}{4}\alphab_{AB}-\frac{1}{2}\rho\gs_{AB}\qquad H_{AB}=-\frac{1}{4}\ld\alpha_{AB}+\frac{1}{4}\ld\alphab_{AB}-\frac{1}{2}\sigma\gs_{AB}
\end{gather}
\end{subequations}

We now compare the components of $\widehat{\mathcal{L}_nE}$, $\widehat{\mathcal{L}_nH}$  to the null components of the  modified Lie derivative of $W$ with respect to $ N=\Omega n$.

\begin{lemma}\label{lemma:maxwell:sources}
For any Weyl field $W$ decomposed into electric and magnetic parts $E$, and $H$ with respect to the normal $n$, we have the following relations to the null decomposition of $\MLie{N}W$ where $N=\Omega n$:
\begin{subequations}
  \begin{multline}
  \Bigl\lvert \widehat{\mathcal{L}_nE}(X,X)+\frac{1}{2}k\times E - \frac{1}{\Omega} \rho[\MLie{N}W] \Bigr\rvert \leq C \bigl\lvert \Lbh q+\Lh q^{-1}+q(2\omegabh-\tr\chib)+q^{-1}(2\omegah-\tr\chi)\bigr\rvert|\rho|\\+C\lvert \bigl(\zeta+\ds\log q,\beta+\betab\bigr)\rvert+C\lvert \Bigl( q\chibh+q^{-1}\chih,\alpha+\alphab\Bigr)\rvert
\end{multline}
\begin{multline}
  \Bigl\lvert \widehat{\mathcal{L}_nE}(e_A,X)+\frac{1}{2}(k\times E)(e_A,X)-\frac{1}{2\Omega}\Bigl(\betab_q[\MLie{N}W]+\beta_q[\MLie{N}W]\Bigr) \Bigr\rvert\leq \bigl|\Lbh q+\Lh q^{-1}\bigr|\bigl(|\betab|+|\beta|\bigr)\\
  +2\bigl(|\eta|+|\etab|+|\ds\log q|\bigr)|\rho|+\bigl(|\alpha|+|\alphab|\bigr) \bigl(|\eta|+|\etab|+|\ds\log q|\bigr)+\bigl(q|\chibh|+q^{-1}|\chih|\bigr)\bigl(|\betab|+|\beta|\bigr)
\end{multline}

\begin{multline}
  \Bigl\lvert \widehat{\mathcal{L}_nE}+\frac{1}{2}(k\times E)+\frac{1}{2}\bigl(  \widehat{\mathcal{L}_nE}(X,X)+\frac{1}{2}(k\times E)(X,X) \bigr) \gs-\frac{1}{4\Omega}\bigl(\alpha_q[\MLie{N}W]+\alphab_q[\MLie{N}W]\bigr)\Bigr\rvert_{\gs}\leq\\
\leq q\bigl|2\omegabh-\tr\chib\bigr| |\rho|+q^{-1}\bigl|2\omegah-\tr\chi\bigr| |\rho|+\lvert \Lbh q+\Lh q^{-1} \rvert |\rho|+\bigl(q|\chibh|+q^{-1}|\chih|\bigr)|\rho|\\
+\Bigl(q\bigl|\tr\chib-2\omegabh\bigr|+q^{-1}\bigl|\tr\chi-2\omegah\bigr|+|\Lbh q|+|\Lh q^{-1}|+q|\chibh|+q^{-1}|\chih|\Bigr)\bigl(|\alpha_q|+|\alphab_q|\bigr)\\
+2\bigl(|\zeta|+|\ds\log q|\bigr)\bigl(|\beta_q|+|\betab_q|\bigr)
\end{multline}
\end{subequations}
Similarly for the magnetic part $H$.
\end{lemma}
\begin{proof}
  Note first
  \begin{equation*}
    g(X,X)=1\qquad g(X,e_A)=0
  \end{equation*}
Define
\begin{equation*}
  \gb{r}(E^\sharp\cdot Y,Z)=E(Y,Z)
\end{equation*}
then
\begin{gather*}
  E^\sharp\cdot X=\rho X+\frac{1}{2}\betab_Ae_A+\frac{1}{2}\beta_A e_A\\
  E^\sharp\cdot e_A=\frac{1}{2}\bigl(\betab_A+\beta_A) X+\frac{1}{4}\bigl(\alpha^{\sharp B}_{A}+\alphab^{\sharp B}_{A}\bigr)\cdot e_B-\frac{1}{2}\rho e_A
\end{gather*}

Moreover
\begin{gather*}
  k^\sharp\cdot X=\frac{1}{2}\bigl(q\omegabh+q^{-1}\omegah+\Lbh q+\Lh q^{-1}\bigr)X  +\frac{1}{2}\eta^\sharp-\frac{1}{2}\etab^\sharp\\
  k^\sharp\cdot e_A=\frac{1}{2}\bigl(\eta_A-\etab_A\bigr)X+\frac{1}{2}q\chib^{\sharp B}_{A}e_B+\frac{1}{2}q^{-1}\chi^{\sharp B}_{A}e_B
\end{gather*}
because
\begin{gather*}
  k(X,X)=g(\nabla_X n,X)=\frac{1}{2}\bigl(q\omegabh+q^{-1}\omegah+\Lbh q+\Lh q^{-1}\bigr)\\
  k(X,e_A)=\frac{1}{2}\eta_A-\frac{1}{2}\etab_A\\
  k(e_A,e_B)=\frac{1}{2}q\chib_{AB}+\frac{1}{2}q^{-1}\chi_{AB}
\end{gather*}

Thus
\begin{equation*}
  g(k^\sharp\cdot  X, E^\sharp \cdot X)=\frac{1}{2}\bigl(q\omegabh+q^{-1}\omegah+\Lbh q+\Lh q^{-1}\bigr)\rho+\frac{1}{4}\bigl(\eta-\etab,\beta+\betab\bigr)
\end{equation*}
\begin{multline*}
    g(k^\sharp \cdot e_A, E^\sharp\cdot e_B\bigr)=\frac{1}{4}\bigl(\eta_A-\etab_A\bigr)\bigl(\beta_B+\betab_B)+\frac{1}{8}\bigl(\alpha+\alphab\bigr)_{B}^{\sharp C}\bigl(q\chibh_{AC}+q^{-1}\chih_{AC}\bigr)\\
    -\frac{1}{4}\bigl(q\chibh+q^{-1}\chih\bigr)_{AB}\rho+\frac{1}{16}\bigl(q\tr\chib+q^{-1}\tr\chi\bigr)\bigl(\alpha+\alphab\bigr)_{AB}
    -\frac{1}{8}\bigl(q\tr\chib+q^{-1}\tr\chi\bigr) \gs_{AB}\rho
\end{multline*}
\begin{equation*}
  g(k^\sharp\cdot e_A,E^\sharp\cdot X)=\frac{1}{2}\bigl(\eta_A-\etab_A\bigr)\rho+\frac{1}{4}\bigl(q\chib+q^{-1}\chi\bigr)_A^\sharp\cdot\bigl(\betab+\beta\bigr)
\end{equation*}
\begin{equation*}
  g(k^\sharp\cdot X,E^\sharp\cdot e_A)=\frac{1}{4}\bigl(q\omegabh+q^{-1}\omegah+\Lbh q+\Lh q^{-1}\bigr)\bigl(\betab_A+\beta_A)+\frac{1}{8}\bigl(\alpha_A^\sharp+\alphab_A^\sharp\bigr)\cdot \bigl(\eta-\etab\bigr)-\frac{1}{4}\bigl(\eta_A-\etab_A\bigr)\rho
\end{equation*}
\begin{multline*}
  k\cdot E=k^{ij}E_{ij}=k(X,X)E(X,X)+\gs^{AB}k(X,e_A)E(X,e_B)+\gs^{AC}\gs^{BD}k_{CD}E_{AB}\\
  =\frac{1}{4}\bigl(2q\omegabh+2q^{-1}\omegah-q\tr\chib-q^{-1}\tr\chi+2\Lbh q+2\Lh q^{-1}\bigr)\rho+\frac{1}{4}\bigl(\eta-\etab,\beta+\betab \bigr)+\frac{1}{8}\Bigl( q\chibh+q^{-1}\chih,\alpha+\alphab\Bigr)
\end{multline*}
and by \CKLemma{4.4.2}
\begin{multline*}
  \frac{1}{2}(k\times E)(X,X)= g(k^\sharp\cdot X,E^\sharp\cdot X)-\frac{1}{3}(k\cdot E) g(X,X)\\
  =\frac{1}{2}\frac{1}{6}\bigl(4q\omegabh+4q^{-1}\omegah+q\tr\chib+q^{-1}\tr\chi+4\Lbh q+4\Lh q^{-1}\bigr)\rho+\frac{1}{6}\bigl(\eta-\etab,\beta+\betab\bigr)\\
-\frac{1}{3}\frac{1}{8}\Bigl( q\chibh+q^{-1}\chih,\alpha+\alphab\Bigr)
\end{multline*}
\begin{multline*}
  (k\times E)(e_A,X)=\frac{1}{4}\bigl(\eta_A-\etab_A\bigr)\rho+\frac{1}{4}\bigl(q\chibh+q^{-1}\chih\bigr)_A^\sharp\cdot\bigl(\betab+\beta\bigr)\\
+\frac{1}{8}\bigl(2q\omegabh+2q^{-1}\omegah+q\tr\chib+q^{-1}\tr\chi+2\Lbh q+2\Lh q^{-1}\bigr)\bigl(\betab_A+\beta_A)+\frac{1}{8}\bigl(\alpha_A^\sharp+\alphab_A^\sharp\bigr)\cdot \bigl(\eta-\etab\bigr)
\end{multline*}
\begin{multline*}
  (k\times E)(e_A,e_B)=\frac{1}{4}\bigl(\eta-\etab\bigr)\otimesh \bigl(\beta+\betab\bigr)_{AB}+\frac{1}{12}\bigl(\eta-\etab,\beta+\betab \bigr)\gs_{AB}\\+\frac{1}{24}\bigl(\alpha+\alphab,q\chibh+q^{-1}\chih\bigr)\gs_{AB}
    -\frac{1}{2}\bigl(q\chibh+q^{-1}\chih\bigr)_{AB}\rho+\frac{1}{8}\bigl(q\tr\chib+q^{-1}\tr\chi\bigr)\bigl(\alpha+\alphab\bigr)_{AB}\\
-\frac{1}{3}\frac{1}{4}\bigl(4q\omegabh+4q^{-1}\omegah+q\tr\chib+q^{-1}\tr\chi+4\Lbh q+4\Lh q^{-1}\bigr)\gs_{AB}\rho
\end{multline*}
Finally since
\begin{gather*}
  [n,X]=-\frac{1}{2}\Bigl(2\eta-2\etab-\bigl(q\omegabh+\Lbh q\bigr)e_4+\bigl(q^{-1}\omegah+\Lh q^{-1}\bigr)e_3\Bigr)\\
  [n,e_A]=\frac{1}{2}\bigl(\nablas_3e_A+\nablas_4 e_A\bigr)+\frac{1}{2}\eta_A e_3+\frac{1}{2}\etab_A e_4-\frac{1}{2}\bigl(q\chib+q^{-1}\chi\bigr)_A^\sharp-\bigl(\zeta+\ds\log q\bigr)_AX
\end{gather*}
we have
\begin{equation*}
  \mathcal{L}_nE(X,X)=n\rho+\bigl(\eta-\etab,\beta+\betab\bigr)+\bigl(q\omegabh+q^{-1}\omegah+\Lbh q+\Lh q^{-1}\bigr)\rho
\end{equation*}
\begin{multline*}
  \mathcal{L}_nE(e_A,X)=\frac{1}{4}\frac{1}{\Omega}\Bigl(\Db\betab+\Db\beta+D\betab+D\beta\Bigr)+\frac{1}{4}\bigl(q\omegabh+q^{-1}\omegah+\Lbh q+\Lh q^{-1}\bigr)\bigl(\betab_A+\beta_A\bigr)\\
  +\frac{1}{2}\bigl(\zeta+\ds\log q-2\eta+2\etab\bigr)_A\rho+\frac{1}{4}\bigl(\alpha+\alphab\bigr)_A^\sharp\cdot\bigl(\eta-\etab\bigr)
\end{multline*}
\begin{multline*}
   \mathcal{L}_nE(e_A,e_B)=\frac{1}{8}\frac{1}{\Omega}\Bigl(\Dbh\alpha+\Dbh\alphab+\Dh\alpha+\Dh\alphab\bigr)-\frac{1}{4}\Bigl(2 n \rho +q\tr\chib \rho+q^{-1}\tr\chi\rho\Bigr)\gs_{AB}\\
   -\frac{1}{4}\bigl(\eta_A-\etab_A-4\zeta-4\ds\log q\bigr)\bigl(\betab_B+\beta_B\bigr)   -\frac{1}{4}\bigl(\eta_B-\etab_B-4\zeta-4\ds\log q\bigr)\bigl(\betab_A+\beta_A\bigr)\\
   +\frac{1}{8}\Bigl(q\chibh+q^{-1}\chih,\alphab+\alpha\Bigr)\gs_{AB}
\end{multline*}
and thus
\begin{multline*}
  \widehat{\mathcal{L}_nE}(X,X)=\mathcal{L}_nE(X,X)+\frac{2}{3}(k\cdot E)g(X,X)=\\
  =n\rho+\frac{7}{6}\bigl(\eta-\etab,\beta+\betab\bigr)+\frac{1}{3}\frac{1}{4}\Bigl( q\chibh+q^{-1}\chih,\alpha+\alphab\Bigr)\\
+\frac{1}{3}\frac{1}{2}\bigl(8q\omegabh+8q^{-1}\omegah-q\tr\chib-q^{-1}\tr\chi+8\Lbh q+8\Lh q^{-1}\bigr)\rho
\end{multline*}
\begin{multline*}
    \widehat{\mathcal{L}_nE}(e_A,e_B)=\mathcal{L}_nE(e_A,e_B)+\frac{2}{3}(k\cdot E)g_{AB}=\frac{1}{8}\frac{1}{\Omega}\Bigl(\Dbh\alpha+\Dbh\alphab+\Dh\alpha+\Dh\alphab\bigr)_{AB}\\
-\frac{1}{2} n \rho \gs_{AB} +\frac{1}{3}\frac{1}{4}\bigl(4q\omegabh+4q^{-1}\omegah-5q\tr\chib-5q^{-1}\tr\chi+4\Lbh q+4\Lh q^{-1}\bigr)\gs_{AB}\rho\\
   -\frac{1}{4}\bigl(\eta-\etab\bigr)\otimesh\bigl(\betab+\beta\bigr)_{AB} +\bigl(\zeta_A+\ds_A\log q\bigr)\bigl(\betab_B+\beta_B\bigr) +\bigl(\zeta_B+\ds_B\log q\bigr)\bigl(\betab_A+\beta_A\bigr)\\
   +\frac{5}{24}\Bigl(q\chibh+q^{-1}\chih,\alphab+\alpha\Bigr)\gs_{AB} -\frac{1}{12}\bigl(\eta-\etab,\beta+\betab \bigr)\gs_{AB}
\end{multline*}
Therefore
\begin{multline*}
  \widehat{\mathcal{L}_nE}(X,X)+\frac{1}{2}(k\times E)(X,X)=n\rho+\frac{4}{3}\bigl(\eta-\etab,\beta+\betab\bigr)+\frac{1}{3}\frac{1}{8}\Bigl( q\chibh+q^{-1}\chih,\alpha+\alphab\Bigr)\\
+\frac{1}{6}\frac{1}{2}\bigl(20q\omegabh+20q^{-1}\omegah-q\tr\chib-q^{-1}\tr\chi+20\Lbh q+20\Lh q^{-1}\bigr)\rho
\end{multline*}
\begin{multline*}
  \widehat{\mathcal{L}_nE}(e_A,X)+\frac{1}{2}(k\times E)(e_A,X)=\frac{1}{4}\frac{1}{\Omega}\Bigl(\Db\betab+\Db\beta+D\betab+D\beta\Bigr)\\+\frac{1}{16}\bigl(6q\omegabh+6q^{-1}\omegah+q\tr\chib+q^{-1}\tr\chi+6\Lbh q+6\Lh q^{-1}\bigr)\bigl(\betab_A+\beta_A\bigr)\\
  +\frac{1}{8}\bigl(4\zeta+4\ds\log q-7\eta+7\etab\bigr)_A\rho+\frac{5}{16}\bigl(\alpha_A^\sharp+\alphab_A^\sharp\bigr)\cdot \bigl(\eta-\etab\bigr)+\frac{1}{8}\bigl(q\chibh+q^{-1}\chih\bigr)_A^\sharp\cdot\bigl(\betab+\beta\bigr)
\end{multline*}
\begin{multline*}
  \widehat{\mathcal{L}_nE}(e_A,e_B)+\frac{1}{2}(k\times E)(e_A,e_B)+  \frac{1}{2}\widehat{\mathcal{L}_nE}(X,X)\gs_{AB}+\frac{1}{4}(k\times E)(X,X) \gs_{AB}=\\
=\frac{1}{8}\frac{1}{\Omega}\Bigl(\Dbh\alpha+\Dbh\alphab+\Dh\alpha+\Dh\alphab\bigr)
+\frac{1}{2}\bigl(2q\omegabh+2q^{-1}\omegah-q\tr\chib-q^{-1}\tr\chi+2\Lbh q+2\Lh q^{-1}\bigr)\gs_{AB}\rho\\
   -\frac{1}{8}\bigl(\eta-\etab\bigr)\otimesh\bigl(\betab+\beta\bigr)_{AB} +\bigl(\zeta_A+\ds_A\log q\bigr)\bigl(\betab_B+\beta_B\bigr) +\bigl(\zeta_B+\ds_B\log q\bigr)\bigl(\betab_A+\beta_A\bigr)\\
   +\frac{1}{4}\Bigl(q\chibh+q^{-1}\chih,\alphab+\alpha\Bigr)\gs_{AB} +\frac{5}{8}\bigl(\eta-\etab,\beta+\betab \bigr)\gs_{AB}\\
    -\frac{1}{4}\bigl(q\chibh+q^{-1}\chih\bigr)_{AB}\rho+\frac{1}{16}\bigl(q\tr\chib+q^{-1}\tr\chi\bigr)\bigl(\alpha+\alphab\bigr)_{AB}\\
\end{multline*}
and in comparison to Lemma~\ref{lemma:null:decomposition:LieMW}:
  \begin{multline*}
  \widehat{\mathcal{L}_nE}(X,X)+\frac{1}{2}(k\times E)(X,X)= \frac{1}{\Omega} \rho[\MLie{N}W]+\frac{4}{3}\bigl(\Lbh q+\Lh q^{-1}\bigr)\rho\\+\frac{1}{6}\bigl(4\zeta-3\ds\log q,\beta+\betab\bigr)+\frac{1}{3}\frac{1}{8}\Bigl( q\chibh+q^{-1}\chih,\alpha+\alphab\Bigr)+\frac{11}{24}\bigl(q(2\omegabh-\tr\chib)+q^{-1}(2\omegah-\tr\chi)\bigr)\rho
\end{multline*}
\begin{multline*}
  4\widehat{\mathcal{L}_nE}(e_A,X)+2(k\times E)(e_A,X)=\frac{2}{\Omega}\Bigl(\betabt_q+\betat_q\Bigr)+\frac{5}{4}\bigl(\Lbh q+\Lh q^{-1}\bigr)\bigl(\betab_A+\beta_A\bigr)\\
  +\frac{1}{2}\bigl(\eta-\etab+10\ds\log q\bigr)_A\rho+\frac{1}{4}\bigl(\alpha_A^\sharp+\alphab_A^\sharp\bigr)\cdot \bigl(3\eta-3\etab-2\ds\log q\bigr)+\frac{3}{2}\bigl(q\chibh+q^{-1}\chih\bigr)_A^\sharp\cdot\bigl(\betab+\beta\bigr)
\end{multline*}
\begin{multline*}
  8\widehat{\mathcal{L}_nE}(e_A,e_B)+4(k\times E)(e_A,e_B)+  4\widehat{\mathcal{L}_nE}(X,X)\gs_{AB}+2(k\times E)(X,X) \gs_{AB}=\\
=\frac{2}{\Omega}\Bigl(\alphat_q+\alphabt_q\bigr)_{AB}
+4\bigl(2q\omegabh+2q^{-1}\omegah-q\tr\chib-q^{-1}\tr\chi+2\Lbh q+2\Lh q^{-1}\bigr)\gs_{AB}\rho\\
+\frac{1}{4}\bigl(3q\tr\chib+3q^{-1}\tr\chi-6q\omegabh-6q^{-1}\omegah-7\Lbh q+\Lh q^{-1}\bigr)\alpha_{AB}\\+\frac{1}{4}\bigl(3q\tr\chib+3q^{-1}\tr\chi-6q^{-1}\omegah-6q\omegabh-7\Lh q^{-1}+\Lbh q\bigr)\alphab_{AB}\\
   +10\bigl(\zeta+\ds\log q\bigr)\otimesh\bigl(\betab+\beta\bigr)_{AB}+2\ds\log q\otimesh\bigl(\betab+\beta\bigr)_{AB}+2\bigl(\zeta-4\ds\log q,\beta+\betab \bigr)\gs_{AB}\\
   +2\Bigl(q\chibh+q^{-1}\chih,\alphab+\alpha\Bigr)\gs_{AB}
    -2\bigl(q\chibh+q^{-1}\chih\bigr)_{AB}\rho\\
\end{multline*}

\end{proof}

Also note that
\begin{equation}\label{eq:maxwell:sources:cross}
  (E\wedge k)(X)=\frac{1}{4}\bigl(\ld\eta-\ld\etab,\betab+\beta\bigr)+\frac{1}{8}\bigl(\alpha+\alphab\bigr)\wedge\bigl(q\chibh+q^{-1}\chih\bigr)
\end{equation}
contains no terms linear in $q\tr\chib$, $q^{-1}\tr\chi$ due to the antisymmetry of the product. Similarly for the other components.

\subsection{Elliptic estimate for Maxwell system}
\label{sec:electro:elliptic}

We quote the following result from \CKS{4.4} for symmetric Hodge systems.

\begin{lemma}[Corollary 4.4.2.1 in \cite{ch:kl}]\label{lemma:maxwell:elliptic}
  Let $E$ and $H$ be symmetric traceless tensors on a 3-dimensional Riemannian manifold $(\Sigma,\gb{})$ which satisfy the system \eqref{eq:maxwell:hodge}.
Then there is a constant $C>0$ such that
\begin{multline}
  \int_\Sigma\Bigl(|\nabla E|^2+|\nabla H|^2\Bigr)\dm{\gb{}}\leq C \int_{\Sigma} |\rho_E|^2+|\rho_H|^2+|\sigma_E|^2+|\sigma_H|^2 \dm{\gb{}}\\
  +C\int_{\Sigma}|\Ric(\gb{})|\bigl(|E|^2+|H|^2\bigr)\dm{\gb{}}
\end{multline}

\end{lemma}

Here we will  make for simplicity an assumption directly on the Ricci curvature of~$\Sigma_r$:
\begin{equation}
  \lvert \overline{\Ric} \rvert \leq \frac{C_0}{r^2} \tag{\emph{\textbf{BA:III}:ii}} \label{eq:BA:III:Ric}
\end{equation}
We note however that in view of the Gauss equations of the embedding of $\Sigma_r$ in $(\mathcal{D},g)$ we can reduce this condition to an assumption on the scalar curvature of $\Sigma_r$.

\begin{lemma}
  Suppose $(\Sigma,\gb{r})$ is  embedded as a spacelike hypersurface in  $(\mathcal{M},g)$, which is a solution to \eqref{eq:eve:lambda:intro}.
Assume (\textbf{BA:I}.ii,iv-vii) hold for some $C_0>0$. 
% \begin{subequations}
% \begin{gather}
%        \lvert 2\omega-\Omega \tr\chi\rvert\leq C \tr\chi      \qquad  \lvert 2\omegab-\Omega\tr\chib\rvert\leq C \tr\chib\\
%      \lvert D \log q\rvert \leq C \tr\chi \qquad \lvert \Db \log q \rvert \leq C \tr\chib \\
%       \qquad q\tr\chib\leq C\qquad \qquad q^{-1}\tr\chi\leq C\\
%       q \Omega |\chibh|\leq C\qquad \Omega |\zeta| \leq C\qquad q^{-1} \Omega\lvert\chih\rvert\leq C
%    \end{gather}
%  \end{subequations}
 Then there is a constant $C>0$,
\begin{equation}
  \lvert \overline{\Ric } \rvert^2  \leq C\Bigl( \frac{1}{\Omega^2}+\Rb^2+|W|^2\Bigr)
\end{equation}

\end{lemma}

\begin{remark}
  While this estimate shows that the assumption \eqref{eq:BA:III:Ric} is consistent with \eqref{eq:assumption:intro:Omega}, and the decay properties of the Weyl curvature, we point out that the assumption \eqref{eq:BA:III:Ric} is not redundant in the context of the assumptions we have already made.
In fact, to retrieve \eqref{eq:BA:III:Ric} directly from the assumptions on the null second fundamental form we would have to make the stronger assumption that
\begin{equation}
  \Omega^2 \bigl(2\omegah-\tr\chi\bigr)\leq C\qquad \Omega^2\bigl(2\omegabh-\tr\chib\bigr)\leq C
\end{equation}
which reflects the behavior in (Schwarzschild-)de Sitter, but goes beyond the assumptions already made in (\emph{\textbf{BA:I}.ii}).
\end{remark}

\begin{proof}

Recall the Gauss equation
\begin{equation*}
    \overline{\Ric}_{ij}+\tr k k_{ij}-k_{i}^{\phantom{i}m}k_{mj}=\Ric_{ij}+R_{0i0j}
\end{equation*}
From the Einstein equations
\begin{equation*}
  \Ric_{ij}=\Lambda g_{ij}
\end{equation*}
and with \eqref{eq:weyl:lambda}
\begin{equation*}
  R_{0i 0j} =    W_{0i 0j}-\frac{\Lambda}{3}g_{ij}\,.
\end{equation*}

Then using the formulas obtained in the proof of Lemma~\ref{lemma:maxwell:sources},
\begin{gather*}
  \tr k= \frac{1}{2}\bigl(q\omegabh+q^{-1}\omegah+q\tr\chib+q^{-1}\tr\chi+\Lbh q+\Lh q^{-1}\bigr)\\
  |k|^2 = \frac{1}{4}\bigl(q\omegabh+q^{-1}\omegah+\Lbh q+\Lh q^{-1}\bigr)^2+\frac{1}{4}\lvert q\chibh+q^{-1}\chih\rvert^2+\frac{1}{8}\bigl(q\tr\chib+q^{-1}\tr\chi\bigr)^2
\end{gather*}
we have
  \begin{gather*}
        \overline{\Ric}(X,X)=-\tr k\, k(X,X)+g(k^\sharp\cdot X,k^\sharp\cdot X)+\frac{2}{3}\Lambda+E_{XX}\\
        \overline{\Ric}(e_A,X)=-\tr k\, k(X,e_A)+g(k^\sharp\cdot e_A,k^\sharp\cdot X)+E_{AX}\\
        \overline{\Ric}_{AB}=-\tr k\, k_{AB}+g(k^\sharp\cdot e_A,k^\sharp\cdot e_B)+\frac{2}{3}\Lambda \gs_{AB}+E_{AB}
  \end{gather*}

Now the Gauss equation tells us
\begin{equation*}
    \Rb+(\tr k)^2-\lvert k\rvert^2=2\Lambda
\end{equation*}
so we can write
  \begin{gather*}
        \overline{\Ric}(X,X)=\frac{1}{3}\tr k\bigl(\tr k-3 k(X,X)\bigr)+\frac{1}{3}\bigl(3g(k^\sharp\cdot X,k^\sharp\cdot X)-|k|^2\bigr)+\frac{1}{3}\Rb+\rho\\
        \overline{\Ric}(e_A,X)=-\tr k\, k(X,e_A)+g(k^\sharp\cdot e_A,k^\sharp\cdot X)+\frac{1}{2}\bigl(\betab+\beta\bigr)_A\\
        \overline{\Ric}_{AB}=\frac{1}{3}\tr k \bigl(\tr k\gs_{AB}-3 k_{AB}\bigr)+\frac{1}{3}\bigl(3g(k^\sharp\cdot e_A,k^\sharp\cdot e_B)-\lvert k\rvert^2\gs_{AB}\bigr)+\frac{1}{3}\Rb \gs_{AB}+E_{AB}
  \end{gather*}
Since
\begin{gather*}
  \tr k-3k(X,X)=\frac{1}{2}\bigl(q\tr\chib+q^{-1}\tr\chi-2q\omegabh-2q^{-1}\omegah-2\Lbh q-2\Lh q^{-1}\bigr)\\
3g(k^\sharp\cdot X,k^\sharp\cdot X)-|k|^2= \frac{1}{8}\bigl(2q\omegabh+2q^{-1}\omegah-q\tr\chib-q^{-1}\tr\chi+2\Lbh q+2\Lh q^{-1}\bigr)\times\notag\\\times\bigl(2q\omegabh+2q^{-1}\omegah+q\tr\chib+q^{-1}\tr\chi+2\Lbh q+2\Lh q^{-1}\bigr)  +3\lvert \zeta \rvert^2
\end{gather*}
\begin{multline*}
 g(k^\sharp\cdot e_A,k^\sharp\cdot X)=\frac{1}{8}\bigl(2q\omegabh+2q^{-1}\omegah+q\tr\chib+q^{-1}\tr\chi+2\Lbh q+2\Lh q^{-1}\bigr)\bigl(\eta_A-\etab_A\bigr)\\+\frac{1}{4}\bigl(q\chibh_A^\sharp+q^{-1}\chih_A^\sharp\bigr)\cdot \bigl(\eta-\etab\bigr)
\end{multline*}
\begin{multline*}
  \tr k\gs_{AB}-3 k_{AB}=\frac{1}{4}\bigl(2q\omegabh+2q^{-1}\omegah-q\tr\chib-q^{-1}\tr\chi+2\Lbh q+2\Lh q^{-1}\bigr)\gs_{AB}\\
-\frac{3}{2}q\chibh_{AB}+\frac{3}{2}q^{-1}\chih_{AB}
\end{multline*}
\begin{multline*}
  3g(k^\sharp\cdot e_A,k^\sharp\cdot e_B)-\lvert k\rvert^2\gs_{AB}= +\frac{3}{4}\bigl(q\tr\chib+q^{-1}\tr\chi\bigr)\bigl(q\chibh+q^{-1}\chih\bigr)_{AB}\\
+\frac{3}{4}\bigl(\eta_A-\etab_A\bigr)\bigl(\eta_B-\etab_B\bigr)+\frac{1}{8}\lvert q\chibh+q^{-1}\chih\rvert^2\gs_{AB}\\
+\frac{1}{16}\bigl(q\tr\chib+q^{-1}\tr\chi-2q\omegabh-2q^{-1}\omegah-2\Lbh q-2\Lh q^{-1}\bigr)\times\\\times\bigl(q\tr\chib+q^{-1}\tr\chi+2q\omegabh+2q^{-1}\omegah+2\Lbh q+2\Lh q^{-1}\bigr)\gs_{AB}
\end{multline*}
the result follows from
\begin{equation*}
  \lvert \overline{\Ric} \rvert^2 = \bigl(\overline{\Ric}(X,X)\bigr)^2+2\gs^{AB}\overline{\Ric}(X,e_A)\overline{\Ric}(X,e_B)+\gs^{AC}\gs^{BD}\overline{\Ric}_{AB}\overline{\Ric}_{CD}\,.
\end{equation*}
\end{proof}

\begin{corollary}\label{cor:elliptic}
  Let $W$ be a solution to \eqref{eq:W:Bianchi}, and $E$ and $H$ the electric and magnetic parts of $W$ relative to $\Sigma_r$.
Assume (\textbf{BA:I},ii,iv-vi)  and \eqref{eq:BA:III:Ric} hold for some $C_0>0$. Then
\begin{multline}
  \int_{\Sigma_r}|\nablab E|^2+|\nablab H|^2\dm{\gb{r}}\leq C \int_{\Sigma_r} \frac{1}{\Omega^2}\Bigl[|\alphabt|^2+|\betabt|^2+\rhot^2+\sigmat^2+|\betat|^2+|\alphat|^2\Bigr]\dm{\gb{r}}\\+C\int_{\Sigma_r}\Bigl(\frac{1}{r^2}+\frac{1}{\Omega^2}\Bigr)\Bigl(|E|^2+|H|^2\Bigr)\dm{\gb{r}}
\end{multline}

\end{corollary}

\begin{proof}
Since $W$ solves the Bianchi equations, $E$, and $H$ solve \eqref{eq:maxwell:hodge} --- cf.~\eqref{eq:maxwell} --- with 
\begin{gather*}
      \rho_E=H\wedge k\qquad 
      \sigma_E=\nablab\log\phi\wedge E-\frac{1}{2}k\times H-\widehat{\mathcal{L}_nH}\\
    \rho_H= -E\wedge k\qquad 
    \sigma_H=\nablab\log\phi\wedge H+\frac{1}{2}k\times E+\widehat{\mathcal{L}_nE}
\end{gather*}
By Lemma~\ref{lemma:maxwell:sources}, and the assumptions (\emph{\textbf{BA:I},ii,iv-vi}) we have
\begin{multline*}
 | \widehat{\mathcal{L}_nE}+\frac{1}{2}k\times E |^2 + | \widehat{\mathcal{L}_nH}+\frac{1}{2}k\times H |^2 \leq \\ \leq \frac{1}{\Omega^2}\bigl[\rhot^2+\sigmat^2+|\betabt|^2+|\betat|^2+|\alphabt|^2+|\alphat|^2+\rho^2+\sigma^2+|\betab_q|^2+|\beta_q|^2+|\alphab_q|^2+|\alpha_q|^2\bigr]
\end{multline*}
Moreover by \eqref{eq:maxwell:sources:cross}, and the assumptions (\emph{\textbf{BA:I},ii,iv-vi}),
\begin{equation*}
  | E\wedge k |^2 + | H \wedge k |^2 \leq \frac{C}{\Omega^2}\bigl[ |\alphab_q|^2+|\betab_q|^2+\rho^2+\sigma^2+|\beta_q|^2+|\alpha_q|^2\bigr]
\end{equation*}
%Todo: write out the gradient term also
The statement of the Corollary then follows immediately from Lemma~\ref{lemma:maxwell:elliptic}, and \eqref{eq:electro:null}.

\end{proof}

%%% Local Variables:
%%% mode: latex
%%% TeX-master: "weyl"
%%% End:

\section{Sobolev inequalties}
\label{sec:sobolev}

We will prove a Sobolev trace inequality on the spacelike hypersurfaces $\Sigma_r$ to relate the bounds on the energy fluxes obtained in Sections~\ref{sec:energy:estimates}-\ref{sec:commuted:redshift} to $\mathrm{L}^p$ estimates on the spheres.

\subsection{Preliminaries} 

Here $(\Sigma_r,\gb{r})$ are 3-dimensional Riemannian manifolds diffeomorphic to a cylinder $\mathbb{R}\times\mathbb{S}^2$, and there exists a differentiable function $u:\Sigma_r\to\mathbb{R}$, namely the restriction of the null coordinate $u$ to $\Sigma_r$, such that the level sets of $u\rvert_{\Sigma_r}$ are spheres $S_{u,r}$ diffeomorphic to $\mathbb{S}^2$, and $u$ has no critical points. In fact, in  Lemma~\ref{lemma:g:induced} we proven  that the metric on $\Sigma_r$ in $(u,\vartheta^1,\vartheta^2)$ coordinates, takes the canonical form
  \begin{equation}
    \gb{r}=q^{-2}\Omega^2\ud u^2+\gs_{AB}\ud \vartheta^1\ud\vartheta^2\,.
  \end{equation}
We also found that the volume form is given by
  \begin{equation}
    \dm{\gb{r}}=q^{-1}\Omega\,\ud u\wedge\dm{\gs}=q^{-1}\Omega\sqrt{\det\gs}\:\ud u \wedge \ud \vartheta^1\wedge \ud\vartheta^2
  \end{equation}
The ``lapse function''  of the foliation of $\Sigma_r$ by spheres $S_u:=S_{u,r}$ is thus given by
\begin{equation}
    a=\frac{\Omega}{q}
\end{equation}

Recall that in Section~\ref{sec:electromagnetic:null} we introduced $X$ as the unit tangent vector to $\Sigma_r$, orthogonal to $S_u$. Thus we have $aX f=\partial_u f$, for any function $f(u,\vartheta^1,\vartheta^2)$ on $\Sigma_r$. We denote by $\theta$ the second fundamental form of the surfaces $S_u$ embedded in $\Sigma_r$. Viewed as a tensor on $\Sigma_r$ we can write
\begin{equation}
  \theta(Y,Z)=\gs(\overline{\nabla}_{\Pi Y} X,\Pi Z)
\end{equation}
where $\Pi$ denotes the projection to the tangent space of $S_u$, and thus
\begin{equation*}
  \theta_{ij}=\nabla_i X_j+X_i\nablas_j \log a
\end{equation*}
because $\overline{\nabla}_X X=-\nablas\log a$. Therefore
\begin{equation}\label{eq:div:X}
  \tr_{\gb{}}\theta=\gb{}^{ij}\theta_{ij}=\nablab^i X_i=\divb X\,.
\end{equation}

We also have the first variational formula
\begin{equation*}
  \theta_{AB}=\frac{1}{2a}\partial_u \gamma_{AB}\qquad \theta_{AB}=\theta(\partial_{\vartheta^A},\partial_{\vartheta^B})
\end{equation*}
which gives
\begin{equation*}
  \tr\theta = \frac{1}{2 a }\partial_u\log \det( \gs_{AB} )
\end{equation*}
Consequently,
\begin{equation}
  \frac{\ud}{\ud u}A(u)=\int_{S_u}a\tr\theta\dm{\gs}
\end{equation}
and more generally, for any function $f$,
\begin{equation}
  \frac{\ud }{\ud u}\int_{S_u} f\dm{\gs}=\int_{S_u}a \bigl\{ X f+ f \tr\theta\bigr\}\dm{\gs}
\end{equation}
\begin{remark}
  In the spherically symmetric setting, $\det\gs = r^4$ is constant on $\Sigma_r$,  leading to the \emph{vanishing of the mean curvature $\tr\theta=0$}.
Thus it cannot be expected that in the present setting $\tr\theta$ has a sign.
\end{remark}

\subsection{Isoperimetric Sobolev Inequalities}

We can adapt the proof of the following Sobolev inequality of \CKProp{3.2.1}.
There the geometric setting is different, the manifold $\Sigma$ being diffeomorphic to $\mathbb{R}^3$, and $u$ being a radial function whose level sets have \emph{positive} mean curvature. However, as already remarked there, the proof carries over more generally, if the constant is allowed to depend on the mean curvature.
Here we require that
\begin{equation}
  r|\tr\theta|\leq C
\end{equation}
Since
\begin{equation}
  \theta_{AB}=\frac{1}{2}q\chib -\frac{1}{2}q^{-1}\chi_{AB}
\end{equation}
the assumption amounts to
\begin{equation}\label{eq:BA:I:viii:weak}
  \lvert q\,r\tr\chib- q^{-1}\,r\tr\chi \rvert \leq C_0 \tag{\emph{\textbf{BA:I}.viii}${}_w$}
\end{equation}
which follows from (\emph{\textbf{BA:I}:vii,viii}); in fact \eqref{eq:BA:I:viii:weak} is \emph{weaker} than (\emph{\textbf{BA:I}:viii}).

Furthermore we will assume that the lapse function of the foliation of $\Sigma_r$ by level sets of $u$ is bounded above and below:
  \begin{equation}\label{eq:BA:III:lapse}
    a_m=\inf_{\Sigma_r}a >0\qquad a_M=\sup_{\Sigma_r} a<\infty \tag{\emph{\textbf{BA:III}.iii}}
  \end{equation}
The Sobolev inequality on $\Sigma_r$ stated below relies on the isoperimetric inequalities on the spheres $S_u$, and we  thus need some control on the \emph{isoperimetric constant} $I(u)$ of each sphere $S_u$:
  \begin{equation}\label{eq:BA:III:iso}
    I:=\sup_uI(u)<\infty\tag{\emph{\textbf{BA:III}.iv}}
  \end{equation}

  \begin{remark}
 We note here that in the context of the recovery of the assumptions $(\textbf{BA})$, the bounds on the isoperimetric constants of $S_{u,v}$ are derived from corresponding assumptions on the initial data --- which imply that the isoperimetric constants of the spheres foliating $\mathcal{C}\cup\overline{\mathcal{C}}$ are bounded --- and the validity of the assumptions $(\textbf{BA:I})$.
 We refer the reader to Lemma~6.1 in \cite{schlue:optical}, where \eqref{eq:BA:III:iso} is derived in a simplified setting from \eqref{eq:assummption:intro:average} using Lemma~5.12 in \cite{schlue:optical}.
 The role of the isoperimetric constants of $S_{u,v}$ is also discussed more broadly in \CCh{5.2}.
  \end{remark}

\begin{lemma}\label{lemma:sobolev}
  Assume (\textbf{BA:III}.iii,iv) hold. Moreover assume \eqref{eq:BA:I:viii:weak} for some $C_0>0$.
  Then for any $S_u$-tangent tensorfield $\theta$ we have
  \begin{subequations}
    \begin{gather}
      \Bigl(\int_{\Sigma_r} r^6\lvert \theta\rvert^6 \Bigr)^\frac{1}{6}\leq C\Bigl(\int_{\Sigma_r}\lvert \theta\rvert^2+r^2\lvert \nablab \theta \rvert^2 \Bigr)^\frac{1}{2}\label{eq:sobolev:L:six}\\
      \sup_u\Bigl(\int_{S_u} r^4\lvert\theta\rvert^4\Bigr)^\frac{1}{4}\leq C\Bigl(\int_{\Sigma_r}\lvert \theta\rvert^2+r^2\lvert \nablab \theta \rvert^2\Bigr)^\frac{1}{2}\label{eq:sobolev:L:four}
    \end{gather}
  \end{subequations}
  where $C$ depends on  $a_M/a_m$, $I$, and
  \begin{equation}
    h=\sup_{\Sigma_r}\lvert r\tr\theta \rvert <\infty\,.
  \end{equation}
\end{lemma}

\begin{proof}
  
Recall the isoperimetric inequality on the sphere:
\begin{equation*}
  \int_{S_u}\bigl(\Phi-\overline{\Phi}\bigr)^2\dm{\gs}\leq I(u)\Bigl(\int_{S_u}\lvert\nablas\Phi\rvert\dm{\gs}\Bigr)^2
\end{equation*}
Apply to $\Phi=\lvert \theta \rvert^3$ to obtain
\begin{equation*}
  \int_{S_u}\lvert\theta\rvert^6\lesssim_ I\Bigl(r^{-2}\int_{S_u}\lvert\theta\rvert^4\Bigr)\Bigl(\int_{S_u}|\theta|^2+r^2|\nablas\theta|^2\Bigr)
\end{equation*}
and therefore:
\begin{multline*}
  \int_{\Sigma}r^6 |\theta|^6=\int\int_{S_u}r^6|\theta|^6a\dm{\gs}\ud u
  \lesssim_ I\int\Bigl(r^4\int_{S_u}\lvert\theta\rvert^4\Bigr)\Bigl(\int_{S_u}|\theta|^2+r^2|\nablas\theta|^2\Bigr)a_M(u)\ud u\\
  \lesssim_I \frac{a_M}{a_m}\sup_u\Bigl(r^4\int_{S_u}\lvert\theta\rvert^4\Bigr)\int_\Sigma |\theta|^2+r^2|\nablas\theta|^2
\end{multline*}
Now by the divergence theorem on \footnote{We can think of the cylinder $\Sigma$ being divided in two components by $S_u$, $\Sigma_u$ being one component, a 3-dimensional manifold with boundary $S_u$. }
\begin{equation*}
  \Sigma^u=\cup_{l\geq u} S_l
\end{equation*}
we obtain
\begin{multline*}
  \int_{S_u}r^4|\theta|^4=-\int_{\Sigma^u}\divb\bigl(r^4|\theta|^4 X)
  \leq \int_{\Sigma^u}r^4|\theta|^4 |\divb X | +4 r^4|\theta|^3|\nablab_X\theta|+4r^3|X r||\theta|^4\\
  \leq \int_{\Sigma^u} r^4 |\theta|^4\bigl( |\tr\theta| +\frac{2}{a}\lvert \overline{a\tr\theta} \rvert \bigr)+4r^3|\theta^3|r|\nablab_X\theta|\\
\lesssim \frac{a_M}{a_m} \sup_{\Sigma^u}|r\tr\theta| \Bigl(\int_{\Sigma^u}r^6|\theta|^6\Bigr)^\frac{1}{2}\Bigl(\int_{\Sigma_u}|\theta|^2\Bigr)^\frac{1}{2}+4 \Bigl(\int_{\Sigma^u}r^6|\theta|^6\Bigr)^\frac{1}{2}\Bigl(\int_{\Sigma_u}r^2|\nablab_X\theta|^2\Bigr)^\frac{1}{2}
\end{multline*}
where we used \eqref{eq:div:X} and 
\begin{equation*}
  X r=\frac{r}{2a}\overline{r\tr\theta}\,.
\end{equation*}
Inserting this above yields \eqref{eq:sobolev:L:six}:
\begin{equation*}
  \int_{\Sigma}r^6 |\theta|^6\lesssim_{I,\frac{a_M}{a_m},h} \Bigl(\int_\Sigma |\theta|^2+r^2|\nablab\theta|^2\Bigr)^3
\end{equation*}
Finally inserting this again above then also yields \eqref{eq:sobolev:L:four}:
\begin{equation*}
  \int_{S_u}r^4|\theta|^4\lesssim_{I,\frac{a_M}{a_m},h}\Bigl(\int_\Sigma |\theta|^2+r^2|\nablab\theta|^2\Bigr)^2 
\end{equation*}

\end{proof}

See also Chapter~2 in \cite{ch:kl} for a broader discussion.

\subsection{Applications}

Let $\theta$ be any of the null components of the Weyl curvature $W$:
\begin{equation*}
  \theta=\Bigl\{\alphab,\betab,\rho,\sigma,\beta,\alpha\Bigr\}
\end{equation*}
Note then $\theta$ is a function, or a 1-, or 2-covariant tensorfield on $S_{u}$.
Consider the dimensionless $\mathrm{L}^4$ norm\footnote{These norms will play a prominent role in the ``recovery'' of the assumptions (\emph{\textbf{BA}}), see already Section~7 in \cite{schlue:optical}. Here they are introduced for convenience, but they also provide the correct ``scaling'' in comparison to the $\mathrm{L}^\infty(S_{u,v})$ norm, in the sense that under suitable assumptions the following Sobolev inequality holds for any $S$-tangent tensorfield $\xi$: \begin{equation}\|\xi\|_{L^\infty(S)}\lesssim \nLq{\xi}+\nLq{r\nablas\xi}\end{equation} In particular, if the right hand side decays as a function of $r(S)$, then $\xi$ decays in $\mathrm{L}^\infty(S)$ at the same rate. We refer the reader to Section~6 of \cite{schlue:optical} for the proof of this inequality in the current setting, and to \CCh{5} for Sobolev inequalities on spheres more generally.} of $\theta$ on a sphere $S_u$:
\begin{equation}
  \nLq{\theta}:= \Bigl(\frac{1}{4\pi r^2}\int_{S_u}|\theta|^4_{\gs}\Bigr)^\frac{1}{4}
\end{equation}

Let us now assume the validity of all ``bootstrap assumptions'' (\emph{\textbf{BA:I-III}}).
Then by the Sobolev inequality of Lemma~\ref{lemma:sobolev}  above
\begin{equation}
  \nLq{\theta}\lesssim\frac{1}{r^\frac{3}{2}}\Bigl(\int_{\Sigma_r}|\theta|^2+r^2\lvert\nablab\theta\rvert^2\Bigr)^\frac{1}{2}
\end{equation}
Moreover by Corollary~\ref{cor:elliptic} in view of \eqref{eq:assumption:intro:Omega} 
\begin{multline}
  \int_{\Sigma_r}r^2|\nablab \theta |^2 \dm{\gb{r}}\lesssim \int_{\Sigma_r}|\alphabt|^2+|\betabt|^2+\rhot^2+\sigmat^2+|\betat|^2+|\alphat|^2\dm{\gb{r}}\\+\int_{\Sigma_r}|\alphab|^2+|\betab|^2+\rho^2+\sigma^2+|\beta|^2+|\alpha|^2\dm{\gb{r}}
\end{multline}
Finally by the main results of Proposition~\ref{prop:conclusion:energy} 
\begin{gather*}
  \int_{\Sigma_r}|\alphabt|^2+|\betabt|^2+\rhot^2+\sigmat^2+|\betat|^2+|\alphat|^2\dm{\gb{r}}\lesssim  r^3\int_{\Sigma_r}\ld P^q[\MLie{N}W]\lesssim \frac{1}{r^{3}}\\
  \int_{\Sigma_r}|\alphab|^2+|\betab|^2+\rho^2+\sigma^2+|\beta|^2+|\alpha|^2\dm{\gb{r}}\lesssim  r^3\int_{\Sigma_r}\ld P^q[W]\lesssim \frac{1}{r^{3}}\\
\end{gather*}
which implies
\begin{equation}
  \nLq{\theta}\lesssim \frac{1}{r^{3}}\,.
\end{equation}

This concludes the proof of the theorem.

%%% Local Variables:
%%% mode: latex
%%% TeX-master: "weyl"
%%% End:

%\input{sobolev}

\appendix
\section{Conformal geometry}
\label{sec:bach}

In this appendix we recall a relationship between the Weyl curvature of the ambient spacetime $(\mathcal{M},g)$ and the intrinsic conformal geometry of $\Sigma_r$. More precisely we recall that the \emph{Bach tensor} --- which provides a local characterisation of conformal flatness --- of $\Sigma_r$ is related to the magnetic part of the Weyl curvature. The fact that  in our main theorem \emph{all components} of the Weyl curvature are on equal footing --- in particular that the magnetic part decays at the \emph{same rate} as the electric part indicates that in general $\Sigma_r$ is \emph{not} conformally flat.\footnote{In fact, we expect that \emph{future null infinity} is intrinsically not conformally flat. This statement is stronger, and relates to the precise decay rate of the magnetic part, and can only be proven in the context of a suitable conformal completion, which we do not discuss here. }

%In this Section we are interested in the intrinsic conformal geometry of hypersurfaces $\Sigma_r$, and in particular of future null infinity $\Scri$.

Recall the \emph{Bach tensor} of a $3$-dimensional Riemannian manifold $(\Sigma,\gb{})$:
\begin{equation}\label{eq:bach}
  B=\curl \hat{S}
\end{equation}
here $\hat{S}$ is the trace-free part of the Ricci curvature $S$ of $\Sigma$,
\begin{equation}
  \hat{S}_{ij}=S_{ij}-\frac{1}{3}{\gb{}}_{ij}\tr S
\end{equation}
and $\curl$ denotes the \emph{symmetric curl}.\footnote{In general if $\theta$ is a symmetric 2-covariant tensorfield in a 3-dimensional manifold $(\Sigma,\gb{})$ the symmetric curl of $\theta$ is defined by:\[    (\curl{\theta})_{ij}=\frac{1}{2}\Bigl(\epsilon_i^{\phantom{i}ab}\nablab_a\theta_{bj}+\epsilon_j^{\phantom{j}ab}\nablab_a\theta_{bi}\Bigr) \]  where $\nablab$  is the covariant derivative of $\gb{}$ and $\epsilon$ is the volume form of $\gb{}$, and indices are raised with respect to $\gb{}$. }

% \todo{Clean up this sketch of a remark:}
% \begin{remark}
%   More commonly referred to as Cotton tensor, defined as antisymmetric derivative of Schouten tensor, see eg Fefferman-Graham paper.
%   However, antisymmetric part of this is trivial by the Bianchi equations, and symmetric part coincides with symmetric curl of the trace-free part of the Ricci curvature, see Demetri's lecture notes, in remarks on Bach's theorem.
% \end{remark}

  \emph{The fundamental property of $B$ is that its vanishing locally identifies conformal flatness.}
  %\todo{State Bach's theorem. See Eisenhart `Riemannian geometry'. Valid for cylinder?}
 Here we will prove a formula that relates $B$ to the magnetic part of the ambient Weyl curvature $W$ relative to $\Sigma_r$:

\begin{proposition}\label{prop:B:H}
  Let $B$ be the Bach tensor of $(\Sigma_r,\gb{r})$,
  and $(\Sigma_r,\gb{r},k)$ be a hypersurface in $(M,g)$, which satisfies the Einstein equations $\Ric(g)=\Lambda g$.
  Then $B$ is given in terms of $H$ by
  \begin{equation}
    B=-\widehat{\mathcal{L}_nH}+\nablab\log\phi\wedge E
    -\frac{1}{3}\nablab \tr k\wedge \hat{k} - \frac{1}{2} \hat{k}\wedge \nablab\hat{k} -  \hat{k}\times H 
  \end{equation}
\end{proposition}

% \todo{Do calculations related to conformal changes:}
% \begin{remark}
%   This is the Bach tensor of $\Sigma_r\subset(\mathcal{M},g)$; the relevant quantity however might be the Bach tensor of $\Sigma^+\subset (\mathcal{M},\Omega^2g)$.
%   Ashtekar computes the Cotton tensor of $\mathcal{I}$ as a surface in the conformal completion of the spacetime; he uses coordinates where the extrinsic curvature (with respect to the conformal metric) vanishes; see Ashtekar paper I pages 7-9.
%   Understand this gauge choice! See if you can prove the theorem that $B=0$ iff $H=0$, where $H$ is now the magnetic part of the \emph{rescaled} Weyl curvature; see in particular formulas (2.5), and (2.18) in Ashtekar, Bonga, Kesavan.
  
% \end{remark}

We recall the Gauss-Codazzi equations of the embedding of $\Sigma_r\subset \mathcal{M}$,
and express the curvature that appears on the right hand sides of \eqref{eq:codazzi:ij} and \eqref{eq:gauss:sigma} in terms of the electric and magnetic parts $E$, $H$ of the ambient Weyl curvature $W$:
\begin{gather}
  \nablab_i k_{jm}-\nablab_j k_{im}=R_{m0ij}=-\epsilon_{ij}^{\phantom{ij}k}H_{km}\label{eq:codazzi:H}\\
  \overline{\Ric}_{ij}+\tr k k_{ij}-k_{i}^{\phantom{i}m}k_{mj}=\Ric_{ij}+R_{0i0j}=E_{ij}+\frac{2}{3}\Lambda g_{ij}\\
    \Rb+(\tr k)^2-\lvert k\rvert^2=2\Lambda
\end{gather}

We can now relate the Bach tensor $B$ of $\Sigma_r$ to the Weyl curvature $W$ of the ambient spacetime; here we always denote \[S:=S_r:=\overline{\Ric}:=\Ric(\gb{r})\]

\begin{lemma}
  The traceless part of the Ricci curvature of $\Sigma_r\subset\mathcal{M}$ is given by
  \begin{equation}
    \hat{S}=E+\frac{1}{2}\hat{k}\times\hat{k}-\frac{1}{3}\tr k \hat{k}
  \end{equation}
\end{lemma}

\begin{proof}
We have already seen in \eqref{eq:gauss:sigma} that the Codazzi equations read
\begin{equation}
    (S_r)_{ij}+\tr k k_{ij}-k_{i}^{\phantom{i}m}k_{mj}=\Ric_{ij}(g)+R_{0i0j}=\Lambda (\gb{r})_{ij}+R_{0i0j}\\
\end{equation}
and
\begin{equation}
    \tr_{\gb{r}} S=-(\tr k)^2+\lvert k\rvert^2+2\Lambda
  \end{equation}
by the Einstein equations, and thus
\begin{equation}
  \hat{S}_{ij}=\hat{k}_i^{\phantom{i}m}\hat{k}_{mj}-\frac{1}{3}|\hat{k}|^2(\gb{r})_{ij}-\frac{1}{3}\tr k \hat{k}_{ij}+\frac{1}{3}\Lambda (\gb{r})_{ij}+R_{0i0j}
\end{equation}
We have also seen in \eqref{def:em} that the electric part of the Weyl curvature relates to the relevant component of the Riemann curvature in view of \eqref{eq:weyl:lambda} as follows:
\begin{equation}
  E_{ij}=W_{0i0j}=R_{0i0j}+\frac{\Lambda}{3}(\gb{r})_{ij}
\end{equation}
Therefore
\begin{equation}
  \hat{S}_{ij}=E_{ij}+\hat{k}_i^{\phantom{i}m}\hat{k}_{mj}-\frac{1}{3}|\hat{k}|^2(\gb{r})_{ij}-\frac{1}{3}\tr k \hat{k}_{ij}
\end{equation}
which yields the formula of the Lemma, in view of the fact
\begin{equation}
  (\hat{k}\times \hat{k})_{ij} = 2\hat{k}_i^{\phantom{i}m}\hat{k}_{mj}-\frac{2}{3}\hat{k}\cdot\hat{k}g_{ij}
\end{equation}
\end{proof}

For the curl of $\hat{S}$ we then have
\begin{equation}
  \curl \hat{S}=\curl E+\frac{1}{2}\curl \hat{k}\times\hat{k}-\frac{1}{6}\nablab \tr k\wedge \hat{k}-\frac{1}{3}\tr k \curl \hat{k}
\end{equation}
and it remains to calculate the curl of
\begin{equation}
  \bigl(\hat{k}\times\hat{k}\bigr)_{ij}=\epsilon_i^{\phantom{i}ab}\epsilon_j^{\phantom{j}cd}\hat{k}_{ac}\hat{k}_{bd}+\frac{1}{3}|\hat{k}|^2 g_{ij}
\end{equation}
Since
\begin{equation}
  \epsilon_i^{\phantom{i}ab}\epsilon_b^{\phantom{b}mn}=\delta^m_i\delta^{an}-\delta^n_i\delta^{am}
\end{equation}
we obtain
\begin{equation}
  \begin{split}
    (\curl \hat{k}\times\hat{k})_{ij}&=\epsilon_i^{\phantom{i}cd}\nabla^a(\hat{k}_{ad}\hat{k}_{jc}\bigr)+\epsilon_j^{\phantom{i}cd}\nabla^a(\hat{k}_{ad}\hat{k}_{ic}\bigr)\\&=-\bigl(\divergence \hat{k}\wedge \hat{k}\bigr)_{ij}+\epsilon_i^{\phantom{i}cd}\hat{k}_{ad}\nabla^a\hat{k}_{jc}+\epsilon_j^{\phantom{i}cd}\hat{k}_{ad}\nabla^a\hat{k}_{ic}
  \end{split}
\end{equation}
Note that the Codazzi equation \eqref{eq:codazzi:H} implies that
\begin{subequations}
\begin{gather}
  \epsilon_{a\phantom{j}b}^{\phantom{a}j}\nablab_i k_{jm}=\epsilon_{a\phantom{j}b}^{\phantom{a}j}\nablab_j k_{im}-\delta_{ai}H_{bm}+\delta_{bi}H_{am}\\
  (\divergence k)_j-\nablab_j\tr k=0
\end{gather}
\end{subequations}
or alternatively for the trace-free parts
\begin{subequations}
\begin{gather}
  \epsilon_{a\phantom{j}b}^{\phantom{a}j}\nablab_i \hat{k}_{jm}=-\frac{1}{3}\epsilon_{amb}\nablab_i \tr k+\epsilon_{a\phantom{j}b}^{\phantom{a}j}\nablab_j \hat{k}_{im}+\frac{1}{3}\epsilon_{a\phantom{j}b}^{\phantom{a}j}g_{im}\nablab_j \tr k-\delta_{ai}H_{bm}+\delta_{bi}H_{am}\\
  (\divergence \hat{k})_j=\frac{2}{3}\nablab_j\tr k
\end{gather}
\end{subequations}
which then implies that
\begin{multline}
  \epsilon_i^{\phantom{i}cd}\hat{k}_{ad}\nablab^a\hat{k}_{jc}+\epsilon_j^{\phantom{i}cd}\hat{k}_{ad}\nablab^a\hat{k}_{ic}=\\
  =\epsilon_{i}^{\phantom{i}cd}\hat{k}_{ad}\nablab_c \hat{k}^a_{\phantom{a}j}+\epsilon_{j}^{\phantom{i}cd}\hat{k}_{ad}\nablab_c \hat{k}^a_{\phantom{a}i}+\frac{1}{3}\epsilon_{i}^{\phantom{i}cd}\hat{k}_{jd}\nablab_c \tr k+\frac{1}{3}\epsilon_{j}^{\phantom{i}cd}\hat{k}_{id}\nablab_c \tr k-\hat{k}_{id}H_{\phantom{d}j}^d-\hat{k}_{jd}H_{\phantom{d}i}^d\\
  =- \bigl(\hat{k}\wedge\nablab\hat{k}\bigr)_{ij}+\frac{2}{3}\hat{k}\cdot H g_{ij}+\frac{1}{3}\bigl(\nablab \tr k\wedge \hat{k}\bigr)_{ij}-(\hat{k}\times H)_{ij}-\frac{2}{3}(\hat{k}\cdot H) g_{ij}
\end{multline}
where we introduced the notation:
%\todo{This is not introduced in Ch-Kl; a better way to express this structure?}
\begin{equation}
  (\hat{k}\wedge \nablab\hat{k})_{ij}=\epsilon_{i}^{\phantom{i}cd}\hat{k}_{ca}\nablab_d \hat{k}^a_{\phantom{a}j}+\epsilon_{j}^{\phantom{i}cd}\hat{k}_{ca}\nablab_d \hat{k}^a_{\phantom{a}i}-\frac{2}{3}\Bigl(\epsilon^{mcd}\hat{k}_{ca}\nablab_d \hat{k}^a_{\phantom{a}m}\Bigr)g_{ij}
\end{equation}
and observed that by the Codazzi equations
\begin{gather}
  \epsilon^{jcd}\hat{k}_{ca}\nablab_d \hat{k}^a_{\phantom{a}j}=-\epsilon^{jcd}\hat{k}_{ca}\nablab_d \hat{k}^a_{\phantom{a}j}+\epsilon^{jcd}\hat{k}_{ca}\epsilon_{jd}^{\phantom{jd}m}H_{m}^{\phantom{m}a}\notag\\
    \epsilon^{jcd}\hat{k}_{ca}\nablab_d \hat{k}^a_{\phantom{a}j} = -\frac{2}{2}\delta^{cm}\hat{k}_{ca}H_m^{\phantom{m}a}=-\hat{k}\cdot H
\end{gather}

In conclusion,
\begin{equation}
  \curl \hat{k}\times\hat{k} = -\frac{1}{3} \nablab \tr k\wedge \hat{k} - \hat{k}\wedge \nablab\hat{k} -\hat{k}\times H
\end{equation}
and we have proven that
\begin{equation}\label{eq:curl:S:pre}
  \begin{split}
    \curl \hat{S}&=\curl E+\frac{1}{2}\curl \hat{k}\times\hat{k}-\frac{1}{6}\nablab \tr k\wedge \hat{k}-\frac{1}{3}\tr k \curl \hat{k}\\
    &=\curl E -\frac{1}{3}\nablab \tr k\wedge \hat{k} - \frac{1}{2} \hat{k}\wedge \nablab\hat{k} - \frac{1}{2} \hat{k}\times H +\frac{1}{3}\tr k H
  \end{split}
\end{equation}
because again by \eqref{eq:codazzi:H},
\begin{equation}
  \curl k=-H
\end{equation}

The proof of Prop.~\ref{prop:B:H} is now immediate 
in view of the electromagnetic decomposition of the Bianchi equations \eqref{eq:maxwell}, in particular \eqref{eq:curl:E},
  \begin{equation*}
      \curl E=-\widehat{\mathcal{L}_nH}+\nablab\log\phi\wedge E-\frac{1}{2}k\times H
  \end{equation*}
  Indeed, it follows from \eqref{eq:curl:S:pre} that
  \begin{equation*}
    B=-\widehat{\mathcal{L}_nH}+\nablab\log\phi\wedge E
    -\frac{1}{3}\nablab \tr k\wedge \hat{k} - \frac{1}{2} \hat{k}\wedge \nablab\hat{k} -  \hat{k}\times H \,.
  \end{equation*}

%%% Local Variables:
%%% mode: latex
%%% TeX-master: "weyl"
%%% End:

\section{Assumptions}
\label{sec:BA}

We list here in three tables the assumptions on the metric and connection coefficients that are made in this paper under the label that they are referred to along with the page numbers where they are first introduced and further discussed in the text. Here
\begin{equation}
  q=\sqrt{\frac{\overline{\Omega\tr\chi}}{\overline{\Omega\tr\chib}}}
\end{equation}
and $C_0>0$ is a fixed constant.

\begin{center}
\begin{tabular}[c]{r @{.} l|c|c}
  \multicolumn{2}{c|}{Label} & Assumption & Page\\
  \hline
  \textbf{BA:I}&i & $\tr\chi>0\quad \tr\chib>0$ & \pageref{eq:assumption:intro:tr}\\
  \textbf{BA:I}&ii & $\Omega \lvert 2\omegah - \tr\chi \rvert \leq C_0 \tr\chi$ & \pageref{eq:assumption:intro:omega}\\
  && $ \Omega \lvert 2\omegabh - \tr\chib \rvert \leq C_0 \tr\chib$ & \\
  \textbf{BA:I}&iii & $  \lvert \Omega \tr\chi - \overline{ \Omega\tr\chi } \rvert \leq  C_0\Omega^{-1} \overline{\Omega \tr\chi }$   & \pageref{eq:assummption:intro:average}\\
  && $   \lvert \Omega \tr\chib - \overline{ \Omega\tr\chib } \rvert \leq  C_0 \Omega^{-1} \overline{\Omega \tr\chib } $ & \\
  \textbf{BA:I}&iv & $\Omega \lvert \chih \rvert_{\gs} \leq C_0\qquad \Omega \lvert \chibh \rvert \leq C_0$ & \pageref{eq:assumption:intro:zeta} \\
\textbf{BA:I}&iv${}^\prime$ &  $\Omega^2 \lvert  \chibh \rvert \leq C_0 \tr\chib \qquad \Omega^2 \lvert \chih \rvert \leq C_0 \tr\chi$ & \pageref{eq:BA:I:iv:e}\\
                             && $\Omega \lvert \zeta \rvert \leq C_0$ & \\
  \textbf{BA:I}&v & $\lvert D\log q \rvert \leq C_0 \tr\chi$ & \pageref{eq:BA:I:v}\\
                             && $\lvert \Db \log q\rvert \leq C_0\tr\chib$ &\\
  \textbf{BA:I}&vi & $\Omega \lvert \eta \rvert + \Omega \lvert \etab \rvert\leq C_0 (q\tr\chib+q^{-1}\tr\chi)$ & \pageref{eq:BA:I:v}\\
                             && $\Omega\lvert \ds\log q\rvert \leq C_0 (q\tr\chib+q^{-1}\tr\chi)$ &\\
  \textbf{BA:I}&vi${}^\prime$ &  $  \Omega^2 \lvert \eta \rvert + \Omega^2 \lvert \etab \rvert + \Omega^2 \lvert \ds\log q\rvert\leq C_0   (q\tr\chib+q^{-1}\tr\chi)$ & \pageref{eq:BA:I:vi:e}\\
  \textbf{BA:I}&vii & $q\tr\chib\leq C_0\qquad q^{-1}\tr\chi\leq C_0$ & \pageref{eq:BA:I:vii}\\
  \textbf{BA:I}&vii & $\Omega  \lvert q\tr\chib -q^{-1}\tr\chi \rvert\leq C_0  \bigl( q\tr\chib+q^{-1}\tr\chi\bigr)$ & \pageref{eq:BA:I:vii}
\end{tabular}
\end{center}

\begin{center}
\begin{tabular}[c]{r @{.} l|c|c}
  \multicolumn{2}{c|}{Label} & Assumption & Page\\
  \hline
  \textbf{BA:II}&i& $  \lvert D (\Omega\chibh) \rvert \leq C_0 \tr\chi\tr\chib \qquad   \lvert \Db (\Omega\chih) \rvert \leq C_0 \tr\chi\tr\chib $ & \pageref{eq:BA:II:i}\\
                             &&  $\lvert D (\Omega\chih) \rvert \leq C_0 \tr\chi\tr\chi \qquad \lvert \Db (\Omega\chibh) \rvert \leq C_0 \tr\chib\tr\chib$ & \\
  \textbf{BA:II}&ii& $\Omega \lvert \nablas (\Omega\chibh) \rvert \leq C_0\tr\chib\qquad \Omega \lvert \nablas (\Omega\chih) \rvert \leq C_0 \tr\chi$ & \pageref{eq:BA:II:i} \\
  \textbf{BA:II}&iii & $\lvert D\bigl(\Omega\tr\chib-2\omegab\bigr) \rvert \leq C_0 \tr\chi\tr\chib$ & \pageref{eq:BA:II:remaining} \\ & & $\lvert \Db\bigl(\Omega\tr\chi-2\omega\bigr) \rvert \leq C_0\tr\chib\tr\chi$                      & \\
                             &&      $\lvert D\bigl(\Omega\tr\chi-2\omega\bigr) \rvert  \leq C_0 \tr\chi \tr\chi$ & \\ && $  \lvert \Db\bigl(\Omega\tr\chib-2\omegab\bigr) \rvert \leq C_0 \tr\chib\tr\chib$ & \\
  \textbf{BA:II}&iv & $\Omega \bigl\lvert \nablas(\Omega\tr\chi-2\omega) \bigr\rvert \leq C_0\tr\chi \qquad     \Omega \bigl\lvert \nablas (\Omega\tr\chib-2\omegab) \bigr\rvert \leq C_0\tr\chib$ & \pageref{eq:BA:II:remaining}\\
      \textbf{BA:II}&v &      $\lvert D\Db\log q \rvert \leq C_0\tr\chi\tr\chib  \qquad  \lvert DD\log q \rvert \leq C_0\tr\chi\tr\chi$ & \pageref{eq:BA:II:remaining}\\ 
                             && $\lvert \Db\Db\log q \rvert \leq C_0 \tr\chib\tr\chib \qquad \lvert \Db D\log q \rvert\leq C_0 \tr\chib\tr\chi$ & \\
  \textbf{BA:II}&vi & $\lvert \Omega \nablas D \log q \rvert \leq C_0 \tr\chi \qquad \lvert \Omega \nablas \Db \log q \rvert \leq C_0 \tr\chib$ & \pageref{eq:BA:II:remaining}\\
     \textbf{BA:II}&vii &  $\lvert D(\Omega\zeta) \rvert + \lvert D(\Omega\ds\log q\bigr) \rvert\leq C_0 \tr\chi$ & \\ && $\lvert \Db (\Omega \zeta ) \rvert + \lvert \Db(\Omega\ds\log q\bigr) \rvert\leq C_0 \tr\chib$ & \pageref{eq:BA:II:remaining}\\
  \textbf{BA:II}&viii &$\Omega | \nablas (\Omega\zeta) | + \Omega |\nablas (\Omega\ds\log q) |\leq C_0 \bigl(q\tr\chib+q^{-1}\tr\chi\bigr)$ & \pageref{eq:BA:II:remaining}
\end{tabular}
\end{center}

\begin{center}
\begin{tabular}[c]{r @{.} l|c|c}
  \multicolumn{2}{c|}{Label} & Assumption & Page\\
  \hline
  \textbf{BA:III}&i& $C_0^{-1}\, r \leq \Omega \leq C_0\,r$ & \pageref{eq:assumption:intro:Omega}\\
                                     \textbf{BA:III}&ii & $\lvert \overline{\Ric} \rvert \leq \frac{C_0}{r^2}$ & \pageref{eq:BA:III:Ric} \\
  \textbf{BA:III}&iii &   $a_m=\inf_{\Sigma_r}a >0\qquad a_M=\sup_{\Sigma_r} a<\infty$ & \pageref{eq:BA:III:lapse}          \\
  \textbf{BA:III}&iv &   $I:=\sup_uI(u)<\infty$ & \pageref{eq:BA:III:iso}
\end{tabular}
\end{center}

 % \bibliographystyle{amsalpha}
 % \bibliography{weyls}

\begin{thebibliography}{ABK15b}

\bibitem[ABK15a]{ashtekar:I}
Abhay Ashtekar, B{\'e}atrice Bonga, and Aruna Kesavan, \emph{Asymptotics with a
  positive cosmological constant: {I}. {B}asic framework}, Classical Quantum
  Gravity \textbf{32} (2015), no.~2, 025004, 41. \MR{3291776}

\bibitem[ABK15b]{ashtekar:II}
\bysame, \emph{Asymptotics with a positive cosmological constant. {II}.
  {L}inear fields on de {S}itter spacetime}, Phys. Rev. D \textbf{92} (2015),
  no.~4, 044011, 14. \MR{3441014}

\bibitem[ABK16]{ashtekar:PRL}
\bysame, \emph{Gravitational waves from isolated systems: Surprising
  consequences of a positive cosmological constant}, Phys. Rev. Lett.
  \textbf{116} (2016), 051101.

\bibitem[BZ09]{bieri}
Lydia Bieri and Nina Zipser, \emph{Extensions of the stability theorem of the
  {M}inkowski space in general relativity}, AMS/IP Studies in Advanced
  Mathematics, vol.~45, American Mathematical Society, Providence, RI;
  International Press, Cambridge, MA, 2009. \MR{2531716}

\bibitem[CBIP07]{isenberg:pollack}
Yvonne Choquet-Bruhat, James Isenberg, and Daniel Pollack, \emph{The constraint
  equations for the {E}instein-scalar field system on compact manifolds},
  Classical Quantum Gravity \textbf{24} (2007), no.~4, 809--828. \MR{2297268}

\bibitem[Chr91]{demetri:notes}
Demetrios Christodoulou, \emph{Notes on the geometry of null hypersurfaces},
  1991.

\bibitem[Chr07]{ch:shocks}
\bysame, \emph{The formation of shocks in 3-dimensional fluids}, EMS Monographs
  in Mathematics, European Mathematical Society (EMS), Z\"urich, 2007.
  \MR{2284927}

\bibitem[Chr09]{ch:blue}
\bysame, \emph{The formation of black holes in general relativity}, EMS
  Monographs in Mathematics, European Mathematical Society (EMS), Z\"urich,
  2009. \MR{2488976}

\bibitem[CK90]{ch:kl:lin}
D.~Christodoulou and S.~Klainerman, \emph{Asymptotic properties of linear field
  equations in {M}inkowski space}, Comm. Pure Appl. Math. \textbf{43} (1990),
  no.~2, 137--199. \MR{1038141}

\bibitem[CK93]{ch:kl}
Demetrios Christodoulou and Sergiu Klainerman, \emph{The global nonlinear
  stability of the {M}inkowski space}, Princeton Mathematical Series, vol.~41,
  Princeton University Press, Princeton, NJ, 1993. \MR{1316662}

\bibitem[CNO19]{costa:nohair}
Jo\~{a}o~L. Costa, Jos\'{e} Nat\'{a}rio, and Pedro Oliveira, \emph{Cosmic
  no-hair in spherically symmetric black hole spacetimes}, Ann. Henri
  Poincar\'{e} \textbf{20} (2019), no.~9, 3059--3090. \MR{3995897}

\bibitem[DHR16]{dhr:linear}
Mihalis Dafermos, Gustav Holzegel, and Igor Rodnianski, \emph{The linear
  stability of the {S}chwarzschild solution to gravitational perturbations},
  arXiv:1601.06467v1, 2016.

\bibitem[DL17]{dl:C0}
Mihalis Dafermos and Jonathan Luk, \emph{The interior of dynamical vacuum black
  holes i: The c0-stability of the kerr cauchy horizon}, arXiv:1710.01722
  [gr-qc], 2017.

\bibitem[DR13]{dr:clay}
Mihalis Dafermos and Igor Rodnianski, \emph{Lectures on black holes and linear
  waves}, Evolution equations, Clay Math. Proc., vol.~17, Amer. Math. Soc.,
  Providence, RI, 2013, pp.~97--205. \MR{3098640}

\bibitem[dS17]{desitter}
Willem de~Sitter, \emph{On einstein's theory of gravitation and its
  astronomical consequences}, Monthly Notices of the Royal Astronomical Society
  \textbf{77} (1917), 155--184.

\bibitem[Dya11a]{sd:beyond}
Semyon Dyatlov, \emph{Exponential energy decay for {K}err--de {S}itter black
  holes beyond event horizons}, Math. Res. Lett. \textbf{18} (2011), no.~5,
  1023--1035. \MR{2875874}

\bibitem[Dya11b]{sd:quasi}
\bysame, \emph{Quasi-normal modes and exponential energy decay for the
  {K}err-de {S}itter black hole}, Comm. Math. Phys. \textbf{306} (2011), no.~1,
  119--163. \MR{2819421 (2012g:58055)}

\bibitem[Ein17]{einstein:kosmos}
Albert Einstein, \emph{Kosmologische {B}etrachtungen zur allgemeinen
  {R}elativit\"atstheorie}, Sitzungsberichte der Preu{\ss}ischen Akad. d.
  Wissenschaften (1917), 142--152,
  http://einsteinpapers.press.princeton.edu/vol6-doc/568.

\bibitem[FG85]{fefferman:conformal}
Charles Fefferman and C.~Robin Graham, \emph{Conformal invariants}, no.
  Num\'{e}ro Hors S\'{e}rie, 1985, The mathematical heritage of \'{E}lie Cartan
  (Lyon, 1984), pp.~95--116. \MR{837196}

\bibitem[Fri86]{friedrich:desitter}
Helmut Friedrich, \emph{On the existence of {$n$}-geodesically complete or
  future complete solutions of {E}instein's field equations with smooth
  asymptotic structure}, Comm. Math. Phys. \textbf{107} (1986), no.~4,
  587--609. \MR{868737}

\bibitem[GH77]{gibbons:hawking}
Gary~W. Gibbons and Stephen~W. Hawking, \emph{Cosmological event horizons,
  thermodynamics, and particle creation}, Phys. Rev. D \textbf{15} (1977),
  no.~10, 2738--2751.

\bibitem[GVK17]{valiente}
Edgar Gasper\'{\i}n and Juan~A. Valiente~Kroon, \emph{Perturbations of the
  asymptotic region of the {S}chwarzschild--de {S}itter spacetime}, Ann. Henri
  Poincar\'{e} \textbf{18} (2017), no.~5, 1519--1591. \MR{3635961}

\bibitem[Hin16]{hi:quasi}
Peter Hintz, \emph{Global analysis of quasilinear wave equations on
  asymptotically de {S}itter spaces}, Ann. Inst. Fourier (Grenoble) \textbf{66}
  (2016), no.~4, 1285--1408. \MR{3494174}

\bibitem[HS15]{had:speck:FLRW:dust}
Mahir Had{\v{z}}i{\'c} and Jared Speck, \emph{The global future stability of
  the {FLRW} solutions to the dust-{E}instein system with a positive
  cosmological constant}, J. Hyperbolic Differ. Equ. \textbf{12} (2015), no.~1,
  87--188. \MR{3335528}

\bibitem[HV14]{hi:vasy:trapping}
Peter Hintz and Andr{\'a}s Vasy, \emph{Non-trapping estimates near normally
  hyperbolic trapping}, Math. Res. Lett. \textbf{21} (2014), no.~6, 1277--1304.
  \MR{3335848}

\bibitem[HV16]{hintz:vasy:16:global}
\bysame, \emph{Global analysis of quasilinear wave equations on asymptotically
  kerr-de sitter spaces}, International Mathematics Research Notices
  \textbf{2016} (2016), no.~17, 5355--5426.

\bibitem[HV18]{hi:vasy:stability}
Peter Hintz and Andr\'{a}s Vasy, \emph{The global non-linear stability of the
  {K}err--de {S}itter family of black holes}, Acta Math. \textbf{220} (2018),
  no.~1, 1--206. \MR{3816427}

\bibitem[Kot18]{kottler:sds}
F.~Kottler, \emph{{\"U}ber die physikalischen {G}rundlagen der {E}insteinschen
  {G}ravitationstheorie}, Ann. Phys. \textbf{56} (1918), 401--462.

\bibitem[KS19]{KS:GCM}
Sergiu Klainerman and Jeremie Szeftel, \emph{Effective results on
  uniformization and intrinsic gcm spheres in perturbations of kerr}, 2019.

\bibitem[Lem27]{lemaitre}
Georges Lema{\^\i}tre, \emph{Un univers homog\`ene de masse constante et de
  rayon croissant, redant compte de la vitesse radiale des n\'ebuleuses
  extra-galactiques}, Annales de la Soci\'et\'e scientifique de Bruxelles
  \textbf{47} (1927), no.~A, 49--59.

\bibitem[NB09]{nussbaumer}
Harry Nussbaumer and Lydia Bieri, \emph{Discovering the expanding universe},
  Cambridge University Press, 2009.

\bibitem[Ren04]{rendall:asymptotics}
Alan~D. Rendall, \emph{Asymptotics of solutions of the {E}instein equations
  with positive cosmological constant}, Ann. Henri Poincar\'e \textbf{5}
  (2004), no.~6, 1041--1064. \MR{2105316}

\bibitem[Rin08]{ringstrom:invent}
Hans Ringstr{\"o}m, \emph{Future stability of the {E}instein-non-linear scalar
  field system}, Invent. Math. \textbf{173} (2008), no.~1, 123--208.
  \MR{2403395}

\bibitem[RS13]{rod:speck:FLRW:euler}
Igor Rodnianski and Jared Speck, \emph{The nonlinear future stability of the
  {FLRW} family of solutions to the irrotational {E}uler-{E}instein system with
  a positive cosmological constant}, J. Eur. Math. Soc. (JEMS) \textbf{15}
  (2013), no.~6, 2369--2462. \MR{3120746}

\bibitem[RS14a]{rod:speck:bang:linear}
\bysame, \emph{A regime of linear stability for the {E}instein-scalar field
  system with applications to nonlinear big bang formation}, arXiv:1407.6293,
  2014.

\bibitem[RS14b]{rod:speck:bang}
\bysame, \emph{Stable big bang formation in near-{FLRW} solutions to the
  {E}instein-scalar field and {E}instein-stiff fluid systems}, arXiv:1407.6298,
  2014.

\bibitem[Sch14]{schlue:gr}
Volker Schlue, \emph{General relativity}, Lecture notes, available at
  \verb+https://blogs.unimelb.edu.au/volker-schlue/home-page/teaching/+, 2014.

\bibitem[Sch15]{glw}
\bysame, \emph{Global results for linear waves on expanding {K}err and
  {S}chwarzschild de {S}itter cosmologies}, Comm. Math. Phys. \textbf{334}
  (2015), no.~2, 977--1023. \MR{3306609}

\bibitem[Sch19]{schlue:optical}
\bysame, \emph{Optical functions in de sitter}, arXiv:1910.05799 [math.AP],
  2019.

\bibitem[Spe12]{speck:FLRW:euler}
Jared Speck, \emph{The nonlinear future stability of the {FLRW} family of
  solutions to the {E}uler-{E}instein system with a positive cosmological
  constant}, Selecta Math. (N.S.) \textbf{18} (2012), no.~3, 633--715.
  \MR{2960029}

\bibitem[Spe13]{speck:stabilize}
\bysame, \emph{The stabilizing effect of spacetime expansion on relativistic
  fluids with sharp results for the radiation equation of state}, Arch. Ration.
  Mech. Anal. \textbf{210} (2013), no.~2, 535--579. \MR{3101792}

\bibitem[Vas13]{vasy:13}
Andr{\'a}s Vasy, \emph{Microlocal analysis of asymptotically hyperbolic and
  {K}err-de {S}itter spaces (with an appendix by {S}emyon {D}yatlov)}, Invent.
  Math. \textbf{194} (2013), no.~2, 381--513. \MR{3117526}

\bibitem[Wey19]{weyl:sds}
Hermann Weyl, \emph{{\"U}ber die statischen kugelsymmetrischen {L}\"osungen von
  {E}insteins kosmologischen {G}ravitationsgleichungen}, Phys. Z. \textbf{20}
  (1919), 31--34.

\end{thebibliography}

\providecommand{\bysame}{\leavevmode\hbox to3em{\hrulefill}\thinspace}
\providecommand{\MR}{\relax\ifhmode\unskip\space\fi MR }
\providecommand{\MRhref}[2]{%
  \href{http://www.ams.org/mathscinet-getitem?mr=#1}{#2}
}
\providecommand{\href}[2]{#2}

\end{document}